\documentclass[a4paper, 12pt]{amsart} 

\usepackage[pdftex]{graphicx} 

\usepackage{amssymb}
\usepackage{amsmath}
\usepackage{amsthm}
\usepackage{amscd}

 \usepackage{amssymb}

\usepackage{amssymb}
\usepackage{amsmath}
\usepackage{amsthm}
\usepackage{amscd}

\usepackage{comment}
\usepackage[all]{xy}
\usepackage{color}

 \usepackage{multirow}  
 \usepackage{hhline}
 \usepackage{ifthen}
 \usepackage{amscd}

\usepackage[utf8]{inputenc}

\usepackage{tikz}
\usetikzlibrary{positioning}
\usetikzlibrary{intersections}
\usetikzlibrary{calc, quotes, angles}

 \theoremstyle{plain}
  \newtheorem{thm}{Theorem}[section]
  \newtheorem{lem}[thm]{Lemma}
  \newtheorem{slem}[thm]{Sublemma}
  \newtheorem{cor}[thm]{Corollary}
  \newtheorem{prop}[thm]{Proposition}
  \newtheorem{clm}[thm]{Claim}

  \newtheorem{ass}[thm]{Assertion}
  \newtheorem{obs}[thm]{Observation}

\theoremstyle{definition}
  \newtheorem{defn}[thm]{Definition}
  \newtheorem{conj}[thm]{Conjecture}
  \newtheorem{ex}[thm]{Example}
  
  \newtheorem{prob}[thm]{Problem}
  \newtheorem{asmp}[thm]{Assumption}

\theoremstyle{remark}
  \newtheorem{rem}[thm]{Remark}

  \newtheorem{ack}{Acknowledgment}

\renewcommand{\labelenumi}{(\theenumi)}

\numberwithin{equation}{section}

\newenvironment{proclaim}[1]
  {\par \vspace{1.5ex} \noindent #1 \ }
  {\par \vspace{1.5ex}}

\DeclareMathOperator{\diam}{diam}               
\DeclareMathOperator{\interior}{int}            
\DeclareMathOperator{\capp}{cap}                 
\DeclareMathOperator{\rad}{rad}                 
\DeclareMathOperator{\grad}{grad}               
\DeclareMathOperator{\Ext}{Ext}                 

\newcommand{\bdy}{\partial}                     

\newcommand{\field}[1]{\mathbb{#1}}

\newcommand{\C}{\field{C}}                      
\newcommand{\R}{\field{R}}                      
\newcommand{\Z}{\field{Z}}                      

\DeclareMathOperator{\isom}{Isom}

\DeclareMathOperator{\CCap}{Cap}

\DeclareMathOperator{\Mo}{\mbox{M\"o}}

\DeclareMathOperator{\sys_1}{\mbox{sys}_1}

\newcommand{\cal}{\mathcal}
\renewcommand{\setminus}{-}

\newcommand{\pa}[0]{\partial}

\newcommand{\pmed}[0]{\par\medskip}
\newcommand{\pbig}[0]{\par\bigskip}

\newcommand{\beq}{\begin{equation}}
\newcommand{\eeq}{\end{equation}}

\begin{document}

                \setcounter{section}{-1}

\title{Collapsing 4-manifolds\\
  under a lower curvature bound\\
}

\author{Takao Yamaguchi}

\address{Faculty of Mathematics, Kyushu University, Fukuoka
  812-8581, JAPAN}

\curraddr{Institute of Mathematics, University of Tsukuba, Tsukuba
 \hspace*{0.4cm}  305-8571, JAPAN}

\email{takao@math.tsukuba.ac.jp}


\footnote{
This is a slightly modified version of arXiv:1205.0323, which was first distributed in June 2002.}

\date{\today}

\subjclass{Primary 53C20, 53C23; Secondary 57N10, 57M99}

\keywords{collapsing, Gromov-Hausdorff convergence, Alexandrov
  spaces, topology of 4-manifolds }

\begin{abstract}
  In this paper we describe the topology
  of 4-dimensional closed orientable Riemannian manifolds 
  with a uniform lower bound
  of sectional curvature and with a uniform upper bound of diameter
  which collapse to metric spaces of lower dimensions.
  This enables us to understand the set of homeomorphism classes of
  closed orientable $4$-manifolds with those
  geometric bounds on curvature and diameter.
  In the course of the proof of the above results,
  we obtain the soul theorem for 4-dimensional complete
  noncompact Alexandrov spaces with nonnegative curvature.
  A metric classification for 3-dimensional complete
  Alexandrov spaces with nonnegative curvature is also
  given.
\end{abstract}

\maketitle

                  \tableofcontents


\section{Introduction} \label{sec:intro}

The study of the Gromov-Hausdorff convergence of 
Riemannian manifolds has been a significant branch 
in differential geometry. In the study of convergence or collapsing
of Riemannian manifolds,  one usually considers  
a curvature bound of Riemannian manifolds.
Our main concern is the study of the collapsing phenomena
of Riemannian manifolds under a uniform lower bound on the 
sectional curvature $K$. 
When the absolute value $|K|$ is uniformly bounded, Cheeger, Fukaya and 
Gromov \cite{CFG} developed a general theory of collapsing,
where the collapsing phenomena were described in terms of the
generalized group actions by nilpotent groups,
called $N$-structures.
It should be noted that these actions are not permitted to have 
fixed points.
Now if we turn attention to the case  when 
$K$ has only a uniform lower bound,
we recognize several kinds of different collapsing 
phenomena in this new situation:
\begin{itemize}
 \item  {\it Any} effective action on a compact manifold by 
    a compact connected Lie group of positive dimension causes a collapsing 
    under a lower curvature bound (\cite{SY:3mfd});
 \item There is a phenomenon of curvature-concentration, for instance  
    in the convergence of surfaces
    to a singular surface with conical singularities, 
    under a lower curvature bound.
\end{itemize}
Thus the study of collapsing of Riemannian manifolds
under a lower curvature bound will enable us to 
understand a wider class of collapsing phenomena
than the case of absolute bound on $K$. 

In spite of several studies of collapsing 
or convergence under a lower sectional curvature bound,
the general collapsing structure is still unclear.
Recently in \cite{SY:3mfd}, we have classified all the collapsing 
phenomena of $3$-manifolds under a lower curvature bound.
In the present paper, we discuss and determine the collapsing of 
$4$-manifolds under a lower curvature bound.
Note that $S^4$, $\pm\C P^2$, $S^2\times S^2$ and arbitrary 
connected sum of them all admit nontrivial $S^1$-actions,
and therefore can collapse under $K\ge -1$ while
no simply connected closed $4$-manifold can
collapse under $|K|\le 1$ because of nonzero Euler 
characteristics (\cite{Ch:finite}, \cite{CG:collapseI}).

For a given $D>0$, let $\cal M(4,D)$ denote the family of closed 
orientable Riemannian $4$-manifolds $M$ with 
a lower sectional curvature bound $K\ge -1$ and an upper diameter 
bound $\diam (M)\le D$. Since $\cal M(4,D)$ is precompact(\cite{GLP}),
it is quite natural to consider an infinite sequence 
$M_i^4$ in $\cal M(4,D)$ converging to a compact metric space $X$
with respect to the Gromov-Hausdorff distance.
Now it is well-known that $X$ has the structure of Alexandrov 
space with curvature $\ge -1$ and of dimension $\le 4$.
It is a quite important problem to reproduce the topology of $M_i^4$
for sufficiently large $i$  by using the information on the singularities
of $X$. 
This enables us to understand the elements of a small neighborhood
of $X$ in $\cal M(4,D)$ and therefore to describe the structure of 
$\cal M(4,D)$ itself because of the precompactness.

In the case of $\dim X=4$ or $\dim X=0$, we know some
answers for this problem ( \cite{Pr:alex2}, \cite{FY:fundgp}).
In this paper, we discuss the cases of $1\le \dim X\le 3$.
Our results can be simply stated as follows:

\begin{thm} \label{thm:general}
Suppose $1\le \dim X\le 3$. Then $M_i^4$ has a singular fibre 
structure over $X$ in a generalized sense.
\end{thm}

Here the fibre type is constant along each strata of a stratification 
of $X$, but may change when the strata changes.
In particular, $M_i^4$ is a fibre bundle over $X$ if $X$ has 
no {\it essential singular points}.
 
To state our results more precisely and more explicitly, let us describe
the collapsing structure in each case of the dimension of $X$.

We begin with the case of $\dim X=3$.

\begin{thm} \label{thm:dim3}
Suppose that $M_i^4$ collapses to a three-dimensional  compact Alexandrov
space $X^3$ under $K\ge -1$. 
Then there exists a locally smooth, local $S^1$-action 
$\psi_i$ on $M_i^4$ such that the orbit space 
$M_i^4/\psi_i$ is homeomorphic to $X^3$.
\end{thm}

\begin{rem} \label{rem:dim3infom}
In Theorem \ref{thm:dim3}, one can obtain the information on the 
slice representations of the isotropy groups at the fixed points or 
the exceptional orbits of $\psi_i$ from the information on the 
singularities of the corresponding singular points of $X$ as described
below.  See Theorems \ref{thm:patch} and 
\ref{thm:patchwbdy} for details:
\begin{enumerate}
 \item  We have a bijection $\iota_i$ of the fixed point set 
   $F(\psi_i)$ to the union of $\partial X^3$ and a subset of 
   $\Ext(\interior X^3)$, where $\Ext(\interior X^3)$, the set of 
   {\it extremal points} of $\interior X^3$, denotes the set of points 
   $p$ of $\interior X^3$ whose spaces of directions, $\Sigma_p$,
   have diameters $\le\pi/2$.
 \item The singular locus $S(\psi_i)\subset X$ of $\psi_i$
   is a quasigeodesic consisting of essential singular points of $X^3$.
\end{enumerate}
\end{rem}

\begin{rem} \label{rem:dim3mfd}
In a course of the proof of Theorem \ref{thm:dim3}, 
we prove that the limit space $X^3$ is a topological manifold
(Proposition \ref{prop:topmfd}).
\end{rem}

 In \cite{Fn:circleI}, Fintushel classified locally smooth 
$S^1$-actions on simply connected
closed $4$-manifolds. Applying this result to Theorem \ref{thm:dim3},
we obtain 

\begin{cor} \label{cor:sum}
Under the hypothesis of Theorem \ref{thm:dim3}, suppose that 
$M_i^4$ is simply connected. Then 
$M_i^4$ is homeomorphic to the connected sum 
  \begin{equation}
      S^4\,\#\, k_i\C P^2\,\#\, \ell_i(-\C P^2)\,\#\, 
              m_i(S^2\times S^2), \label{eq:conn-sum}
  \end{equation}
    where $k_i$, $\ell_i$ and $m_i$ have a definite upper bound
    in terms of the number of the extremal points of $\interior X^3$ and 
    the number, say $\alpha(\partial X^3)$, of the connected components 
    of $\partial X^3:$
     $$
      k_i+\ell_i+2m_i\le \#\, \Ext(\interior X^3)
            + 2 \alpha(\partial X^3) -2.
     $$
\end{cor}
 
Actually $M_i^4$ is diffeomorphic to the connected sum
\eqref{eq:conn-sum} (see Remark \ref{rem:conn-sum}).
\par

Now we turn to the case of $\dim X=2$. First we consider the case when 
$X^2$ has no boundary.
By definition, a closed $4$-manifold is a Seifert $T^2$-bundle
over $X^2$ if there exists a surjective continuous map
$f:M^4\to X^2$ such that any point $p\in X^2$ has a 
neighborhood $U$ 
admitting the following commutative diagram $:$
    \begin{equation*}
     \begin{CD}
       f^{-1}(U) @>{\simeq} >> (T^2\times D^2)/\Z_m \\
          @V{f} VV      @VV{\pi} V    \\
           U @>    {\simeq} >>    D^2/\Z_m,
     \end{CD}
    \end{equation*}
where $(T^2\times D^2)/\Z_m$ denotes a free diagonal $\Z_m$-quotient;
the $\Z_m$-action is free on the $T^2$-factor and 
by rotation on the $D^2$-factor,
$p$ corresponds to the $\Z_m$-fixed point of $D^2$ in the above
diagram, and $m$ is a positive integer called the multiplicity  
of the torus fibre $f^{-1}(p)$.

\begin{thm} \label{thm:dim2nobdy}
Suppose that $M_i^4$ collapses to a two-dimensional compact Alexandrov
space $X^2$ without boundary under $K\ge -1$. 
Then $M_i^4$ is homeomorphic to either an $S^2$-bundle over $X^2$ or 
a Seifert $T^2$-bundle over $X^2$.
\end{thm}

Thus if the general fibre is a torus, $M_i^4$ may have 
singular torus-fibres (multiple tori) near essential singular
points of $X^2$, while in the sphere-fibre case,
$M_i^4$ is just a sphere-bundle even if $X^2$ has serious 
singular points. It should be remarked that the multiplicity $m$
of a multiple torus near a essential singular point $p\in X$
can be estimated as $m\le 2\pi/L(\Sigma_p)$, where 
$L(\Sigma_p)$ denotes the length of $\Sigma_p$.

Next we consider the case when the compact Alexandrov 
surface $X^2$ has nonempty boundary. In this case, 
we have a decomposition of $M_i^4$ into two parts:
one is a part of $M_i^4$ near a compact domain of $\interior X^2$
approximating $\interior X^2$ which is either an $S^2$-bundle
or a Seifert $T^2$-bundle (Theorem \ref{thm:dim2nobdy}).
The other is a part of $M_i^4$ near $\partial X^2$,
on which we can define a 
singular fibre structure as indicated in the theorem below.

\begin{thm} \label{thm:dim2wbdy}
Suppose that $M_i^4$ collapses to a two-dimensional compact Alexandrov
space $X^2$ with boundary under $K\ge -1$. 
Then $M_i^4$ is homeomorphic to a singular fibre space 
$\cal F(X^2)$ over $X^2$. 
\end{thm}

As a particular consequence of Theorem \ref{thm:dim2wbdy},
there is a continuous surjective map
$f_i: M_i^4\to X^2$ such that $f_i$ restricted to $\interior X^2$
is either an $S^2$-bundle or a Seifert $T^2$-bundle and 
$f_i$ restricted to $\partial X^2$ is a singular fibration 
whose fibres are ones of a point, $S^1$, $S^2$, the real projective 
plane $P^2$ and the Klein bottle $K^2$.
The fiber types of $f_i|_{\partial X^2}$ may change at
extremal points $\in {\rm Ext}(X^2)\cap \partial X^2$.
(see Theorem \ref{thm:dim2wbdy-break} for details).

Finally we assume $\dim X=1$. If $X$ is a circle, we can apply the 
fibration theorem (\cite{Ym:collapsing}) to get a fibre bundle structure 
on $M_i^4$ over $X$. More precisely,

\begin{thm}[\cite{FY:fundgp}] \label{thm:circle}
Suppose that $M_i^4$ collapses to a circle under $K\ge -1$. 
Then $M_i^4$ fibers over $S^1$ with fibre 
finitely covered by $S^1\times S^2$, $T^3$, a nilmanifold or 
a homotopy $3$-sphere.
\end{thm}

Thus the essential case here is the case when $X$ is isometric to a 
closed interval:

\begin{thm} \label{thm:dim1}
Suppose that $M_i^4$ collapses to a one-dimensional closed interval
under $K\ge -1$. 
Then $M_i^4$ is homeomorphic to a gluing $U_i^4\cup V_i^4$ along 
their boundaries, where  each of $U_i^4$ and $V_i^4$ is either
a disk bundle over $k$-dimensional closed manifold $N^k$ with
$0\le k\le 3$, or a gluing of two disk-bundles over
$k_j$-dimensional closed manifolds
$Q^{k_j}$, $j=1,2$, with $0\le k_j\le 2$,
where $N^k$ and $Q^{k_j}$ have nonnegative Euler numbers,
and if $k=3$,  $N^3$ is one of the closed 3-manifolds
appearing as the fibre types in Theorem \ref{thm:circle}.
\par
Therefore $M_i^4$ is a gluing of at least two and at most four
disk-bundles.
\end{thm}

In the case when $X$ is a point, rescaling metric 
of $M_i^4$ with the new diameter $1$, we can reduce the problem 
to the cases of $1\le \dim X\le 4$. \par

As a conclusion of the results in this paper together with the
stability theorem (\cite{Pr:alex2}) and the above remark, 
we have a description of the homeomorphism classes in 
$\cal M(4,D)$ as follows:

\begin{cor}
For a given $D>0$, there exist finitely many elements
$N_1^4,\ldots, N_k^4$ of $\cal M(4,D)$, where $k=k(D)$, such that
for any element $M^4$ of $\cal M(4,D)$ one of the following holds:
\begin{enumerate}
  \item $M^4$ is homeomorphic to one of $\{ N_1^4,\ldots, N_k^4\}$;
  \item $M^4$ is homeomorphic to a closed $4$-manifold 
    as described in Theorem \ref{thm:dim3},
    \ref{thm:dim2nobdy}, \ref{thm:dim2wbdy}, \ref{thm:circle}
    or \ref{thm:dim1} with a suitable compact Alexandrov space $X$
    of $1\le\dim X\le 3$.
\end{enumerate}
\end{cor}

In the proof of each result mentioned above, 
it is crucial to understand
the topology of a small (but of a fixed radius) metric ball 
in a collapsed 
$4$-manifold which is very close to a singular point of 
the limit space.
As the important first step, we establish the rescaling 
technique in dimension $4$ which makes it possible 
to pursue the topology of the metric ball. 
Under the rescaling of metrics, we obtain a new
limit space, a complete noncompact Alexandrov space with 
nonnegative curvature and with dimension larger than that of
the original limit space.
This reduces the problem to the study of collapsing
or convergence to a higher dimensional spaces
(with nonnegative curvature).
The second step is to analyze the new limit space in several aspects.
Among them, we especially need to establish the generalized soul 
theorem (Theorem \ref{thm:soul}) for $4$-dimensional 
complete open Alexandrov spaces with nonnegative curvature.
In the third step, we actually determine the structure of collapsing
to the new limit space, and obtain the topology 
of the small metric ball using an inductive procedure.
Using this topological information, we define a 
local fibre structure caused by collapsing.
Through some patching argument under the consideration of the singular
point set of the limit space, we glue those locally 
defined fibre structures to obtain a globally defined 
fibre structure.

A similar strategy was used in \cite{SY:3mfd}. However 
the present situation is more  involved in several aspects like;
a generalization of the rescaling argument in
\cite{SY:3mfd} (our approach seems to be possible to generalize to 
the general dimensions),
the complexity of the singular sets in the limit spaces,
the generalized soul theorem, classification of the collapses 
to the complete noncompact nonnegatively curved spaces with 
dimension less than four.

Our results in this paper suggest a possibility of the 
study of collapse of general $n$-dimensional Riemannian manifolds with 
a lower curvature bound by the following program:
\par\medskip\noindent
{\bf Assumption $(A_k)$}  \,\, 
   We already know the fibre structure of closed $n$-manifolds 
   collapsing to spaces of dimension $\ge k$,
   where the orbit types of the singular fibres can be determined 
   or estimated in terms of the singularities of the corresponding
   singular points of the limit spaces 
   (this is always satisfied for $k=n$).\par
\noindent
{\bf Purpose $(P_{k-1})$} \,\,
   We want to determine the structure of closed 
   $n$-manifolds collapsing to $(k-1)$-dimensional spaces.\par
\medskip\noindent
To realize $(P_{k-1})$, in the first step, we obtain the 
topology of a small metric ball using the following procedure
\begin{itemize}
 \item[(a)] a suitable rescaling of metrics;
 \item[(b)] the (metric) classification of the complete noncompact 
       spaces of dimension $\ge k$ and with nonnegative curvature 
       by means of the geometric invariants of them (like souls,
       the dimensions of the ideal boundaries, the extremal point
       sets and the essential singular point sets);
 \item[(c)] the classification of the phenomena of collapsing to
       the spaces in (b) using $(A_k)$;
\end{itemize}
In the second step, we first define a fibre structure on the small
metric ball, and then do a gluing process under the 
presence of the singular point set of the limit space
and realize the purpose $(P_{k-1})$.

As stated above, this program has been carried out for $n=4$ in the 
present paper.

The organization of this paper is as follows:
In Section \ref{sec:cap}, we provide some basic notions
about Alexandrov spaces with curvature bounded below and 
the Gromov-Hausdorff convergence.
We formulate a main result in this section,
the fibration-capping theorem, which gives a fibre
structure on a collapsed manifold in the case when 
the limit space has nonempty boundary. The proof of this 
result is postponed to Sections 
\ref{sec:cap-state}-\ref{sec:cap-lip}
(Part\ref{part:cap}), where we actually prove an equivariant
version of it.

In Section \ref{sec:action}, we prepare the basic
results about $S^1$-actions on oriented $4$-manifolds
(\cite{Fn:circleI})
providing the properties of singular loci and
the equivariant classification of such actions
in terms of weighted orbit data. We apply 
those results to obtain a topological classification of 
$S^1$-action on compact $4$-manifolds with certain
simple orbit data in terms of $D^2$-bundles over $S^2$.
This result is useful to classify the collapsing phenomena 
of $4$-manifolds to $3$-dimensional complete noncompact 
spaces with nonnegative curvature.

In Section \ref{sec:rescal}, we establish a key result
to obtain a topology of a small metric ball 
of a collapsed $4$-manifold close to a singular point of the limit 
space by generalizing a rescaling argument in \cite{SY:3mfd}. 
This makes it possible to reduce the problem to
the study of collapsing to higher dimensional limit spaces
of nonnegative curvature.

The rescaling argument in Section \ref{sec:rescal}
makes clear that the first step in the study of collapsing
of $4$-manifolds should be the case when the limit space is of 
dimension three. In Section \ref{sec:dim3alex}, as a preliminary
section, we establish two results about the properties of 
the essential singular point set of $3$-dimensional 
Alexandrov spaces with curvature bounded below:
One is a characterization of certain essential singular point set
in terms of quasigeodesics, and 
the other is about the existence of a metric collar neighborhood.

From Section \ref{sec:dim3nobdyloc} to Section \ref{sec:dim3wbdy},
we analyze the collapsing phenomena of orientable $4$-manifolds to
$3$-dimensional Alexandrov spaces. 
In Section \ref{sec:dim3nobdyloc}, we concentrate on 
finding the collapsing structure on a small
metric ball in a collapsed $4$-manifold which is close to a 
small metric ball in the interior of the limit three-space.
As the local collapsing structure, we construct an $S^1$-action 
on a small perturbation of the small metric ball 
extending the $S^1$-bundle structure 
on a regular part of the $4$-manifold.
In Section \ref{sec:dim3nobdyglob}, we patch those locally defined 
$S^1$-actions to obtain a globally defined local $S^1$-action
on a collapsed $4$-manifold.
In Section \ref{sec:dim3wbdy}, we determine the structure of 
collapsing to a $3$-dimensional Alexandrov space $X^3$ with boundary
using the collar neighborhood theorem established in 
Section \ref{sec:dim3alex} and the capping theorem,
a generalization of the fibration theorem.
We prove that a part of a collapsed $4$-manifold close to the 
boundary $\partial X^3$ is a $D^2$-bundle over $\partial X^3$,
compatible with the $S^1$-bundle structure near $\partial X^3$.
This provides a globally defined, local $S^1$-action on the 
$4$-manifold.

Before proceeding to the study of collapsing of $4$-manifolds 
to $2$-dimensional Alexandrov spaces, in Section 
\ref{sec:dim3positive}  we classify all the phenomena of collapsing of 
pointed complete $4$-manifolds to a pointed complete
noncompact $3$-dimensional Alexandrov spaces with nonnegative 
curvature. Using $S^1$-actions with the study of the singular
locus, we determine the topology of large metric balls in the
$4$-manifolds centered at the reference points. 
We prove that they are homeomorphic to either a disk-bundle
or a gluing of two disk-bundles.

In Sections \ref{sec:dim2nobdysph} and \ref{sec:dim2nobdytor},
we apply the classification result in Section 
\ref{sec:dim3positive} to obtain the topological information 
on a small metric ball in a $4$-manifold collapsed to
$2$-dimensional space without boundary.
Considering the new limit space under the rescaling of 
metrics, we come to know what kinds of collapsing phenomena
really occur near a singular fibre. 
After those discussions together with the equivariant fibration 
theorem,  we prove Theorem \ref{thm:dim2nobdy}
putting the desired fibre structures.

In Section \ref{sec:dim2wbdy}, we consider the case
when the limit surface has nonempty boundary, 
and prove Theorem \ref{thm:dim2wbdy}.

In Section \ref{sec:dim2positive}, 
we classify all the phenomena of collapsing of 
pointed complete $4$-manifolds to a pointed complete
noncompact $2$-dimensional Alexandrov spaces with nonnegative 
curvature using Theorems \ref{thm:dim2nobdy} and \ref{thm:dim2wbdy}. 
Theorem \ref{thm:dim1} is also proved there.

In Sections \ref{sec:local}--\ref{sec:soulII} of Part \ref{part:alex}, 
we prove the generalized soul theorem
in dimension $4$,  and in Section \ref{sec:alexwbdy} we classify 
all the $3$-dimensional complete
nonnegatively curved Alexandrov spaces with nonempty boundary. 

In Sections \ref{sec:cap-state}--\ref{sec:cap-lip} of Part \ref{part:cap}, 
we give the proof of the equivariant fibration-capping theorem.

\begin{ack}
  The author would like to thank Fuquan Fang, Kenji Fukaya, 
  Karsten Grove, Yukio Otsu, Takashi Shioya and Ayato Mitsuishi
  for discussions concerning this work.  
  This research was partially supported
  by  Institut des Hautes \'Etudes Scientifiques
  and Universidade Federal do Cear\'a. He is grateful
  for their hospitality and financial supports.
\end{ack}


\section{Alexandrov spaces and the Gromov-Hausdorff 
      convergence}\label{sec:cap}

In this section, we present some results on Alexandrov 
spaces with curvature bounded below and on 
the Gromov-Hausdorff convergence which will be used 
in the subsequent sections.
\par\medskip

\noindent 
{\bf 1. Alexandrov spaces.} \quad
We begin with some preliminary results on Alexandrov spaces.
We refer to \cite{BGP} for the basic materials and the details 
on Alexandrov spaces mentioned below.

Let $X$ be a finite-dimensional complete Alexandrov space with 
curvature bounded below, say $\ge \kappa$.
For any two point $x$ and $y$, let $xy$ denote a minimal geodesic
joining $x$ to $y$.
For any geodesic triangle 
$\Delta xyz$ in $X$ with vertices $x, y$ and $z$, if we denote by 
$\tilde\Delta xyz$ a {\it comparison triangle} 
in the $\kappa$-plane $M_{\kappa}^2$, 
the simply connected complete surface with constant curvature 
$\kappa$, the natural
map $\Delta xyz \to \tilde\Delta xyz$ is non-expanding,
where we assume that the perimeter 
of $\Delta xyz$ is less than $2\pi/\sqrt\kappa$ if $\kappa>0$.
This property is called the {\it Alexandrov convexity}. 
\par
  The angle between the geodesics $xy$ and $yz$ in $X$ is denoted by 
$\angle xyz$, and the corresponding angle of $\tilde\Delta xyz$
by $\tilde\angle xyz$.
It holds that $\angle xyz \ge \tilde\angle xyz$.
We denote by $\Sigma_p=\Sigma_p(X)$ the space of directions at $p\in X$,
and by $K_p=K_p(X)$ the tangent cone at $p$ with vertex $o_p$,
the Euclidean cone $K(\Sigma_p)$ over $\Sigma_p$.
It is known that $\Sigma_p$ (resp. $K_p$) is an Alexandrov space with 
curvature $\ge 1$ (resp $\ge 0$). 
More generally if $\Sigma$ is a compact Alexandrov space with
curvature $\ge 1$, then the Euclidean cone $K(\Sigma)$ over
$\Sigma$ is a complete Alexandrov space with nonnegative 
curvature.

For a compact set $A\subset X$ and $p\in X\setminus A$, we denote by 
$A^{\prime}=A^{\prime}_p$ the closed set of $\Sigma_p$ consisting 
of all the directions of minimal geodesics from $p$ to $A$.

  Let $k$ be the dimension of $X$, and  $\delta>0$.
A system of $k$ pairs of points, $(a_i,b_i)_{i=1}^k$ is called an 
$(k,\delta)$-\emph{strainer} at $p\in X$ if it satisfies 
\begin{align*}
   \tilde\angle a_ipb_i > \pi - \delta, \quad & 
                \tilde\angle a_ipa_j > \pi/2 - \delta, \\
      \tilde\angle b_ipb_j > \pi/2 - \delta,\quad &
                \tilde\angle a_ipb_j > \pi/2 - \delta,
\end{align*}
for every $i\neq j$. 
The number $\min\,\{ d(a_i,p), d(b_i,p)\,|\, 1\le i\le k\}$ 
is called the {\em length} of the strainer. \par

 Let   $R_{\delta}(X)$ denote the set of $(k,\delta)$-strained
points of $X$.  
Note that every point in $R_{\delta}(X)$ has a small neighborhood 
almost isometric to an open subset of $\R^k$ for small $\delta$,
where the almost isometry is given by 
$f(\,\cdot\,)=(d(a_1, \cdot\,),\ldots,d(a_k, \cdot\,))$.
In particular, $R_{\delta}(X)$ is a Lipschitz $k$-manifold.
The set $S_{\delta}(X):=X-R_{\delta}(X)$ is called the 
{\it $\delta$-singular set} of $X$.

  A point $p$ of an Alexandrov space $X$ 
is called an {\em extremal} point if $\diam(\Sigma_p) \le \pi/2$.
This coincides with the case when an extremal subset in the sense of 
\cite{PrPt:extremal} is a point.
The set of all extremal points of $X$ is denoted 
$\Ext(X)$. A point $p$ of $X$ 
is called an {\em essential} singular point if the radius 
of $\Sigma_p$ satisfies 
$\rad(\Sigma_p):=\min_{\xi}\max_{\eta}\angle(\xi,\eta) \le \pi/2$,
and the set of all essential singular points of $X$ is denoted 
$ES(X)$.
Notice that if a point $p\in X$ is not an essential
singular point, then $\Sigma_p$ is homeomorphic to a sphere
(\cite{GP:radius}, \cite{PrPt:extremal}), and a small metric ball around $p$ 
is homeomorphic to $\R^k$ (\cite{Pr:alex2}).

  We also say that $p$ is a {\em topological} singular point if
$\Sigma_p$ is not homeomorphic to a sphere.  
If $p$ is not a topological singular point, then it is called 
topologically regular. 
$X$ is called {\em topologically regular} if any point of $X$ 
is topologically regular.
It is well known (cf.\cite{Dav:decomp})that 
the double suspension $S^2(\Sigma^3)$
of the Poincare homology $3$-sphere $\Sigma^3$ 
is a topological sphere. Since $S^2(\Sigma^3)$
carries a metric as an Alexandrov space of curvature $\ge 1$,
one recognizes that
a topological singular point of an Alexandrov space can be a 
manifold point at least in dimension $\ge 5$.
The non-collapsing limit of a sequence of 
$n$-manifolds with a uniform lower sectional curvature bound
is called {\em smoothable}. 
Recently in \cite{Kp:non-collapse}, V. Kapovitch has proved that 
a smoothable Alexandrov space is topologically regular.
More precisely, 

\begin{prop}[\cite{Kp:non-collapse}] \label{prop:smoothable}
Let $X^n$ be a smoothable Alexandrov space.
Then for any $p_0\in X^n$, $p_1\in\Sigma_{p_0}X,
\ldots, p_i\in\Sigma_{p_{i-1}}\cdots\Sigma_{p_0}X$,
the iterated space of directions 
$\Sigma_{p_i}\Sigma_{p_{i-1}}\cdots\Sigma_{p_0}X$
is a topological sphere.
\end{prop}

We call a point $p_0$ of an Alexandrov space $X^n$ 
{\em topologically nice} if it
satisfies the conclusion of  Proposition
\ref{prop:smoothable} for any $p_1\in\Sigma_{p_0}X,
\ldots, p_i\in\Sigma_{p_{i-1}}\cdots\Sigma_{p_0}X$.
We also call $X^n$ 
{\em topologically nice} if any point of $X^n$ is topologically
nice. It is obvious that for $n\le 4$,
if $X^n$ is topologically
regular, then it is topologically nice.
When $m\ge 3$, $S^m(\Sigma^3)$ is not topologically nice but
topologically regular.
\par\medskip
\noindent 
{\bf 2. Gromov-Hausdorff convergence.} \quad
A (not necessarily continuous) map $\varphi:Y\to Z$ between metric spaces
is called an {\em  $\epsilon$-approximation} if 
\begin{enumerate}
\item  $|d(x,y)-d(\varphi(x),\varphi(y))|<\epsilon$ \quad for all $x,y\in Y$,
\item $\varphi(Y)$ is $\epsilon$-dense in $Z$.
\end{enumerate}
The Gromov-Hausdorff distance $d_{GH}(Y,Z)$ between $Y$ and $Z$ is defined
to be the infimum of such $\epsilon$ that there exist
$\epsilon$-approximations  $Y\to Z$ and $Z\to Y$.
The pointed Gromov-Hausdorff distance 
$d_{p.GH}((Y,y),(Z,z))$ between pointed metric spaces $(Y,y)$ and 
$(Z,z)$ is defined as 
the infimum of such $\epsilon$ that there exist
$\epsilon$-approximations  $B(y,1/\epsilon)\to B(z,1/\epsilon)$
and $B(z,1/\epsilon)\to B(y,1/\epsilon)$ sending
$y$ to $z$ and $z$ to $y$ respectively.
%

In the study of collapsing of three-dimensional 
Riemannian manifolds with a lower curvature bound,
the so-called fibration theorem (\cite{Ym:collapsing}, etc.) 
has played one of fundamental roles
(\cite{SY:3mfd}).
This theorem  shows that if a complete Riemannian 
manifold  $M$ with a definite lower bound on sectional curvature 
is Gromov-Hausdorff close to a Riemannian manifold  $X$ of 
lower dimension without boundary 
(or more generally an Alexandrov space with only weak singularities 
(\cite{Ym:conv})), then  $M$  fibers over  $X$ with 
almost nonnegatively curved fibre.   
We generalize this result by considering $X$ with nonempty boundary
as follows :

 Let $X$ be a $k$-dimensional complete Alexandrov space with curvature 
$\ge -1$. 
Now we assume that $X$ has nonempty boundary, and
denote by $D(X)$ the double of $X$, which is also 
an Alexandrov space with curvature $\ge -1$
(see \cite{Pr:alex2}). By definition, $D(X)=X\cup X^*$ glued
along their boundaries, where $X^*$ is another copy of $X$.

A $(k,\delta)$-strainer $\{ (a_i,b_i)\}$ of $D(X)$ at $p\in X$
is called {\it admissible} if 
$a_i\in X$, $b_j\in X$ for every $1\le i\le k$, $1\le j\le k-1$
(obviously, $b_k\in X^*$ if $p\in\partial X$ for instance).
Let  $R_{\delta}^D(X)$ denote the set of points of $X$ at which 
there are admissible $(k,\delta)$-strainers.
This has the structure of a Lipschitz $k$-manifold with boundary. 
Note that every point of $R_{\delta}^D(X)\cap \partial X$ 
has a small neighborhood in $X$ almost isometric to an open subset 
of the half space $\R^n_+$ for small $\delta$.

If $Y$ is a closed domain of $R_{\delta}^D(X)$, then
the $\delta_D$-{\it strain radius} of $Y$, denoted 
$\text{$\delta_D$-{\rm str.rad}$(Y)$}$, is 
defined as the infimum of positive numbers $\ell$ such that 
there exists an admissible $(k,\delta)$-strainer of length $\ge \ell$ 
at every point $p\in Y$. 

For a small $\nu>0$, we put
$$
   Y_{\nu} := \{ x\in Y\,|\, d(\partial X, x)\ge\nu\}.
$$
We use the following special notations:
$$
   \partial_0 Y_{\nu} := Y_{\nu}\cap \{ d_{\partial X}=\nu\},\quad 
             \interior_0 Y_{\nu}:=Y_{\nu}-\partial_0 Y_{\nu}.
$$
Let $M^n$ be another $n$-dimensional complete Alexandrov space 
with curvature $\ge -1$.
A surjective map $f:M\to X$ is called an 
$\epsilon$-{\it almost Lipschitz submersion} if \par
\begin{enumerate}
 \item  it is an $\epsilon$-approximation;
 \item  for every $p,q\in M$  
    $$
     \left|\frac{d(f(p), f(q))}{d(p, q)}-\sin\theta_{p,q}\right| < \epsilon.
   $$
  where $\theta_{p,q}$ denotes the infimum of $\angle qpx$ when 
  $x$ runs over $f^{-1}(f(p))$.
\end{enumerate}

We denote by $\tau(\epsilon_1,\ldots,\epsilon_k)$ a function depending on
a priori constants and $\epsilon_i$ satisfying 
$\lim_{\epsilon_i\to 0}\tau(\epsilon_1,\ldots,\epsilon_k)=0$. 

\begin{thm}[Fibration-Capping Theorem] \label{thm:orig-cap}
Given $k$ and $\mu>0$ there exist positive numbers $\delta=\delta_k$,
$\epsilon=\epsilon_{k}(\mu)$  and $\nu=\nu_k(\mu)$ satisfying 
the following $:$\,\, Let $X^k$ and $M^n$ be as above, and 
let  $Y\subset R_{\delta}^D(X)$ be a closed domain
such that $\text{$\delta_D$-{\rm str.rad}$(Y)$}\ge\mu$.
Suppose that $X$ has only weak singularities, i.e., 
$M^n=R_{\delta_n}(M^n)$ for some small $\delta_n>0$.
If $d_{GH}(M^n,X^k)<\epsilon$ for some $\epsilon\le\epsilon_{n}(\mu)$,
then there exists a closed domain $N\subset M^n$ and 
a decomposition 
$$
   N = N_{\rm int} \cup N_{\rm cap}
$$
of $N$ into two closed domains glued along their boundaries and a
Lipschitz map  $f:N \to Y_{\nu}$ such that 
\begin{enumerate}
 \item $N_{{\rm int}}$ is the closure of 
       $f^{-1}(\interior_0 Y_{\nu})$, and
        $N_{\rm cap} = f^{-1}(\partial_0 Y_{\nu})$;
 \item both the restrictions 
     $f_{\rm int}:=f|_{N_{\interior}}: N_{\interior} \to Y_{\nu}$ and
     $f_{\rm cap}:=f|_{N_{\rm cap}} : N_{\rm cap} \to \partial_0
     Y_{\nu}$ are 
   \begin{enumerate}
     \item  locally trivial fibre bundles;
     \item  $\tau(\delta,\nu,\epsilon/\nu)$-Lipschitz submersions.
   \end{enumerate}
\end{enumerate}
\end{thm}
\medskip
\begin{rem}
\begin{enumerate}
 \item The fibres of both $f_{\rm int}$ and $f_{\rm cap}$ 
    have almost nilpotent fundamental groups (\cite{FY:fundgp})
    and the first Betti numbers less than or equal to the dimensions 
    of them (\cite{Ym:collapsing}).
\item If $X$ has no boundary, then $N=N_{\rm int}$ and 
      $f=f_{\rm int}$. In this case, Theorem \ref{thm:orig-cap} is 
      stated in 
      \cite{BGP}, and will be called the fibration 
      theorem as usual. See \cite{Ym:conv} and \cite{SY:3mfd}
      in the case when $M$ is a Riemannian manifold.
\end{enumerate}
\end{rem}

For instance, if $X$ is a closed interval and if 
$d_{GH}(M, X)$ is sufficiently small, then we have an obvious 
decomposition $M=M_{\rm int}\cup M_{\rm cap}$, where
$M_{\rm int}\simeq F\times I$, $F$ is a general fibre
and $M_{\rm cap}$ consists of two components.
If $X$ is a Riemannian disk $D^k$, we have a
decomposition $M=M_{\rm int}\cup M_{\rm cap}$, where
$M_{\rm int}\simeq F\times D^k$ and $M_{\rm cap}$ is a fibre bundle over 
$S^{k-1}$  whose fibre has boundary homeomorphic to $F$.

Applying the results of the present paper together with 
\cite{FY:fundgp}, \cite{SY:3mfd}, we have the following 

\begin{cor} \label{cor:cap}
 Under the situation of Theorem \ref{thm:orig-cap},
 if the codimension  
 $m:=\dim M - \dim X$ is less than or equal to three, 
 the topology of the fibre $F_{\rm cap}$ of $f_{\rm cap}$ 
 is described as follows:
 \begin{enumerate}
  \item If $m=1$, $F_{\rm cap}$ is homeomorphic to $D^2;$
  \item If $m=2$, $F_{\rm cap}$ is homeomorphic to one of 
     $D^3$, $P^2\tilde\times I$, $S^1\times D^2$, and 
     $K^2\tilde\times I;$
  \item If $m=3$, $F_{\rm cap}$ is homeomorphic to either a disk-bundle
     or a gluing of two disk-bundles having the same topology as 
     $U_i^4$ described in Theorem \ref{thm:dim1}.
 \end{enumerate}
%
\end{cor}

Note that when $m\le 3$, the topology of the fibre of  $f_{\rm int}$
is already determined in \cite{Ym:collapsing} and \cite{FY:fundgp}.

The proofs of Theorem \ref{thm:orig-cap} and Corollary \ref{cor:cap}
are differed to Part \ref{part:cap},
where we actually prove an equivariant version of Theorem \ref{thm:orig-cap}
(see Theorem \ref{thm:cap}).

  When no collapsing occurs, we have the following stability result
due to \cite{Pr:alex2}.\par

\begin{thm}[Stability Theorem \cite{Pr:alex2}]\label{thm:stability}
Let a sequence of compact $n$-dimensional Alexandrov spaces $X_i$ with
curvature $\ge -1$ converge to 
a compact Alexandrov space  $X$ of dimension $n$. Then 
$X_i$ is homeomorphic to $X$ for sufficiently large $i$.
\end{thm}

In the proof of the theorem above, the notion of regular 
maps is crucial. Here we recall the most basic case of it.

Consider a map $f=(f_1,\ldots,f_m):U\to\R^m$ on an open set $U$ of
an Alexandrov space $X$ defined by the distance functions  
$f_j(x)=d(A_j,x)$ from compact subsets $A_j\subset X$. 
The map $f$ is said to be  $(c,\epsilon)$-\emph{regular} at $p\in U$ 
if there is a point $w\in X$ such that 
\begin{enumerate}
 \item $\angle((A_j)_p^{\prime},(A_k)_p^{\prime})>\pi/2-\epsilon$;
 \item $\angle(w_p^{\prime},(A_j)_p^{\prime})>\pi/2+c$,
\end{enumerate}
for every $j\neq k$.
In \cite{Pr:alex2}, it was proved that if 
$f=(f_1,\ldots,f_m):U\to\R^m$ is $(c,\epsilon)$-regular
on $U$ for some $\epsilon>0$ sufficiently small compared to $c$,
then it is a topological submersion (see also \cite{Pr:morse}).
We simply say that $f$ is regular in this case.
Together with \cite{Si:stratified}, this implies that 
if $f:U\to\R^m$ is regular and proper, then it is a locally
trivial fibre bundle over its image.

Under the same assumption as Stability Theorem
\ref{thm:stability},
suppose in addition that a regular map $f:U\to\R^m$ is 
given as above on an open subset $U\subset X$. 
Using an approximation map $X\to X_i$,
we can define a regular map $f_i:U_i\to\R^m$ 
on an open subset $U_i\subset X_i$ for any sufficiently large $i$.

The following is the respectful version of Stability Theorem
\ref{thm:stability}.

\begin{thm}[\cite{Pr:alex2}] \label{thm:stability-respect}
Under the situation above, there exists a homeomorphism $h_i:X\to X_i$ 
such that $f_i\circ h_i=f$ holds on every compact set 
$K\subset U$ and for sufficiently large $i\ge i(f,K)$.
\end{thm}

For any compact set $A$ of an Alexandrov space $X$, $d_A=d(A,\,\cdot\,)$
denotes the distance function from $A$.
In case $X$ is a Riemannian manifold, we denote
by $\tilde d_A=\tilde d(A,\,\cdot\,)$ a smooth approximation of 
$d_A$.

 
\section{Preliminaries on complete Alexandrov spaces with
       nonnegative curvature} \label{sec:ideal}

Let $C$ be a complete nonnegatively curved Alexandrov space 
with nonempty boundary. In \cite{Pr:alex2}, it was proved 
that the function $d(\partial C,\,\cdot\,)$ is concave
on $C$. We first show the rigidity for this function.

\begin{prop} \label{prop:concave-rigid}
For a unit speed geodesic segment 
$\gamma:[\,0,a\,]\to C$ joining $p_0$ to $p_1$,
suppose that $d(\partial C, \gamma(t))$ is constant.
Then for any minimal geodesic $\gamma_0$ from $p_0$ 
to $\partial C$, there is a 
 minimal geodesic $\gamma_1$ from $p_1$ 
to $\partial C$, such that $\{ \gamma_0,\gamma,\gamma_1\}$ bounds a 
flat totally geodesic rectangle.
\end{prop}

\begin{proof}
This follows from a modification of Proposition 9.10 in
\cite{SY:3mfd}. For completeness, we give a proof below.

Let $q_0\in\partial C$ denote the end point of
$\gamma_0$. 
We show that
$\Sigma_{q_0}(D(C))$ is the spherical suspension over 
$\partial\Sigma_{q_0}(C)$, where
$D(C)$ is the double of $C$.
Let $r_0\in D(C)$ be another copy of $p_0$, and 
$\xi_+$ and $\xi_-$ the directions at $q_0$  represented 
by geodesics $q_0p_0$ and $q_0r_0$ respectively.
Obviously $\angle(\xi_+,\xi_-)=\pi$.
It follows that 
$\Sigma_{q_0}(D(C))$ is the spherical suspension over 
the hypersurface $S:=\{ v\,|\, \angle(\xi_{\pm},v)=\pi/2\}$.
Clearly  $\partial\Sigma_{q_0}(D(C))$ is a subset of 
$\{ v\,|\, \angle(\xi_{\pm},v) \ge \pi/2\}$ which coincides with
$S$. An obvious argument then implies that
$\partial\Sigma_{q_0}(D(C))=S$.
\par
  
Let $\sigma:[\,0,b\,]\to X$ be a unit speed minimal geodesic joining 
$q_0$ to $p_1$. 
We consider the concave function $f(u)=d(\sigma(u),\partial C)$. 
Put $\alpha=\angle(\dot\sigma(0),\partial\Sigma_{q_0}(C))$, 
where $\dot\sigma(0)$ denotes the direction at $\sigma(0)$ 
represented by $\sigma$.
From a standard argument,
$$
             f^{\prime}(0)=\sin \alpha.
$$
Consider the triangle 
$\triangle q_0^{\prime}p_0^{\prime}p_1^{\prime}$ on $\R^2$ 
such that
$d(p_0^{\prime},q_0^{\prime})=t$, $d(p_0^{\prime},p_1^{\prime})=a$ and 
$\angle q_0^{\prime}p_0^{\prime}p_1^{\prime}=\pi/2$,
where $t:=d(p_0,q_0)$.
Set
$b^{\prime}=d(q_0^{\prime},p_1^{\prime})$,
$\alpha^{\prime}=\angle q_0^{\prime}p_1^{\prime}p_0^{\prime}$, 
$\theta^{\prime}=\angle p_0^{\prime}q_0^{\prime}p_1^{\prime}$,
and  $\theta=\angle p_0 q_0 p_1=\pi/2-\alpha$. 
Since $\angle q_0 p_0 p_1=\pi/2$, we have
$b^{\prime}\ge b$. 
It follows from  the concavity of $f$ that 
$$
  f^{\prime}(0)\ge \frac tb\ge \frac t{b^{\prime}}.
$$
Thus we obtain that 
\begin{equation}
    \alpha\ge\alpha^{\prime} \qquad \text{and} \qquad
            \theta\le\theta^{\prime}.
\end{equation}
Consider now a
comparison triangle $\tilde\triangle q_0 p_0 p_1$ in $\R^2$
and put $\tilde\theta =\tilde\angle p_0 q_0 p_1$, 
$\tilde\alpha=\tilde\angle q_0 p_1 p_0$. 
Since we may assume for 
our purpose that $t>a$, it follows from an obvious consideration
with $b>t$ that 
$\alpha^{\prime}\le\tilde\alpha\le\pi/2$, 
$\theta^{\prime}\le\tilde\theta$ and hence 
\begin{equation}
  \theta^{\prime}=\theta=\tilde\theta, \quad
   \alpha^{\prime}=\tilde\alpha=\alpha,\quad 
   b=b^{\prime}\quad \text{and}\quad
   \tilde\angle q_0 p_0 p_1=\pi/2.  
\end{equation}
It follows from the rigidity argument(cf.\cite{Shm:alex}) that  
$\triangle q_0 p_0 p_1$ spans a 
totally geodesic flat triangle isometric to 
$\tilde\triangle q_0 p_0 p_1$.
Furthermore,
$f^{\prime}(0)=t/b$. It follows from the concavity of $f$ that
$f(u)=tu/b$ for all $u$. 
Let $x_u$ and $y_u$ be the points on $\partial C$ and $q_0 p_1$ 
respectively such that $f(u)=d(\sigma(u),x_u)$ and 
$d(p_0,y_u)=ua/b$. 
Then it follows together with the comparison argument that
$d(x_u,y_u)\le d(x_u,\sigma(u))+d(\sigma(u),y_u)\le t$.
Thus $\sigma$ lies on the minimal connections from  
the points of $\gamma$ to $\partial C$.   
\par

By repeating the argument above for $x_u,y_u,p_1$ in place of
$q_0,p_0,p_1$, we conclude that the set of minimal connections 
$x_uy_u$, $0\le u\le b$, provides a  
totally geodesic flat rectangle.
\end{proof}

Let $X$ be a complete noncompact   
Alexandrov space with nonnegative curvature.
Consider the Busemann function associated with a reference point
$p\in X$ defined by 
\begin{equation}
    b(x)=\sup_{\gamma}b_{\gamma}(x),
         \label{eq:Busemann}
\end{equation}
where 
$\gamma$ runs over all the geodesic rays emanating from $p$
and

$$
   b_{\gamma}(x)=\lim_{t\to\infty} t-d(x,\gamma(t)).
$$
Applying the Cheeger-Gromoll basic construction(\cite{CG:str},
\cite{Pr:alex2}),
we obtain a sequence of finitely many 
nonempty compact totally convex sets:
$$
  C(0)\supset C(1)\supset C(2)\supset\cdots\supset C(k),
$$
where $C(0)$ is the minimum set of $b$, $n>\dim C(0)$, 
$\dim C(i) > \dim C(i+1)$ and $C(k)$ has no boundary.
Then a soul $S$ of $X$ is defined as $S=C(k)$. 
It was proved in \cite{Pr:alex2} that $X$ is 
homotopy equivalent to $S$.

The compact totally convex set $C(0)$ depends on the reference 
point $p$ and is denoted by $C_p(0)$ for a moment.
Now consider the integer
$$
  m_X := \inf_{p\in X} \dim C_p(0).
$$

We denote by $X(\infty)$ the ideal boundary of $X$ equipped 
with the Tits distance. 

\begin{lem} \label{lem:C0codim1}
If $m_X=n-1$, then $X^n(\infty)$ 
consists of at most two elements.
\end{lem}

\begin{proof}
Consider $C(0)$ with an arbitrary reference point.
By Proposition \ref{prop:concave-rigid}, we have a map
$f:C(0)\times\R$ satisfying:
\begin{enumerate}
 \item $f(x,0)=x$ for any $x\in C(0);$
 \item Both $f(x\times\R_+)$ and $f(x\times\R_-)$ represent 
       geodesic rays from $x$ for any $x\in C(0)$, 
       where $\R_+=[0,\infty)$, $\R_-=(\,-\infty,0\,);$ 
 \item For any $x,y\in C(0)$,  $f(x\times\R_{\pm})$, a geodesic
       $\gamma_{x,y}$ joining $x$ and $y$ and  $f(y\times\R_{\pm})$
       span flat rectangles.
\end{enumerate}
For any $p\in\interior C(0)$, 
consider $C_p(0)$. Clearly $p\in C_p(0)$.
Let $\gamma_{\pm}$ denote
the geodesic rays $f(p\times\R_{\pm})$.
Suppose that there is a geodesic ray $\sigma$ starting from $p$ 
different from $\gamma_{\pm}$, and consider
the minimum set $C$ of $b_{\sigma}$ on the compact set $C_p(0)$.
Since $b_{\sigma}$ is locally nonconstant on $C_p(0)$, 
$\dim C<\dim C_p(0)=n-1$. Obviously
$C_q(0)\subset C$ for any $q\in C$, contradicting the assumption 
$m_X=n-1$.
\end{proof}

\begin{ex}
Let $X$ be the double of the closed domain
$\{ (x,y)\,|\, x,y\ge 0, \,y\ge x-1\,\}$ on the $(x,y)$-plane.
Note $\dim X(\infty)=1$.
For $p=(1/2,0)$, $C_p(0)$ coincides with the segment
$[0,1/2]\times 0$, and for
$q=(0,0)$, $C_q(0)=\{ q\}$. 
\end{ex}

The above example shows that $\dim C(0)=\dim X-1$ does not imply
$\dim X(\infty)=0$.

\begin{prop}\label{prop:ideal}
    $\dim X(\infty)\le \mbox{\rm codim\,\,} S-1$.
\end{prop}

Independently Shioya has obtained the above proposition 
in a similar way.

\begin{lem} \label{lem:perp}
Let $\Sigma$ be an compact Alexandrov space with curvature $\ge 1$, 
and $\Sigma_0\subset\Sigma$ be a closed locally convex set
of positive dimension and without boundary as an Alexandrov space.
Let $\xi\in\Sigma$ be such that $d(\xi,\Sigma_0)\ge\pi/2$.
Then $d(\xi,v)=\pi/2$ for every $v\in \Sigma_0$.
\end{lem}

\begin{proof}
Let $v_0\in \Sigma_0$ be a foot from $\xi$ to $\Sigma_0$.
We proceed by induction on $\dim \Sigma_0$. 
If $\dim \Sigma_0=1$, then $\Sigma_0$ is isometric to a circle.
Since $\angle \xi v_0v=\pi/2$, the lemma follows from the Alexandrov convexity.
\par
  Suppose the lemma is true for $\dim \Sigma_0 -1$. 
Put $\Sigma^{\prime}=\Sigma_{v_0}(\Sigma)$,
$\Sigma^{\prime}_0=\Sigma_{v_0}(\Sigma_0)$. 
Let $\xi^{\prime}=\xi_{v_0}^{\prime}\in\Sigma^{\prime}$,
$v^{\prime}=v_{v_0}^{\prime}\in\Sigma_0^{\prime}$ for every 
$v\in\Sigma_0$.
It is easy to check that $\Sigma_0^{\prime}$ is a closed locally convex
subset of $\Sigma^{\prime}$. Applying the induction hypothesis, we see that
$\angle(\xi^{\prime},v^{\prime})=\pi/2$ for every $v\in\Sigma_0$.
The Alexandrov convexity then implies the conclusion.
\end{proof}

\begin{proof}[Proof of Proposition \ref{prop:ideal}]
It suffices to show that for any regular point $p$ of $S$ and any geodesic 
ray $\gamma$ from $p$, $\dot\gamma(0)$ is perpendicular to $S$.
Let $x\in S$ be a foot from $\xi=\gamma(\infty)\in X(\infty)$ to 
$S$, and $\sigma$ a geodesic ray from $x$ asymptotic to 
$\gamma$. Namely we have sequences $x_n\in S$, $y_n\in X$ satisfying
\begin{enumerate}
 \item $x_n\to x$, $y_n\to\xi$, $x_ny_n\to\sigma$;
 \item $d(y_n,x_n)=d(y_n,S)$.
\end{enumerate}
Note that $\Sigma_x(S)$ is a closed locally convex set of 
$\Sigma_x(X)$. 
By the previous lemma, $\lim \angle qx_ny_n=\pi/2$ for 
any $q\in S$, and hence $\angle qx\xi=\pi/2$. 
The Alexandrov convexity then implies that
\begin{equation}
  \tilde\angle xqy_n > \pi/2 + o_n
\end{equation}
with $\lim_{n\to\infty} o_n = 0$.
Thus we have 
$\lim_{n\to\infty}|d(y_n,q) - d(y_n,x)| = 0$,
and hence $b_{\sigma}(x)=b_{\sigma}(q)$,
where $b_{\sigma}$ denotes the Busemann function associated with 
the ray $\sigma$.
It follows that $b_{\gamma}(x)=b_{\gamma}(q)$  for any $q\in S$. 
Thus $\gamma$ must be perpendicular to $S$ at $p$.
\end{proof}

We denote by $B(S,\epsilon)$ the closed $\epsilon$-metric ball 
around $S$.

\begin{thm}[Generalized Soul Theorem] \label{thm:soul} 
Let $X$ be a $4$-dimensional complete open 
$($i.e, noncompact and boundaryless$)$
Alexandrov space with nonnegative curvature.
Suppose in addition that $X$ is topologically regular.
Then there exists a positive number $\epsilon$ such that 
\begin{enumerate}
 \item $X$ is homeomorphic to $\interior B(S,\epsilon);$ 
 \item $B(S,\epsilon)$ is homeomorphic to a disk-bundle 
       over $S$, called the normal bundle of $S$.
\end{enumerate}
\end{thm}

The proof of Theorem \ref{thm:soul} is given in 
Part \ref{part:alex}.

\begin{cor} \label{cor:soul}
Suppose that a sequence of 
$4$-dimensional pointed complete Riemannian
manifolds $(M_i^4,p_i)$ with $K \ge -1$ converges to a pointed
complete noncompact $4$-dimensional Alexandrov space $(Y^4,y_0)$ with 
nonnegative curvature. Then for a sufficiently large $R>0$,
$B(p_i,R)$ is homeomorphic to the normal closed disk-bundle
over the soul of $Y^4$.
\end{cor}

\begin{proof}
Note first that $Y^4$ is topologically regular
(\cite{Kp:non-collapse}).
By Theorem \ref{thm:stability-respect},
$B(p_i, R)$ is homeomorphic to $B(y_0, R)$ any sufficiently large $R$
because of the nonnegativity of 
the curvature of $Y^4$.
By Theorem \ref{thm:soul}, $B(y_0, R)$ is homeomorphic to 
the normal closed disk-bundle over the soul of $Y^4$ 
\end{proof}


\section{$S^1$-actions on oriented $4$-manifolds}\label{sec:action}

In this section, we present some results on 
$S^1$-actions on oriented $4$-manifolds which will be used 
in the subsequent sections.
\par

First we fix the notation and terminology about local 
$S^1$-actions. Suppose that an open covering $\{ U_{\alpha}\}$
of a manifold $N$ and an $S^1$-action $\psi_{\alpha}$
on $U_{\alpha}$ are given in such a way that
both the actions $\psi_{\alpha}$ and $\psi_{\beta}$
coincide up to orientation on the intersection
$U_{\alpha}\cap U_{\beta}$.
We say that $\{(U_{\alpha},\psi_{\alpha})\}$ defines a local 
$S^1$-action, denoted $\psi$, on $N$.
The local $S^1$-action $\psi$ is {\em locally smooth}
if each $x\in N$ has a slice $V_x$ which is a disk invariant under 
te action of the isotropy group $S^1_x$ at $x$  and if this
action is equivariant to an orthogonal action. 

Let $N^*=N/\psi$ be the orbit space and
$\pi:N\to N^*$ the orbit map.
An orbit is called {\em exceptional} if the isotropy group at 
a point on the orbit is non-trivial but finite.
Let $F=F(\psi)$ and $E=E(\psi)$ denote
the fixed point set and the union of the exceptional orbits respectively.
In the present paper, an orbit in $S=S(\psi):=F\cup E$ is called 
a singular orbit.
The images $F^*:=\pi(F(\psi))$, $E^*:=\pi(E(\psi))$ and 
$S^*:=\pi(S(\psi))$ are
called the fixed point locus, the exceptional locus and the singular locus
respectively.

\begin{lem} \label{lem:simply}
 \begin{enumerate}
  \item If $N$ is simply connected, then so is $N^*;$
  \item If both $N$ and $N^*-S^*$ are orientable, then $\psi$ is 
        actually an $S^1$-action.
 \end{enumerate}
\end{lem}

\begin{proof}
\,\,$(1)$. Let $\gamma^*$ be any loop at a point $p^*\in N^*-S^*$.
Since $S^*\cap\interior N^*$ is of codimension $\ge 2$,
one can slightly perturb $\gamma^*$ to a loop $\sigma^*$
in $N^* - S^*$. Let $\sigma$ be a lift of $\sigma^*$ to $N-S$.
From the assumption, $\sigma$ is homotopic to a path in 
$\pi^{-1}(p^*)$ keeping endpoints fixed. Thus
$\sigma^*$ and hence $\gamma^*$ is null-homotopic. \par

$(2)$. For each $x\in N-S$, we have the orientation on the 
slice $V_x$ at $x$ induced from the orientation of $N^*-S^*$.
Then the orientation of $N$ induces the orientation of 
the orbit $S^1 x$ in such a way that the orientation 
of $V_x$ followed by the orientation of $S^1 x$ coincides with
the original orientation of $N$. Since $E^*$ is of codimension 
$\ge 2$, those orientations of the orbits in $N-S$ extend to 
the orientations of the orbits in $N-F$.
\end{proof}

\begin{defn} Let $D^4(1)=\{ (z_1,z_2)\in \C^2\,|\,|z_1|^2+|z_2|^2=1\}$, and 
let $a, b$ be relatively prime integers. The $S^1$-action $\psi_{a,b}$ on 
$D^4(1)$ defined by
$$
    z\cdot(z_1,z_2)  =  (z^az_1,z^bz_2),
$$
is called the {\it canonical $S^1$-action of type $(a,b)$}.
For any $0<t\le 1$, the restriction of $\psi_{a,b}$ to 
$S^3(t)=\{ (z_1,z_2)\in D^4(1)\,|\,|z_1|^2+|z_2|^2 = t^2\}$
gives a Seifert bundle $S^3(t)\to S^3(t)/\psi_{a,b}\simeq S^2$. 
\end{defn}

  Note that the restriction  $\psi_{1,1}|_{S^3(t)}$ is the Hopf-fibration.
Note also that the orbit space $D^4(1)/\psi_{a,b}$ with the standard 
quotient metric is an Alexandrov space with nonnegative curvature 
whose space of directions $\Sigma_o$ at the origin
$o\in D^4(1)/\psi_{a,b}$ has at most two singular points, 
say $\xi_1$ and $\xi_2$, where 
$$
   \{ |a|, |b|\} = \{ 2\pi/L(\Sigma_{\xi_1}(\Sigma_o)), 
                       2\pi/L(\Sigma_{\xi_2}(\Sigma_o))\}.
$$

 We describe the equivariant classification of 
$S^1$-actions on oriented $4$-manifolds $M$
slightly extending the treatment in \cite{Fn:circleI}.  
In \cite{Fn:circleI}, $M$ is assumed to be simply connected 
and to have no boundary.
In our later use however, we need the more general case when
$M$ has nonempty boundary and only the orbit space $M^*$ 
is assumed to be simply connected.
For convenience, we describe the equivariant classification 
in our general setting.

Let $M$ be an oriented $4$-manifold (possibly with boundary) with 
a locally smooth $S^1$-action.
We assume that there are no fixed points on $\partial M$.
Let $\pi:M\to M^*$ be the orbit map. The proof of the 
following Proposition is similar to Proposition (3.1) in 
\cite{Fn:circleI}, and hence omitted.

\begin{prop}
Under the situation above, $M^*$ is a $3$-manifold for
which the general properties of the singular locus are 
described as follows$:$
\begin{enumerate}
 \item $\partial M^* - \pi(\partial M)\subset F^*;$
 \item $F^*-\partial M^*$ is discrete$;$
 \item The closure $\bar{E}^*$ of $E^*$ is a disjoint union of arcs and 
       simply closed curves in $M^*-(\partial M^* -\pi(\partial M))$.
       Each connected component of $E^*$ is one of the following$;$
  \begin{enumerate}
   \item A simple closed curve in $\interior M^*;$
   \item An arc whose closure joins two points of 
         $(F^*\cap\interior M^*)\cup\pi(\partial M)$.
  \end{enumerate}
\end{enumerate}
\end{prop}

Note that the orbit type is constant on each component of 
$E^*$. From now on we assume that $M^*$ is simply connected.

\begin{lem}
If $M^*$ is simply connected, then each component of
$F^*\cap\partial M^*$ is homeomorphic to either
$S^2$ or $D^2$.
\end{lem}

\begin{proof}
Note that each component of $\partial M^*$ is a sphere.
On the double $D(M)$ we have an $S^1$-action induced from the original 
$S^1$-action. By Van Kampen's theorem, the orbit space $D(M)^*$ is 
simply connected. Suppose that there is a component $F^*_0$ of
$F^*\cap\partial M^*$ which is neither a disk nor a sphere.
Then it turns out the double $D(F_0^*)$ is a component of
$\partial (D(M)^*)$ and a surface with genus $\ge 1$, 
a contradiction to the simple connectivity of $D(M)^*$.
\end{proof}

The orientation of $M$ induces an orientation of $M^*$ in such a way
that the orientation of $M^*$ followed by the natural orientation
of the orbit coincides with the original orientation of $M$.
For a given oriented submanifold $X^*$ of $M^*$ we also use this
convention to orient $\pi^{-1}(X^*)$.

We assign the following orbit data to $M^*$:
\begin{enumerate}
 \item For a given component $F_i^*$ of $F^*\cap\partial M^*$
       homeomorphic to $S^2$, take a collar neighborhood 
       $F^*_i\times [0,1]$ of $F_i$, and orient 
       $F^*_i\times \{ 1\}$ by the outward normal. We assign
       to $F_i^*$ the Euler number of the $S^1$-bundle 
       $\pi^{-1}(F_i^*\times \{ 1\})\to F_i^*\times \{ 1\};$
 \item For each $x^*\in F^*\cap\interior M^*-\bar{E}^*$, take a small disk
       neighborhood $B^*$ of $x^*$. We assign to $x^*$ the 
       Euler number $+1$ or $-1$ of the $S^1$-bundle
       $\pi^{-1}(\partial B^*)\simeq S^3\to\partial B^*\simeq S^2;$
 \item Let $L^*$ be a simple closed curve in 
       $E^*\cup(F^*\cap\interior M^*)$ and fix an orientation 
       of $L^*$. For each component $J^*$ of $L^*-F^*$, let
       $y^*$ be the end point of $J^*$, choose a sufficiently small
       disk domain $B^*$ of $y^*$ and orient $\partial B^*$ by a
       normal with direction of $J^*$. We assign to $J^*$ the 
       Seifert invariant $(\alpha,\beta)$ of the orbit in 
       $\pi^{-1}(\partial B^*)$ with image $J^*\cap\partial B^*$.
       The weights assigned to $L^*$ consists of the orientation 
       and a finite system of the Seifert invariants, abbreviated by
       $\{ (\alpha_1,\beta_1),\ldots,(\alpha_n,\beta_n)\}$. If 
       the orientation of $L^*$ is reversed, the finite system
       changes to 
       $\{ (\alpha_n,\alpha_n-\beta_n),\ldots,
       (\alpha_1,\alpha_1-\beta_1)\}$. We regards these two weights 
       to be equivalent$;$
 \item Let $A^*$ be an arc component of 
       $E^*\cup(F^*\cap\interior M^*)$. Choosing an orientation of 
       $A^*$ , we define a finite system of Seifert invariants as in 
       $(3)$. Let $y^*$ be either the initial point or the end point
       of $A^*$, and choose a sufficiently small disk domain $B^*$
       of $y^*$ and orient $\partial B^*$ as in $(3)$. We assign
       to $y^*$ the obstruction number $b$ for the Seifert bundle 
       $\pi^{-1}(\partial B^*)\to\partial B^*$ with exactly
       one singular orbit. The weights assigned to $A^*$ consists 
       of the orientation, a finite system of the Seifert invariants
       and the two obstruction numbers at the end points, 
       abbreviated by
       $\{ b';(\alpha_1,\beta_1),\ldots,(\alpha_n,\beta_n);b''\}$.
       If the orientation of $A^*$ is reversed, the weights change
       to 
       $\{-1- b''';(\alpha_n,\alpha_n-\beta_n),\ldots,
       (\alpha_1,\alpha_1-\beta_1);-1-b'\}$.
       We regards these two weights to be equivalent.
\end{enumerate}
$M^*$ together with the above collection of weights is called a 
{\it weighted orbit space}.

\begin{lem}[\cite{Fn:circleI}]\label{lem:adj-inv}
The weights defined above have some restriction stated below$:$
\begin{enumerate}
 \item If $(\alpha_i,\beta_i)$ and $(\alpha_{i+1},\beta_{i+1})$
       denotes the Seifert invariants assigned to adjacent arcs 
       in some weighted arc or circle, then 
   \begin{equation*}
      \det \begin{pmatrix}
             \alpha_i & \beta_i \\
             \alpha_{i+1} & \beta_{i+1}
            \end{pmatrix}
          =\pm 1.
   \end{equation*}
 \item If $\{ b';(\alpha_1,\beta_1),\ldots,(\alpha_n,\beta_n);b''\}$
       denotes the weights of a weighted arc, then 
       $$
         b'\alpha_1 + \beta_1=\pm 1,\qquad
            b''\alpha_n + \beta_n=\pm 1.
       $$
\end{enumerate}
\end{lem}

\begin{prop}[\cite{Fn:circleI}] \label{prop:circ-class}
Let $M_1$ and $M_2$ be compact oriented $4$-manifolds
$($possibly with boundary$)$ with $S^1$-action such that 
their weighted orbit spaces are isomorphic. Then $M_1$ is 
orientation-preservingly equivariantly homeomorphic 
to $M_2$.
\end{prop}

\begin{proof}
This can be proved by the same methods as \cite{Fn:circleI}
by using Propositions (3.6), (3.9), Lemma (6.1) and Theorem (6.2)
in \cite{Fn:circleI} together with Proposition (9.1) in
\cite{Fn:circleII}.
\end{proof}

Next we discuss $S^1$-actions on $D^2$-bundles over $S^2$
(see \cite{Or:Seif} for details),
which will be used in the subsequent sections to construct local 
structures of collapsing in several cases.

Let $S^2=B_1\cup B_2$ be the union of its upper and lower 
hemispheres, and $D_j^2=D^2(1)$ the unit disks, $j=1,2$. 
Put polar coordinates on $B_j\times D_j^2$.
For an arbitrary integer $\omega$, we consider the 
$D^2$-bundle over $S^2$ defined as 
$$
   S^2\tilde\times_{\omega} D^2 :=
           B_1\times D_1^2\bigcup_{f_{\omega}} B_2\times D_2^2,
$$
where $f_{\omega}:\partial B_1\times D_1^2\to \partial B_2\times D_2^2$
is the gluing homeomorphism defined by
$$
  f_{\omega}(e^{i\varphi}, se^{i\phi})=
         (e^{-i\varphi}, se^{i(-\omega\varphi+\phi)}).
$$
This bundle has Euler number $\omega$, which coincides with 
the self-intersection number of the zero section. Note that
the boundary of $S^2\tilde\times_{\omega} D^2$ is homeomorphic to 
the lens space $L(|\omega|,1)$.
Conversely any $D^2$-bundle over $S^2$ with the self-intersection
number $\omega$ of its zero section is homeomorphic to 
$S^2\tilde\times_{\omega} D^2$.

For relatively prime integers $a_j$ and $b_j$, define an 
$S^1$-action $\psi_j$ on $B_j\times D_j^2$ by
$$
  e^{i\theta}(re^{i\varphi},se^{i\phi}) =
       (re^{i(\varphi+a_j\theta)},se^{i(\phi+b_j\theta)}).
$$
Those actions $\psi_j$ define an $S^1$-action, denoted
$\hat \psi(a_1,b_1)$, on the bundle $S^2\tilde\times_{\omega} D^2$
if and only if $a_2= -a_1$ and $b_2=-\omega a_1+b_1$.

We show that some $S^1$-actions with simple orbit spaces 
essentially come from 
$S^1$-actions on $S^2\tilde\times_{\omega} D^2$ with suitable
$\omega:$

\begin{prop} \label{prop:plumbing}
Let $M$ be a compact oriented $4$-manifold with boundary and with
a locally smooth $S^1$-action $\psi_0$ whose orbit space $M^*$
is homeomorphic to $D^3$. Suppose that 
\begin{enumerate}
 \item the closure of any component of the exceptional locus $E^*$
       of $\psi_0$ meets  $F^*\cap\interior M^*;$
 \item $F^*$ consists of $m$ points of $\interior M^*$ and $n$
       pieces of disjoint disks on $\partial M^*$ with 
       $m+n=2$, $m,n\ge 0$.
\end{enumerate}
Then there is an $S^1$-action $\psi$ on 
$S^2\tilde\times_{\omega}D^2$ as described above which is
equivariantly homeomorphic to $(M,\psi_0)$, where $\omega$
is determined by the weighted orbit invariants of $\psi_0$ as
follows$:$
\begin{enumerate}
 \item[(a)] If $m=2$ and $E^*$ is empty, then $\omega\in\{ 1,2\};$
 \item[(b)] If $m=2$ and $E^*$ is an oriented arc with Seifert 
            invariants $(\alpha'',\beta'')$ joining a point of $F^*$
            to a point of $\partial M^*$, then 
            $|\omega|=\alpha''\pm 1;$
 \item[(c)] Suppose that $m=2$ and $E^*$ consists of two oriented 
            arcs $J'$ and $J''$  with Seifert invariants 
            $(\alpha',\beta')$ and $(\alpha'',\beta'')$ respectively,
            where $J'$ joins a point of $\partial M^*$
            to a point of $F^*$ and $J''$ joins the other point of
            $F^*$ to a point of $\partial M^*$. Then 
            $|\omega|=|\alpha'\pm\alpha''|;$
 \item[(d)] If $m=2$ and $E^*$ is an arc joining the two points of
            $F^*$, then $\omega=0;$
 \item[(e)] If $m=2$ and $E^*$ consists of an oriented arc $J$ joining
            the two points of $F^*$ with Seifert invariants 
            $(\alpha,\beta)$, and an oriented arc from the end point
            of $J$ to a point of $\partial M^*$ with Seifert
            invariants $(\alpha'',\beta'')$, then 
            $|\omega|=(\alpha''\pm 1)/\alpha;$
 \item[(f)] Suppose that $m=2$ and $E^*$ consists of three oriented
            arcs $J'$, $J$ and $J''$ with Seifert invariants 
            $(\alpha',\beta')$, $(\alpha,\beta)$  and 
            $(\alpha'',\beta'')$ respectively, where $J'$ 
            joins a point of $\partial M^*$
            to a point of $F^*$, $J$ joins the end point of $J'$ 
            to the other point of $F^*$ and $J''$ joins the end point of
            $J$ to a point of $\partial M^*$. Then 
            $|\omega|=|\alpha'\pm \alpha''|/\alpha;$
 \item[(g)] If $m=n=1$ and $E^*$ is empty, then $|\omega|=1;$
 \item[(h)] If $m=n=1$ and $E^*$ is an arc from a point of 
            $\partial M^*$ to $F^*\cap\interior M^*$ with 
            Seifert invariants  $(\alpha',\beta')$, then 
            $|\omega|=\alpha';$
 \item[(i)] If $n=2$, then $\omega=0$.
\end{enumerate}
\end{prop}

\begin{center}
\begin{tikzpicture}
[scale = 0.5]
\draw(0,0) circle[x radius=1.3,y radius=1];
\fill (-0.5,0)  circle (2pt);
\fill (0.5,0) circle (2pt);
\fill (0,-1) circle (0pt) node [below]{$(a)$};
\draw(4,0) circle[x radius=1.3,y radius=1];
\fill (3.5,0)  circle (2pt);
\fill (4.5,0) circle (2pt);
\draw(4.5,0)--(5.35,0);
\draw[->](4.5,0)--(5,0);
\fill (4,-1) circle (0pt) node [below]{$(b)$};
\draw(8,0) circle[x radius=1.3,y radius=1];
\fill (7.5,0)  circle (2pt);
\fill (8.5,0) circle (2pt);
\draw(8.5,0)--(9.35,0);
\draw[->](8.5,0)--(9,0);
\draw(7.5,0)--(6.65,0);
\draw[->](6.65,0)--(7.15,0);
\fill (8,-1) circle (0pt) node [below]{$(c)$};
\draw(12,0) circle[x radius=1.3,y radius=1];
\fill (11.5,0)  circle (2pt);
\fill (12.5,0) circle (2pt);
\draw(11.5,0)--(12.5,0);
\draw[->](11.5,0)--(12.15,0);
\fill (12,-1) circle (0pt) node [below]{$(d)$};
\draw(0,-4) circle[x radius=1.3,y radius=1];
\fill (-0.5,-4)  circle (2pt);
\fill (0.5,-4) circle (2pt);
\draw(0.5,-4)--(1.35,-4);
\draw[->](0.5,-4)--(1.0,-4);
\draw(-0.5,-4)--(0.5,-4);
\draw[->](-0.5,-4)--(0.1,-4);
\fill (0,-5) circle (0pt) node [below]{$(e)$};
\draw(4,-4) circle[x radius=1.3,y radius=1];
\fill (3.5,-4)  circle (2pt);
\fill (4.5,-4) circle (2pt);
\draw(4.5,-4)--(5.35,-4);
\draw[->](4.5,-4)--(5.05,-4);
\fill (4,-5) circle (0pt) node [below]{$(f)$};
\draw(6.6,-3.3) -- (9,-3.3);
\draw(6.6,-4.8) -- (9,-4.8);
\draw [thick] (6.6,-3.3) to [out=225, in=135] (6.6,-4.8);
\draw [thick,dotted] (6.6,-3.3) to [out=-45, in=45] (6.6,-4.8);
\filldraw [fill=lightgray,thick] (9,-3.3) to [out=225, in=135] (9,-4.8);
\filldraw [fill=lightgray,thick] (9,-3.3) to [out=-45, in=45] (9,-4.8);
\fill (7.8,-4) circle (2pt);
\fill (9.2,-4) circle (0pt)  node [right]{$F^*$};
\fill (8,-5) circle (0pt) node [below]{$(g)$};
\draw(11.5,-3.3) -- (13.9,-3.3);
\draw(11.5,-4.8) -- (13.9,-4.8);
\draw [thick] (11.5,-3.3) to [out=225, in=135] (11.5,-4.8);
\draw [thick,dotted] (11.5,-3.3) to [out=-45, in=45] (11.5,-4.8);
\filldraw [fill=lightgray,thick] (13.9,-3.3) to [out=225, in=135] (13.9,-4.8);
\filldraw [fill=lightgray,thick] (13.9,-3.3) to [out=-45, in=45] (13.9,-4.8);
\fill (12.6,-4) circle (2pt);
\draw(11.2,-4)--(12.6,-4);
\draw[->](11.2,-4)--(12,-4);
\fill (14.1,-4) circle (0pt)  node [right]{$F^*$};
\fill (12.6,-5) circle (0pt) node [below]{$(h)$};
\fill (6,-7) circle (0pt) node[below]{Figure 1};
\end{tikzpicture}
\end{center}

%

%
\begin{proof}
This follows from Proposition \ref{prop:plumbing} and the 
examples of $S^1$-actions on 
the $D^2$-bundles over $S^2$ given in Section 4 of \cite{Fn:circleI}.
\end{proof}

\begin{rem}
In Proposition \ref{prop:plumbing}, if $m+n=1$, then $M$ is 
homeomorphic to $D^4$
and $(M,\psi_0)$ is equivariantly homeomorphic to either 
$(D^4,\psi_{a,b})$ for some relatively prime integers $a$ and $b$
$(m=1)$ or $(D^2\times D^2, 1\times \text{rotation})$
$(n=1)$.
\end{rem}


\part{Reproducing collapsed  $4$-manifolds} \label{part:collapse}

 In this Part \ref{part:collapse}, we use some results from 
Parts \ref{part:alex} and \ref{part:cap}, for instance
Theorem \ref{thm:soul} and Theorem \ref{thm:cap}.

\section{Rescaling argument} \label{sec:rescal}

\bigskip

 We denote by  $\tau(a_1,\ldots,a_k|\epsilon)$ a function depending on
a priori constants, $a_1,\ldots,a_k$ and $\epsilon$ satisfying 
$\lim_{\epsilon\to 0}\tau(a_1,\ldots,a_k|\epsilon)=0$
for each fixed $a_1,\ldots,a_k$.

 For a closed set $S$ in a metric space, we denote by
$A(S;r_1,r_2)$ the annulus
$B(S,r_2)-\interior B(S,r_1)$.

 Let a sequence of pointed complete $n$-dimensional Riemannian
manifolds $(M_i^n,p_i)$ with $K \ge -1$ converge to a pointed
$k$-dimensional Alexandrov space $(X^k,p)$ with respect to the
pointed Gromov-Hausdorff convergence, where $k \le n-1$.  
In this section, we establish a basic result to study the topology 
of the metric ball $B(p_i,r)$ for any sufficiently large 
$i$ compared to a fixed small $r > 0$.
If $p$ is a regular point of $X^k$, then $B(p_i,r)$ fibers over 
$B(p,r)$ with almost nonnegatively curved fibre (\cite{Ym:conv}).
Hence from now on, we assume that $p$ is a singular point.
\par
  The purpose of this section is to prove the following
result.

\begin{thm}\label{thm:rescal}
Suppose $n=4$ and that $B(p_i,r)$ is not homeomorphic to $D^4$
for each $r>0$ and for any sufficiently large $i$. Then 
there exist an $r=r_p>0$ and sequences $\delta_i\to 0$ and 
$\hat p_i\in B(p_i,r)$ such that 
\begin{enumerate}
  \item $\hat p_i\to p$ under the convergence 
       $B(p_i,r)\to B(p,r)$;
  \item  $B(p_i,r)$ is homeomorphic to $B(\hat p_i,R\delta_i)$ for every 
      $R\ge 1$ and large $i$ compared to $R$;
  \item  for any limit $(Y,y_0)$ of $(\frac{1}{\delta_i}M_i^4, \hat p_i)$, 
     we have  $\dim Y\ge k+1$.
 \end{enumerate}
\end{thm}

\begin{rem}
\begin{enumerate}
 \item The homeomorphism in Theorem \ref{thm:rescal} $(2)$ is given 
       by the flow curves of a gradient-like vector field of 
       the distance function $d_{\hat p_i}=d(\hat p_i,\,\cdot\,);$
 \item The rescaling constant $\delta_i$ can be thought of as 
       the size of the (singular) fibre which is not visible yet.
 \item The above theorem has been proved in \cite{SY:3mfd} under the
       hypothesis that $k=2$ and $\diam(\Sigma_p)<\pi;$
  \item A generalization of Theorem \ref{thm:rescal} actually 
     holds in the general dimension. This will appear in a 
     forthcoming paper.
\end{enumerate}
\end{rem}

Let $(Y,y_0)$ be a complete noncompact Alexandrov space with
nonnegative curvature given in Theorem \ref{thm:rescal},
and $Y(\infty)$ denote the ideal boundary of $Y$ with 
the Tits metric.

The following lemma has been proved in Lemma 3.7 of \cite{SY:3mfd}.

\begin{lem} \label{lem:expand}
  There is an expanding  map $\Sigma_p\to Y(\infty)$.
  In particular, $\dim Y(\infty) \ge \dim \Sigma_p$.
\end{lem}

Actually we have the following lemma in more general setting.
The proof is similar to Lemma 3.7 of \cite{SY:3mfd}, and 
hence omitted.

\begin{lem} \label{lem:expanding2}
Suppose that a sequence $(X_i^n,x_i)$ of $n$-dimensional 
complete noncompact pointed Alexandrov spaces with 
curvature $\ge -\epsilon_i$, $\lim_{i\to\infty}\epsilon_i=0$,
converges to a pointed space $(Z,z_0)$. For $\mu_i>0$ converging 
to $0$, let $(W,w_0)$ be any limit of $(\frac{1}{\mu_i}X_i^n, x_i)$.
Then we have an expanding map $\Sigma_{z_0}\to W(\infty)$.
\par
  In particular, $\dim W(\infty) \ge \dim Z - 1$.
\end{lem}

We go back to the situation of Theorem \ref{thm:rescal}.
Let  $S$ be a soul of $Y$.  

\begin{prop}\label{prop:ideal-dims}
   $\dim S\le \dim Y - \dim X$.
\end{prop}

This immediately follows from Proposition \ref{prop:ideal}
and Lemma \ref{lem:expand}.
\par

As a preliminary to the proof of Theorem \ref{thm:rescal},
we recall an argument in Section 3 of \cite{SY:3mfd}.  

For a compact set $A\subset X^k$ and $\epsilon>0$, let
$\beta_A(\epsilon)$ denote the maximal number of points in $A$
having distance $\ge\epsilon$. 

Take  a $\mu_i$-approximation $\phi_i : B(p,1/\mu_i)\to B(p_i,1/\mu_i)$ 
with $\phi_i(p)=p_i$, where $\mu_i\to 0$ as $i \to \infty$.  
For any $\epsilon > 0$, we take an 
$\epsilon$-net
$\{\xi^j\}_{j=1,\dots,\beta_{\Sigma_p}(\epsilon)}$ of $\Sigma_p$.  For
a small enough $r > 0$ compared to $p$ and $\epsilon$, take $x^j
\in \bdy B(p,r)$, $j = 1, \dots, \beta_{\Sigma_p}(\epsilon)$, such
that the direction $\eta^j$ at $p$ of a minimal segment from $p$ to
$x^j$ satisfies 
$\angle(\xi^j,\eta^j) < \tau(r) < \epsilon^2$.  
Set $x^j_i = \phi_i(x^j)$ and
\begin{equation}
  f_i = f_{\epsilon,r,i} = 
  \frac{1}{\beta_{\Sigma_p}(\epsilon)}\sum_j\frac{1}{r}d(x^j_i,\,\cdot\,)
     \colon M_i^n \to \R. \label{eq:f_i}
\end{equation}
Consider the measure 
$m_{\epsilon}=\frac{1}{\beta_{\Sigma_p}(\epsilon)} \sum_j \delta_{\xi^j}$,
where $\delta_x$ is the Dirac $\delta$-measure.
We have a sequence $\epsilon_\ell \to 0$ such that the measure
$m_{\epsilon_{\ell}}$ 
converges to some Borel measure $m_p$ on
$\Sigma_p$ as $\ell \to\infty$ in the weak${}^*$ topology.
Note that 
\begin{enumerate}
  \item $m_p$ coincides with the normalized Hausdorff measure 
        over $\Sigma_p$ if $k \le 2$;
  \item $m_p$ is regular in the sense that  $m_p(U)>0$ 
        for any non-empty open subset $U$ of $\Sigma_p$.
\end{enumerate}
The latter property comes from the Bishop-Gromov volume
comparison theorem (cf.\cite{Ym:conv}) for $\Sigma_p$.\par

We make the identification $\Sigma_p=\Sigma_p\times \{ 1\}\subset K_p$.
Let 
$\psi_r:B(o_p,1/\nu_r)\to B(p,\nu_r;\frac{1}{r}X^k)$ be a 
$\nu_r$-approximation such that $\psi_r(o_p)=p$ and
$\psi_r(\xi^j)=x^j$, where
$\lim_{r \to 0} \nu_r = 0$, and let
\begin{equation*}
 \bar f = \int_{\xi \in \Sigma_p} d(\xi,\,\cdot\,) \; dm_p  : K_p
\to \R. 
\end{equation*}
We then have
\begin{equation}
  |f_i \circ  \phi_i \circ \psi_r - \bar f|
    < \tau(\ell,r|1/i)+\tau(\ell| r)+\tau(1/\ell), \label{eq:aver}
\end{equation}
on any fixed compact set of $K_p$, simply denoted as
$$
  \lim_{\ell\to\infty}\lim_{r\to 0}\lim_{i\to\infty} 
       \bar f_i\circ\phi_i\circ\psi_r = \bar f.
$$

\begin{lem}[\cite{SY:3mfd}, Lemmas 3.4, 3.5] \label{lem:locmax-general}
If $\bar f$ takes a strictly local maximum 
at the vertex $o_p\in K_p$, then we can find $\delta_i\to 0$ and 
$\hat p_i\in B(p_i,r)$ satisfying the conclusion of 
Theorem \ref{thm:rescal}.
\end{lem}

We give a sketch of the essential idea of the proof of
Lemma \ref{lem:locmax-general}.
\par

From the assumption,
one can take as the point $\hat p_i$ as a local maximum point of
$f_i$ converging to $o_p$, and as $\delta_i$
to be the maximum distance between $\hat p_i$ and
the critical point set (\cite{GS:sphere}, \cite{Gr:betti}) 
of $d_{\hat p_i}$ within  $B(p,r)$.
Let $\hat q_i$ be a critical point of $d_{\hat p_i}$ within
$B(p,r)$ realizing $\delta_i$, and let $z_0\in Y$ be the 
limit of $\hat q_i$ under the convergence
$(\frac{1}{\delta_i} M_i, \hat p_i)\to (Y,y_0)$.
Let $x_{\infty}^j\in Y(\infty)$ denote the point at infinity 
defined by the limit ray, say $\gamma_j$ from $y_0$ 
of the geodesic $\hat p_ix_i^j$
under the convergence 
$(\frac{1}{\delta_i} M_i, \hat p_i)\to (Y,y_0)$. 
For fixed $\epsilon=\epsilon_{\ell}$ and $r$, $f_i$ 
after some normalization converges to the
function 
$g=\frac{1}{\beta_{\Sigma_p}(\epsilon)}\sum_j b_{x_{\infty}^j}$,
where $b_{x_{\infty}^j}$ can be thought of as 
a generalized Busemann function associated with $\gamma_j$.
Let $v\in\Sigma_{y_0}$ be a direction of a minimal geodesic from 
$y_0$ to $z_0$, and let 
$\theta_j$ denote the angle between $v$ and $\gamma_j$.
Since $z_0$ is a critical point of $d_{y_0}$ and since
$Y$ has nonnegaitve curvature, 
we obtain that $\theta_j\ge\pi/2$.
On the other hand, since $y_0$ is a maximum point of $g$,
we have
$$
0\ge g_{y_0}'(v)= - \frac{1}{\beta_{\Sigma_p}(\epsilon)}\sum_j 
                      \cos\,\theta_j \ge 0.
$$
It follows that $\theta_j=\pi/2$.
Therefore $\{ w_j\}_{1\le j\le\beta_{\Sigma_p}(\epsilon)}$,
which have pairwise distance $\ge\epsilon$,
must be contained in a metric sphere
$\partial B(v,\pi/2;\Sigma_{y_0})$. 
Since this holds for any sufficiently small $\epsilon$ with 
$\beta_{\Sigma_p}(\epsilon)\sim\text{\rm const}\epsilon^{\dim X - 1}$,
we can conclude $\dim Y\ge k+1$.
\medskip

\begin{lem} \label{lem:loc-max}  
Suppose one of the following two cases.
\begin{enumerate}
 \item $k=2$ and $\diam(\Sigma_p)<\pi;$
 \item  $\diam (\Sigma_p)\le \pi/2$.
\end{enumerate}
Then the function $\bar f$ takes a strictly
local maximum at  $o_p\in K_p$. In particular, 
Theorem \ref{thm:rescal} holds in those cases.
\end{lem}

\begin{proof}
The case (1) is proved in \cite{SY:3mfd}, and 
Suppose the case (2) and 
that there exists a sequence $x_i\to o_p$ satisfying 
$\bar f(x_i)\ge\bar f(o_p)$ and take $\xi_i\in \Sigma_p$ with
$1=d(o_p,\xi_i)=d(o_p,x_i)+d(x_i, \xi_i)$.
Putting  $r_i=d(o_p,x_i)$, we have 
$d_{\xi}(x_i)\le 1-r_i/2$ for every 
$\xi\in \Sigma_p$ with $d(\xi, \xi_i)\le\delta$,
where $\delta$ is a small positive constant independent of $i$.
Now by the assumption, 
$d_{\eta}(x_i)\le 1+O(r_i^2)$ for all $\eta\in\Sigma_p$.
It follows  that 
  \begin{align*}
   \bar f(x_i) &= \int_{B(\xi,\delta;\Sigma_p)}d(\xi, x_i)\,dm_p
                 + \int_{B(\xi,\delta;\Sigma_p)^c} d(\xi, x_i)\,dm_p \\
         &\le (1-r_i/2)m_p(B(\xi,\delta;\Sigma_p)) + 
              (1+O(r_i^2))m_p(B(\xi,\delta;\Sigma_p)^c) \\
 &\le (1+O(r_i^2))m_p(\Sigma_p) - \frac{r_i}{2} m_p(B(\xi,\delta;\Sigma_p))\\
     & < m_p(\Sigma_p)=\bar f(o_p),
  \end{align*}
for large $i$. This is a contradiction.
\end{proof}

\begin{lem} \label{lem:split-rescal}
Suppose $k\le 3$. 
If  $\diam(\Sigma_p)=\pi$, then  Theorem \ref{thm:rescal} holds.
\end{lem}

\begin{proof}
We first assume that $\dim X=3$ and $p$ is not a boundary point of $X^3$. 
Note that $K_p=K(\Sigma)\times \R$, where $\Sigma$ is a circle of
length $<2\pi$. 
To apply the previous argument to $K(\Sigma)$, 
let us consider the function
$$
   \tilde f = \int_{\xi\in\Sigma} d(\xi,\,\cdot\,) d\cal H 
           : K(\Sigma) \to \R,
$$
where 
$d\cal H$ denotes the normalized Hausdorff measure of $\Sigma$.
By Lemma \ref{lem:loc-max}, 
$\tilde f$ takes a strictly local maximum at $o_p$.
Take $\tau_r>0$ with $\lim_{r\to 0}\frac{\tau_r}{\nu_r}=\infty$,
and for a point $v_r$  on the line
$o_p\times \R\subset K(\Sigma)\times\R$ 
with $d(v,o_p)=1/\tau_r$,
let  $w_r:=\psi_{r}(v_r)$ and 
$w^i_r:=\phi_i(w_r)$.
Let $S_r^i$ be the three-dimensional submanifold defined as 
$$
   S_r^i := \{ d_{w_r^i}=d_{w_r^i}(p_i)\}\cap \{ d_{p_i}\le r\}.
$$
We construct a function $\tilde f_i$ on $M_i^4$  by using an 
$\epsilon_{\ell}$-net $\{\xi^j\}$ of $\Sigma$, 
$\{ x^j\}\subset\partial B(p,r)$ and  $\{ x_i^j=\phi_i(x^j)\}$
in a way similar to the above construction of $f_i$ in 
\eqref{eq:f_i}:
\begin{equation*}
 \tilde f_i = \tilde f_{\epsilon_{\ell},r,i} = 
  \frac{1}{\beta_{\Sigma_p}(\epsilon_{\ell})}\sum_j 
    \frac{1}{r} d(x^j_i,\,\cdot\,)\colon M_i^4 \to \R. 
\end{equation*}
Similarly to \eqref{eq:aver}, we have
$$
  \lim_{\ell \to \infty} \lim_{r\to 0} \lim_{i\to\infty} \tilde f_i \circ
  \phi_i \circ \psi_r = \tilde f.
$$
Now take a local maximum point $\hat p_i\in S_r^i$ of
the function $\tilde f_i$ restricted to $S_r^i$  such that 
$\hat p_i\to p$. This is possible because $\tilde f$ 
takes a strictly local maximum at $o_p$ and 
$(\frac{1}{r}S_r^i,\tilde p_i)\to (K_1(\Sigma), o_p)$ as $i\to\infty$,
$r\to 0$, where $K_1(\Sigma)$ denotes the unit subcone
of $K(\Sigma)$. 
Let 
$\delta_i$ be the maximum distance from $\hat p_i$ to 
the critical point set of $d_{\hat p_i}$ on $B(p_i,r)$.
Note that $\delta_i\to 0$ since $\hat p_i\to p$. 
Let $(Y,y_0)$ be any limit of 
$(\frac{1}{\delta_i}M_i^4, \hat p_i)$. 
By the splitting theorem, $Y$   is isometric to 
a product $Y_0\times\R$.
Let  $q_i\in B(p_i,r)$ be a critical point of $d_{\hat p_i}$ with 
$d(\hat p_i,q_i)=\delta_i$.
We may assume that $q_i$ converges to a point $ z_0\in Y$ 
under the convergence 
$(\frac{1}{\delta_i}M_i^4, \hat p_i) \to (Y,y_0)$.
The fact that $z_0$ is a critical point of $d_{y_0}$ yields 
$\pi_2(y_0)=\pi_2(z_0)$, where $\pi_2:Y_0\times\R\to\R$ is the projection.
For  $a\gg 1$, take $b_i\to\infty$ such that 
$g_i(\hat p_i)=a$ for  $g_i:=\tilde f_i - b_i$.
We may assume that $g_i$ converges to a function $g_{\infty}$
on $Y$. Since $g_i$ also takes a local maximum on $S_r^i$ at $\hat p_i$,
it follows from 
construction that  $g_{\infty}$ takes a maximum on 
$Y_0\times \{ \pi_2(y_0)\}$ at $y_0$.
Now the argument of Lemma 3.5 in \cite{SY:3mfd} yields   
$\dim Y\ge 4$, and the theorem holds.  
\par

The proof for the case that $\Sigma=[0,\ell]$, $\ell<\pi$ or 
the case that $\dim X=2$, $\Sigma_p=[0,\pi]$ 
is similar to the above argument. \par
  
Finally let us consider the case that $\Sigma=[0,\pi]$, that is,
$(K_p,o_p)=(\R^2\times [0,\infty), 0)$. 
Take  $v_{1,r},v_{2,r}\in\R^2\times 0$ 
with $d(o_p,v_{\alpha,r})=\tau_r$, 
and $\angle(v_{1,r} o_p v_{2,r})=\pi/2$.
Set  $w_{\alpha,r}^i=\phi_i\circ\psi_r(v_{\alpha,r})$  
as before, and 
consider the two-dimensional submanifold
$$
   S_r^i := \{ d_{w_{1,r}^i}=d_{w_{1,r}^i}(p_i)\}\cap
               \{ d_{w_{2,r}^i}=d_{w_{2,r}^i}(p_i)\}
                \cap \{ d_{p_i}\le r\}.
$$
Then we complete the proof by a similar argument.
\end{proof}

 In the rest of this section, we consider the remaining case 
that $\pi>\diam(\Sigma_p)>\pi/2$. From now on, we assume $n=4$.
In view of Lemma \ref{lem:split-rescal} and \cite{SY:3mfd}, we may assume 
$\dim X=3$.
For $\xi,\eta\in \Sigma_p$ with $d(\xi,\eta)=\diam(\Sigma_p)$,
we put $x=\exp r\xi^{\prime}$, $y=\exp r\eta^{\prime}$,
where $\xi^{\prime}$, $\eta^{\prime}$ are sufficiently close to 
$\xi$ and $\eta$ respectively, and $r$ is sufficiently small.
Take $0<\epsilon_1\ll \epsilon\ll r\le \mbox{const}_p$,
and put $w:=\exp(\epsilon \xi^{\prime})\in px$,
$$
  U_{\epsilon_1}(p,w) := B(w,\epsilon_1)\cap \partial B(p,\epsilon).
$$
We may assume that for some fixed constant $\delta < \diam(\Sigma_p)-\pi/2$
 
\begin{enumerate}
  \item $\tilde\angle xzy >\pi/2 + \delta$ \quad for any $z\in B(p,\epsilon)$;
  \item $\angle pzy - \tilde\angle pzy < \tau(r)$
         \quad for any $z\in B(p,\epsilon)$;
  \item $\angle pzy\le \pi/2 - \delta$ \quad  for all 
     $z\in  U_{\epsilon_1}(p,w)$.
\end{enumerate}

For $x_i, y_i\in M_i^4$ with $x_i\to x$, $y_i\to y$, 
we  have $\tilde\angle x_iz_iy_i >\pi/2 + \delta$ for every 
$z_i\in B(p_i,\epsilon)$ and large $i$.
Let $w_i\in M_i^4$ be such that $w_i\to w$, and 
$$
     U_{\epsilon_1}(p_i,w_i) := \partial B(p_i,\epsilon)\cap B(w_i,\epsilon_1).
$$
Applying  Lemma \ref{lem:split-rescal}
to the convergence $(M_i^4, w_i)\to (X^3,w)$,
we have sequences $\delta_i\to 0$ and $\hat w_i\to w$ satisfying 
\begin{enumerate}
  \item[(a)] $B(w_i,\epsilon_1)\simeq B(\hat w_i,R\delta_i)$ for every 
     $R\ge 1$ and large $i$;
  \item[(b)] for any limit $(Y,y_0)$ of $\frac{1}{\delta_i}(M_i^4,\hat w_i)$, 
     we have $\dim Y=4$.
\end{enumerate}
Note that $Y$ is isometric to a product $\R\times Y_0$.
By Lemma \ref{lem:expand},  $\dim Y(\infty)\ge 2$.
It follows from Proposition \ref{prop:ideal}
that the dimension of the soul $S_0$  of $Y_0$ is $0$ or $1$.
Since $Y_0$ has no boundary, the generalized soul theorem 
in \cite{SY:3mfd} implies that
$$
   B(S_0,R;Y_0)\simeq D^3 \quad \mbox{or} 
    \quad S^1\times D^2
$$
for large $R>0$.

\begin{lem} \label{lem:u_1} We have 
 \begin{align*}
  U_{\epsilon_1}(p_i,w_i)& \simeq B(S_0,R;Y_0) \\
                   & \simeq D^3 \quad \mbox{or} \quad
           S^1\times D^2.  
 \end{align*}
\end{lem}

\begin{proof}
Let $R_1\gg R$ be  large numbers,
and let $q_i$ be a point on $p_i\hat w_i$ with $d(q_i,\hat w_i)=\delta_iR_1$.
Let $q\in Y$ be the limit of $q_i$ under the convergence 
$(\frac{1}{\delta_i} M_i^4,\hat w_i)\to (Y,y_0)$.
From the standard Morse theory for distance functions together with 
the proof of Lemma \ref{lem:split-rescal},
 \begin{align*}
   U_{\epsilon_1}(p_i,w_i) & \simeq U_{\epsilon_1}(p_i,\hat w_i) \\
             & \simeq U_{\delta_i R}(p_i,\hat w_i)\\
             & \simeq U_{\delta_i R}(q_i,\hat w_i).
 \end{align*}
It should be noted here that to show the second $\simeq$ above, we actually
need to take a smooth approximation $\tilde d_{p_i}$ of $d_{p_i}$. 
Since this is only a technical point, we omit the detail of 
this argument.\par

Now consider the convergence $(\frac{1}{\delta_i} M_i^4,\hat w_i)\to
(Y,y_0)$ and  note that  $(d_{y_0},d_q)$ is regular near 
$(d_{y_0},d_q)^{-1}(R, R_1)$.
It follows from Theorem \ref{thm:stability-respect} that 
 \begin{align*}
   U_{\delta_i R}(q_i,\hat w_i) & \simeq U_{R}(q, y_0) \\
             & \simeq B(y_0,R;Y_0) \\
             & \simeq B(S_0,R;Y_0).
 \end{align*}
This completes the proof.
\end{proof}

To complete the proof of Theorem \ref{thm:rescal}, we need to 
determine the topology of $B(p_i,r)$. This is based on 
the following lemma.

\begin{lem}\label{lem:flow-rescal}
There exist a positive number $\epsilon$ independent of $i$ and 
a unit vector field $V_i$ on $B(p_i,\epsilon)$
satisfying:
 \begin{enumerate}
   \item  The flow curves of $V_i$ are gradient-like for $d_{y_i}$;
   \item  Let $h_{z_i}^i(t)$ denote the flow curve of $V_i$ starting from 
      a point $z_i$. 
      Then we have 
    \begin{enumerate}
      \item $h_{z_i}^i(t)$ is transversal to $\partial B(p_i,\epsilon)$
         for every $z_i\in U_{\epsilon_1}(p_i,w_i)$;
      \item $$
    |\angle(V_i(h^i_{z_i}(t),(p_i)_{h^i_{z_i}(t)}^{\prime})-
          \angle(y_{\phi_i(z_i)}^{\prime},p_{\phi_i(z_i)}^{\prime})|
                     < \tau(\epsilon,1/i),
            $$
         for every $z_i\in U_{\epsilon_1}(p_i,w_i)-U_{\epsilon_1/2}(p_i,w_i)$. 
    \end{enumerate}
 \end{enumerate}
\end{lem}

\begin{proof}
We construct a local gradient-like vector field for $d_{y_i}$ on a small 
neighborhood of each point 
of $B(p_i,\epsilon)$ and use a partition of unity to get a 
global one.\par

It is easy to construct a vector field $W_i$ on a small neighborhood $A_i$
of a broken geodesic  $w_ip_i\cup p_iy_i$ so as to satisfy 
$(1)$, $(2)$-(a) and that the flow curve of $V_i$ through $w_i$ passes $p_i$.
Extend $W_i$ on $B(p_i,\epsilon_2)$, $\epsilon_2=10^{-10}\epsilon_1$
so as to satisfy $(1)$.
For each $q_i\in B(p_i,\epsilon)-A_i-B(p_i,\epsilon_2)$, 
let $v_{q_i}$ denote a unit vector at 
$q_i$ tangent to a minimal geodesic from $q_i$ to $y_i$.
The property $(2)$ stated in the preceding paragraph of Lemma \ref{lem:u_1}
yields that
\begin{equation} 
    |\angle(v_{q_i},(p_i)_{q_i}^{\prime})-\angle(y_q^{\prime},p_{q}^{\prime})|
       < \tau(\epsilon,1/i), \label{eq:vf-alex}
  \end{equation}
where $q$ is a limit point of $q_i$.
Let $V_{q_i}$ be a smooth extension of $v_{q_i}$ to a small 
neighborhood of $q_i$.
By patching $\{ V_q\}$ with a partition of unity, 
we construct a vector field $\tilde W_i$ 
on $B(p_i,\epsilon)- A_i-B(p_i,\epsilon_2)$ satisfying \eqref{eq:vf-alex}
in  place of $v_{q_i}$. 
To obtain the  required vector field, it suffices to patch 
$W_i$ and $\tilde W_i$.
\end{proof}

\begin{lem} \label{lem:dim3balltop}
$$
     B(p_i,\epsilon)\simeq U_{\epsilon_1}(p_i,w_i)\times I.
$$
\end{lem}

\begin{proof}
Let $\ell_i := p_iw_i\cup p_iy_i$ and consider a smooth approximation 
$f_i$ of the function $d(\ell_i, \cdot )$, which is regular near 
$f_i^{-1}(\epsilon_1)$ for some small $\epsilon_1$ independent of $i$.
Let $S_i=B(p_i,\epsilon)\cap \{ f_i=\epsilon_1\}$.
By using  Lemma \ref{lem:flow-rescal}, one can construct 
a gradient-like vector field $W_i$ for $d_{y_i}$ on 
$B_i:= \{ f_i \le \epsilon_1\}\cap B(p_i,\epsilon)$ such that
$W_i$ is tangent to $S_i$. To do this, it suffices to take the tangential 
component of the vector field $V_i$ to $S_i$ in the orthogonal decomposition 
of $V_i$, where $V_i$ is the vector field given in 
Lemma \ref{lem:flow-rescal}.
Thus we have $B_i\simeq U_{\epsilon_1}(p_i,w_i)\times I$.
Since $f_i$ is regular on $B(p_i,\epsilon)-\interior B_i$
and the gradient flow of it is transversal to 
$\partial (B(p_i,\epsilon)-\interior B_i)$, 
we conclude that $B(p_i,\epsilon)\simeq B_i$. 
\end{proof}

\begin{proof}[Proof of Theorem \ref{thm:rescal} in the case of 
               $\pi/2 < \diam(\Sigma_p) < \pi$]

By the previous lemma together with the topological assumption on 
$B(p_i,r)$,  $B(p_i,\epsilon)$ is homeomorphic to $S^1\times D^3$.
Let $\Gamma_i=\pi_1(B(p_i,\epsilon))$, and 
$\tilde B(p_i,\epsilon)$ the universal cover of $B(p_i,\epsilon)$. 
For a point $\tilde p_i\in\tilde B(p_i,\epsilon)$ over $p_i$ ,
put  
$$
  \delta_i := \min_{\gamma\in\Gamma_i-\{1\}} d(\gamma\tilde p_i,\tilde p_i),
$$
and choose a sequence  $\tau_i\to 0$  satisfying 
$$
\lim_{i\to\infty}\frac{\tau_i}
   {\max\{ \delta_i, d_{GH}(B(p_i,\epsilon),B(p,\epsilon)\}}
   =\infty.
$$
Let $(Z,z_o,G)$ be any limit of 
$(\frac{1}{\tau_i}\tilde B(p_i,\epsilon), \tilde p_i,\Gamma_i)$ 
with respect to the pointed equivariant Gromov-Hausdorff
topology (see \cite{FY:fundgp} for instance for 
the definition). Note that $G$ is a lie group (\cite{FY:isomgp}).
By the choice of $\tau_i$, $\Gamma_i$ collapses, that is,
$\dim G>0$. Since $\dim Z/G=3$, it follows that $\dim Z=4$.
Since $Z$ is simply connected and $\dim Z(\infty)\ge 2$,
the soul of $Z$ must be a point, and therefore 
Generalized Soul Theorem \ref{thm:soul} implies that $Z$ is 
homeomorphic to $\R^4$.

Let $(W,w_0,\Gamma_{\infty})$ be any limit of 
$(\frac{1}{\delta_i}\tilde B(p_i,\epsilon), \tilde p_i,\Gamma_i))$ 
with respect to the pointed equivariant Gromov-Hausdorff
topology. The previous argument shows $\dim Y=4$.
Moreover $\Gamma_i$ does not collapse under this convergence,
and therefore $\Gamma_{\infty}\simeq \Gamma_i\simeq \Z$.
As before, since $Y$ is simply connected and $\dim Y(\infty)\ge 2$,
Generalized Soul Theorem \ref{thm:soul} and Proposition 
\ref{prop:ideal} imply $Y\simeq \R^4$. 
Hence $Y/\Gamma_{\infty}$, the limit of $(\frac{1}{\delta_i} M_i^4,p_i)$,
has a soul $\simeq S^1$. Corollary \ref{cor:soul} then yields that 
$B(p_i,R\delta_i)\simeq S^1\times D^3$ for large $R>0$ and $i$.
This completes  the proof of Theorem \ref{thm:rescal}.
\end{proof}

\begin{cor} \label{cor:dim3top}
If a sequence of pointed $4$-dimensional complete Riemannian
manifolds $(M_i^4,p_i)$ with $K \ge -1$ converges to a pointed 
$3$-dimensional complete Alexandrov space $(X^3,p)$, then 
there exists a positive number $r=r_p$ depending only on $p$ such that
$B(p_i,r)$ is homeomorphic to either $D^4$ or $S^1\times D^3$ for
sufficiently large $i$.  
\end{cor}

\begin{proof}
Applying  Theorem \ref{thm:rescal} to the convergence 
$(M_i^4,p_i)\to (X^3,p)$, we have sequences 
$\delta_i\to 0$ and $\hat p_i\in M_i$ such that 
\begin{enumerate}
  \item $\hat p_i\to p$ under the convergence 
       $M_i^4\to X^3$;
  \item  $B(p_i,r)$ is homeomorphic to $B(\hat p_i,R\delta_i)$ for every 
      $R\ge 1$ and large $i\ge i(R)$;
  \item  for a limit $(Y,y_0)$ of $(\frac{1}{\delta_i}M_i, \hat p_i)$, 
     we have  $\dim Y = 4$.
\end{enumerate}
Since $Y(\infty)$ has dimension at least two (Lemma \ref{lem:expand}),
the soul of $Y$ has dimension
at most one (Proposition \ref{prop:ideal}). 
Thus Corollary \ref{cor:soul} implies the conclusion.
\end{proof}


\section{Geometry of Alexandrov three-spaces}\label{sec:dim3alex}

From Theorem \ref{thm:rescal}, it is significant in the first step
to analyze
collapsing of $4$-manifolds to $3$-dimensional Alexandrov 
spaces, which will be discussed in Sections 
\ref{sec:dim3nobdyloc}, \ref{sec:dim3nobdyglob},
\ref{sec:dim3wbdy} and \ref{sec:dim3positive}.
In this section, we investigate the structure near singular points
of $3$-dimensional Alexandrov spaces with curvature bounded below, 
and obtain two results:
One is about the structure of the essential singular point set of a 
$3$-dimensional Alexandrov space in terms of quasigeodesics
(Proposition \ref{prop:graph}), 
and the other is about the existence of a collar neighborhood 
of $3$-dimensional Alexandrov spaces with nonempty boundary
(Theorem \ref{thm:collar}).

We first prepare some material on extremal subsets (see
\cite{PrPt:extremal} for the details).
By definition, a closed subset $F$ in an Alexandrov space $X$ 
with curvature bounded below is called {\em extremal}
if the following holds: 
For any $p\in X-F$, consider the distance function 
$d_p=d(p,\cdot)$, and let $q\in F$ be a local minimum point 
for the restriction $d_p|_{F}$. Then $q$ is infinitesimally a local
minimum of $d_p$, namely 
$$
    \limsup_{q_i\in X\to q} \frac{d_p(q_i)-d_p(q)}{d(q_i,q)} \le 0.
$$
or equivalently, 
$$
      \max_{\eta\in \Sigma_q} \min_{\xi\in p_q^{\prime}}
                    \angle(\xi,\eta) \le \pi/2.
$$
If $F$ is a point, $F$ is extremal if and only if 
it is an extremal point in the sense of Section \ref{sec:cap}.

\begin{ex} \label{ex:extI}
 \begin{enumerate}
  \item $\partial X$ is an extremal subset of $X;$
  \item $\R\times o$ is an extremal subset of $\R\times K(S^1_{\ell})$,
        where $S^1_{\ell}$ denotes the circle of length $\ell\le\pi$
        and $o$ is the vertex of the cone $K(S^1_{\ell})$.
 \end{enumerate}
\end{ex}

Now consider an Alexandrov space $\Sigma$ with curvature $\ge 1$.
A closed set $\Omega\subset\Sigma$ is called {\em extremal} if 
it satisfies the above condition together with 
\begin{enumerate}
 \item $\Sigma=B(\Omega,\pi/2)$;
 \item $\diam(\Sigma)\le\pi/2$ \, if $\Omega$ is either a point
       or empty.
\end{enumerate}

\begin{ex}   \label{ex:extII}
 \begin{enumerate}
  \item Let $\Sigma$ be the spherical suspension over a circle of 
        length $\le\pi$. Then the set consisting of the two vertices 
        of $\Sigma$ is extremal;
  \item  Let $\xi_1$, $\xi_2$ and $\xi_3$ be three points on the unit
         sphere $S^2(1)$ with $d(\xi_1,\xi_i)=\pi/2$, $(i=2,3)$, and 
         $d(\xi_2,\xi_3)\le \pi/2$. Let $D(\Delta)$ denote the double of
         the triangular region bounded by $\triangle \xi_1\xi_2\xi_3$.
         Then any nonempty subset of
         $\{ \xi_1,\xi_2,\xi_3\}$ is extremal in $D(\Delta)$.
 \end{enumerate}
\end{ex}

\begin{lem}[\cite{PrPt:extremal}] \label{lem:reduct}
$F\subset X$ is extremal if and only if 
$\Sigma_p(F)\subset \Sigma_p$ is extremal
for any $p\in F$, where the space of directions $\Sigma_p(F)$ of $F$ 
at $p$ is defined as the set of directions $\xi\in\Sigma_p$
such that for some 
sequence $x_n\in F$ with $x_n\to p$, $(x_n)'_p$ converges 
to $\xi$. 
\end{lem}

Here we present some information on the essential singular point set in
a $3$-dimensional Alexandrov space.  First we need

\begin{lem} \label{lem:1vertex}
Let $\Sigma$ be a compact Alexandrov space with curvature $\ge 1$ and with 
boundary. Then there exists at most one extremal point of
$\interior\Sigma$. \par
 If $\interior\Sigma$ has the unique extremal point,
then it must be the unique maximum point of the distance function 
from the boundary $\partial \Sigma$.
\end{lem}

\begin{proof}
Note that $f=d(\partial\Sigma,\,\cdot\,)$ is a strictly concave function on 
$\Sigma$ (\cite{Pr:alex2}), and therefore $f$ has a unique maximum point.
From the concavity, it is obvious to see that any non-maximum point
of $\interior\Sigma$ has the space of directions 
whose diameter is greater than $\pi/2$.
\end{proof}

\begin{lem} \label{lem:essential}
Let $\cal C$ be a closed subset of the essential singular point set 
$ES(X^3)$ of a $3$-dimensional Alexandrov space $X^3$ 
satisfying the following:
\begin{enumerate}
 \item For any $p\in \cal C$, either $\Sigma_p(\cal C)$ contains 
       at least two elements, or it consists of only a point.
       In the latter case, $p$ must be an extremal point of $X^3;$     
 \item $\cal C\cap\partial X^3$ is either empty or consists of
       components of $\partial X^3$.
\end{enumerate}
Then $\cal C$ is extremal. 
\end{lem}

\begin{proof}
Obviously for any $p\in \cal C$, every direction 
$\xi\in\Sigma_p(\cal C)$ is an essential singular point of 
$\Sigma_p(X^3)$. Therefore by Example \ref{ex:extI} and 
Lemma \ref{lem:1vertex}, we may assume that 
$\cal C\subset \interior X^3$.
By Lemma \ref{lem:reduct}, it suffices to show
that for any $p\in\cal C$, 
$\Sigma_p(\cal C)$ is an extremal subset of $\Sigma_p(X^3)$.
From assumption, it suffices to consider the case when 
$\Sigma_p(\cal C)$ contains distinct elements  $\xi_1$ and $\xi_2$.
Then we have
\begin{equation}
   L(\Sigma_{\xi_i}(\Sigma_p(X^3)))\le\pi, \quad i=1,2.\label{eq:ext}
\end{equation}
If $\diam(\Sigma_p)\le\pi/2$, clearly 
$\Sigma_p(\cal C)$ is extremal. If 
$\diam(\Sigma_p) > \pi/2$, then $\xi_1$ and $\xi_2$ realize the 
diameter of $\Sigma_p$, because any other point $\xi_3$ from
$\xi_1$, $\xi_2$ does not satisfy \eqref{eq:ext}.  
In view of \eqref{eq:ext}, the Alexandrov convexity implies that 
$B(\xi_1,\pi/2)$ and $B(\xi_2,\pi/2)$ cover $\Sigma_p$, and hence 
$\Sigma_p(\cal C)$ is extremal.
%
\end{proof}

\begin{rem}
In Lemma \ref{lem:essential}, the two conditions on $\cal C$ are 
essential. For instance, let $\Sigma^2$ be a compact Alexandrov 
surface with curvature $\ge 1$ such that 
\begin{enumerate}
 \item the radius of $\Sigma^2$ is less than or equal to $\pi/2;$
 \item $\Sigma^2$ contains at most one essential singular point$;$
  \begin{enumerate}
   \item if $\Sigma^2$ has no essential singular point, then 
         $\diam(\Sigma^2)>\pi/2;$
   \item if $\Sigma^2$ has an essential singular point $\xi$,
         then there is an $\eta$ with $d(\xi,\eta) > \pi/2$.
  \end{enumerate}
\end{enumerate}
For $X^3:=K(\Sigma^2)$, $ES(X^3)$ is not extremal because
the extremal condition does not hold at the vertex of the cone
$K(\Sigma^2)$.
\end{rem}

Let $X$ be a complete Alexandrov space space with 
curvature $\ge \kappa$. For a curve $\gamma:[a,b]\to X$
and a point $p\in X$, a curve $\tilde\gamma$
on the $\kappa$-plane $M^2_{\kappa}$ is called  
a {\it development} of $\gamma$  from $p$ if 
$d(\tilde\gamma(t),\tilde p)=d(\gamma(t), p)$
for some point $\tilde p$.
The curve $\gamma:[a,b]\to X$
is called a {\it quasigeodesic} (see \cite{PrPt:extremal})
if for any point $p\in X$,
the development $\tilde \gamma$  of $\gamma$ from $p$ 
defines a convex subset 
bounded by $\tilde p\tilde\gamma(t_1)$, $\tilde p\tilde\gamma(t_2)$ 
and $\tilde\gamma([t_1,t_2])$ for any $a\le t_1<t_2\le b$.

\begin{prop} \label{prop:graph}
Let $X^3$ be a $3$-dimensional complete Alexandrov space with 
curvature bounded below, and let $\cal C$ be a subset of 
$ES(\interior X^3)$ satisfying the 
condition $(1)$ of Lemma \ref{lem:essential}.  
Then $\cal C$ 
has the structure of a finite metric graph such that
\begin{enumerate}
 \item every vertex of $\cal C$ as a graph except the endpoints
       has order three$;$
 \item every subarc of $\cal C$ is a quasigeodesic.
\end{enumerate}
\end{prop}

\begin{proof}
 By Lemma \ref{lem:essential},
$\cal C$ is extremal. It follows that $\cal C$ is locally connected 
by quasigeodesics (\cite{PrPt:extremal}) and that  
there exists a quasigeodesic starting at $p$ in a  
direction $\xi\in\Sigma_p(\cal C)$. The uniqueness of such a 
quasigeodesic follows from 
the  argument in the proof of Assertion \ref{ass:inj}.
Note that $\Sigma_p$ has at most three essential singular 
points (see Appendix in \cite{SY:3mfd}). 
This completes the proof.
\end{proof}

\begin{ex}
Let $\Sigma$ be a compact Alexandrov surface with curvature $\ge 1$ and 
with exactly three essential singular points, say $p_1$, $p_2$ and $p_3$. 
We consider the spherical suspension $X^3=\Sigma\times_{\sin t} [0,\pi]$ 
of $\Sigma$,  where we make an identification 
$\Sigma=\Sigma\times \{\pi/2\}$.  $X$ has the essential 
singular point set consisting of three minimal geodesic segments
joining the two poles of $X^3$ through $p_i$.
\end{ex}

\begin{ex}\label{ex:convex}
We construct a $3$-dimensional complete open Alexandrov space $X^3$  
with nonnegative curvature whose essential singular point set consists of 
a countable set with a limit as follows:
First we construct a noncompact convex body $E$ in the $(x,y,z)$-space
as follows.
Let $C$ be a compact convex ``polygon'' on the $(x,y)$-plane
with countable edges such that 
a sequence $p_i$ of consecutive vertices of $\partial C$ converges 
to a point $p\in\partial C$.
We denote by $\{ e_i\}$ the line segment from $p_{i-1}$ to $p_i$ in 
$\partial C$. 
For a point $q_i\in\R^2-C$ sufficiently close to the midpoint of $e_i$,
let $\ell_i$ be the geodesic ray in $\R^3$ 
starting from $(q_i,1)$ in the direction to the
positive $z$-axis.
Let $E_i$ be the minimal convex set containing 
$C\times [0,\infty)$ and $\ell_1\cup \ell_2 \cup\cdots$.
If $d(q_i,e_i)$ is sufficiently small, then 
$p_i\times [0,\infty)$ is contained in $\partial E$.
Finally we take the double $X^3:=D(E)$.
$X^3$ has nonnegative curvature, and $ES(X^3)$ consists of the sequence 
$p_i$, $i=1,2,\ldots$ with the limit $p$.
\end{ex}

Example \ref{ex:convex} shows that in Proposition \ref{prop:graph} 
one cannot drop the condition $(1)$ in Lemma \ref{lem:essential}.
\par\medskip

Next we turn to the other subject of this section, 
the existence of a collar neighborhood
of the boundary of an Alexandrov space with curvature bounded below.

\begin{prop} \label{prop:collar-general}
Let $X^n$ be an Alexandrov space with 
curvature bounded below and with nonempty boundary.
Then $\partial X^n$ has a collar neighborhood,i.e., an 
open neighborhood of $\partial X^n$ homeomorphic to 
$\partial X^n\times [0,1)$.
\end{prop}

\begin{proof}
This is done by induction on $n$. The case of $n=1$ is 
clear. By \cite{Br:flat}, it suffices to show that 
$\partial X$ is locally collared.
For any $p\in\partial X$, $\partial\Sigma_p$ has a collar
neighborhood. This implies that $\partial K_p$ has a collar
neighborhood. Thus there is an $\epsilon>0$ such that 
$B(p,\epsilon)\cap\partial X$ has a collar neighborhood.
\end{proof}

Since the collar neighborhood given in 
Proposition \ref{prop:collar-general} is only topological,
it is not enough for 
our purpose. We need a metric collar neighborhood at least
when $\partial X$ is compact.

We put for $\epsilon>0$
$$
   X_{\epsilon}=\{ x\in X\,|\,d(x,\partial X)\ge\epsilon \}.
$$

\begin{conj} \label{conj:collar}
Let $n$ be any positive integer, and let $X^n$ be an $n$-dimensional 
Alexandrov space with curvature bounded below 
and with nonempty compact boundary.
Then there is a positive number $\epsilon$ such that 
$X^n-X_{\epsilon}^n$ provides a collar neighborhood of $\partial X^n$.
\end{conj}

The author is not certain if the method in \cite{Br:flat} can be 
applied to solve Conjecture \ref{conj:collar} affirmatively.
Here is a conjecture related with Conjecture \ref{conj:collar}.
\par

\begin{conj}[\cite{PrPt:extremal}] \label{conj:bdyalex}
If $X^n$ is an Alexandrov space with curvature $\ge\kappa$, then 
so is the boundary $\partial X^n$ (if it is nonempty) 
with respect to the length metric induced 
from $X^n$. 
\end{conj}

Conjecture \ref{conj:collar} will be true  if we come to know 
Conjecture \ref{conj:bdyalex} to be true: \par

\begin{obs} \label{obs:collar}
If Conjecture \ref{conj:bdyalex} is true, then 
Conjecture \ref{conj:collar} is also true. \par
\end{obs}

\begin{proof}
Let $Y^n$ denote the gluing
$$
   Y^n := X^n\cup_{\partial X^n} \partial X^n\times [0,\infty).
$$
By the assumption,  $\partial X^n\times [0,\infty)$
is an Alexandrov space with curvature bounded below.
It follows from \cite{Pt:apply} that $Y^n$ is also an 
Alexandrov space with curvature bounded below.
Put $Z:=\partial X^n\times 1\subset Y^n$ and consider 
$$
  f = d(Z,\, \cdot\,) : Y^n\to \R.
$$
Obviously $f$ is regular on $\partial X^n\times 0$. 
Observation \ref{obs:collar} immediately follows from 
\cite{Pr:alex2}.
\end{proof}

\begin{thm} \label{thm:collar}
Conjecture \ref{conj:collar} is true for $n=3$.
\end{thm}

\begin{proof}
First note that each point $p$ of $\partial X^3$ has a small 
spherical neighborhood homeomorphic to $\R^3_+$.
Consider the convergence
$(\frac{1}{r}X^3,p)\to (K_p,o_p)$ as $r\to 0$.
Since $\Sigma_p-(\Sigma_p)_{\epsilon}$ provides a collar
neighborhood of $\partial\Sigma_p$ for small $\epsilon$,
it follows that $K_p-(K_p)_{\epsilon}$ also provides a
collar neighborhood of $\partial K_p$.
Note that $d_{\partial K_p}$ and $(d_{\partial K_p},d_{o_p})$
are regular on $\{ 0<d_{\partial K_p}\le\epsilon\}$ and 
$\{ 0<d_{\partial K_p}\le\epsilon, d_{o_p}\ge 1/2\}$ 
respectively. This ensures the existence of a positive number
$r$ such that $B(p,r)-X^3_{\epsilon}$ provides a collar neighborhood
of $\partial X^3\cap B(p,r)$ for $\epsilon\ll r$.
We take a finite covering 
$\{ B(p_j,r_j/2)\}_{1\le j\le N}$ of $\partial X^3$, 
where $\epsilon_j\ll r_j$ are chosen for $p_j$ as above.
Let $K$ be a triangulation of $\partial X^3$ by Lipschitz curves,
and take sufficiently small $\epsilon$ compared with 
$\min \{\epsilon_j\}$ and the sizes of the simplices of $K$.
Let $K^i$ denote the $i$-skeleton of $K$.
For each $p\in K^0$, choose a minimal geodesic $\gamma_p$
from $p$ to $\partial X_{\epsilon}$.
The disjoint union of $\{ \gamma_p\}_{p\in K^0}$ provides 
a collar of $K_0$. 
For each edge $e\in K^1$, choose a Lipschitz curve $e_1$ on 
$\partial X_{\epsilon}$ such that 
\begin{enumerate}
 \item the union of $e$, $e_1$, $\gamma_p$ and $\gamma_q$
   bounds a $2$-disk in $X^3-\interior X^3_{\epsilon}$, 
   denoted $D_{e}$, giving a collar of $e$, satisfying
   $\interior D_e\subset\interior X^3-X^3_{\epsilon}$,
   where $p$ and $q$ are the endpoints of $e;$
 \item if two edges $e$ and $e'$ of $K$ meet at a vertex $p$, 
   then $D_{e}\cap D_{e'}=\gamma_p$.
\end{enumerate}
This is possible because of the local collar structure 
mentioned above.
Now for every $2$-simplex $\Delta^2$ of $K$,
let $D_{\partial \Delta^2}$ be the collar of $\partial \Delta^2$
constructed above. 
Let $\hat\Delta^2$ denote a disk domain of $\partial X^3_{\epsilon}$
bounded by $D_{\partial \Delta^2}\cap \partial X_{\epsilon}$.
Since the union of $\Delta^2$, $D_{\partial \Delta^2}$ and 
$\hat\Delta^2$ is  homeomorphic to $S^2$ and locally flat,
by the generalized Schoenflies theorem, 
it bounds a closed domain homeomorphic to $D^3$,
which gives a collar of $\Delta^2$ extending 
the collar structure given by $D_{\partial \Delta^2}$.
Thus we obtain a collar structure on $X^3-X^3_{\epsilon}$.
\end{proof}


\section{Collapsing to three-spaces without boundary- \\
local construction} 
\label{sec:dim3nobdyloc}

 Let a sequence of complete $4$-dimensional pointed Riemannian
manifolds $(M_i^4,p_i)$ with $K \ge -1$ converge to a 
$3$-dimensional pointed Alexandrov space $(X^3,p)$. 
Throughout this section, we assume that 
$p$ is an interior point of $X^3$. The purpose of this section 
is to construct an $S^1$-action on a small perturbation of $B(p_i,r)$,
for a sufficiently small $r>0$ depending only on $p$.

By Fibration Theorem \ref{thm:orig-cap}, we have a locally 
trivial $S^1$-bundle 
$f_i:M_i^{\prime}\to \interior B(p,1)-S_{\delta_3}(X^3)$ which is 
$\epsilon_i$-approximation, $\lim_{i\to\infty}\epsilon_i=0$,
where $M_i^{\prime}$ is an open subset of $M_i^4$ and 
$\delta_3 >0$ is sufficiently small.

Our first step is to prove that $\partial B(p_i,r)$ is a Seifert 
fibred space over $\partial B(p,r)$.

Let $\delta\ll r$ and consider $B(q,\delta)$ for every $q\in \partial B(p,r)$.
We take a point $q_i\in \partial B(p_i,r)$ Gromov-Hausdorff 
close to $q\in\partial B(p,r)$.

\begin{lem}  \label{lem:bdytorus}
There exist positive numbers $r_p$ and $c_q$ such that
if $r\le r_p$  and $\delta/r\le c_q$, then 
for some  $r_1<r<r_2$ and $\delta_1<\delta<\delta_2$ sufficiently 
close to $r$ and $\delta$ respectively, there exists
a homeomorphism 
$h_i:(d_{p_i},d_{q_i})^{-1}([r_1,r_2]\times [\delta_1,\delta_2])\to
T^2 \times [r_1,r_2]\times [\delta_1,\delta_2]$ 
which respects $(d_{p_i},d_{q_i})$, that is, 
$p_r\circ h_i=(d_{p_i}, d_{q_i})$, where $p_r$ is the projection 
to the factor $[r_1, r_2]\times [\delta_1, \delta_2]$.
In particular, $\partial B(p_i,r)\cap \partial B(q_i,\delta)$ is 
homeomorphic to $T^2$.
\end{lem}

\begin{proof}
Note that $(d_p,d_q)$ is regular near $(d_p,d_q)^{-1}(r,\delta)$
for $\delta\ll r\le\mbox{const}_p$. By \cite{Pr:alex2}, 
$(d_p,d_q)^{-1}(r,\delta)\simeq S^1$ and  we have a homeomorphism 
$h:(d_p,d_q)^{-1}([r_1,r_2]\times [\delta_1,\delta_2])\to
S^1\times [r_1,r_2]\times [\delta_1,\delta_2]$ which respects $(d_p,d_q)$
for some  $r_1<r<r_2$ and $\delta_1<\delta<\delta_2$ 
sufficiently close to $r$ and $\delta$ respectively.
Since $(d_p,d_q)^{-1}([r_1,r_2]\times [\delta_1,\delta_2])$
does not meet $S_{\delta_3}(X^3)$, Fibration Theorem \ref{thm:orig-cap}
together with a standard Morse theory for distance functions yields
\begin{align*}
(d_{p_i},d_{q_i})^{-1}&([r_1,r_2] \times [\delta_1,\delta_2]) \\
    & \simeq \mbox{$S^1$-bundle over 
      $(d_p,d_q)^{-1}([r_1,r_2]\times [\delta_1,\delta_2])$}, \\
    & \simeq \mbox{($S^1$-trivial bundle over 
      $S^1)\times [r_1,r_2]\times [\delta_1,\delta_2]$}.
\end{align*}
From the regularity of $(d_{p_i},d_{q_i})$, we have the required 
homeomorphism. 
\end{proof}

Recall that 
$$
     U_{\delta}(p_i,q_i):=\partial B(p_i,r)\cap B(q_i,\delta),
$$
for $q_i\in\partial B(p_i,r)$.

\begin{lem}\label{lem:bdysolid}
$U_{\delta}(p_i,q_i)\simeq S^1\times D^2$ for sufficiently small
$r$, $\delta\ll r$ and large $i$.
\end{lem}

\begin{proof}
It follows from Lemma \ref{lem:u_1} that 
$U_{\delta}(p_i,q_i)$ is homeomorphic to either $S^1\times D^2$
or $D^3$ for any sufficiently large $i$.
Hence the result follows from  Lemma \ref{lem:bdytorus}.
\end{proof}

\begin{prop} \label{prop:bigdiam}
If $p$ is an interior point of $X^3$ and if
$\diam(\Sigma_p)>\pi/2$, then $B(p_i,r)\simeq S^1\times D^3$.
\end{prop}

\begin{proof}
This follows from Lemmas \ref{lem:bdysolid} and \ref{lem:dim3balltop}.
\end{proof}

Since $S_{\delta}(\Sigma_p)$ is finite for every $\delta>0$,
$S_{\delta}(K_p)$ is the union of finitely many geodesic rays 
starting from $o_p$. In view of the convergence
$(\frac{1}{r} X^3,p) \to (K_p,o_p)$ as $r\to 0$,
it is possible to take a small positive number $\epsilon=\epsilon_p$ 
and a finite subset $\{ q^j \}_{j=1,\ldots,k(p)}$ of
$\partial B(p,r)\cap S_{\delta_3}(X^3)$ for any sufficiently small
$r=r(p,\epsilon)>0$ such that 
\begin{enumerate}
 \item $\{ U_{\epsilon r}(p,q^j) \}_{j}$ covers  
   $\partial B(p,r)\cap S_{\delta_3}(X^3);$
 \item $(U_{\epsilon r}(p,q^j)-U_{\epsilon^{10}r}(p,q^j))$ 
       does not meet $S_{\delta}(X^3);$
 \item $\{ U_{\epsilon r}(p,q^j)\}_j$ is disjoint.
\end{enumerate}
Since $\Sigma_p$ is topologically regular, 
$U_{\epsilon r}(p,q^j)$  is homeomorphic to $D^2$
for a small $\epsilon$.

Let $\xi_j\in S_{\delta_3}(\Sigma_p)$ be the direction with 
$\angle(\xi_j,(q^j)_p^{\prime})<\tau(r)$, and 
take $q^j_i\in \partial B(p_i,r)$ with $q^j_i\to q^j$.

\begin{lem} \label{lem:singslice}
Let $\epsilon=\epsilon_p$ be as above. Then there exists a positive 
number $r_p$ such that 
for every $0<r\le r_p$  we have a three-dimensional submanifold 
$U_{j,i}$ of $M_i^4$, a small perturbation of 
$U_{\epsilon r}(p_i,q^j_i)$, satisfying the  following:
\begin{enumerate}
 \item $d_{GH}(U_{j,i}, U_{\epsilon r}(p,q^j))\to 0$ as $i\to\infty;$
 \item  $U_{j,i}$ is a fibred solid torus over 
    $U_{\epsilon r}(p,q^j)$ such that
   \begin{enumerate}
     \item $\partial U_{j,i} = 
               f_i^{-1}(\partial U_{\epsilon r}(p, q^j))$
       and the fibre structure on $\partial U_{j, i}$ induced from
       that of $U_{j,i}$ is compatible to the $S^1$-bundle structure 
       on $\partial U_{j,i}$  defined by $f_i;$ 
     \item the Seifert invariants of the singular fibre in $U_{j,i}$
       $($if it exists$)$ do not exceed 
       $$
           2\pi/L(\Sigma_{\xi_j}(\Sigma_p));
       $$
   \end{enumerate}
 \item Let
       $$
           \partial B^{f_i}(p_i,r) := 
   f_i^{-1}(\partial B(p,r)-\cup_{j} U_{\epsilon r}(p,q^j))\bigcup 
               (\cup_{j} U_{j,i}),
       $$
   and let $B^{f_i}(p_i,r)$ denote the closed domain bounded by 
   $\partial B^{f_i}(p_i,r)$ containing $p_i$. 
   Then $B^{f_i}(p_i,r)\simeq B(p_i,r)$. \par
\end{enumerate}
In particular, $\partial B^{f_i}(p_i,r)$ is a Seifert fibred space 
over $\partial B(p,r)$. 
\end{lem}

\begin{proof}
First note that both $d_{p_i}$-flow curves and $d_{q_i}$-flow curves
are transversal to $f_i^{-1}(\partial U_{\epsilon r}(p,q^j))$.
Hence using some $d_{q_i}$-flow curves starting from 
$f_i^{-1}(\partial U_{\epsilon r}(p,q^j))$, we can 
extend $f_i^{-1}(\partial U_{\epsilon r}(p,q^j))$ to a 
three-dimensional submanifold $\hat U_{j,i}$ such that 
\begin{enumerate}
 \item $\partial\hat U_{j,i}$ is the disjoint union of  
       $f_i^{-1}(\partial U_{\epsilon r}(p,q^j))$ and 
       $\partial U_{\epsilon r/2}(p_i,q^j_i);$
 \item $\hat U_{j,i}\simeq f_i^{-1}(\partial U_{\epsilon r}(p,q^j))
       \times I$ through $d_{q_i}$-flow curves $;$ 
 \item $d_{p_i}$-flow curves are transversal to $\hat U_{j,i};$
 \item The Hausdorff distance between $\hat U_{j,i}$ and 
       $U_{\epsilon r}(p_i,q^j_i)-U_{\epsilon r/2}(p_i,q^j_i)$ is 
       less than $\tau_i r$, where $\lim_{i\to\infty} \tau_i=0$.
\end{enumerate}
We let $U_{j,i}$ denote the union of $\hat U_{j,i}$ and
$U_{\epsilon r/2}(p_i,q^j_i)$.
It follows from $(3)$ that 
$B^{f_i}(p_i,r)\simeq  B(p_i,r)$, where 
$B^{f_i}(p_i,r)$ is defined as above. 
By $(1)$ and $(2)$, $U_{j,i}\simeq D^2\times S^1$.
Therefore from construction, we only have to prove the conclusion (2)-(b) 
for $U_{j,i}$ chosen in this way.
Let $V_{i,j}^r$ denote the compact domain consisting of 
the flow curves of $d_{p_i}$ through $U_{\epsilon r}(p_i,q_i^j)$
and contained in $A(p_i;r_1,r_2)$ for $r_1=0.9r$ and
$r_2=1.1r$.
We may assume that $V_{i,j}^r$ converges to a compact 
domain $V_j^r$ containing $U_{\epsilon r/2}(p,q^j)$.
Since $(\frac{1}{r}V_j^r,q^j)$ converges to 
$\Omega_j:=K(U_{\epsilon/2}(o_p,\xi_j))\cap\{ 0.9\le d_{o_p}\le 1.1\}$, 
we have 
$$
  d_{p.GH}\left((\frac{1}{r}V_{i,j}^r,q_i^j),(\Omega_j,\xi_j)\right)
          < \mu_i^r/r + \tau(r),
$$
where $\lim_{i\to\infty} \mu_i^r=0$ (for any fixed $r$).
Therefore for a sequence $\nu_{\alpha}\to 0$ as $\alpha\to\infty$,
we see
$$
  d_{p.GH}\left((\frac{1}{\nu_{\alpha} r}V_{i,j}^r,q_i^j), 
    (K_{\xi_j}(\Sigma_p)\times\R, (o_{\xi_j},0))\right)
      < \mu_i^r/\nu_{\alpha} r + \tau(r)/\nu_{\alpha}+\tau({1/\alpha}),
$$
which converges to zero if we take a sequence 
depending on $(\alpha,r,i)$ in such a way
that $\alpha\gg 1$ and $\tau(r)/\nu_{\alpha}\ll 1$ and
$\mu_i^r/\nu_{\alpha} r\ll 1$.
Let $\tilde V_{i,j}^r \to V_{i,j}^r$ be the universal cover and 
$\Gamma_{i,j}$ the deck transformation group of it.
Passing to a subsequence, we may assume that 
$(\frac{1}{\nu_{\alpha} r}V_{i,j}^r,q_i^j,\Gamma_{i,j})$ 
converges to a triplet $(Z, z_0, G)$with respect to the
pointed equivariant Gromov-Hausdorff convergence,
where $Z$ is a complete noncompact Alexandrov space
with nonnegative curvature such that $Z/G$ is isometric
to $K(S^1_{\ell_0})\times\R$ with 
$\ell_0 = L(\Sigma_{\xi_j}(\Sigma_p))$.
In a way similar to Section 4 of \cite{SY:3mfd}, we have
\begin{enumerate}
 \item $Z$ is isometric to a product $K(S^1_{\ell})\times\R\times\R$;
 \item $G$ is isomorphic to $\Z_{\mu}\times\R$,
    where $\mu\le 2\pi/\ell_0$ and 
    $$
     \Z_{\mu}\times\R\subset\isom(K(S^1_{\ell}))\times\isom(\R)
     \subset SO(2)\times\R,
    $$
\end{enumerate}
from which we obtain a fibred solid torus structure on 
$U_{\epsilon r}(p_i,q_i^j)$ of type $(\mu,\nu)$ for some $\nu$, 
compatible to the fibres of $f_i$ on the boundary.
From the construction of $U_{j,i}$, the fibre structure on 
$U_{\epsilon r}(p_i,q_i^j)$ defines a fibre structure on 
$U_{j,i}$ of the same type as $U_{\epsilon r}(p_i,q_i^j)$,
compatible to the fibres of $f_i$ on the boundary.
\end{proof}

Let $\tilde f_i:\partial B^{f_i}(p_i,r)\to \partial B(p,r)$ 
be the Seifert fibration given in Lemma \ref{lem:singslice}, 
\par

\begin{prop}\label{prop:bdyseif}
There exists a positive number $r_0=r_0(p)$ such that 
for any $0<r\le r_0$ and sufficiently large $i$ 
the Seifert fibration  
$\tilde f_i:\partial B^{f_i}(p_i,r)\to \partial B(p,r)$ 
satisfies that 
 \begin{enumerate}
   \item the number of singular fibres is at most two;
   \item for any singular fibre over a point $q\in\partial B(p,r)$, 
      there is an essential singular point $\xi$ of $\Sigma_p$ with 
      $\angle(\xi,q_p^{\prime})<\tau(r)$ such that 
      the Seifert invariants of the singular fibre do not exceed 
      $$
         \frac{2\pi}{L(\Sigma_{\xi}(\Sigma_p))};
      $$
   \item $\tilde f_i$ comes from some $S^1$-action on
      $\partial B^{f_i}(p_i,r)$.
 \end{enumerate}
\end{prop}

In view of Lemma \ref{lem:singslice}, for the proof of 
Proposition \ref{prop:bdyseif}, it suffices to check $(1)$ and $(3)$. 
We need the following 

\begin{prop}\label{prop:topmfd}
Suppose that a sequence  of pointed complete 
orientable Riemannian $4$-manifolds $(M_i^4,p_i)$ collapses to a pointed 
Alexandrov $3$-space $(X^3,p)$ under $K\ge -1$.
Then $X^3$ is a topological manifold.
\end{prop}

\begin{proof}
In view of Stability Theorem \ref{thm:stability}, it suffices to show that 
$\Sigma_p$ is homeomorphic to $S^2$ or $D^2$.
If $p\in\partial X^3$, then $\Sigma_p$ is an Alexandrov surface with
curvature $\ge 1$ and with nonempty boundary,  and hence it is 
homeomorphic to $D^2$. If $p\in\interior X^3$, then $\Sigma_p$ is 
homeomorphic to either $S^2$ or $P^2$. Suppose $\Sigma_p\simeq P^2$.
By lemma \ref{lem:singslice}, we have a Seifert fibration
$\tilde f_i:\partial B(p_i,r)\to \Sigma_p$. Note that the number 
$m_i$ of the singular orbits of $\tilde f_i$ satisfies $m_i\le 1$. 
Otherwise,
the universal cover $\tilde\Sigma_p$ of $\Sigma_p$ 
would contain more than three essential singular points.
This is a contradiction to the curvature condition that 
$\tilde \Sigma_p$ has curvature $\ge 1$ (see for instance 
Appendix in \cite{SY:3mfd}).
If $m_i=0$, then $\partial B^{f_i}(p_i,r)$ is an $S^1$-bundle
over $P^2$, denoted $P^2\tilde\times S^1$.
This is a contradiction to Lemma \ref{lem:dim3balltop}.
If $m_i=1$, then $\partial B^{f_i}(p_i,r)$ is homeomorphic to 
$P^2\tilde\times S^1$ or a prism manifold 
(see \cite{Or:Seif} for instance). 
This is also a contradiction to Lemma \ref{lem:dim3balltop}.
\end{proof}

\begin{proof}[Proof of Proposition \ref{prop:bdyseif}].
By Corollary \ref{cor:dim3top}, $B(p_i,r)$ is homeomorphic to 
either $S^1\times D^3$ or $D^4$. 
If $B(p_i,r)\simeq S^1\times D^3$, 
\cite{Sc:geometry}, p.459, shows that the number $m_i$ of singular fibres of 
$\tilde f_i:\partial B^{f_i}(p_i,r)\simeq S^1\times S^2 \to 
         \partial B(p,r)\simeq S^2$
satisfies that $m_i=0$ or $2$. 
If $B(p_i,r)\simeq D^4$, then the Seifert bundle 
$\tilde f_i:\partial B^{f_i}(p_i,r)\simeq S^3\to\partial B(p,r)\simeq S^2$
has at most two singular fibres (see \cite{Or:Seif}).
In either case, $\tilde f_i$ comes from an $S^1$-action on 
$\partial B^{f_i}(p_i,r)$ (see \cite{Or:Seif}).
This completes the proof of Proposition \ref{prop:bdyseif}.
\end{proof}

Our next step is to construct an $S^1$-action on $B^{f_i}(p_i,r)$ 
extending $\tilde f_i$.

Let us begin with the following lemma, which
easily follows from the convergence 
$(\frac{1}{r} X^3,p)\to (K_p,o_p)$ and the finiteness of 
$S_{\delta}(\Sigma_p)$ for any $\delta>0$.

\begin{lem}  \label{lem:coverS}
For any $p\in X^3$, let $\epsilon=\epsilon_p>0$ be 
as in Lemma \ref{lem:singslice}. Then there exist 
a positive integer $k=k(p)$ and $r=r(p,\epsilon)>0$ such that
for some $x^1,\ldots,x^k\in\partial B(p,10r)$,
\begin{enumerate}
 \item $B(p,10r)\cap S_{\delta_3}(X^3)\subset 
       \bigcup_{j=1}^k B(px^j,\epsilon r);$
 \item $(B(px^j,\epsilon r)-B(px^j,\epsilon^{10} r))\cap
          A(p;r/100,10r)$ does not meet $S_{\delta_3}(X^3);$
 \item $\{ B(px^j,\epsilon r)\cap A(p;r/100,10r)\}_j$ 
       is disjoint.
\end{enumerate}
\end{lem}

For each $1\le j\le k$ and for $\epsilon_1\in (0,2\epsilon]$, 
$r_2\le r_1\le 2r$, we consider 

\begin{gather*}
    A_{j}(p;\epsilon):=B(px^j,\epsilon r)\cap B(p,2r),\\
    A_{j}(p;\epsilon_1;r_1) := A_{j}(p;\epsilon_1)\cap B(p, r_1) \\
 A_{j}(p;\epsilon_1;r_2, r_1):= A_{j}(p;\epsilon_1)\cap A(p;r_2,r_1)\\
 A^{f_i}(p_i;r_1,r) := 
          \{ x\in B^{f_i}(p_i,r) \,|\, d(p_i,x)\ge r_1\}.  
\end{gather*}
Fibration Theorem \ref{thm:orig-cap} enables us to take a closed 
domain  $A^{f_i}_{j}(p_i;\epsilon)\subset B(p_i,2r)$   
such that 
\begin{enumerate}
  \item $A_{{j}}^{f_i}(p_i;\epsilon)$ converges to $A_j(p;\epsilon)$ 
        under the convergence $M_i^4\to X^3$;
  \item $\partial A_{j}^{f_i}(p_i;\epsilon)\cap
         \interior A(p_i;r/100,2r)$ coincides with 
        $f_i^{-1}(\partial A_{j}(p;\epsilon)\cap
          \interior A(p_i;r/100,2r))$.
\end{enumerate}
Take $x_i^j\in\partial B(p_i,10r)$ close to 
$x^j\in\partial B(p,10r)$, and
let $\tilde d_{p_ix_i^j}$ be a smooth approximation of the distance
function $d_{p_ix_i^j}$.  We consider 
\begin{gather*}
  A_{j}^i(p_i;\epsilon) := 
             \{ \tilde d_{p_ix_i^j}\le\epsilon r\}\cap B(p_i,2r),\\
 A_{j}^i(p_i;\epsilon_1;r_2,r_1):=A_{j}^i(p_i;\epsilon_1)\cap
                  A(p_i; r_2, r_1) \\
 A^{f_i}_{j}(p_i;\epsilon;r_1,r) := 
          A^{f_i}(p_i;r_1,r)\cap A_{j}^{f_i}(p_i;\epsilon).  
\end{gather*}

\begin{lem} \label{lem:af}
$A_{j}^{f_i}(p_i;\epsilon;r/100,r)\simeq (S^1\times D^2)\times I$.
\end{lem}

\begin{proof}
By using the flow curves of a gradient vector field 
for $\tilde d_{p_ix_i^j}$, we have 
$$
   A_{j}^{f_i}(p_i;\epsilon;r/100,r) 
        \simeq   A_{j}^i(p_i;\epsilon) \cap A^{f_i}(p_i;r/100,r).
$$
By  Lemma \ref{lem:bdysolid} together with 
the flow curves of a gradient-like vector field for $d_{p_i}$,
$$
   A_{j}^i(p_i;\epsilon) \cap A^{f_i}(p_i;r/100,r) 
        \simeq (S^1\times D^2)\times I,
$$
and hence 
$A_{j}^{f_i}(p_i;\epsilon;r/100,r)\simeq (S^1\times D^2)\times I$.
\end{proof}

Now we construct an $S^1$-action on $B^{f_i}(p_i,r)$ extending 
$\tilde f_i$.
Let us first consider 

\begin{proclaim}{\emph{Case} I.}
    $B(p_i,r)\simeq S^1\times D^3$.
\end{proclaim}

In this case, the number  $m_i$ of singular fibres of 
$\tilde f_i:\partial B^{f_i}(p_i,r)\simeq S^1\times S^2\to 
             \partial B(p,r)\simeq S^2$
satisfies that $m_i=0$ or $2$. 
If $m_i=0$, $\tilde f_i$ is a trivial bundle and 
extends to a trivial bundle $B^{f_i}(p_i,r)\to B(p,r)$
(compare Lemma \ref{lem:extend} below).
Thus we obtain a free $S^1$-action on $B^{f_i}(p_i,r)$ 
whose orbit space is homeomorphic to $B(p,r)$, 
as the collapsing structure on $B^{f_i}(p_i,r)$.
\par

Before considering the essential case $m_i=2$, we need the following 
elementary lemma on extending $S^1$-actions.  
Probably this is obvious for specialist, but for the lack of
references we shall give a proof.

Let $\phi_{\mu,\nu}$ denote the $S^1$-action 
on the solid torus $S^1\times D^2$ given by the 
canonical fibred solid torus structure of type $(\mu,\nu)$.
Namely, $\phi_{\mu,\nu}$ comes from the 
standard $\Z_{\mu}$-action on $S^1\times D^2$ generated by
$$
   \tau_{\mu\nu}(e^{i\theta_1}, re^{i\theta_2}) = 
  (e^{i \, (\theta_1+\frac{2\pi}{\mu})},
  r e^{i \, (\theta_2+\frac{2\nu}{\mu}\pi)}).
$$
We consider the canonical orientation of $S^1\times D^2\times I$,
and the orientation on $\partial(S^1\times D^2\times I)$
as boundary.
             
\begin{lem} \label{lem:extend}
Consider 
an $S^1$-action $\varphi$ on $(S^1\times\partial D^2\times I)\cup
(S^1\times D^2\times \{ 0\})$ such that
\begin{itemize}
 \item $\varphi$ defines an $S^1$-action, denoted $\phi_0$, on 
       $S^1\times D^2\times \{ 0\}$,
       equivariant to $\phi_{\mu,\nu};$
 \item $\varphi$ defines a free $S^1$-action on 
       $S^1\times\partial D^2\times I$.
\end{itemize}
Then the following holds:
\begin{enumerate}
 \item $\varphi$ extends to a locally smooth $S^1$-action $\psi$ on 
       $S^1\times D^2\times I;$
 \item Any such an extension $\psi$ is equivariant to the product 
       action $\psi_{\mu,\nu}\times \text{\rm id};$
 \item If we are also given an $S^1$-action, denoted $\phi_1$, on 
       $S^1\times D^2\times \{ 1\}$ equivariant to 
       $\phi_{\mu,\mu-\nu}$ and compatible to $\varphi$ on
       the boundary, then $\varphi\cup\phi_1$ extends to
       a locally smooth $S^1$-action on $S^1\times D^2\times I$.
\end{enumerate}
\end{lem}

\begin{proof}
We first show (3). 
Let $V_j$, $j=0,1$, be a subsolid torus of 
$S^1\times\interior D^2\times\{ j\}$ invariant under the action of 
$\phi_j$. Choose a closed domain $W$ of 
$S^1\times\interior D^2\times I$
homeomorphic to $S^1\times D^2\times I$ such that
$W\cap (S^1\times D^2\times\{ j\})=V_j$.
Let $\rho:W\to S^1\times D^2\times I$ be a homeomorphism,
and $V_t:=\rho^{-1}(S^1\times D^2\times \{ t\})$, $t\in I$.
Let 
$f_0:(S^1\times D^2,\phi_{\mu,\nu})\to (V_0,\phi_0)$
and 
$f_1:(S^1\times D^2,\phi_{\mu,\mu-\nu})\to (V_1,\phi_1)$
be equivariant homeomorphisms. 
From the assumption, 
$g_0:=f_1^{-1}\circ\rho_1^{-1}\circ\Pi_{0,1}\circ\rho_0\circ f_0$
is isotopic the identity,  where 
$\rho_t:=\rho|_{V_t}:V_t\to S^1\times D^2\times \{ t\}$ and 
$\Pi_{s,t}:S^1\times D^2\times \{ s\}\to S^1\times D^2\times \{ t\}$
is the natural identification.
Let $g_t$ be an isotopy between $g_0$ and 
the identity ($=g_1$).
Then 
$f_t:=\rho_t^{-1}\circ\Pi_{1,t}\circ\rho_1\circ f_1\circ g_t$
gives a continuous family of homeomorphisms 
$:S^1\times D^2\to V_t$ joining $f_0$ and $f_1$.
Using $f_t$, we can extend the orbit structure on $V_0\cup V_1$ 
to that on $W$.
Thus we have an $S^1$-action $\bar\phi$ on $W$ extending
$(V_0,\phi_0)$ and $(V_1,\phi_1)$. 
Note $S^1\times D^2\times I-\interior W\simeq S^1\times A^2\times I$,
where $A^2$ is an annulus of $D$ such that 
$D-A^2$ is an open disk.
From now on we identify $S^1\times D^2\times I-\interior W$
with $S^1\times A^2\times I$.
Now we have a free $S^1$-action $\phi$ on 
$S^1\times\partial(A^2\times I)\simeq T^3$ given by $\varphi$, 
$\phi_1$ and $\bar\phi$. Since $\phi$ gives rise to a trivial 
bundle $S^1\to T^3\to T^2$, it provides an imbedding 
$g:S^1\times\partial(A^2\times I)\to S^1\times A^2\times I$
such that 
\begin{enumerate}
 \item $g$ leaves $S^1\times\partial A^2\times\{ t\}$ invariant for 
       any $t\in I;$
 \item $g(S^1\times (x,t))$ gives the $\phi$-orbit for any 
       $(x,t)\in\partial(A^2\times I);$
 \item $g:\{ 1\}\times\partial(A^2\times I):
       \{ 1\}\times\partial(A^2\times I)\to 
                   S^1\times\partial(A^2\times I)$
       is a section to $\phi$.
\end{enumerate}
Therefore we have a homeomorphism 
$$
h:S^1\times\partial(A^2\times I)-g(\{ 1\}\times\partial(A^2\times I))
   \simeq J\times\partial(A^2\times I),
$$
where $J$ is an open interval.
Extend the section $g$ to an embedding 
$G:\{ 1\}\times A^2\times I\to S^1\times A^2\times I$
so that the $\phi$-orbits meet the $G$-image only with the 
$g$-image. Then 
$h$ extends to a homeomorphism
$$
H:S^1\times A^2\times I-G(\{ 1\}\times A^2\times I)
   \simeq J\times A^2\times I.
$$
Since $H$ extends to a homeomorphism
$\bar H:S^1\times A^2\times I\simeq S^1\times A^2\times I$
we can extend $\phi$ to an $S^1$-action on 
$S^1\times A^2\times I$, proving (3).

(1) is immediate from (3) if one extend 
$\varphi|_{S^1\times\partial D^2\times \{ 1\}}$ to an 
$S^1$-action $\phi_1$ on $S^1\times D^2\times \{ 1\}$
which is equivariant to $\phi_{\mu,\mu-\nu}$.

Note that a nontrivial Seifert bundle 
$S^2\times S^1=\partial (S^1\times D^2\times I) \to
S^2=\partial (D^2\times I)$ is essentially unique.
It has two singular orbits, and 
the Seifert invariants are $(\alpha,\beta)$ and 
$(\alpha,\alpha-\beta)$ for some $\alpha>\beta$
(see \cite{Sc:geometry},p.459).
Let $\psi$ be any extension of $\varphi$ to an 
$S^1$-action on $S^1\times D^2\times I$. From the above argument,
$E^*(\psi)$ consists of a segment along which Seifert invariants
are $(\mu,\nu)$. Obviously $F^*(\psi)$ is empty.
Therefore Proposition \ref{prop:circ-class} yields that 
$\psi$ must be equivariant to 
$\phi_{\mu,\nu}\times\text{\rm id}$.
\end{proof}

\begin{prop} \label{prop:locseif}
Suppose that $B(p_i,r)\simeq S^1\times D^3$ and $m_i=2$. Then
there exists a Seifert fibration 
$\hat f_i:B^{f_i}(p_i,r)\to B(p,r)$ such that 
\begin{enumerate}
 \item $\hat f_i = f_i$ on 
 $A^{f_i}(p_i;r/100,r)-\bigcup_{j=1}^k A_{j}^{f_i}(p_i;\epsilon)$;
 \item $\hat f_i = \tilde f_i$ on 
        $\partial B^{f_i}(p_i,r)$;
 \item the singular fibres of $\hat f_i$ are contained in the union
   of two of $\{ A_{j}^i(p_i;\epsilon)\}$, say, 
   $A_1^i(p_i;\epsilon)\cup A_2^i(p_i;\epsilon)$;
 \item the singular locus $\cal{C}_i$ of $\hat f_i$ is a connected 
    quasi-geodesic containing $p$ in its interior and consisting of 
    essential singular points of $X^3$;
 \item the Seifert invariants of the singular fibre of $\hat f_i$ 
       do not exceed 
       $$
\frac{2\pi}{\max\{ L(\Sigma_{\xi_1}(\Sigma_p)), L(\Sigma_{\xi_2}(\Sigma_p))\}},
       $$
\end{enumerate}
where $\xi_i$, $i=1,2$, are the directions at $p$ represented by $\cal{C}_i$
\end{prop}

\begin{proof}
Identify 
$A_{j}^{f_i}(p_i;\epsilon;r/100,r) \simeq (S^1\times D^2)\times I$.
Now we have the  $S^1$-action, say $\phi_i$, on 
$A_{j}^{f_i}(p_i;\epsilon)\cap \partial B^{f_i}(p_i,r) = 
                 S^1\times D^2\times \{ 1\}$
given by Proposition \ref{prop:bdyseif}.
We also have the $S^1$-action on 
$(S^1\times\partial D^2)\times I$ coming from the $S^1$-bundle structure.
Note that those two actions are compatible on the intersection
$S^1\times\partial D^2\times \{ 1\}$.
%
%
Using Lemma \ref{lem:extend}, we extend those  
$S^1$-actions to one on $S^1\times D^2\times I$  
which is equivariant to the product action of 
the action $\phi_i$ on $S^1\times D^2\times \{ 1\}$ and 
the trivial action on $I$.
Thus we have extended the $S^1$-bundle 
$\tilde f_i:A^{f_i}(p_i;r/100,r)-\bigcup_{j=1}^k  
 A_{j}^{f_i}(p_i;\epsilon) \to 
   A(p;r/100,r)- \bigcup_{j=1}^k A_{j}(p;\epsilon)$
to a Seifert bundle $\bar f_i:A^{f_i}(p_i;r/100,r)\to A(p;r/100,r)$.
Next,  extend it to a Seifert bundle 
$\hat  f_i:B^{f_i}(p_i,r)\to B(p,r)$ as follows. 
Since $m_i=2$, the Seifert invariants of the Seifert bundle
$:\partial B(p_i,r/100)\to \partial B(p,r/100)$, 
the restriction of $\bar f_i$,  are expressed as
$(\mu_i,\nu_i)$ and $(\mu_i, \mu_i-\nu_i)$.
Therefore this Seifert bundle is isomorphic as a fibred space 
to the one on $\partial(V_{\mu_i,\nu_i}\times I)$ induced from the 
product action $\phi_{\mu,\nu}\times \text{\rm id}$ on 
$V_{\mu_i,\nu_i}\times I$, 
where $V_{\mu_i,\nu_i}$ denotes the fibred solid
torus of type $(\mu_i,\nu_i)$.
This provides a Seifert bundle 
$B(p_i,r/100)\to B(p,r/100)$ such that 
\begin{enumerate}
 \item it is equivalent to the product $V_{\mu_i,\nu_i}\times I$;
 \item the Seifert bundle
  structures on $A^{f_i}(p_i;r/100,r)$ and $B(p_i,$ $r/100)$ are compatible
  on their boundaries;
\end{enumerate}
The properties $(1)$, $(2)$ and $(3)$ are now obvious.\par
   To prove $(4)$ and $(5)$,  we first show

\begin{ass} \label{ass:surj}
There exist $r=r(p)>0$ and $\epsilon=\epsilon(p)>0$ such that 
$\partial B(p,s)\cap A_{j}(p;\epsilon)$ contains an essential 
singular point for each $0<s\le r$ and $j\in\{ 1,2\}$.
\end{ass}

\begin{proof}
Suppose that $A(p;s_1,s_2)\cap A_{j}(p;\epsilon)$
does not meet $ES(X^3)$ for some $s_1<s_2\le r$.
Then $A_j^{f_i}(p_i;\epsilon;s_1,s_2)$ has an 
$S^1$-bundle structure compatible with $f_i$. 
This is a contradiction to Lemma \ref{lem:extend}.
\end{proof}

\begin{ass} \label{ass:inj}
There exist $r=r(p)>0$ and $\epsilon=\epsilon(p)>0$ such that 
$$
  \frac{|d_p(y)-d_p(x)|}{d(x,y)}\ge 1-\tau(r),
$$
for any $x,y\in ES(X^3)\cap A_{j}(p;\epsilon)$ with 
$y$ sufficiently close to $x$ and  for 
each $j\in\{ 1,2\}$.
\end{ass}

\begin{proof}
Suppose the assertion does not hold. Then
there would exist sequences 
$x_i$, $y_i$ of $ES(X^3)\cap A_{j}(p;\epsilon_i)$
with $r=r_i\to 0$ and $\epsilon_i\to 0$ such that
\begin{enumerate}
 \item \[
         \frac{|d_p(y_i)-d_p(x_i)|}{d(x_i,y_i)} \le 1-\mu,
       \]
       for some $\mu>0;$
 \item $s_i:=d(p,x_i)\to 0$ and $d(x_i,y_i)/s_i\to 0$.
\end{enumerate}
We may assume that $(\frac{1}{r_i}X^3,x_i)$ converges to 
a space $(Y^3, x_{\infty})$, where $Y^3$ is isometric to 
the tangent cone $K_p$. Putting 
$\epsilon_i=d(x_i,y_i)$, we consider any limit 
$(\hat Y^3,\hat x_{\infty})$ of another rescaling 
$(\frac{1}{\epsilon_i}X^3,x_i)$,
where $\hat Y^3$ is isometric to a product $\R\times Y_0$. 
Let $\hat y_{\infty}$ denote the 
limit of $y_i$ under this convergence. 
Since $\hat x_{\infty},\hat y_{\infty}\in ES(\hat Y^3)$,
and since $\pi(\hat x_{\infty}) \neq \pi(\hat y_{\infty})$, where 
$\pi:\hat Y^3\to Y_0$ is the projection, 
$Y_0$ contains two essential singular 
points with distance, say $a<1$. Thus  
$Y_0$ must be isometric to the double $D([0,a]\times [0,\infty))$.
In particular, $\dim\hat Y(\infty)=1$ while
$\dim Y(\infty)=2$. This contradicts Lemma \ref{lem:expanding2}.
\end{proof}

From Assertions \ref{ass:surj} and \ref{ass:inj}, 
$ES(X^3)\cap (A_{1}(p;\epsilon)\cup A_{2}(p;\epsilon))$ defines 
a continuous curve, denoted $\cal{C}_p$, through $p$,
such that $\Sigma_x(\cal{C}_p)$ contains at least two elements 
for any $x\in\cal{C}_p$. Therefore
$\cal{C}_p$ is a quasigeodesic by Proposition \ref{prop:graph}.
Obviously we can arrange the fibre structure defined by
$\hat f_i$ so that the singular locus $\cal{C}_i$ of 
$\hat f_i$ coincides with $\cal{C}_p$.
$(5)$ follows from Proposition \ref{prop:bdyseif}.
This completes the proof of Proposition 
\ref{prop:locseif}.
\end{proof}

Now we consider the other case.

\begin{proclaim}{\emph{Case} II.}
  $B(p_i,r)\simeq D^4$.
\end{proclaim}

In this case, by \cite{Or:Seif}, the Seifert bundle 
$\tilde f_i:\partial B^{f_i}(p_i,r)\simeq S^3\to\partial B(p,r)\simeq S^2$
is actually given by an $S^1$-action equivariant to the restriction of 
a canonical action $\psi_{a_ib_i}$ on $D^4(1)$ to 
$S^3(1)=\partial D^4(1)$ for suitable relatively prime integers 
$a_i,b_i$. From Lemma \ref{lem:adj-inv}, the Seifert invariants of 
$\tilde f_i:S^3\to S^2$ have forms $(a_i,a_i')$ and $(b_i,b_i')$ with
\begin{equation*}
  \begin{vmatrix}
     a_i  &  b_i \\
     a_i' &  b_i'
  \end{vmatrix}
        = \pm 1.
\end{equation*}

\begin{prop}  \label{prop:lochopf}
Suppose $B(p_i,r)\simeq D^4$. Then 
 \begin{enumerate}
  \item  $p$ is an extremal point of $X^3$;
  \item  there is an $i_0$ such that 
     for every $i\ge i_0$ we have 
     a locally smooth $S^1$-action $\psi_i$ on $B^{f_i}(p_i,r)$ 
     extending $\tilde f_i$ such that the Seifert bundle
     $\hat f_i:A^{f_i}(p_i;r/100,r)\to A(p;r/100,r)$ 
     given by $\psi_i$ satisfies the following:
    \begin{enumerate}
     \item $\hat f_i=f_i$ on 
      $A^{f_i}(p_i;r/100,r)-\bigcup_{j=1}^k 
      A_{j}^{f_i}(p_i;\epsilon;r/100,r)$;
     \item the singular locus of $\hat f_i$ $($if it exists$)$
      extends to a continuous quasi-geodesic $\cal{C}_i$ containing $p$.
      $\cal{C}_i$ is contained in the union of at most two of 
      $A_{j}(p;\epsilon)$'s,
      and consisting of essential singular points of $X^3$;
     \item $\psi_i$ is equivariant to the canonical action $\psi_{a_i,b_i}$ 
      for some relatively 
      prime integers $a_i$ and $b_i$, where 
       $$
          |a_i|\le \frac{2\pi}{L(\Sigma_{\xi_1}(\Sigma_p))}, \quad  
          |b_i|\le \frac{2\pi}{L(\Sigma_{\xi_2}(\Sigma_p))},
       $$
      for some $\xi_1\neq \xi_2$ and at least one of them is a direction of 
      $\cal{C}_i$ at $p$, and the other is another direction 
      of $\cal{C}_i$ if it exists.
     \end{enumerate}
 \end{enumerate}
\end{prop}

\begin{proof}
$(1)$ follows from Proposition \ref{prop:bigdiam}.
$(2)$-(a), (b) are similar to Proposition \ref{prop:locseif}.
$(2)$-(c) follows from an argument similar to Proposition \ref{prop:bdyseif}.
\end{proof}


\section{Collapsing to three-spaces without boundary- \\
          global construction}  \label{sec:dim3nobdyglob}

 Let a sequence of closed orientable $4$-dimensional Riemannian
manifolds $M_i^4$ with $K \ge -1$ converge to a 
$3$-dimensional compact Alexandrov space $X^3$ without boundary. 
By Fibration Theorem \ref{thm:orig-cap}, we have a locally 
trivial $S^1$-bundle 
$f_i:M_i^{\prime}\to X^3-S_{\delta_3}(X^3)$ which is an
$\epsilon_i$-approximation, $\lim_{i\to\infty}\epsilon_i=0$,
where $M_i^{\prime}$ is an open subset of $M_i^4$.
The purpose of this section is to construct a globally defined 
locally smooth, local $S^1$-action on $M_i^4$.

\begin{thm} \label{thm:patch}
Suppose that $M_i^4$ collapses to a $3$-dimensional  compact Alexandrov
space $X^3$ without boundary under $K\ge -1$. 
Then there exists a locally smooth, local $S^1$-action 
$\psi_i$ on $M_i^4$ satisfying the following$:$
 \begin{enumerate}
  \item $M_i^4/\psi_i\simeq X^3;$
  \item We have an inclusion $F^*(\psi_i)\subset\Ext(X^3);$
  \item For each $x_{\alpha}\in F^*(\psi_i)$, there are 
    $p_i^{\alpha}\in M_i^4$ close to $x_{\alpha}$ and  
    $r_{\alpha} > 0$ independent of $i$ such that
    $B^{f_i}(p_i^{\alpha}, r_{\alpha})$ is $\psi_i$-invariant and 
    $(B^{f_i}(p_i^{\alpha}, r_{\alpha}), \psi_i)$ is equivariantly 
    homeomorphic to $(D^4(1),\psi_{a_ib_i})$ for some relatively 
    prime integers $a_i,b_i$ satisfying
       $$
          |a_i|\le \frac{2\pi}{\ell_1}, \quad  |b_i|\le \frac{2\pi}{\ell_2},
       $$
    where $\ell_1$ and $\ell_2$ are described as in 
    Proposition \ref{prop:lochopf} (2)$;$
  \item Each connected component $\cal{C}_i\subset X^3$ of the singular  
    locus $S^*(\psi_i)$ 
    is either a periodic quasi-geodesic in $X^3-F^*(\psi_i)$ or 
    a quasi-geodesic path or loop in $X^3$ connecting some points of 
    $F^*(\psi_i);$
  \item The order of the isotropy subgroup along every component 
    $\hat{\cal C}_i$ of the exceptional locus $E^*(\psi)$ does not exceed 
    $$
          \inf_{x\in\hat{\cal C}_i} \frac{2\pi}{L_{\xi_x}(\Sigma_x)},
    $$
    where $\xi_x$ is a direction at $x$ defined by $\hat{\cal C}_i$.
 \end{enumerate}
\end{thm}

\begin{rem}
 Theorem \ref{thm:patch} $(2)$ means that for any fixed point 
 $p_i^{\alpha}\in {\rm Fix}(\psi_i)$,
 there corresponds a unique extremal point $x_{\alpha}\in \Ext(X^3)$ 
 which is Gromov-Hausdorff close to $p_i^{\alpha}$.
\end{rem}

\begin{ex} \label{ex:S^4}
For relatively prime integers $a$, $b$, let us consider 
the restriction of the canonical action $\psi_{ab}$ 
on $D^4(1)$ to $S^3(1)$.
This naturally extends to an isometric $S^1$-action 
$\tilde\psi_{a,b}$ on $S^4(1)$,
the spherical suspension over $S^3(1)$, with two fixed points,
say $x_1$ and $x_2$. 
Let $X^3$ be the quotient space
$S^4(1)/\tilde\psi_{a,b}$, which is the spherical suspension over 
$S^3(1)/\psi_{a,b}\simeq S^2$.
By \cite{Ym:collapsing}, we have a sequence of metrics 
$g_i$ on $S^4$ such that 
\begin{enumerate}
 \item the sectional curvature $K_{g_i}$ has a uniform lower bound;
 \item $(S^4,g_i)$ collapses to $X^3$.
\end{enumerate}
Note that 
\begin{enumerate}
 \item[$($3$)$] the diameters of the spaces (say $\Sigma$) of 
   directions at the 
   poles of $X^3$ are both equal to $\pi/2$;
 \item[$($4$)$] $\Sigma$ has at most two singular points, 
   say $\xi$, $\eta$, realizing the diameter of it,  where 
   $$
       L(\Sigma_{\xi}(\Sigma) = 2\pi/|a|,\quad
       L(\Sigma_{\eta}(\Sigma) = 2\pi/|b|.
   $$ 
\end{enumerate}
The singular locus $S^*(\tilde\psi_{a,b})$ of $\tilde\psi_{a,b}$ is
a simple closed loop (resp. a simple arc) joining the two poles 
if $|a|\ge 2$ and $|b|\ge 2$ (resp. if just one of $|a|$ and $|b|$ 
is greater than $1$).
\end{ex}

\begin{ex} \label{ex:CP^2}
Consider the $S^1$-action $\psi$ on $\C P^2$ given by
$$
    z\cdot [z_0:z_1:z_2] = [z^{a_0}z_0:z^{a_1}z_1:z^{a_2}z_2],
$$
for suitably chosen integers $a_i$, $0\le i\le 2$.
We assume that $a_{i+1}-a_{i}$ and $a_{i+2}-a_{i}$
are relatively prime, where $i$ is understood in $mod\,\, 3$. 
Note that $F^*(\psi)$  consists of $x_0=[1:0:0]$,
$x_1=[0:1:0]$, and $x_2=[0:0:1]$.
The action $\psi$ induces actions equivariant to 
the canonical actions of type $(a_{i+1} - a_{i}, a_{i+2} - a_{i})$
on a neighborhood of $x_i$.
For  an invariant metric  $g$ on $\C P^2$,
take a sequence of metrics $g_i$ on $\C P^2$
such that 
\begin{enumerate}
 \item the sectional curvature $K_{g_i}$ has a uniform lower bound;
 \item $(\C P^2,g_i)$ converges to the quotient space 
   $X^3=(\C P^2,g)/S^1$.
\end{enumerate}
Note that 
\begin{enumerate}
 \item[$($3$)$] $\diam(\Sigma_{\bar x_i})=\pi/2$, where 
   $\bar x_i\in X^3$ is 
   the image of $x_i$ under the projection $\C P^2\to X^3$;
 \item[$($4$)$] $\Sigma_{\bar x_i}$ has exactly two singular points, 
   say $\xi_{i,i+1}$, $\xi_{i,i+2}$ realizing the diameter of 
   $\Sigma_{\bar x_i}$, where 
   $$
     L(\Sigma_{\xi_{i,i+1}}(\Sigma_{\bar x_i})) = 2\pi/|a_{i+1}-a_i|,\quad
     L(\Sigma_{\xi_{i,i+2}}(\Sigma_{\bar x_i})) = 2\pi/|a_{i+2}-a_i|.
   $$
\end{enumerate}
Now suppose $a_0=0$ for instance. Then
the singular locus  $S^*(\psi)$ is
a simple arc joining $x_1$ and $x_2$ through $x_0$. 
If $a_{i+1}-a_{i}$ and $a_{i+2}-a_{i}$
are relatively prime for every $i$  in $mod\,\, 3$,
then $\cal{C}$ is a simple closed loop joining $x_0$, $x_1$ and $x_2$.
\end{ex}

\begin{ex} \label{ex:S^2S^2}
Let $S^2(1)\subset \C\times\R$ be the unit sphere and 
consider the $S^1$-action $\psi$ on $S^2(1)\times S^2(1)$ given by
$$
    z\cdot (z_1,t_1,z_2,t_2) = (z^{a}z_1,t_1,z^{b}z_2,t_2),
$$
for relatively prime integers $a$ and $b$,
where $z_i\in\C$, $t_i\in\R$, $i=1,2$.
The action $\psi$ induces actions equivariant to 
the canonical actions of type $(a, b)$
on a neighborhood of each of the four fixed point of $\psi$, up to 
orientation.
We can take a sequence of metrics $g_i$ on $S^2(1)\times S^2(1)$
such that 
\begin{enumerate}
 \item the sectional curvature $K_{g_i}$ has a uniform lower bound;
 \item $(S^2(1)\times S^2(1),g_i)$ converges to the quotient space 
   $X^3=(S^2(1)\times S^2(1))/S^1$.
\end{enumerate}
Note that $X^3$ is an Alexandrov space with nonnegative curvature
having no boundary.
The singular locus  $S^*(\psi)$ consists of 
$(\rm i)$ a simple loop joining those four fixed points
if $|a|, |b| \neq 1$, $({\rm ii})$ two segments
joining those four fixed points
if just one of $|a|$ and $|b|$ is equal to $1$,
$({\rm iii})$ those four fixed points
if both $|a|$ and $|b|$ are equal to $1$.
\end{ex}

In view of Examples \ref{ex:S^4}, \ref{ex:CP^2} and \ref{ex:S^2S^2},
taking connected sum around fixed points, 
we can construct an $S^1$-action on every connected sum
of $S^4$, $\pm \C P^2$ and $S^2\times S^2$ whose
orbit space $X^3$ has no boundary.
\par\medskip

\begin{proof}[Proof of Theorem \ref{thm:patch}]
Recall that for each $p\in S_{\delta_3}(X^3)$, there are
sufficiently small $\epsilon=\epsilon_p>0$ and $r=r(p,\epsilon)>0$
for which we have a locally smooth $S^1$-action 
$\psi_{p,i}$ on $B^{f_i}(p_i,r)$
for sufficiently large $i$, where $p_i\in M_i^4$ is close to $p$.
To prove Theorem \ref{thm:patch}, we patch those $S^1$-actions 
and a local $S^1$-action defined by $f_i$ to obtain a globally 
defined locally smooth, local $S^1$-action.
Note that if $p$ is an isolated point of $S_{\delta_3}(X^3)$, we 
have nothing to do for the patching. 

For any non-isolated point $p$ of $S_{\delta_3}(X^3)$, 
let $\epsilon=\epsilon_p$ and $r=r(p,\epsilon )$ be so small that 
\begin{enumerate}
  \item $\{ A_{j}(p;\epsilon)\}_{j=1,\ldots, k(p)}$ covers 
        $B(p,2r)\cap S_{\delta_3}(X^3)$, where 
        $A_j(p;\epsilon)=B(px^j,\epsilon r)\cap B(p,2r)$ is defined as 
        in the previous section;
  \item $(A_{j}(p;\epsilon)-A_{j}(p;\epsilon^{10}))\cap
         A(p;r/100, 2r)$ does not meet $S_{\delta_3}(X^3);$
  \item $\{ A_j(p;\epsilon)-B(p,\epsilon/100)\}_j$ is disjoint.
\end{enumerate}
From the compactness of $S_{\delta_3}(X^3)$, take $p^1,\ldots,p^N$ 
such that 
$S_{\delta_3}(X^3)\subset \bigcup_{\alpha=1}^N B(p^{\alpha},r_{\alpha}/3)$,
where $r_{\alpha}=r(p^{\alpha},\epsilon_{\alpha})$ and $\epsilon_{\alpha}$
are defined as above. 
We assume that 
$r_1\ge r_2\ge\cdots\ge r_N$.
Note that if $d(p^{\alpha},p^{\beta})\le r_{\alpha}/3$ with 
$\alpha <\beta$, then 
$B(p^{\alpha},r_{\alpha})\supset B(p^{\beta},r_{\beta}/3)$.
Hence considering the covering  
$\{ B(p^{\alpha},r_{\alpha})\}$  
of $S_{\delta_3}(X^3)$ in stead of $\{ B(p^{\alpha},r_{\alpha}/3)\}$,
and removing some of $\{ r_2,\ldots, r_N\}$ if necessary, 
we may assume that 
$B(p^{\alpha}, r_{\alpha}/3)$ does not contain $p^{\beta}$
for every $1\le\alpha<\beta\le N-1$. In particular, 
$\{B(p^{\alpha},r_{\alpha}/6)\}$ is disjoint. \par

 Let $p_i^{\alpha}\in M_i^4$ be a point close to $p^{\alpha}$. 
By Fibration Theorem \ref{thm:orig-cap}, we have an 
$S^1$-bundle structure $\cal{F}_i$ on 
$f_i^{-1}(X^3 - S_{\delta_3}(X^3))\subset M_i^4$.
On each $B^{f_i}(p_i^{\alpha},r_{\alpha})$, we have the
structure $\cal{S}_i^{\alpha}$ of locally smooth  $S^1$-action 
$\psi_{p^{\alpha},i}$
given by Proposition \ref{prop:locseif}
or \ref{prop:lochopf}. We have to patch all of those structures
$\cal{F}_i$ and  $\cal{S}_i^{\alpha}$ together to obtain a globally 
defined structure on $M_i^4$ of locally smooth, local $S^1$-action as 
stated in Theorem \ref{thm:patch}.
\par 







   Let $A_{j_{\alpha}}(p^{\alpha};\epsilon_{\alpha})$, 
$A_{j_{\alpha}}(p^{\alpha};\epsilon_1;r_1), \ldots, 
1\le j_{\alpha}\le k(p^{\alpha})$, be as defined right after 
Lemma \ref{lem:coverS} for $p^{\alpha}$.  
%
%
%
%
%
%
%
%
%
Now we are going to patch 
$\cal{F}_i$, $\cal{S}_i^{\alpha}$ and $\cal{S}_i^{\beta}$ assuming
$A_{j_{\alpha}}(p^{\alpha};\epsilon_{\alpha};r_{\alpha})\cap 
A_{j_{\beta}}(p^{\beta};\epsilon_{\beta};r_{\beta})$ to be nonempty 
with $\alpha<\beta$.

If $A_{j_{\alpha}}(p^{\alpha};\epsilon_{\alpha};r_{\alpha})\cap 
A_{j_{\beta}}(p^{\beta};\epsilon_{\beta};r_{\beta})$
does not meet $S_{\delta_3}(X^3)$, there is no need to 
patch $\psi_{p^{\alpha},i}$ and $\psi_{p^{\beta},i}$ 
on the intersection $B^{f_i}(p_i^{\alpha},r_{\alpha})\cap
B^{f_i}(p_i^{\beta},r_{\beta})$.
Hence assume that there is a point $x$ in the intersection of
$A_{j_{\alpha}}(p^{\alpha};\epsilon_{\alpha};r_{\alpha})$, 
$A_{j_{\beta}}(p^{\beta};\epsilon_{\beta};r_{\beta})$
and $S_{\delta_3}(X^3)$.

We consider the following two cases depending on 
the value $\angle x p^{\alpha}p^{\beta}$.

\begin{proclaim}{\emph{Case} I.}
 $\angle x p^{\alpha}p^{\beta} >\epsilon_{\alpha}/2$.
\end{proclaim}

Then $p^{\beta}$ must be contained in some  
$A_{k_{\alpha}}(p^{\alpha};\epsilon_{\alpha};r_{\alpha})$
with $k_{\alpha}\neq j_{\alpha}$.
Extending the $S^1$-action on $B^{f_i}(p^{\alpha},r_{\alpha})$
to one on $B^{f_i}(p^{\alpha},2r_{\alpha})$ in a similar way, 
we can borrow the $S^1$-action structure on
$B^{f_i}(p^{\alpha},2r_{\alpha})$ as a compatible $S^1$-action 
structure on 
$A_{j_{\beta}}(p^{\beta};\epsilon_{\beta};r_{\beta})$.
See Figure 2.
\vspace{0.5cm}
\begin{center}
\begin{tikzpicture}
[scale = 0.5]
\draw [-,thick] (-4,0) -- (4.5,0);
\draw [-,thick] (-4,-2) -- (4.5,-2);
\fill (-4,-1) circle (2pt) node [left]{$p_\alpha$};
\draw [-, thick] (4.5,0) to [out=290, in=70] (4.5,-2);
\draw [-,thick] (4.5,-4.6) -- (2.5,-1.2);
\draw [-,thick] (6,-4.0) -- (4,-0.6);
\fill (3.8,-1.2) circle (2pt) node [below]{$x$};
\draw [-, thick] (2.5,-1.2) to [out=45, in=175] (4,-0.6);
\fill (5.6,-4.5) circle (2pt) node [below]{$p_\beta$};
\fill (0, -6)  circle (0pt) node [below] {Figure 2};
\end{tikzpicture}
\end{center}
\vspace{0.5cm}


\begin{proclaim}{\emph{Case} II.}
 $\angle x p^{\alpha}p^{\beta} \le\epsilon_{\alpha}/2$.
\end{proclaim}

In this case, we actually have $p^{\beta}\in
A_{j_{\alpha}}(p^{\alpha};\epsilon_{\alpha}^{10};r_{\alpha})$
if $d(x,p^{\beta})\le\epsilon_{\alpha} r_{\alpha}$, then 
taking $r_{\alpha}$ slightly larger if necessary,
we may assume that 
$B^{f_i}(p_i^{\beta},r_{\beta})\subset 
          B^{f_i}(p_i^{\alpha},r_{\alpha})$.
Therefore we can remove $\beta$ from $\{ \alpha\}$, and 
we may assume $d(x,p^{\beta}) >\epsilon_{\alpha} r_{\alpha}$.
It follows that $\tilde\angle p^{\alpha}xp^{\beta}$ is sufficiently 
small or sufficiently close to $\pi$. In the former case,
taking $r_{\alpha}$ slightly larger if necessary,
we may assume that 
$B^{f_i}(p_i^{\beta},r_{\beta})\subset 
          B^{f_i}(p_i^{\alpha},r_{\alpha})$.
Therefore we can remove $\beta$ from $\{ \alpha\}$ as before.
In the latter case, 
assuming $\epsilon_{\alpha}r_{\alpha}\ge\epsilon_{\beta}r_{\beta}$
(the other case is proved similarly),
we have
$$
A^{f_i}_{j_{\beta}}(p_i^{\beta};\epsilon_{\beta};r_{\beta}/100,r_{\beta})
  \cap \interior B^{f_i}(p_i^{\alpha},r_{\alpha})
  \subset \interior 
   A^{f_i}_{j_{\alpha}}(p_i^{\alpha};
     2\epsilon_{\alpha};r_{\alpha}/100,r_{\alpha}).
$$
 See Figure 3. \par

\begin{center}
\begin{tikzpicture}
[scale = 0.5]
\draw [-,thick] (-4,0) -- (4.5,0);
\draw [-,thick] (-4,-2) -- (4.5,-2);
\fill (-4,-1) circle (2pt) node [left]{$p_\alpha$};
\draw [-, thick] (4.5,0) to [out=290, in=70] (4.5,-2);
\draw [-,thick] (2.5,-0.5) -- (8.5,-0.5);
\draw [-,thick] (2.5,-1.5) -- (8.5,-1.5);
\fill (8.7,-1) circle (2pt) node [right]{$p_\beta$};
\draw [-, thick] (2.5,-0.5) to [out=225, in=135] (2.5,-1.5);
\fill (3.5,-1) circle (2pt) node [right]{$x$};
\fill (2, -4)  circle (0pt) node [below] {Figure 3};

\end{tikzpicture}
\end{center}


\noindent
We perturb $\cal{S}_i^{\alpha}$ to a new locally smooth $S^1$-action 
$\tilde{\cal{S}}_i^{\alpha}$ on 
$A^{f_i}_{j_{\alpha}}(p_i^{\alpha};2\epsilon_{\alpha};
r_{\alpha}/100,r_{\alpha})\cap B^{f_i}(p_i^{\beta},r_{\beta})$ 
as follows: 
we let $\tilde{\cal{S}}_i^{\alpha}$ coincide 

\begin{itemize}
 \item  with $\cal{S}_i^{\beta}$ on
        $A^{f_i}_{j_{\alpha}}(p_i^{\alpha}; 
        2\epsilon_{\alpha};r_{\alpha}/100,r_{\alpha}) \cap       
   A^{f_i}_{j_{\beta}}(p_i^{\beta};\epsilon_{\beta};r_{\beta}/100,r_{\beta})$
   and $;$ 
 \item with $\cal{F}_i$ on the complement of
   $A^{f_i}_{j_{\beta}}(p_i^{\beta};\epsilon_{\beta};
    r_{\beta}/100,r_{\beta})$ in 
   $A^{f_i}_{j_{\alpha}}(p_i^{\alpha};$ $2\epsilon_{\alpha}
   ;r_{\alpha}/100,r_{\alpha})\cap B^{f_i}(p_i^{\beta},r_{\beta})$.
\end{itemize}
Note that this is well-defined. Using 
$$ 
A^{f_i}_{j_{\alpha}}(p_i^{\alpha};2\epsilon_{\alpha};r_{\alpha}/100,r_{\alpha})
- \interior B^{f_i}(p^{\beta},r_{\beta}) \simeq (S^1\times D^2)\times
I,   
$$
 together with lemma \ref{lem:extend}, 
 we can extend $\tilde{\cal{S}}_i^{\alpha}$ to a locally smooth
 $S^1$-action 
 on the complement of $B^{f_i}(p_i^{\beta},r_{\beta})$ in 
 $A^{f_i}_{j_{\alpha}}(p_i^{\alpha};\epsilon_{\alpha};
 r_{\alpha}/100,r_{\alpha})$, 
 which is compatible with $\cal{S}_i^{\alpha}$ on 
 $\{ d_{p_i^{\alpha}}=r_{\alpha}/100\}\cap
 A^{f_i}_{j_{\alpha}}(p_i^{\alpha};\epsilon_{\alpha};
 r_{\alpha}/100,r_{\alpha})$.
 Thus we obtain the required locally smooth
 $S^1$-action on the union of 
 $A^{f_i}_{j_{\alpha}}(p_i^{\alpha};2\epsilon_{\alpha};
 r_{\alpha}/100,r_{\alpha})$  and 
 $A^{f_i}_{j_{\beta}}(p_i^{\beta};\epsilon_{\beta};
 r_{\beta}/100,r_{\beta})$,
 which is compatible with $\tilde{\cal{S}}_i^{\alpha}$, 
 $\cal{F}_i$ and $\cal{S}_i^{\beta}$.
 Repeating this local patching procedure finitely many times, 
 we obtain the global collapsing structure and complete the proof of 
 Theorem \ref{thm:patch}.
\end{proof}


\section{Collapsing to three-spaces with boundary} 
\label{sec:dim3wbdy}

In this section, we consider the situation that a sequence of closed 
orientable $4$-dimensional Riemannian
manifolds $M_i^4$ with $K \ge -1$ converge to a 
$3$-dimensional compact Alexandrov space $X^3$ with boundary. 
We construct a globally defined locally smooth, local 
$S^1$-action on $M_i^4$ to complete the proof of Theorem 
\ref{thm:dim3}.
\par

\begin{thm} \label{thm:patchwbdy}
Suppose that $M_i^4$ collapses to a $3$-dimensional  compact Alexandrov
space $X^3$ with boundary under $K\ge -1$. 
Then there exists a locally smooth, local $S^1$-action 
$\psi_i$ on $M_i^4$ satisfying the following:
 \begin{enumerate}
  \item $M_i^4/\psi_i\simeq X^3$;
  \item $F^*(\psi_i)$ coincides with the union of $\partial X^3$ 
    and a subset of 
    ${\rm Ext}(\interior X^3)$,
    say $\{ x_{\alpha}\}_{\alpha=1,\ldots,k};$
  \item For each $1\le\alpha\le k$, there is an $r_{\alpha} > 0$ 
    independent of $i$ such that if $p^{\alpha}_i\in M_i^4$ denotes the 
    $\psi_i$-fixed point corresponding to $x_{\alpha}$, then 
    $B^{f_i}(p_i^{\alpha}, r_{\alpha})$ is $\psi_i$-invariant and 
    $(B^{f_i}(p_i^{\alpha}, r_{\alpha}), \psi_i)$ is equivariantly 
    homeomorphic to $(D^4(1), \psi_{a_i,b_i})$ for some relatively 
    prime integers $a_i,b_i$ satisfying
    \[
          |a_i|\le \frac{2\pi}{\ell_1}, \quad  |b_i|\le \frac{2\pi}{\ell_2},
    \]
    where $\ell_1$ and $\ell_2$ are described as in 
    Proposition \ref{prop:lochopf}$;$
  \item Each connected component ${\cal C}_i$ of the singular locus  
    of $\psi_i$ in $\interior X^3$ 
    is either a periodic quasi-geodesic in 
    $\interior X^3-\{ x_{\alpha}\}$ or 
    a quasi-geodesic path or loop in $\interior X^3$ 
    connecting some of $\{ x_{\alpha}\};$
 \item The order of the isotropy subgroup along every component 
    $\hat{\cal C}_i$ of the exceptional locus $E^*(\psi)$ does not exceed 
    \[
     \inf_{x\in\hat{\cal C}_i}\frac{2\pi}{L(\Sigma_{\xi_x}(\Sigma_x))},
    \]
    where $\xi_x$ is a direction at $x$ defined by $\hat{\cal C}_i$.
  \end{enumerate}
\end{thm}

\begin{ex} \label{ex:S^4wbdy}
Let us consider 
the $S^1$-action $\psi$ 
on $S^4=D^3\times S^1\cup S^2\times D^2$, 
where $S^1$ acts canonically only on 
$S^1$-factors and $D^2$-factors of 
$D^3\times S^1$ and $S^2\times D^2$ respectively.
Then $S^4/\psi=D^3$ and ${\rm Fix}(\psi)=\partial D^3$.
Fixing a $\psi$-invariant metric $g$ on $S^4$, 
by \cite{Ym:collapsing}, we have a sequence of metrics 
$g_i$ on $S^4$ such that 
$(S^4,g_i)$ collapses to $(S^4,g)/\psi$ under a lower sectional
curvature bound.
\end{ex}

\begin{ex} \label{ex:CP^2wbdy}
Consider the $D^2$-bundle $S^2\tilde\times_{-1} D^2$ over
$S^2$ with the Euler number $-1$ and choose a fibre metric on 
$S^2\tilde\times_{-1} D^2$
whose fibre is isometric to the unit disk $D^2(1)$. 
Let us consider the $S^1$-action $\psi$ on 
$\C P^2=D^4(1)\cup_{S^3}S^2\tilde\times_{-1} D^2 $ 
such that the action of $\psi$ coincides with the 
canonical action $\psi_{1,1}$ 
on $D^4(1)$ and $\psi$ acts on each 
$S^2\tilde\times_{-1} S^1(r)\subset S^2\tilde\times_{-1} D^2$, 
$0<r\le 1$, as the Hopf fibration. 
Then $F(\psi)$ is the disjoint union of the origin of $D^4(1)$ and 
the zero-section of $S^2\tilde\times_{-1} D^2$,
and $\C P^2/\psi\simeq D^3$. 
For  an invariant metric  $g$ on $\C P^2$,
we obtain a sequence of metrics $g_i$ on $\C P^2$
such that $(\C P^2,g_i)$ collapses to the quotient space 
$X^3=(\C P^2,g)/\psi$ under a lower sectional
curvature bound.
\end{ex}

\begin{ex} \label{ex:S^2S^2wbdy}
Let us consider the $S^1$-action on $S^2\times S^2$
such that $S^1$ acts only on one $S^2$-factor.
Then $S^2\times S^2/S^1=I\times S^2$ and 
the fixed point set is the disjoint union of 
two copies of $S^2$.
For  an invariant metric  $g$ on $S^2\times S^2$,
we obtain a sequence of metrics $g_i$ on $S^2\times S^2$
such that the sectional curvature $K_{g_i}$ has a uniform lower bound
and $(S^2\times S^2,g_i)$ collapses to the quotient space 
$X^3=(S^2\times S^2,g)/S^1$.
\end{ex}

In Section \ref{sec:dim3nobdyglob}, we have constructed
$S^1$-action on every connected sum
of $S^4$, $\pm \C P^2$ and $S^2\times S^2$ whose
orbit space $X^3$ has no boundary.
In view of Examples \ref{ex:S^4wbdy}, \ref{ex:CP^2wbdy} 
and \ref{ex:S^2S^2wbdy} together with the construction 
in the previous section, taking connected sum around 
fixed points, 
we can construct an $S^1$-action on every connected sum
of $S^4$, $\pm \C P^2$ and $S^2\times S^2$ whose
orbit space $X^3$ has nonempty boundary.\par
\medskip

Let us go back to the situation of Theorem \ref{thm:patchwbdy}.
Take a small $\nu>0$ such that $X^3-X_{\nu}^3$ provides a collar
neighborhood of $\partial X^3$ (Theorem \ref{thm:collar}). 
By Theorem \ref{thm:patch}, we obtain 
a closed domain $M^4_{\nu i}$ of $M_i^4$ such that 
\begin{enumerate}
 \item it collapses to $X_{\nu}^3$;
 \item there exists a locally smooth, local $S^1$-action on $M^4_{\nu i}$
    such that $M^4_{\nu i}/S^1 \simeq X_{\nu}^3$.
\end{enumerate}

 Let $M_i^{\partial}$ denote the closure of $M_i^4 - M_{\nu i}^4$.
The rest of this section is almost devoted to prove 

\begin{thm}  \label{thm:D2along}
$M_i^{\partial}$ is homeomorphic to a $D^2$-bundle over 
$\partial X^3_{\nu}$ compatible with the $S^1$-bundle structure 
of $\partial M_{\nu i}^4$ induced from the local $S^1$-action on 
$M^4_{\nu i}$.
\end{thm}

From Theorem \ref{thm:D2along}, 
$M_i^4$ is a gluing of $M^4_{\nu i}$ equipped with the local 
$S^1$-action  and a $D^2$-bundle over $\partial X^3_{\nu}$,
where the $S^1$-orbit over a point $x\in \partial X^3_{\nu}$ 
is identified with the boundary of the $D^2$-fibre over
$x$. Therefore it is straightforward to construct a desired 
globally defined, 
locally smooth, local $S^1$-action on $M_i^4$ satisfying the conclusions 
of Theorem \ref{thm:patchwbdy}.

A key in the proof of Theorem \ref{thm:D2along} is the following
lemma.

\begin{lem}  \label{lem:bdyball}
\begin{enumerate}
 \item For a point $p\in \partial X^3$, take $p_i\in M_i^4$ with $p_i\to p$. 
       Then there exists a positive number $r_0$ such that 
       $B(p_i, r)$ is homeomorphic to $D^4$ for every $r\le r_0$
       and each sufficiently large $i\ge i_0$, where  $i_0=i_0(r);$
 \item There are no singular orbits over $\partial X_{\nu}^3$,
       namely  $\partial M_{\nu i}^4$ is an $S^1$-bundle over 
       $\partial X_{\nu}^3$.
\end{enumerate}
\end{lem}

By Theorem \ref{cor:dim3top}, 
\begin{align}
  B(p_i,r) \simeq D^4 \quad \mbox{or} \quad
           S^1\times D^3.  \label{eq:D4S1D3}
\end{align}
We have to exclude the second possibility in \eqref{eq:D4S1D3}. \par

Choose a $\delta_1 r$-net $N(\delta_1)$ of
$A(p;r/10,2r)\cap\partial X$
and take a finite subset 
$N_i(\delta_1)\subset M_i^4$ converging to $N(\delta_1)$, and  
suppose that  $d_{GH}(M_i^4,X^3)<\epsilon_i$ and 
$\epsilon_i \ll \delta_1 r$.

Consider a smooth approximation 
  $$
      f_i = \tilde d(N_i(\delta_1),\,\cdot\,\,).
  $$

Since $f :=d(N(\delta_1),\,\cdot\,\,)$ is regular on 
$f^{-1}([\,\delta_2 r/100,\delta_2 r\,])$,
where $\delta_1\ll \delta_2\ll 1$, $f_i$ is also regular on 
$f_i^{-1}([\,\delta_2 r/100,\delta_2 r\,])$.

\begin{slem} \label{slem:sph-int}
There exist positive numbers $\nu$ and $r$ with $\nu\ll r$
 such that for any sufficiently large $i\ge i(\nu, r)$ 
we have a closed domain $M_i(p_i,r,\nu)$ satisfying 
\begin{enumerate}
  \item $M_i(p_i,r,\nu)$ converges to $A(p;r/10, 2r)\cap X_{\nu}^3$
      under the convergence $M_i^4\to X^3$;
  \item  $M_i(p_i,r,\nu)\cap\partial B(p_i,r)$ is homeomorphic to
      a Seifert bundle over $\partial B(p,r)\cap X_{\nu}^3\simeq D^2$
      having at most one singular orbit;
  \item  $M_i(p_i,r,\nu)$ is homeomorphic to 
      $(M_i(p_i,r,\nu)\cap\partial B(p_i,r))\times I$.
\end{enumerate}
\end{slem}

\begin{proof}
Take $\nu>0$ such that $\interior \Sigma_p - (\Sigma_p)_{\nu}$ 
does not meet $S_{\delta_3}(\Sigma_p)$. Then in view of the convergence
$(\frac{1}{r} X^3,p) \to (K_p,o_p)$ as $r\to 0$, we have such an
$r_0>0$  that $A(p;r/10,2r)\cap (X^3_{\nu^{10} r} - X^3_{\nu r})$ 
does not meet $S_{\delta_4}(X^3)$ for any $r\le r_0$. 
Note that 
$d_{p.GH}((\frac{1}{r} M_i^4,p_i), (K_p,o_p)) < 
\epsilon_r' + \epsilon_i/r$ with 
$\lim_{r\to 0}\epsilon_r' = 0$, $\lim_{i\to\infty}\epsilon_i=0$, 
and that  $\Sigma_p$ satisfies the assumption of Lemma \ref{lem:1vertex}.
Hence if $r$ is sufficiently small and $i$ is sufficiently large
$i\ge i(r)$,  Theorem \ref{thm:patch} applied to 
$A(o_p;1/10, 2)\cap (K_p)_{\nu}$ gives a 
closed domain $M_i(p_i,r,\nu)$ satisfying $(1)$ and $(2)$.

Using $f_i$-flow curves, we can slightly deform 
$M_i(p_i,r,\nu)$ to a closed domain  $M_i'(p_i,r,\nu)$
in such a way that it is $\tilde d_{p_i}$-flow-invariant.
Therefore 
$M_i'(p_i,r,\nu)$ is homeomorphic to
$(M_i'(p_i,r,\nu)\cap\partial B(p_i,r))\times I$,
yielding $(3)$.
\end{proof}

Putting $M_i^{\partial}(p_i,r,\nu)$ to be the closure of 
$A(p_i;r/10,2r) - M_i(p_i,r,\nu)$,
we next investigate the topology of 
$M_i^{\partial}(p_i,r,\nu)\cap\partial B(p_i,r)$.

In a way similar to Sublemma \ref{slem:sph-int} (3), we obtain 
$$
 M_i^{\partial}(p_i,r,\nu) \cap A(p_i;r/5,r) \simeq 
    (M_i^{\partial}(p_i,r,\nu)\cap\partial B(p_i,r)) \times I.
$$

\begin{slem} \label{slem:sph-bdy}
$M_i^{\partial}(p_i,r,\nu)\cap\partial B(p_i,r)$ is a 
$D^2$-bundle over $\partial X_{\nu}\cap \partial B(p,r)\simeq S^1$
which is compatible with the $S^1$-bundle over 
$\partial X_{\nu}\cap \partial B(p,r)$ induced from the fibre structure 
on $M_i(p_i,r,\nu)\cap\partial B(p_i,r)$.

In particular 
$M_i^{\partial}(p_i,r,\nu)\cap\partial B(p_i,r)$ is homeomorphic 
to $D^2\times S^1$.
\end{slem}

\begin{proof}
For a small $\epsilon_1 >0$ we take finitely many points 
$\xi_1,\ldots,\xi_N$ of $\partial\Sigma_p$ satisfying 

\begin{enumerate}
 \item $\xi_j$ is adjacent to $\xi_{j-1}$;
 \item $L(\Sigma_{\xi}(\Sigma_p)) > \pi-\epsilon_1$ for any element
    $\xi$ of the complement of $\{ \xi_1,\ldots,\xi_N\}$ in 
    $\partial\Sigma_p$;
 \item there exist positive numbers $\delta_2 \ll \delta_3\ll 1$ such that
  \begin{enumerate}
      \renewcommand{\labelenumi}{(\alpha{enumi})}
    \item for every $\xi\in B(\xi_j,\delta_3)$, there exists 
      some $\eta\in \Sigma_p$ with 
      $\tilde\angle \xi_j \xi \eta  > \pi-\epsilon_1;$ 
    \item $\tilde\angle \xi_j \xi \xi_{j+1} > \pi-\epsilon_1$ for every 
      $\xi\in B(\widehat{\xi_j \xi_{j+1}},\delta_2) \setminus B(\xi_j,\delta_3)
      \setminus B(\xi_{j+1},\delta_3)$, where $\widehat{\xi_j \xi_{j+1}}$ 
       is the minimizing arc joining $\xi_j$ and $\xi_{j+1}$ 
       in $\partial\Sigma_p;$
   \end{enumerate}
  \end{enumerate}
In view of the convergence
$(\frac{1}{r} X^3,p)\to (K_p,o_p)$, $r\to 0$, we find a small $r$ 
and a finite set 
$\{x^1,\ldots,x^N\}\subset \partial B(p,r)\cap\partial X^3\simeq S^1$
such that 

\begin{enumerate}
 \item $x^j$ is adjacent to $x^{j-1}$;
 \item there exists a sufficiently small positive number 
    $\delta_4 \ll \delta_3$ such that
  \begin{enumerate}
     \renewcommand{\labelenumi}{(\alpha{enumi})}
   \item $d_{GH}(\Sigma_x(X^3), D^2_{+}(1)) < \tau(\epsilon_1)$ 
      for any $x\in\partial B(p,r)\cap\partial X^3$ 
      in the complement of $\delta_4 r$-neighborhood of 
      $\{x^1,\ldots,x^N\}$, where $D^2_{+}(1)\subset S^2(1)$ denotes 
      a closed hemisphere;
   \item for every $y\in A(x^j;\delta_4 r,\delta_3 r)\cap\partial B(p,r)$, 
       there exists
       $z\in X^3$ with $\tilde\angle x^jy z > \pi-\tau(\epsilon_1);$ 
   \item $\tilde\angle x^jxx^{j+1} > \pi-\tau(\epsilon_1)$ for every point
      $x\in B(\widehat{x^jx^{j+1}},\delta_2 r) \setminus B(x^j,\delta_3 r)
      \setminus B(x^{j+1},\delta_3 r)$, where $\widehat{x^jx^{j+1}}$
      is a minimizing arc in $\partial X^3\cap \partial B(p,r)$
      joining $x^j$ and $x^{j+1}$.
  \end{enumerate}
\end{enumerate}

  Take the points $\{ y^1,\ldots,y^N\}\subset \partial B(p,r/2)$
with $y^j\in px^j$. Let us assume $\delta_1\ll \delta_4$.
Taking $x_i^j, y_i^j\in M_i^4$ with $x_i^j\to x^j$ and $y_i^j\to y^j$
under the convergence $M_i^4\to X^3$,  
consider the closed domains of $M_i^4$
$$
    B_i^j = \{ \tilde d_{x^j_i y^j_i}\le\delta_2 r \}\cap 
            \{ r/2 \le \tilde d_{p_i} \le r \}.
$$
Let $\Delta_j$ denote the domain on $\partial X^3$ bounded by
$\widehat{x^jx^{j+1}}$, $x^{j+1}y^{j+1}$, $\widehat{y^jy^{j+1}}$
and $y^jx^j$.
  Let $C_i'$ be a compact
  domain of $M_i^4$ which converges to $B(\Delta_j,\delta_2 r)$, $C_i$
  the closure of $C_i' \setminus B_i^j \setminus B_i^{j+1}$, and $N_i$ the
  closure of $\partial C_i \setminus B_i^j \setminus B_i^{j+1}$.  Applying
  Fibration Theorem \ref{thm:orig-cap} to a neighborhood of 
  $\{ d(\Delta_j, \,\cdot\,) = \delta_2 r \}$, 
  we can take such $C_i'$ that for
  every $x\in (B_i^j\cup B_i^{j+1})\cap N_i$
  $$
  |\,\angle(\xi_1(x), \xi_2(x)) - \pi/2\,| <
     \tau(r)+\tau(\delta_3)+\tau(\delta_3|\delta_2) + 
            \tau(\delta_2,\delta_3 |\epsilon_i),
  $$
  where $\xi_1$ and $\xi_2$ denote the unit normal vector fields to
  $\partial(B_i^j\cup B_i^{j+1})$ and $N_i$ respectively.
  Thus both
  $\partial B_i^j$ and $\partial B_i^{j+1}$ meet $N_i$ transversally,
  and therefore 
  $B_i^j\cap N_i\cap\partial B(p_i,r)\simeq S^1$, 
  $B_i^{j+1}\cap N_i\cap\partial B(p_i,r)\simeq S^1$.  
  It follows that
  $$
      N_i\simeq S^1\times I\times I,\quad   N_i\cap B_i^j\simeq S^1\times I,\quad 
        N_i\cap B_i^{j+1} \simeq S^1\times I. 
  $$
  We next show  $C_i\simeq D^4$.  
  Consider the functions $f_i$,
  $$
      g_i = \tilde d(x_i^jy_i^j,\,\cdot\,) - 
                  \tilde d(x_i^{j+1}y_i^{j+1},\,\cdot\,),
  $$
  and $\tilde d_{p_i}$.
  Note that the gradient of $f_i$ is almost perpendicular to $N_i$
  and that $g_i$ is regular on $C_i$.  Set
  $F_i = f_i^{-1}([\,0,\delta_2 r/2\,])\cap g_i^{-1}(0)$ and denote by $H_i$
  the set consisting of all flow curves of the gradient of $g_i$
  contained in $C_i$ through $F_i$. Clearly,
  $$
       H_i\simeq F_i\times I.
  $$
  Since the gradient of $f_i$ is almost perpendicular to that
  of $g_i$ on $\{ f_i\ge\delta_2 r/100\}\cap C_i$, it follows
  that $\partial (F_i\cap\partial B(p_i,r))\simeq S^1$.  
  Note also that $F_i$ is topologically a product 
  $F_i\simeq (F_i\cap\partial B(p_i,r))\times I$. 
  It follows from the generalized Margulis lemma
  (\cite{FY:fundgp}) that  $\pi_1(F_i)\simeq\pi_1(H_i)$ is almost
  nilpotent, and therefore by the orientability, $F_i\simeq D^2\times I$.
  \par
  It is easy to construct a smooth vector field $V_i$ on a
  neighborhood of $C_i \setminus H_i$ such that
  \begin{enumerate}
  \item $V_i = \grad f_i$ outside a small neighborhood of $\partial
    B_i^j\cup \partial B_i^{j+1}$,
  \item $V_i$ is tangent to $\partial B_i^j\cup \partial B_i^{j+1}$,
  \item $f_i$ is strictly decreasing along the flow curves of $V_i$.
  \end{enumerate}
  Thus we have
  
  \begin{equation}
            C_i\simeq H_i\simeq D^2\times I\times I.
                    \label{eq:CH}
  \end{equation}
  
  Next we show that $B_i^j\simeq D^3\times I$. 
  For the points $x_i^j$ and $x_i^{j-1}$ we
  construct a compact domain $\hat{C}_i$ in the same way as the
  construction of $C_i$.  From  Fibration Theorem \ref{thm:orig-cap}, 
  $(\partial B_i^j \setminus C_i \setminus \hat{C}_i)\cap\partial B(p_i,r)
  \simeq S^1\times I$, which together with \eqref{eq:CH} determines 
  the topological type of the boundary of the $3$-manifold 
  $B_i^j\cap \partial B(p_i,r)$ as:
  \begin{equation}
    \partial (B_i^j\cap \partial B(p_i,r)) \simeq
      D^2\cup (S^1\times I)\cup D^2  \simeq S^2.
  \end{equation}
 It follows from Lemma \ref{lem:u_1} that 

\begin{equation}
  B_i^j\cap \partial B(p_i,r)\simeq U_{\delta_3 r}(p_i, x_i^j)  
                     \simeq D^3. \label{eq:sliced3}
\end{equation}
%
From the standard critical point theory for $d_{x_i^j}$,
$$
  M_i^{\partial}(p_i,r,\delta_2 r)\cap U_{\delta_3 r}(p_i, x_i^j)
               \simeq D^3.
$$
Now we have a locally flat embedding
$\varphi_i:S^2=D^2\cup(S^1\times I)\cup D^2\to M_i^4$
such that 
$\varphi_i(S^2)=\partial(M_i^{\partial}(p_i,r,\delta_2 r)\cap 
U_{\delta_3 r}(p_i, x_i^j))$ which lies in the interior of 
some 3-dimensional submanifold homeomorphic to $D^3$ and containing
$M_i^{\partial}(p_i,r,\delta_2 r)\cap U_{\delta_3 r}(p_i, x_i^j)$.
By the Schoenflies theorem, it extends to an 
embedding $\Phi_i:D^2\times I\to M_i^4$ such that
$F_i(D^2\times I)=M_i^{\partial}(p_i,r,\delta_2 r)\cap 
U_{\delta_3 r}(p_i, x_i^j)$.
This provides a compatible $D^2$-bundle structure 
on 
$M_i^{\partial}(p_i,r,\delta_2 r)\cap U_{\delta_3 r}(p_i, x_i^j)$
over $\partial X\cap U_{\delta_3 r}(p, x^j)$.
Thus 
$M_i^{\partial}(p_i,r,\delta_2 r)\cap \partial B(p_i,r)$
is a compatible $D^2$-bundle over $\partial X^3\cap\partial B(p,r)\simeq S^1$,
and we conclude that 
$$
  M_i^{\partial}(p_i,r,\delta_2 r)\cap \partial B(p_i,r)\simeq S^1\times D^2.
$$
\end{proof}

\begin{proof}[Proof of Lemma \ref{lem:bdyball}]
From Sublemmas \ref{slem:sph-int} and \ref{slem:sph-bdy},
$\partial B(p_i,r)$ is homeomorphic to a gluing 
$S^1\times D^2 \cup S_i(D^2)$ along their boundaries, 
where  $ S_i(D^2)$ denotes a Seifert bundle over $D^2$ 
having at most one singular orbit and 
$\{ x\}\times \partial D^2 \subset S^1\times D^2$ 
is glued with the regular fibre of $S_i(D^2)$ over 
$x\in \partial D^2$.
Thus $\partial B(p_i,r)$ is homeomorphic to either $S^3$ or 
a lens space $L(\mu_i,\nu_i)$, where $\mu_i\neq 0$, or equivalently
$L(\mu_i,\nu_i)$ is not homeomorphic to $S^1\times S^2$. 
Therefore in view of \eqref{eq:D4S1D3},
$\partial B(p_i,r)$ must be homeomorphic to $S^3$ and 
$B(p_i,r)\simeq D^4$. 
In particular $ S_i(D^2)$ has no singular orbits.
This completes the proof of 
Lemma \ref{lem:bdyball}.
\end{proof}

\begin{proof}[Proof of Theorem \ref{thm:D2along}]
We proceed in a way similar to
that of Theorem \ref{thm:patch} as follows.
Let $\nu$ be a small positive number and $\delta_3>0$ be 
as in Fibration-Capping Theorem \ref{thm:orig-cap}.
Let $Y$ be a closed domain of $R^D_{\delta_3}(X^3)$ which 
approximates $R^D_{\delta_3}(X^3)$ in the sense that
$B(Y,\mu)\supset R^D_{\delta_3}(X^3)$ with $\mu\ll\nu$.
Applying Theorem \ref{thm:orig-cap} to $Y$, we have 
a closed domain $M_{\nu i}^{\rm cap}\subset M_i^4$ 
and a map 
$f_{i,{\rm cap}} : M_{\nu i}^{\rm cap}\to \partial_0 Y_{\nu}$ 
such that
\begin{enumerate}
 \item $f_{i,{\rm cap}}$ is a $\tau_i$-approximation with 
       $\lim \tau_i=0;$
 \item $f_{i,{\rm cap}}$ is a locally trivial $D^2$-bundle compatible
   with the $S^1$-bundle structure, say $\mathfrak h_i$, 
   on $\partial M_{\nu i}^4$, in the sense that the boundary 
   of $f_{i,{\rm cap}}^{-1}(x)$ coincides with the orbit 
   coming from $\mathfrak h_i$ for each $x\in \partial_0 Y_{\nu}$.
\end{enumerate}

Next we construct a local $D^2$-bundle structure
at each point of $\partial X_{\nu}$
extending $f_{i,{\rm cap}}$.
The following lemma follows from the convergence 
$(\frac{1}{r} X^3,p)\to (K_p,o_p)$ as $r\to 0$ and the finiteness of 
$S_{\delta}(D(\Sigma_p))$ for any $\delta>0$.
We put $S^D_{\delta_3}(X^3):=X^3-R^D_{\delta_3}(X^3)$.

\begin{lem}  \label{lem:coverSD}
For any $p\in\partial X^3 \cap S^D_{\delta_3}(X^3)$, 
there exist a positive integer $k=k(p)$, positive numbers
$\epsilon=\epsilon_p>0$ and $r=r(p,\epsilon)>0$ such that
for some $x^1,\ldots,x^k\in\partial B(p,10r)\cap\partial X^3$
and for sufficiently small $\nu$
\begin{enumerate}
 \item $B(p,10r)\cap (\partial X_{\nu}^3-R^D_{\delta_3}(X^3))
      \subset \bigcup_{j=1}^k B(px^j,\epsilon r);$
 \item $(B(px^j,\epsilon r)-B(px^j,\epsilon^{10} r))\cap
       A(p;r/100,10r)$ does not meet 
       $S^D_{\delta_3}(X^3)$
       for each $1\le j\le k;$
 \item $\{ B(px^j,\epsilon r)\cap A(p;r/100,10r)\}_j$ 
       is disjoint$;$
 \item $B(px^j,\epsilon r)\cap A(p;r/100,10r)\cap\partial X^3_{\nu}$
       bounds a topological disk, say $\Delta_j(p)$,
       for each $1\le j\le k;$
 \item $A(p;r/100,10r)\cap\partial X^3_{\nu}$ does not meet 
       $S_{\delta_3}(X^3)$
\end{enumerate}
\end{lem}

Here we need 

\begin{thm}[Generalized Sh\"oenflies Theorem\cite{Br:schoenflies}] 
                     \label{thm:shoen}
Let $f:S^3\to \R^4$ be a locally flat topological embedding 
in the sense that 
there is a topological embedding $F: S^3\times(-1,1)\to \R^4$ with
$F(x,0)=f(x)$. Let $E\subset \R^4$ denote the compact domain 
bounded by $f(S^3)$, and let 
$S^3 = D^2\times S^1\cup S^1\times D^2$ be the canonical identification.
Then there exists a homeomorphism $H:D^2\times D^2 \to E$ 
extending $f$.
\end{thm}

Put $I_r:=\Delta_j(p)\cap\partial B(p,r)$.
We have a $D^2$-bundle structure over $\partial I_r$ given by 
$f_{i,{\rm cap}}$, and an $S^1$-bundle structure over $I_r$ 
given by $\mathfrak h_i$.
By \eqref{eq:sliced3} and the 
Generalized Sh\"oenflies Theorem for $\R^3$, those two structures
extends to a trivial $D^2$-bundle structure over $I_r$, on 
some three-dimensional compact submanifold, which is Gromov-Hausdorff
close to a compact domain of $\partial B(p_i,r)$.
Similarly we have a compatible $D^2$-bundle structure over
$I_{r/100}$,  on 
some three-dimensional compact submanifold, which is Gromov-Hausdorff
close to a compact domain of $\partial B(p_i,r/100)$.
On the other hand, 
we have a $D^2$-bundle structure over 
$\partial(\Delta_j(p)\cap B(p,r))-\partial B(p,r)-\partial B(p,r/100)$, 
given by $f_{i,{\rm cap}}$,
and we have an $S^1$-bundle structure over 
$\Delta_j(p)\cap B(p,r)$ given by $\mathfrak h_i$.
Since these two structure are compatible on the intersection, 
by Generalized Sh\"oenflies Theorem \ref{thm:shoen}, these  
extend to a trivial $D^2$-bundle structure over 
$\Delta_j(p)\cap B(p,r)$ for each $1\le j\le k$.
This also provides a $D^2$-bundle structure over 
$\partial B(p,r/100)\cap \partial X_{\nu}^3$, which extends to 
a trivial $D^2$-bundle structure over 
$B(p,r/100)\cap \partial X_{\nu}^3$ again by  
Generalized Sh\"oenflies Theorem \ref{thm:shoen}.
Thus we have constructed a $D^2$-bundle structure 
over $\partial X_{\nu}^3\cap B(p,r)$ which is compatible
with $f_{i,{\rm cap}}$.

In a way similar to the global patching argument constructing
a globally defined local $S^1$-action in 
Section \ref{sec:dim3nobdyglob}, one can patch those 
local $D^2$-bundle structures compatible with $f_{i,{\rm cap}}$ 
and $\mathfrak h_i$ 
to obtain a globally defined $D^2$-bundle structure 
over $\partial X_{\nu}^3$ 
compatible with $f_{i,{\rm cap}}$ and $\mathfrak h_i$.
Note that here we use Theorem \ref{thm:shoen} in place of 
Lemma \ref{lem:extend}. Since the patching 
argument is almost parallel, the detail is omitted.
This completes the proof of Theorem \ref{thm:patchwbdy}.
\end{proof}

\begin{proof}[Proof of Corollary \ref{cor:sum}]
Let $M_i^4$ be as in Corollary \ref{cor:sum}.
By Theorem \ref{thm:dim3}, we have a locally smooth, local $S^1$-action 
on $M_i^4$. Since $M_i^4$ is simply connected, this local $S^1$-action
is actually a global $S^1$-action on $M_i^4$.
According to Fintushel \cite{Fn:circleI} together with Freedman
(see \cite{FQ:4mfd}),
$M_i^4$ is homeomorphic to a connected sum
\begin{equation}
   S^4\,\#\, k_i \C P^2\, \#\, \ell_i(-\C P^2)\,\#\, m_i(S^2\times S^2).
         \label{eq:factor}
\end{equation}
Note that $X^3$ is simply connected, and hence each connected 
component of $\partial X^3$ is a sphere.
It follows from the formula 
$\chi (F(\psi_i))=\chi(M_i^4)$ (see Theorem 10.9 in \cite{Brd:intro})
that 
$$
    k_i + \ell_i + 2m_i + 2 \le 
         \# \text{\rm Ext}(\interior X^3) + 2\alpha(\partial X^3).
$$
\end{proof}

\begin{rem}\label{rem:conn-sum}
Let $M_i^4$ be as in Corollary \ref{cor:sum}.
By \eqref{eq:factor}, there exists a locally smooth $T^2$-action 
on $M_i^4$.
By \cite{OrRy:torusI}, this action can be 
reduced to a smooth 
action, and again \cite{OrRy:torusI} implies that $M_i^4$ is 
diffeomorphic to the above connected sum.
Furthermore we have a sequence $g_{i_j}$, $j=1,2,\ldots$ of metrics on $M_i^4$
such that $(M_i^4, g_{i_j})$ collapses to the quotient space 
$(M_i^4,g_i)/T^2$ under a uniform lower sectional curvature bound,
where $g_i$ is a $T^2$-invariant metric on $M_i^4$.
\end{rem}


\section{Classification of collapsing to noncompact
          three-spaces with nonnegative curvature}
                     \label{sec:dim3positive}

Let a sequence of pointed complete $4$-dimensional orientable Riemannian
manifolds $(M_i^4,p_i)$ with $K \ge -1$ collapses to a pointed
complete noncompact $3$-dimensional Alexandrov space $(Y^3,y_0)$ with 
nonnegative curvature. In this section, using the results of 
Sections \ref{sec:dim3nobdyglob} and \ref{sec:dim3wbdy},
we classify the topology of a large metric ball 
$B(p_i,R)$ in terms of geometric properties of $Y^3$.
The classification result will be used in the subsequent sections
of Part \ref{part:collapse} to describe the phenomena
of orientable $4$-manifolds collapsing to spaces of dimension $\le 2$.

Let 
$$
 Y\supset C(0)\supset C(1)\supset \cdots\supset C(k),
$$
be a sequence of 
nonempty compact totally convex subsets of $Y$
as in Section \ref{sec:ideal}.
Applying Theorem \ref{thm:dim3} to the convergence
$(M_i^4, p_i) \to (Y,y_0)$,
we have a locally smooth, local $S^1$-action $\psi_i$ 
on $B(p_i,R)$ whose orbit
space is homeomorphic to $B(y_0, R)$, where $R$
is a large positive number.
Actually, we have such a local $S^1$-action  
on a small perturbation of $B(p_i,R)$, which is homeomorphic to 
$B(p_i,R)$.
Let $F_i^*:=F^*(\psi_i)$, $E_i^*:=E^*(\psi_i)$, $S_i^*:=S^*(\psi_i)$
and $\cal C_i:=S^*_i - \partial Y$.
We denote by $\pi_i:B(p_i,R)\to B(y_0,R)$ the orbit map.

Let $\ell_i$ denote the number of components of $E_i^*$ homeomorphic 
to a circle, and $m_i$ the number of components of $E_i^*$
homeomorphic to an interval
whose closure does not meet $F_i^*$, and
$n_i$ the number of elements of $F_i^*\cap\interior Y$.

\begin{thm} \label{thm:dim3class}
We have the following classification of the topology of
$B(p_i, R)$ in terms of $\ell_i$, $m_i$, $n_i$ and the geometric 
properties of $Y$:
\begin{proclaim}{\emph{Case} A.}
      $Y$ has no boundary.
\end{proclaim}
\begin{enumerate}
 \item[I.] Suppose $\dim S=2$, yielding $\dim Y(\infty)=0$. Then
     $B(p_i, R)$ is homeomorphic to an $I$-bundle 
     $($resp. a trivial $I$-bundle$)$ over $S_i(S)$,
     a Seifert bundle over $S$ $($resp. if $S\simeq S^2$ $)$.
 \item[II.] Suppose $\dim S=1$.
  \begin{enumerate}
   \item[$(1)$] If $\dim Y(\infty)\ge 1$, then $B(p_i, R)$ is 
      homeomorphic to a $D^2$-bundle over $T^2$ or $K^2;$
   \item[$(2)$] If $\dim Y(\infty)=0$, then $B(p_i, R)$ is 
      homeomorphic to either one of the spaces in $II$-$(1)$  or 
      an $I$-bundle over $S^1\tilde\times K^2$, a $K^2$-bundle over $S^1;$
  \end{enumerate}
 \item[III.] Suppose $\dim S=0$. 
  \begin{enumerate}
   \item[$(1)$] If $\dim C(0)=0$, then $n_i\le 1:$
      $B(p_i, R)$ is 
      homeomorphic to  $D^4$  if $n_i=1$ or to 
      $S^1\times D^3$  if $n_i=0;$ 
   \item[$(2)$] If $\dim C(0)=1$, then $m_i+n_i\le 2:$
      $B(p_i, R)$ is homeomorphic to either one of the spaces in 
      $III$-$(1)$ if $m_i+n_i\le 1$, or  
      \begin{align*}
        (K^2\tilde\times I)\times I \qquad & \text{if}\quad m_i=2,\\
         S^2\tilde\times_{\omega} D^2 
            \qquad & \text{if}\quad (m_i,n_i)=(0,2),
      \end{align*}
      where $\omega$ is explicitly 
      estimated in terms of the singularities of the spaces of 
      directions at the endpoints $\partial C(0);$
   \item[$(3)$] Suppose $\dim C(0)=2$.
      Then $\ell_i\le 1$, $m_i\le 2$, 
      $\max\{ 4\ell_i, 2m_i\}+n_i\le 4$ and 
      $B(p_i, R)$ is homeomorphic to one of the spaces in
      $A$-$III$-$(1)$ if $(m_i,n_i)=(1,0)$, or $m_i=0$, $n_i\le 1$,
      or one in the following list:
      \begin{align*}
       T^2\times D^2\bigcup_{T^2\times I} S^1\times D^3,       
            \hspace{0.2cm} \quad & \text{if} 
                      \quad \ell_i=1, m_i\le 1, \\
    T^2\times D^2\bigcup_{T^2\times I} (K^2\tilde\times I)\times I,      
            \hspace{0.2cm} \quad & \text{if} 
                      \quad (\ell_i,m_i)=(1,2), \\
        (K^2\tilde\times I)\times I,\hspace{1.2cm}\quad & \text{if} \quad
                (m_i,n_i)=(2,0), \\
           S^1\times D^3\bigcup_{S^1\times D^2} D^4,\hspace{0.6cm} \quad
                            & \text{if} \quad   (m_i,n_i)=(1,1), \\
        S^2\tilde\times_{\omega_2} D^2,\hspace{1.5cm} \quad 
             & \text{if}\quad  (m_i,n_i)=(0,2),\\
            S^2\tilde\times_{\omega_2} D^2\bigcup_{S^1\times D^2} D^4,
         \hspace{0.6cm}\quad & if\quad (m_i,n_i)=(0,3), \\
         S^2\tilde\times_{\omega_3} D^2 \bigcup_{S^1\times D^2}
             S^2\tilde\times_{\omega_4} D^2,
          \quad & \text{if}  \quad (m_i,n_i)=(0,4). 
      \end{align*}
     \end{enumerate}
   where $|\omega_1|\in\{ 0,2\}$, $0\le |\omega_3|, |\omega_4| \le 4$
   and 
   $|\omega_2|$ can be explicitly estimated in terms of 
   the ratio of $\pi$ and the angles of $C(0)$
   at the extremal points on $\partial C(0)$.
  \end{enumerate}
\begin{proclaim}{\emph{Case} B.}
      $Y$ has nonempty connected boundary.
\end{proclaim}
\begin{enumerate}
 \item[I.] Suppose $\dim S=2$, yielding $\dim Y(\infty)=0$. Then
   $B(p_i,R)$ is homeomorphic to a $D^2$-bundle over $S;$
 \item[II.] Suppose $\dim S=1$.
  \begin{enumerate}
   \item[$(1)$] If $\dim Y(\infty)\ge 1$, then $B(p_i, R)$ is 
      homeomorphic to  $S^1\times D^3;$
   \item[$(2)$] If $\dim Y(\infty)=0$, then $B(p_i, R)$ is 
      homeomorphic to the space in $B$-$II$-$(1)$,
      $S^1\times (P^2\tilde\times I)$, $S^1\times S^2\times I$ or 
      an $I$-bundle over $S^1\tilde\times S^2$, the nontrivial 
      $S^2$-bundle over $S^1$.
  \end{enumerate}
 \item[III.] Suppose $\dim S=0$. If $Y$ has two ends, then 
      $\ell_i=n_i=0$, $m_i\le 2$ and 
    \begin{equation*}
      B(p_i,r)\simeq 
       \begin{cases}
         \text{a lens space}\times I & \qquad \text{if}\quad m_i\le 1\\
         (P^3\,\#\,P^3)\times I & \qquad \text{if}\quad m_i=2.
       \end{cases}
     \end{equation*}
     Suppose that $Y$ has exactly one end, and consider the maximum set
     $C^*$ $($possibly empty$)$ of $d_{\partial Y}$.
     Then $\cal C_i\subset C^*$.
     If $F^*_i\cap\interior Y$ is empty and $E^*_i$ is nonempty, 
     then $m_i=1$ and
     $B(p_i, R)$ is homeomorphic to
       \begin{equation*} 
           (P^2\tilde\times I)\times I.
       \end{equation*} 
     In the other cases, we have the following$:$
  \begin{enumerate}
   \item[$(1)$] If $C^*$ is empty, then $B(p_i, R)$ is 
      homeomorphic to $D^4;$
   \item[$(2)$] Suppose $C^*$ is nonempty.
    \begin{enumerate}
    \item[$(i)$] If $\dim Y(\infty)\ge 1$, then 
      $B(p_i, R)$ is homeomorphic to either $D^4$ or
      $S^2\tilde\times_{\omega} D^2$, where $|\omega|\in\{ 1,2\};$
    \item[$(ii)$] If $\dim Y(\infty)=0$, then 
      $B(p_i, R)$ is homeomorphic to one of the spaces in 
      $B$-$III$-$(2)$-$(i)$, $S^2\tilde\times_{\omega_1} D^2$ if
      $C^*$ is a ray,
      or $S^2\tilde\times_{\omega_2} D^2\bigcup_{S^1\times D^2} D^4$
      if $\dim C^*=2$, 
      where $\omega_1$ can be estimated in terms of 
      singularities at an interior point of $C^*$ and 
      $0\le |\omega_2|\le 4$.
    \end{enumerate}
  \end{enumerate}
\end{enumerate}
\begin{proclaim}{\emph{Case} C.}
      $Y$ has disconnected boundary.
\end{proclaim}
\begin{enumerate}
 \item[I.] If $\dim Y(\infty)\ge 1$, then
   $B(p_i,R)$ is homeomorphic to a $D^2\times S^2;$
 \item[II.] If $\dim Y(\infty)=0$, then
   $B(p_i,R)$ is homeomorphic to either $S^1\times S^2\times I$
   or an $I$-bundle over the nontrivial $S^2$-bundle 
   $S^1\tilde\times S^2$.
\end{enumerate}
\end{thm}

Later we need to clarify the position of the singular locus 
$\cal C_i=S_i^* - \partial Y$. 
The following lemma will give a general 
information about it.

Let $b$ denote the Busemann function on $Y$ used for the 
construction of $S$ (i.e., $C(0)$ is the minimum set of $b$), 
and let $\mu_0$ be the minimum of $b$.

\begin{lem} \label{lem:locmax}
\begin{enumerate}
 \item $b$ is monotone on each component of $\cal C_i-C(0);$
 \item For any $x\in Y$, $d_x$ has no local minimum
      on $\cal C_i-\{ b\le b(x)\}$. 
 \item If $\cal C_i$ is nonempty, it must meet
      $C(0)$. If $A_i$ denotes a component of $E_i^*$ 
      which is not contained in $\partial C(0)$, then 
      $\bar A_i$ perpendicularly meet $C(0)$.
\end{enumerate}
\end{lem}

Here we say that a curve $c:[a,b]\to Y$ is {\it perpendicular} to 
a closed set $C$ if $c(a)\in C$ and for any $x\in C$, 
the distance $d(x, c([a,b]))$ is realized only at $c(a)$.

\begin{proof}
First we show that $b$ has no local maximum on $\cal C_i-C(0)$.
Suppose that $b$ takes a local maximum $t_0>\mu_0$
at a point $p\in\cal C_i-C(0)$, and 
let $\{ \xi_1,\xi_2\}:=\Sigma_p(\cal C_i)$.
For any $t>t_0$ take a point $q\in b^{-1}(t)$ with 
$d(p,q)=t-t_0$.
Then $d_q$ restricted to $\cal C_i$ takes a local minimum at
$p$. However for any $x$ with $b(x)< t_0$, 
we have $\angle(q_p', x_p')>\pi/2$,
which contradicts the fact that $\cal C_i$ is an 
extremal subset. 

The proof of $(2)$ is similar, and hence omitted.
$(2)$ implies that $b$ has no global minimum on 
$\cal C_i-C(0)$. $(1)$ follows from the previous argument.
(3) follows from (2).
\end{proof}

First we  consider the case when $Y$ has no boundary
and begin with the simplest

\begin{proclaim}{\emph{Case} A-I.}
$Y$ has no boundary and $\dim S=2$.    
\end{proclaim}

In this case, $Y$ is isometric to the normal bundle $N(S)$ of 
$S$, which is flat
(hence if $S\simeq S^2$, then $Y$ is isometric to $S\times \R$).
Let $S_0$ be the zero section of $N(S)$, and $S_i(S_0)$
be the preimage of $S_0$ by $\pi_i$,
which is a Seifert bundle over $S_0$.
Now $B(p_i, R)$ is homeomorphic to 
an $I$-bundle over $S_i(S_0)$.


\begin{proclaim}{\emph{Case} A-II.}
$Y$ has no boundary and $\dim S=1$.    
\end{proclaim}

In this case $Y$ is isometric to a quotient $(\R\times N^2)/\Lambda$
which is homeomorphic to 
$S^1\times N^2$ (resp. the nontrivial bundle $S^1\tilde\times N^2$)
if $Y$ is orientable (resp. if $Y$ is non-orientable),
where $\Lambda\simeq\Z$, $N^2\simeq \R^2$ and the number $k$ of 
essential singular points of $N^2$ is at most two.
If $\dim Y(\infty)=\dim N^2(\infty)\ge 1$,
$k$ is at most 1. 
If $\dim Y(\infty)=0$ and $k=2$, then 
$N^2$ would be isometric to the double of a product $[a,b]\times
[0,\infty)$.

Let $k_i\le 2$ denote the number of the singular orbits of
$\psi_i$ over the $N^2$-factor.
If $k_i=0$, it follows that
$B(p_i,R)$ is an $S^1$-bundle over $B(y_0,R)$,
which is homeomorphic to either $S^1\times D^2$ or the 
twisted product $S^1\tilde\times D^2$. 
Thus $B(p_i,R)$ is a $D^2$-bundle over $T^2$ or $K^2$.

If $k_i=1$, the preimage of the $N^2$-factor by $\pi_i$,
denoted $S_i(N^2)$, is a fibred solid torus with one 
singular fibre. Therefore
$B(p_i,R)$ is an $S_i(N^2)$-bundle over $S^1$, and
$$
 B(p_i,R)\simeq S^1\times D^2\times I\left/(x,0)\sim (f(x),1)\right.,
$$
where $f$ is a gluing homeomorphism of $S^1\times D^2$.
Consider the mapping class group $\cal M_+(S^1\times D^2)$
of orientation-preserving homeomorphisms of $S^1\times D^2$,
which can be thought of as a subgroup of 
$\cal M_+(S^1\times S^1)=SL(2,\Z)$.
Since 
\begin{equation*}
 \cal M_+(S^1\times D^2)=\left\{ \left.
          \begin{pmatrix} 
              \pm 1\,\,\, \omega \\
              \, 0 \,\,\,    \pm 1
           \end{pmatrix}  
        \,\right|\,\omega\in\Z  \right\},
\end{equation*} 
we may assume that $f$ preserves the $D^2$-factors.
Thus  $B(p_i, R)$ is a $D^2$-bundle over $T^2$ or $K^2$.

Suppose $k_i=2$. In a similar way to the above case of $k_i=1$,
$B(p_i,R)$ is an $S(D^2:(2,1), (2,1))$-bundle over $S^1$,
where $S(D^2:(2,1), (2,1))$ is the Seifert bundle over $D^2$
with two singular orbits with Seifert invariants $(2,1)$.
Note that $S(D^2:(2,1), (2,1))\simeq K^2\tilde\times I$.
Since 
$\cal M_+(K^2\tilde\times I)\simeq 
\cal M(K^2)\simeq\Z_2\oplus\Z_2$ ( \cite{Mcc:virt}, 
\cite{Lc:homeo}), in the same way as before,
$B(p_i,R)$ is homeomorphic to an $I$-bundle over 
a $K^2$-bundle over $S^1$.

\begin{proclaim}{\emph{Case} A-III.}
$Y$ has no boundary and $\dim S=0$.    
\end{proclaim}

Since $B(y_0,R)\simeq D^3$ for large $R$ (\cite{SY:3mfd}),
$\psi_i$ is an $S^1$-action (Lemma \ref{lem:simply}).
Suppose first $\dim C(0)\le 1$. By Lemma \ref{lem:C0codim1},
this is the case if $\dim Y(\infty)\ge 1$.
If $\dim C(0)=0$, or equivalently if $C(0)=S$,
then $\psi_i$ has at most one fixed point.
Therefore 
$B(p_i,R)$ is homeomorphic to either $D^4$ or $S^1\times D^3$.

We prepair

\begin{lem} \label{lem:m=n=1}
Let $\psi$  be an $S^1$-action on a compact $4$-manifold
$N$ with boundary such that 
\begin{enumerate}
 \item the orbit space $N^*$ is homeomorphic to $D^3$
       and $F^*(\psi)\subset\interior N^*;$
 \item there is exactly one component $A^*$ of $E^*(\psi)$ 
       whose closure $\bar A^*$ does not meet $F^*(\psi)$,
       and the Seifert invariants along $A^*$ are $(2,1);$
 \item there is a point $x^*\in F^*(\psi)\cap\interior N^*$ 
       such that the number
       of the components of $E^*(\psi)$ whose closure touch $x^*$
       is at most $1$.
\end{enumerate}
Then $N$ is nonorientable.
\end{lem}

\begin{proof}
By Van Kampen's theorem, $\pi_1(N)=\Z_2$.
Let $p:\tilde N\to N$ be the universal cover and 
$\sigma$ the nontrivial deck transformation.
$\tilde N$ has an $S^1$-action $\tilde\psi$
induced from $\psi$.
Note that there is no component of $E^*(\tilde\psi)$ whose closure
does not meet $F^*(\tilde\psi)$.
Let $x_1, x_2\in \tilde N$ be the points over $x^*$, and 
let $\tilde\pi:\tilde N\to\tilde N^*$ be the orbit map.
Divide $\tilde N^*$ by a proper $\sigma$-invariant $2$-disk in 
$\tilde N^*-F^*(\tilde\psi)-E^*(\tilde\psi)$ into two 
$3$-disks $U^*$ and $V^*$ in such a way that
$x_1,x_2\in U^*$ and the other elements of $F^*(\tilde\psi)$
are contained in $V^*$. Put $U:=\tilde\pi^{-1}(U^*)$.
By Proposition \ref{prop:plumbing}, 
$(U,\tilde\psi)$ is equivariatly homeomorphic to 
$S^2\tilde\times_{\omega} D^2$ for some $\omega\in\Z$ 
equipped with a canonical action, denoted $\hat\psi$,
 provided in Section \ref{sec:action}.
Hence we may identify $(U,\tilde\psi)$ with 
$(S^2\tilde\times_{\omega} D^2,\hat\psi)$.
Let 
$$
 S^2\tilde\times_{\omega} D^2 = B_1\times D^2_1\bigcup_{f_{\omega}}
                                        B_2\times D^2_2,
$$
be the gluing as in Section \ref{sec:action}.
Since the $\Z_2$-action defined by $\sigma$ is $S^1$-equivariant,
$\sigma$ preserves $\partial B_1\times D^2$ and the zero-section
$(B_1\cup B_2)\times 0$. Since 
$p(\partial B_1\times 0)$ lies in $E(\psi)$, 
the action of $\sigma$ on $(B_1\cup B_2)\times 0$
is orientation-reversing.  On the other hand,
since $\sigma$ is $S^1$-equivariant, the 
$\sigma$-action on the $D^2$-factor induced from that on
$\partial B_1\times D^2$ must be orientation-preserving.
Therefore $\sigma$ is orientation-reversing and 
$N$ must be nonorientable.
\end{proof}

Suppose $\dim C(0)=1$, that is, $C(0)$ is a geodesic segment.
Since ${\rm Ext}(Y)$ is included in $\partial C(0)$,
so is $F_i^*$.

\begin{lem} \label{lem:dimC(0)=1}
Let $\dim C(0)=1$ and let $A_i$ be a component of $E_i^*$.
\begin{enumerate}
 \item If $A_i$ is not contained in $C(0)$,
    $\bar A_i$ perpendicularly intersects $C(0)$ with
    an endpoint of $\partial C(0);$
 \item If $\bar A_i$ does not meet $F_i^*$,
    one of the following holds$:$
  \begin{enumerate}
   \item If $A_i\cap C(0)$ contains more than one element,
    then $C(0)$ is a subarc of $A_i;$
   \item $A_i$ intersects $C(0)$ with exactly one endpoint of
      $C(0)$ making an angle equal to $\pi/2$. In this case, 
      the Seifert invariants of $\psi_i$ along $A_i$ are
      $(2,1)$.
  \end{enumerate}
\end{enumerate}
\end{lem}

\begin{proof}
Any point $x\in\interior C(0)$ has a small neighborhood 
$U_x$ such that $U_x-C(0)$ has no essential singular
point. Therefore $A_i$ must intersect $C(0)$ with an endpoint.
Suppose $A_i\cap C(0)$ is a point, say $x$.
Let $\xi_1$,$\xi_2$  and $\xi$ denote the directions at $x$ 
determined by $A_i$ and $C(0)$ respectively.
Since $A_i$ is an extremal subset, it follows that
$\angle(\xi_j,\xi)=\pi/2$, $j=1,2$.

\begin{slem} \label{slem:rigid-glue}
 \begin{enumerate}
  \item $d(\xi_1,\xi_2) < \pi;$
  \item $\Sigma_x$ is isometric to the double of a geodesic 
    triangle on $S^2(1)$ with sidelengths
    $\pi/2$, $\pi/2$ and $d(\xi_1,\xi_2)$.
 \end{enumerate}
\end{slem}

\begin{proof}
Suppose (1) first. Note that $L(\Sigma_{\xi_i}(\Sigma_x))\le\pi$,
$i=1,2$ since $A_i\subset E_i^*$.
Let $\blacktriangle$ denote a closed domain bounded by 
$\triangle\xi\xi_1\xi_2$ such that the inner angle, say 
$\alpha$, of $\blacktriangle$ at $\xi$ is at most $\pi$.
Let $\alpha_i$ be the inner angle of $\blacktriangle$ at $\xi_i$.
From the Alexandrov convexity applied to $\blacktriangle$,
\begin{equation}
  \alpha_i\ge\pi/2.  \label{eq:inner-conv}
\end{equation}
Let $\blacktriangle':=\Sigma_x - \interior\blacktriangle$,
and $\alpha'$ and $\alpha_i'$ denote the inner angles of 
$\blacktriangle'$ at $\xi$ and $\xi_i$ respectively.
From \eqref{eq:inner-conv}, $\alpha_i'\le\pi-\alpha_i\le\pi/2$.
If $\alpha'\le\pi$, using the same argument as above,
we conclude $\alpha_i=\alpha_i'=\pi/2$.
The conclusion (2) follows from the rigidity (cf. \cite{Shm:alex}).

For any $\eta\in\xi_1\xi_2$, 
\[
  d^{\blacktriangle'}(\xi,\eta)\ge d(\xi,\eta)\ge \pi/2,
\]
where $d^{\blacktriangle'}(\xi,\eta)$ denotes the inner distance in 
$\blacktriangle'$. Let $\eta_0$ be a farthest point of 
$\xi_1\xi_2$ from $\xi$ with respect to $d^{\blacktriangle'}$.
Divide $\blacktriangle'$ by a minimal segment $\xi\eta$ in 
$\blacktriangle'$ into two triangle regions
$\blacktriangle'_1$ and $\blacktriangle'_2$ with 
$\xi_i\in\blacktriangle'_i$. We may assume that 
the inner angle of $\blacktriangle'_1$ at $\xi$ is at most
$\pi$. 
The Alexandrov convexity applied to $\blacktriangle'_1$
shows
\[
  0\ge \cos d^{\blacktriangle'}(\xi,\eta_0) =
         \sin d(\xi_1,\eta_0) \cos \tilde\alpha_1' \ge 0,
\]
where $\tilde\alpha_1'\le\pi/2$ denotes the comparison angle 
for $\alpha_1'$. It follows that 
$d^{\blacktriangle'}(\xi,\eta_0)=\pi/2$ and 
$\tilde\alpha_1'=\alpha_1'=\alpha_1=\pi/2$.
Thus $\blacktriangle'_1$ is isometric to a triangle in 
$S^2(1)$ with sidelengths $\pi/2$, $\pi/2$ and 
$d(\xi_1,\eta_0)$.

Continue the above argument for $\blacktriangle'_2$
by taking a farthest point of $\eta_0\xi_2$ from 
$\xi$, finally  to conclude that 
$\blacktriangle'$  and $\blacktriangle$ are isometric to 
a triangle in 
$S^2(1)$ with sidelengths $\pi/2$, $\pi/2$ and 
$d(\xi_1,\xi_2)$.

Next we show $(1)$. Suppose $d(\xi_1,\xi_2)=\pi$.
Then $K_p$ is isometric to a product
$\R\times K(S^1_{\ell})$, where the length $\ell$ 
of the circle $S^1_{\ell}$ is at most $\pi$ because of 
$A_i\subset E_i^*$.
Now consider the Busemann function $b$ such that
$C(0)$ coincides with the minimum set of $b$, where we assume
$C(0)=\{ b=0\}$.
Under the convergence $(\frac{1}{r}X,p)\to (K_p,o_p)$,
$\frac{1}{r} b$ converges to a convex function 
$b_{\infty}$ on $K_p$ with Lipschitz constant $1$.
Let $C_{\infty}$ be the limit of $C(0)$ under the above 
convergence, which is a geodesic ray in 
$0\times K(S^1_{\ell})$ from $o_p$ and 
is contained in $\{ b_{\infty}=0\}$.
Since $\{ b_{\infty}=0\}$ is a totally convex set of dimension 
$\le 2$, it follows that 
$\{ b_{\infty}=0\}=0\times K(S^1_{\ell})$.
Obviously $\{ b_{\infty}=1\}=\{\pm 1\}\times K(S^1_{\ell})$
is disconnected while $\{ b=r\}\simeq S^2$
(see Theorem \ref{thm:soul2}) is connected.
In a obvious way, we get a contradiction.
\end{proof}

By Sublemma \ref{slem:rigid-glue},
the Seifert invariants of $\psi_i$ along $A_i$ 
are $(2,1)$.
If  $A_i\cap C(0)$ is more than a point, then 
one of $\xi_j$  must coincide with $\xi$ 
and $C(0)$ is a subarc of $\bar A_i$. 
\end{proof}

By Lemma \ref{lem:dimC(0)=1}, we have $m_i+n_i\le 2$.
Lemmas \ref{lem:m=n=1} and \ref{lem:dimC(0)=1} together with
Lemma \ref{lem:adj-inv}
imply that the case $(m_i,n_i)=(1,1)$ does not occur.

 If $m_i=2$, then $B(y_0,R)\simeq D^3$ consists of continuous 
 one-parameter family $D^2_t$ of mutually disjoint $2$-disks
 each of which transversally intersects $E_i^*$ at 
 exactly two points. 
 Since 
 $$
 \pi_i^{-1}(D^2_t)\simeq S(D^2;(2,1),(2,1))\simeq K^2\tilde\times I,
 $$
 $B(p_i,R)$ is homeomorphic to 
 $(K^2\tilde\times I)\times I$,
 a $D^2$-bundle over $K^2$ whose boundary 
 is homeomorphic to the double $D(K^2\tilde\times I)$.

 If $m_i=1$, then $B(p_i,R)$ is homeomorphic to 
 $S^1\times D^3$.


If $m_i=0$, by the above argument in the case of
$\dim C(0)=0$, 
we may assume  $F_i^*=\partial C(0)$.
In view of Lemma \ref{lem:adj-inv}, it follows from 
Proposition \ref{prop:plumbing} that 
$B(p_i,R)$ is homeomorphic to  $S^2\tilde\times_{\omega} D^2$.
Let $z_j$, $j=0,1$, be the endpoints of $\partial C(0)$.
Then $|\omega|$ does not exceed 
$$
  \max\left\{\frac{2\pi}{L(\Sigma_{\xi_0}(\Sigma_{z_0}(Y)))},1\right\}
 +\max\left\{\frac{2\pi}{L(\Sigma_{\xi_1}(\Sigma_{z_1}(Y)))},1\right\}
$$
for some essential singular point $\xi_j$ (if it exists) of
$\Sigma_{z_j}(Y)$ with distance $\pi/2$ to the direction
determined by $C(0)$.
\par
\medskip
Next suppose  $\dim C(0)=2$.
Since there are exactly two geodesic rays starting from 
each point of the interior of $C(0)$ with directions normal to
$C$, we have a locally isometric imbedding
$f:\interior C(0)\times\R\to Y$ (see Section \ref{sec:local}). 
For any subset $B\subset C(0)$, let 
$\cal N(B)$ denote the union of geodesic rays from the points of $B$
appearing as the limits of geodesic rays 
in $f(\interior C(0))$ normal to $C(0)$.

For the definition of one-normal or two normal 
points in $\partial C(0)$, see the definition 
right before Proposition \ref{prop:codim1-reg}.

We need some geometry of the space of directions at a 
given one-normal point of $\partial C(0)$, 
which is described in the following lemma.

\begin{lem} \label{lem:class-dir}
Suppose $\dim C(0)=2$ and a point $x\in\partial C(0)$ be given.
\begin{enumerate}
 \item If $x$ is an essential singular point of $Y$, then 
    it is a one-normal point $;$
 \item If  a point in $\partial \Sigma_x(C(0))$ is an essential 
    singular point of $\Sigma_x$, then 
    $x$ is a one-normal point $;$
 \item Suppose that $x$ is a one-normal point and there 
    are two essential singular points of $\Sigma_x;$
    one is in $\partial \Sigma_x(C(0))$ and the other is not.
    Let $v$ be the essential singular point in 
    $\Sigma_x-\partial \Sigma_x(C(0))$.
  \begin{enumerate}   
   \item If $L(\Sigma_x(C(0)))\le\pi/2$, then $v$ satisfies
     $$
       L(\Sigma_{v}(\Sigma_x))\ge 2L(\Sigma_x(C(0)));
     $$
   \item If $L(\Sigma_x(C(0)))>\pi/2$, then $v$ lies on a 
     geodesic extension of the geodesic through 
     $\partial \Sigma_x(C(0))$, and $\Sigma_x$ is isometric to
     the double of a geodesic triangle in $S^2(1)$ of
     sidelengths $\pi/2$, $\pi/2$ and 
     $\max\{ d(v,\eta)\,;\,\eta\in\partial\Sigma_x(C(0))\}$.
  \end{enumerate}
 \item If $x$ is an extremal point of $Y$ and if $v$ is an essential
     singular point of $\Sigma_x$ satisfying $d(v,\eta)=\pi/2$ 
     for every $\eta\in\partial\Sigma_x(C(0))$, then  
     $v$ is the normal direction to $C(0)$.
\end{enumerate}
\end{lem}

\begin{proof}
 Let  $\xi_1,\xi_2\in\Sigma_x$ be the normal directions
 to $C(0)$,  and  $\eta_1$, $\eta_2$ be the boundary points of
 $\Sigma_x(C(0))$. 
 Note that $\Sigma_x$ has a geodesic triangulation 
 which consists of four geodesic triangles with vertices
 $\xi_1$, $\xi_2$, $\eta_1$ and $\eta_2$:
 two of which are isometric to a triangle
 in $S^2(1)$ with sidelengths $\pi/2$, $\pi/2$ and 
 $L(\Sigma_x(C(0)))$. Note that the other two triangles
 may be degenerate biangles if $x$ is a one-normal point.
 In view of this geodesic triangulation,
 a direct calculation with 
 the Alexandrov convexity shows that
 if $x$ is a two-normal point ($\xi_1\neq\xi_2$), then the radius of
 $\Sigma_x$ is greater than $\pi/2$, and (1) follows.
 Furthermore, since $L(\Sigma_{\eta_j}(\Sigma_x)) > \pi$ in this case,
 both $\eta_1$ and $\eta_2$ are not essential singular
 points of $\Sigma_x$ yielding (2).

 $(4)$ is obvious from the Alexandrov convexity.

 For a one-normal point $x$, let $\xi$ be a direction at $x$
 normal to $C(0)$. Then $\xi$, $\eta_1$ and $\eta_2$ 
 span  two geodesic triangles of constant curvature $1$.
 Hence  $(3)$ follows from the following sublemma.
\end{proof}

\begin{slem} \label{slem:surf-sing}
Let $\Sigma$ be a compact Alexandrov surface with curvature $\ge 1$
and with no boundary.
Suppose that there are two geodesic triangle regions
of constant curvature
$1$ in $\Sigma$ spanned by three points 
$\xi$, $\eta_1$ and $\eta_2$ such that
$d(\xi,\eta_i)=\pi/2$, $i=1,2$. 
Let $v (\neq \eta_i)$  be an essential singular point 
in $\Sigma-\{ \eta_1,\eta_2\}$.
\begin{enumerate}
 \item  If $\min d(v,\eta_i)\ge\pi/2$, then 
        $L(\Sigma_v(\Sigma))\ge 2d(\eta_1,\eta_2);$
 \item  If $\min d(v,\eta_i) < \pi/2$, then
        $v$ lies on a geodesic extension of a geodesic joining
        $\eta_1$ and $\eta_2$, and the geodesics $v\xi$ is 
        the common edge of the two geodesic triangles spanned 
        by $v$, $\xi$, $\eta_i$ of constant curvature $1$ 
        and of sidelengths $\pi/2$, $\pi/2$ and $\max d(v,\eta_i)$.
\end{enumerate}
\end{slem}

\begin{proof}
If $v=\xi$, then cleary (1) holds.
We assume $v\neq\xi$ and $d(v,\eta_1)\le d(v,\eta_2)$.
Since $v$ is an essential singular point,
one of $d(\eta_1,v)$ and $d(\xi,v)$ is not greater than 
$\pi/2$ and the other is not smaller than $\pi/2$.
We first assume $d(\eta_1,v)\ge \pi/2$ and $d(\xi,v)\le\pi/2$.
Let $\gamma(t)$, $0\le t\le \delta$, be a unit speed 
geodesic from $\xi$ to $v$, and let 
$\theta_t$ denote the minimal angle between $\gamma|_{[t,\delta]}$
and the minimal geodesics from $\gamma(t)$ to $\eta_1$.
Since $d(\eta_1,\,\cdot\,)$ is concave on 
$\Sigma-B(\eta_1,\pi/2)$, it follows that
$\theta_t$ is decreasing. It follows from the assumption on 
the existence of two geodesic triangles spanned by
$\xi$, $\eta_1$ and $\eta_2$ that 
$$
 \angle \eta_1 v\xi \ge \pi-\theta_0\ge 
     \angle \eta_1\xi\eta_2 \ge d(\eta_1,\eta_2),
$$
which implies $(1)$.

If $d(\eta_1,v)\le \pi/2$ and $d(\xi,v)\ge\pi/2 $,
the above argument shows that 
$v$ is on a geodesic through $\eta_1$ starting from $\eta_2$
and $\angle \xi v\eta_1 =\pi/2$, and the conclusion $(2)$ follows.
\end{proof}

\begin{ex} \label{ex:sing-Y}
In $\R^3$ with the usual coordinates $(x,y,z)$,
first consider the sector $C$  on the $(x,y)$-plane
defined by $C=\{ 0\le r\le a, 0\le \theta\le\alpha\}$,
where $(r,\theta)$ denotes the polar coordinates 
and $a>0,\pi/2>\alpha>0$.
Let $\eta_1$ and $\eta_2$  be the directions of the
$(x,y)$-plane at the origin defined by $\partial C$.
\par
(1)\,\,
Let $\xi=(0,0,1)$ be the direction at the origin,
and take a direction $v\neq \xi$, which is sufficiently 
close to $\xi$, such that 
$\pi/2<\angle(v,\eta_1)<\angle(v,\eta_2)$.
Let $A$ be the union of the 
segment $\{ tv\,|\,0\le t\le 1\}$ and 
the ray $\{ v+ (0,0,t)\,|\,t\ge 0\}$.
Then the double $Y$ of the convex hull of 
the union of $C\times [0,\infty)$ and $A$ 
is a complete
open $3$-dimensional Alexandrov space with 
nonnegative curvature such that $C(0)=C$.
Note that the 
space of directions at the origin $(0,0,0)\in Y$ 
is the double of the geodesic quadrangle on $S^2(1)$
spanned by $\eta_1$, $\eta_2$, $\xi$ and $v$.
Note that the union of $A$ and the segment
of $\partial C$ in the direction of $\eta_j$ 
provides a quasigeodesic of $Y$
consisting of essential singular points of $Y$.
\par
(2)\,\,
Take any $\beta$ with $\alpha<\beta<\pi$,
and let $v$ be a direction of the $(x,y)$-plane
at the origin defined by $\theta=\beta$.
Let $A$ be the union of the piece  of parabola,
$\{ (tv,t^2)\,|\,0\le t\le 1\}$, and 
the ray $\{ v+ (0,0,t)\,|\,t\ge 1\}$.
Then the double $Y$ of the convex hull of 
the union of $C\times [0,\infty)$ and $A$ 
is a complete
open $3$-dimensional Alexandrov space with 
nonnegative curvature such that $C(0)=C$.
Note that the 
space of directions at the origin $(0,0,0)\in Y$ 
is the double of the geodesic triangle on $S^2(1)$
of sidelengths $\pi/2$, $\pi/2$ and $\beta$.
Note that the union of $A$ and the segment
of $\partial C$ in the direction of $\theta=0$ 
provides a quasigeodesic of $Y$
consisting of essential singular points of $Y$.
\end{ex}

(1) and (2) of Example \ref{ex:sing-Y}  correspond to
$(a)$ and $(b)$ of Lemma \ref{lem:class-dir} (3) respectively.

\begin{lem} \label{lem:excep-inv}
Suppose $\dim C(0)=2$, and let $A_i$ be an oriented component 
of $E_i^*$
along which the Seifert invariants of $\psi_i$ are
given by $(\alpha_i,\beta_i)$.
\begin{enumerate}
 \item Unless $A_i=\partial C(0)$, $A_i$ can meet $\partial C(0)$ 
   only with an arc of positive length$;$
 \item If  $A_i$ meets $\partial C(0)$, then 
   $(\alpha_i,\beta_i)=(2,1);$
 \item If $A_i$ has no intersection with $\partial C(0)$, 
   then the closure $\bar A_i$ meets $C(0)$ with exactly 
   one point, say $x$, in ${\rm Ext}(C(0))$, and 
   $\bar A_i\subset\cal N(x)$.
  \begin{enumerate}
   \item If $x\in\interior C(0)$, then 
     $x\in A_i\bigcap{\rm Ext}(\interior C(0))$ and 
     $\alpha_i\le\frac{2\pi}{L(\Sigma_x(C(0)))};$
   \item If $x\in\partial C(0)$, then $x\in F_i^*$ and 
     $\alpha_i\le\frac{\pi}{L(\Sigma_x(C(0)))}$.
  \end{enumerate}
\end{enumerate}
Therefore $\cal C_i$ is included in the union of $\partial C(0)$,
quasigeodesics starting from $\rm Ext(C(0))\bigcap\partial C(0)$
which are perpendicular to $C(0)$, 
and the geodesic segments in $\cal N({\rm Ext}(\interior C(0)));$
\end{lem}

\begin{proof}
Suppose that $A_i$ meets $\partial C(0)$ with a point $x$.
Let $\xi_1$ and $\xi_2$ be directions at $x$ 
determined by $A_i$.
Suppose $A_i\cap\partial C(0)=\{ x\}$.
Since $A_i$ is an extremal subset, 
in view of Lemma \ref{lem:locmax}, one can verify that 
both $\xi_1$ and $\xi_2$ are directions normal to $C(0)$,
yielding $x$ being a two-normal point, 
a contradiction to Lemma \ref{lem:class-dir}.
This argument implies that a sufficiently short subarc 
$A_i'$ of $A_i$ starting from 
$x$ is contained in $\partial C(0)$ and $(1)$ holds true.
Taking a boundary regular point $y\in A_i'$ of $\partial C(0)$,
we see that $\Sigma_y$ is isometric to the spherical 
suspension over $S^1_{\pi}$, which yields $(2)$.

Suppose that $A_i$ does not meet $\partial C(0)$.
From Lemma \ref{lem:locmax}, the closure $\bar A_i$ must meet
$C(0)$, say at $x$.
If $x\in\interior C(0)$, 
$\Sigma_x$ is isometric to the spherical 
suspension over $S^1_{\ell}$ with $\ell=L(\Sigma_x(C(0)))$.
It follows that $x$ is not contained in $F_i^*$. Therefore $x$ is an 
element of $A_i\bigcap {\rm Ext}(\interior C(0))$ and  we obtain
$A_i\subset \cal N(x)$ yielding 
$\alpha_i\le 2\pi/\ell$.
If $x\in\partial C(0)$, then it must be contained in $F_i^*$
and hence is an extremal point of $Y$. 
Let $v$ be the direction at 
$x$ determined by $A_i$. 
By Lemma \ref{lem:locmax}, $A_i$ is perpendicular to $C(0)$.
It follows from Lemma \ref{lem:class-dir} (4) that
$v$ must be the normal direction to $C$ at $x$.
From the definition of quasigeodesics, one can see that
$\bar A_i$ must be contained in $\cal N(x)$,
by developing $A_i$ on the Euclidean plane from a point of
$\cal N(x)$.
\end{proof}

\begin{lem} \label{lem:totalbdy}
Suppose $\dim C(0)=2$. Then $\ell_i\le 1$.
 \begin{enumerate}
  \item If $\ell_i=1$, then
     the circle component of $E_i^*$ coincides with $\partial C(0)$, 
     and $Y=D(C(0)\times [0,\infty));$
  \item Every component of $\bar E_i^*\cap\partial C(0)$ contains 
     at most one component of $E_i^*;$
 \end{enumerate}
\end{lem}

\begin{proof}
Lemma \ref{lem:locmax} yields that there is no choice for 
the circle component of $E_i^*$ but $\partial C(0)$,  and hence
$\ell_i\le 1$.
If $\ell_i=1$, Lemma \ref{lem:class-dir} implies that 
each point of $\partial C(0)$ is a one-normal point, from which
$(1)$ follows. Suppose that there are adjacent components $A_i$ and 
$A_i'$ of $\cal C_i$ with nonempty 
intersection $\bar A_i\cap\bar A_i'$ in $F_i^*$.
It turns out that both $A_i$ and $A_i'$ have Seifert invariants
$(2,1)$, a contradiction to Lemma \ref{lem:adj-inv}.
\end{proof}

Suppose $\ell_i=1$, and 
let $p:C(0)\times\R\to Y$ be the map naturally 
extending $f:\interior C(0)\times\R\to Y$. Let 
$C_1$ be a $2$-disk domain in $\interior C(0)$
containing $\interior C(0)\cap E_i^*$, and
$C_2$ the closure of $C(0)-C_1$.
Put $W_i^j=\pi_i^{-1}(p(C_j\times\R))$, $j=1,2$.
Since $\pi_i^{-1}(p(\partial C_1\times\R))\simeq T^2\times I$,
we have
$$
 B(p_i,R)\simeq W_i^1\bigcup_{T^2\times I} W_i^2.
$$
Since $p(C_2\times\R)\cap B(y_0,R)$ is homeomorphic to 
a tubular neighborhood of $\partial C(0)$,
it follows from (3.8) in \cite{Fn:circleI} that
$W_i^2\simeq T^2\times D^2$.
Obviously  $W_i^1\simeq S^1\times D^3$ if $m_i\le 1$.

If $m_i=2$, then as before
$B(p_i,R)$ is homeomorphic to $(K^2\tilde\times I)\times I$.

Next consider the case of $\ell_i=0$.
We need

\begin{lem} \label{lem:class-singsurf}
Let $Z^2$ be a nonnegatively curved compact Alexandrov surface 
with boundary. Let $m$ and $n$ denote the numbers of 
the extremal points in $\interior Z^2$ and 
in $\partial Z^2$ respectively.
Then $2m+n\le 4$, where the equality holds if and only if
$Z^2$ is isometric to either a form 
$D(I\times\{ x\ge 0\})\cap \{ x\le a\}$, the result of cutting 
of the double $D(I^2)$ of a square $I^2$ along the diagonals,
or a rectangle $[a,b]\times [c,d]$.
\end{lem}
 
\begin{proof}
By the Gauss-Bonnet formula, we have the inequality
$2m+n\le 4$, where the equality holds if and only if
\begin{itemize}
 \item  $X^2$ is smooth and flat except those essential singular 
   points of $Z^2;$
 \item the space of directions at an essential singular point 
   $x$ of $Z^2$ has length $\pi$ $($resp. $\pi/2$$)$
   if $x\in\interior Z^2$ (resp. if $x\in\partial Z^2$)$;$ 
 \item $\partial Z^2$ is a broken geodesic with 
   those boundary essential singular points as brerak points.
\end{itemize}
The conclusion easily follows.
\end{proof}

By Lemmas \ref{lem:excep-inv} and \ref{lem:class-singsurf},
$2m_i+n_i\le 4$, where 
the equality holds if and only if $C(0)$ is isometric to 
one of the three types described in Lemma \ref{lem:class-singsurf}.

If $m_i=2$,
$B(p_i,R)$ is homeomorphic to $(K^2\tilde\times I)\times I$
as above.

Suppose $m_i=1$. If $n_i=1$, let $A_i$ denote the component of $E_i^*$ 
whose closure does not touch $F_i^*$. Divide
$B(y_0,R)$ with a proper $2$-disk into two $3$-disk domains 
$B_1$ and $B_2$
in such a way that $A_i\subset B_1$ and 
the other components of $\cal C_i$
is contained in $B_2$. Clearly 
$\pi_i^{-1}(B_1)\simeq S^1\times D^3$ and 
$\pi_i^{-1}(B_2)\simeq D^4$.

If $(m_i,n_i)=(1,2)$, in view of Lemmas \ref{lem:excep-inv} and 
\ref{lem:class-singsurf} together with 
Lemma \ref{lem:adj-inv}, we come to the situation of Lemma
\ref{lem:m=n=1} contradicting to the orientability.
Hence the case $(m_i,n_i)=(1,2)$ does not occur.
\par

Suppose $m_i=0$. 
In view of the argument in the case of $\dim C(0)=1$, 
we may assume that $n_i\ge 2$.
If $n_i=2$, Proposition \ref{prop:plumbing} and Lemma
\ref{lem:excep-inv} imply that 
$B(p_i, R)\simeq S^2\tilde\times_{\omega} D^2$,
where 
\begin{equation}
  |\omega|\le \frac{\pi}{L(\Sigma_{x}(C(0)))}
         +  \frac{\pi}{L(\Sigma_{y}(C(0)))}, \label{eq:omega-est}
\end{equation}
and $\{ x,y\}=F^*_i$.

If $n_i=3$, then $\cal C_i$ is disconnected.
Cut
$B(y_0,R)$ with a $2$-disk into two 
$3$-disk domains $B_1$ and $B_2$ 
in such a way that a connected component of 
$\cal C_i$ containing only one element of $F^*_i$ 
is contained in $B_1$ and the other components are  contained 
in $B_2$.
Then $\pi_i^{-1}(B_1)\simeq D^4$ and 
$\pi_i^{-1}(B_2)\simeq S^2\tilde\times_{\omega} D^2$,
where $|\omega|$ is estimated as in \eqref{eq:omega-est}.

If $n_i=4$, then $\cal C_i$ is disconnected.
Cut
$B(y_0,R)$ with a $2$-disk into two 
$3$-disk domains $B_1$ and $B_2$ 
in such a way that connected components of 
$\cal C_i$ containing exactly two elements of $F^*_i$ 
are contained in $B_1$ and the other components are  contained 
in $B_2$.
Then by Proposition \ref{prop:plumbing}, both $\pi_i^{-1}(B_1)$ and 
$\pi_i^{-1}(B_2)$ are homeomorphic to $S^2\tilde\times_{\omega} D^2$,
where $0\le |\omega|\le 4$.

This completes the proof of Theorem \ref{thm:dim3class}
in the case when $Y$ has no boundary.
\par
\medskip

Next we consider the case when $Y$ has nonempty connected boundary.
We begin with 

\begin{proclaim}{\emph{Case} B-I }
$\partial Y$ is connected and $\dim S=2$.
\end{proclaim}

By Proposition \ref{prop:codim1},
$Y$ is isometric to the product $S\times [0,\infty)$.
Therefore  $Y(\infty)$ is a point, and 
Theorem \ref{thm:patchwbdy} implies that
$B(p_i,R)$ is homeomorphic to a $D^2$-bundle over $S$. 

\begin{proclaim}{\emph{Case} B-II }
$\partial Y$ is connected and $\dim S=1$.
\end{proclaim}

In this case, $Y$ is isometric to a product 
$(\R\times N^2)/\Lambda$, where $N^2$ is either homeomorphic 
to $R^2_+$ or isometric to $I\times \R$, 
and $\Lambda\simeq \Z$. 
Let $H_i$ denote the preimage of an $N^2$-factor by $\pi_i$.
\par

Suppose $N^2\simeq \R^2_+$.
Let $k (\le 1)$ denote the number of the essential singular points of 
$\interior N^2$.
If $k=0$, by Theorem \ref{thm:patchwbdy},
$H_i\simeq D^3$, yielding $B(p_i,R)\simeq S^1\times D^3$.
Suppose $k=1$. Then $N^2$ is isometric to a form 
$D([0,\infty)\times [0,\infty))\bigcap \{ (x,y)\,|\,y\le a\}$,
which implies that $\dim Y(\infty)=0$ and 
$H_i\simeq P^2\tilde\times I$ if $H_i$ has a singular orbit.
Since $B(p_i,R)$ is homeomorphic to an $H_i$-bundle over $S^1$,
it follows from the following lemma that 
$B(p_i,R)\simeq S^1\times(P^2\tilde\times I)$.

\begin{lem}
The mapping class group 
$\cal M_+(P^2\tilde\times I)$ of orientation preserving homeomorphisms
of $P^2\tilde\times I$ is trivial.
\end{lem}

\begin{proof}
Let $\cal M_D(\R P^3)$ be the mapping class group of homeomorphisms of 
$\R P^3$ fixing a disk.
Since any orientation preserving homeomorphism $f$ of $P^2\tilde\times I$
is isotopic to the identity on the boundary sphere,
we way assume $f=1$ on $\partial P^2\tilde\times I$.
Therefor $\cal M_+(P^2\tilde\times I)\simeq \cal M_D(\R P^3)$.
The result follows from \cite{Bnh:diffeotopy},
\cite{HR:involution} and \cite{FW:homotopy}.
\end{proof}

Suppose $N^2$ is isometric to $I\times\R$.
Since $H_i\simeq S^2\times I$, $B(p_i,R)$ is an 
$(S^2\times I)$-bundle over $S^1$.
Since $\cal M_+(S^2\times I)=\Z_2$, 
$B(p_i,R)$ is homeomorphic to either $S^1\times S^2\times I$ or 
an $I$-bundle over $S^1\tilde\times S^2$, the nontrivial $S^2$-bundle 
over $S^1$.

\begin{proclaim}{\emph{Case} B-III }
$\partial Y$ is connected and $\dim S=0$.
\end{proclaim}

If $Y$ has two ends, it is isometric to 
a product $\R\times Y_0^2$, where $Y_0^2\simeq D^2$.
Let $K_i$ denote the preimage of an $Y_0^2$-factor by $\pi_i$.
Then $B(p_i,R)\simeq K_i\times I$.
Note that $Y_0^2$ has either at most one 
essential singular point, say $q$, in its interior, or 
isometric to $D(I\times \{ x\ge 0\})\bigcap \{ x\le a\}$
for some $I$ and $a>0$. In either case, $\ell_i=n_i=0$ and $m_i\le 2$.
In the former case, $K_i$ is homeomorphic 
to $L(\mu_i,\nu_i)$, where $\mu_i\le 2\pi/L(\Sigma_q(Y_0^2))$.
In the latter case, $K_i$ is homeomorphic 
to either $P^3$ ($m_i=1$) or $P^3\,\#\, P^3$ ($m_i=2$).
\par\medskip

In what follows, we assume that $Y$ has exactly one end.
In the argument below about the geometry of $Y$,
we show that $\partial Y$ is homeomorphic to $\R^2$.

Since we have a collar neighborhood of $\partial Y$ 
(Theorem \ref{thm:collar}), we can 
apply the method of \cite{SY:3mfd} to $Y_{\epsilon}$ and obtain that
$B(S,R)$ is homeomorphic to $D^3$ for a large $R>0$.

Note that $(\frac{1}{R} Y,S)$ converges to $(K(Y(\infty),o)$ as 
$R\to\infty$, and denote by $B_R$ the closure of 
$\partial B(S,R) - B(S, R) \cap \partial Y$.
We put $C_R:=\partial Y \cap B(S, R)$.

\begin{ass} \label{ass:B(S,R)capY}
Assume that $Y$ has exactly one end.
Then for any sufficiently large $R$, 
both $B_R$ and $C_R$ are homeomorphic to $D^2$.\par
In particular, $Y$ is homeomorphic to $\R^3_+$.
\end{ass}

\begin{proof}
By using the gradient flows for the distance function $d_S$
from $S$, we see that for sufficiently large $R$,
\begin{enumerate}
 \item $Y - B(S,R) \simeq B_R\times (0,\infty);$
 \item $\partial Y-C_R\simeq \partial B_R\times (0,\infty)$,
\end{enumerate}
where $\partial B_R=\partial C_R$ is the boundary as a topological 
$2$-manifold.
Since $Y$ has exactly one end, $B_R$ must be connected.
We assert that $\partial B_R$ is connected.
Applying the method of \cite{SY:3mfd}, we also 
have a pseudo-gradient flow for 
$d_S$ on $B(S,R) - \{ S\}$ (see Section 10 of \cite{SY:3mfd}
for the definition of pseudo-gradient flows).
For a small $\epsilon$ with $B(S,\epsilon)\subset \interior Y$ and 
$\partial B(S,\epsilon)\simeq S^2$,
let $h:B(S,R) - \interior B(S,\epsilon)\to \partial B(S,\epsilon)$
be the projection along the flow curves.
Since $h(B_R)\subset S^2$ is connected, 
if $\partial B_R$ was disconnected, $h(C_R)$ and hence
$C_R$ would be disconnected, yielding
the disconnectivity of $\partial Y$, a contradiction.
The former conclusion of the assertion then follows from the 
connectivity of both $h(B_R)$ and $h(C_R)$. The latter 
follows easily.
\end{proof}

To determine the topology  of $B(p_i, R)$,
consider the distance function $d_{\partial Y}$ from $\partial Y$.
Let $C^*$ be the maximum set of $d_{\partial Y}$  (possibly empty).

\begin{lem}  \label{lem:sing-incl}
$\cal C_i\subset C^*$ for sufficiently large $i$.
\end{lem}

\begin{proof}
The proof is by contradiction.
In a way similar to Lemma \ref{lem:locmax}, 
we have
\begin{itemize}
 \item $d_{\partial Y}$ is monotone on each component of 
       $\cal C_i - C^*;$
 \item for any $t>0$ less than the maximum of $d_{\partial Y}$,
   and for any $x$ with  $d_{\partial Y}(x)\ge t$,
   $d_x$ has no local minimum on $\cal C_i - \{ d_{\partial Y}\ge
   t\}$.
\end{itemize}
Suppose that $\cal C_i$ is not contained in $C^*$ 
for any sufficiently large $i$.
Then we have a sequence $R_i\to\infty$ and $S^1$-actions 
$\psi_i$ on $B(p_i, R_i)$ extending the original actions.
Since $d_{C^*}$ has no local minimum on $\cal C_i - C^*$,
there is a subarc $A_i(t)$ of $\cal C_i$ with unit speed
parameter, $0\le t\le\ell_i$, $\ell_i\to\infty$,
such that 
$(i)$ \,\,$d_{\partial Y}(A_i(t))$ is monotone decreasing,
$(ii)$\,\,$A_i(\ell_i)\in\partial B(y_0,R_i)-\partial Y$, 
$(iii)$\,\,$A_i(t)$ does not reach $\partial Y$.
Therefore for any $\epsilon>0$ there is an $i$ and 
$s\in(0,\ell_i)$ such that 
$$
 0 \ge (d_{\partial Y}\circ A_i)'(s)=
   - \cos\angle(A_i'(s),(\partial Y)_{A_i(s)}') > -\epsilon.
$$
It follows that 
$$
\angle(A_i'(s),\partial\Sigma_{A_i(s)}(\{ d_{\partial Y}\ge d_s\}))
  <\tau(\epsilon),
$$
where $d_s:=d_{\partial Y}(A_i(s))$.
Take $x$ with $d_{\partial Y}(x)=d_s$ sufficiently close to $A_i(s)$
such that $\angle(A_i'(s),x_{A_i(s)}')<\tau(\epsilon)$.
Now it is easy to see that $d_x$ has a local minimum
on $\cal C_i - \{ d_{\partial Y}\ge d_s\}$,
a contradiction.
\end{proof}

\begin{proclaim}{\emph{Case} (i).}
   $d_{\partial Y}$ has no maximum.
\end{proclaim}

By Lemma \ref{lem:sing-incl}, $\cal C_i$ is empty, and therefore
$B(p_i, r)\simeq D^2\times D^2$.

\begin{proclaim}{\emph{Case} (ii).}
    $d_{\partial Y}$ has a maximum.
\end{proclaim}

From now on we assume that $\cal C_i$  is nonempty.
By the concavity of $d_{\partial Y}$, every geodesic ray of 
$Y$ stating from
any point of $C^*$ is contained in $C^*$.
It follows that $Y(\infty)$ is isometric to $C^*(\infty)$.
Obviously $S$ is isometric to a soul of $C^*$.

If $\dim C^*=2$, as in the case of $\dim C(0)=2$,
for each $x\in\interior C^*$ there are two minimal geodesics 
from $x$ to $\partial Y$. Thus we have a 
locally isometric imbedding 
$f:\interior C^*\times [-a,a]\to Y$, where $a$ is the maximum
of $d_{\partial Y}$.

Suppose that $F^*_i\cap\interior Y$ is empty and $E^*_i$ is 
nonempty. We show $m_i=1$.
If $m_i\ge 2$, then $C^*$ is isometric to $I\times\R$ and
$E_i^*=\partial C^*$.
Therefore it is easy to verify that
$Y$ is isometric to $\R\times
(D(I\times \{ x\ge 0\})\cap\{ x\le a\})$ for some $a>0$,
and thus $Y$ would have two ends, a contradiction to the 
hypothesis. Therefore  $m_i=1$.

Suppose first the special case that $E^*_i$ is a minimal geodesic  
extending to a line. Then $Y$ is isometric to a product
$\R\times Y_0^2$, where $Y_0^2\simeq \R^2_+$.
Since $\interior Y_0^2$ contains an essential singular
point, $Y_0^2$ is isometric to 
$D(\{ x,y\ge 0\})\bigcap \{ y\le a\}$ for some $a>0$.
Since the preimage of a $Y_0^2$-factor by $\pi_i$
is homeomorphic to $P^2\tilde\times I$,
$B(p_i, R)$ is homeomorphic to $(P^2\tilde\times I)\times I$.
Note that the Seifert invariants along $E^*_i$  are
$(2,1)$ and that $C^*$ is isometric to 
$\R^2_+$.

Next consider the general case. 
By contradiction together with the argument above,
we may assume that $\dim C^*=2$ and $E_i^*=\partial C^*$.
It follows that the Seifert invariants along $E^*_i$  are
$(2,1)$. Therefore Proposition \ref{prop:circ-class}
combined with the above argument implies that 
$B(p_i, R)\simeq (P^2\tilde\times I)\times I$.


Suppose  $\dim Y(\infty)\ge 1$.
Then $C^*$  has at most one extremal point on $C^* $.
Since ${\rm Ext}(Y)={\rm Ext}(C^*)$, we have
$\ell_i=0$, $m_i\le 1$ and $n_i\le 1$.
It follows from the same reason as Lemma \ref{lem:totalbdy}(2)
that every component of $\bar E_i^*\cap\partial C^*$ contains 
at most one component of $E_i^*$.
Proposition \ref{prop:plumbing} then 
implies that $B(p_i,R)$ is homeomorphic to either 
$D^4$ or a $S^2\tilde\times_{\omega} D^2$ with 
$|\omega|\in \{ 1, 2\}$.

Next suppose $\dim Y(\infty)=0$.
If $C^*$ was a line, then $Y$ would split as $Y=H^2\times \R$
with $H^2\simeq D^2$, a contradiction to the hypothesis for 
$Y$ having exactly one end.
Thus $C^*$ must be a geodesic ray if 
$\dim C^*=1$.
Let us assume that $C^*$ is a geodesic ray.
Let $\Z_{\mu_i}$, $0\le \mu_i < \infty$, 
be the isotropy group at an interior point $x$ of $C^*$. 
Letting $K_x=\R\times K(S^1_{\ell})$, we obtain
$\mu_i\le 2\pi/\ell$.
In a similar way as above, Proposition \ref{prop:plumbing}
implies that $B(p_i,R)\simeq S^2\tilde\times_{\mu_i} D^2$.

Suppose next that $\dim C^*=2$.
Since we may assume that $C^*$ has two extremal points,
it is isometric to either a product $I\times [0, \infty)$
or the double $D(I\times [0,\infty))$. But in the argument below,
we may assume the former case.
Let $v_1, v_2\in\partial C^*$ be the extremal points of $C^*$.
Here we have the following four cases for $\cal C_i$.
\begin{enumerate}
 \item $\cal C_i$ coincides with $\{ v_1, v_2\};$
 \item $\cal C_i$ is the geodesic segment joining $v_1$ and $v_2;$
 \item $\cal C_i$ consists of $\{ v_1, v_2\}$ and a geodesic segment from
       $v_1$ to $\partial B(y_0, R);$
 \item $\cal C_i$ consists of two geodesic segments from
       $v_1$ and $v_2$  to $\partial B(y_0, R)$.
\end{enumerate}
Let $U$ be a small regular closed neighborhood of $\partial Y$.
Then $U_i := \pi_i^{-1}(U\cap B(y_0,R))$ is homeomorphic to 
$D^4$. Let $V_i$ denote the preimage of the closure of
$B(y_0, R)-U$ by $\pi_i$.
Then $B(p_i, R)\simeq U_i\bigcup_{D^2\times S^1} V_i$. 
Proposition \ref{prop:plumbing} implies that 
$V_i\simeq S^2\tilde\times_{\omega} D^2$, where 
$|\omega|\in \{ 0,2,4\}$ in the cases of $(1)$, $(2)$ and $(4)$,
$|\omega|\in \{ 1,3\}$ in the case of $(3)$.
\par\medskip

Finally we consider

\begin{proclaim}{\emph{Case} C.}
$Y$ has disconnected boundary.
\end{proclaim}

In this case by Theorem \ref{thm:cptcomp}, 
$\partial Y$  consists of two connected components and 
$Y$ is isometric to a product $Z\times I$, where $I$ is a 
closed interval and  $Z$ is a component of $\partial Y$. Note
$\dim Y(\infty)=\dim Z(\infty)$.

Suppose $\dim Y(\infty)\ge 1$. Then $Z$ is  
homeomorphic to $\R^2$.
Let $k$ be the number 
of the essential singular points of $Z$, which is at most $1$.
\par
  We now apply Theorem \ref{thm:patchwbdy}.
If $k=0$,  $B(p_i,R)$ is homeomorphic to 
the gluing 
$$
 (D^2\times I)\times S^1\bigcup (D^2\times \partial I)\times D^2
    \simeq D^2\times S^2.
$$
Note that even if $k=1$ there are no singular loci
of the $S^1$-action $\psi_i$ in $\interior B(y_0, R)$.
For if there were, they would 
correspond to 
$\{ x_0\}\times I$, where $x_0$ is the essential singular point 
of $Z$. This contradicts Theorem \ref{thm:patchwbdy} (4).
Then Theorem \ref{thm:patchwbdy} again implies that 
$B(p_i,R)\simeq D^2\times S^2$.

Next suppose  $\dim Y(\infty)=0$.
In view of the argument above, we may assume that
$Z$ is isometric to a flat cylinder or a flat M\"obius
strip.
By Theorem \ref{thm:patchwbdy},
in the former case,  $B(p_i, R)$ is homeomorphic to
an $S^2$-bundle over $S^1\times I$, which is
homeomorphic to $S^1\times S^2\times I$.
In the latter case,  $B(p_i, R)$ is homeomorphic to
an $S^2$-bundle over the twisted product $S^1\tilde\times I$, which is
homeomorphic to $I$-bundle over the nontrivial bundle 
$S^1\tilde\times S^2$.

This completes the proof of Theorem \ref{thm:dim3class}.
\par\medskip

When a $4$-manifold $W^4$ has boundary homeomorphic to 
$S^3$, we put 
$\CCap W := W\cup_{S^3} D^4$

\begin{prop} \label{prop:find-top}
Under the same assumption of B-III in Theorem \ref{thm:dim3class},
suppose that $n_i = 2$ and $E_i^*$ does not meet 
$\partial B(y_0,R)$. 
Then $\partial B(p_i,R)\simeq S^3$ and 
$\CCap B(p_i, R)$ is homeomorphic to 
either $\C P^2\# (\pm \C P^2)$ or $S^2\times S^2$.
\end{prop}

\begin{proof}
Note that 
$$
\partial B(p_i,R)\simeq D^2\times \partial B_R \bigcup S^1\times B_R
\simeq S^3,
$$
and that 
$\CCap B(p_i,R)$ is simply connected
and admits a local smooth $S^1$-action $\tilde\psi_i$ whose orbit 
space is homeomorphic to $B(y_0,R)$ with 
$F^*(\tilde\psi_i)=\partial B(y_0,R)\cup F_i^*$.
It follows from Corollary \ref{cor:sum} together with the fact
$\chi(\CCap B(p_i,R))=\chi(F(\tilde\psi_i))$
that $\CCap B(p_i,R)$ has a required topological type.
\end{proof}

\begin{rem}
\begin{enumerate}
 \item One can conclude that in Proposition \ref{prop:find-top}, 
   if $\bar E_i^*$ is a
   segment joining the two interior fixed points, then 
   $\CCap B(p_i, R)\simeq S^2\times S^2;$
 \item Proposition \ref{prop:find-top} shows that 
   $B(p_i,R)$ cannot be homeomorphic to a disk-bundle 
   under the situation of B-III if $n_i=2$.
\end{enumerate}
\end{rem}


\section{Collapsing to two-spaces without boundary \\
      (sphere fibre case)} 
\label{sec:dim2nobdysph}

 Let a sequence of pointed complete $4$-dimensional orientable Riemannian
manifolds $(M_i^4,p_i)$ with $K \ge -1$ converge to a pointed
two-dimensional Alexandrov space $(X^2,p)$. 
Throughout this section, we assume that 
$p$ is an interior point of $X^2$.
\par

Now consider the local convergence $B(p_i,2r)\to B(p,2r)$
for a  sufficiently small positive number $r$.
By Fibration Theorem \ref{thm:orig-cap},
$A(p_i;r,2r)$ is homeomorphic to an $F_i$-bundle over 
$A(p;r,2r)\simeq S^1\times I$,
where $F_i$ is either $S^2$ or $T^2$.\par

By Theorem \ref{thm:rescal}, we have sequences $\delta_i\to 0$ and 
$\hat p_i\to p$ such that
\begin{enumerate}
  \item  for any limit $(Y,y_0)$ of $(\frac{1}{\delta_i}M_i^4, \hat p_i)$, 
     we have  $\dim Y\ge 3$;
  \item  $B(p_i,r)$ is homeomorphic to $B(\hat p_i,R\delta_i)$ for every 
      $R\ge 1$ and large $i$ compared to $R$.
\end{enumerate}

Let $S$ be a soul of $Y$.  It follows 
from Lemma \ref{lem:expand} and Proposition \ref{prop:ideal} that 
 
 \begin{align}
   & \dim Y(\infty)\ge 1,\label{eq:dimyinfty}  \\
   & \dim S\le \dim Y - 2.        \label{eq:dims}
 \end{align}

In this section, we consider the case of the general fibre $F_i=S^2$.
Then we have 

\begin{equation}
    \partial B(p_i,r)\simeq S^1\times S^2. \label{eq:s^1s^2}
\end{equation}

We apply Theorem \ref{thm:dim3class} to the 
convergence $(\frac{1}{\delta_i}M_i^4, \hat p_i)\to (Y,y_0)$
under the conditions \eqref{eq:dimyinfty}, \eqref{eq:dims}
and \eqref{eq:s^1s^2}.
The main purpose of this section is to prove

\begin{thm} \label{thm:patchs^2}  There is a positive number $r_p$
such that for any $r\le r_p$ and any sufficiently large $i$ 
compared with $r$,
\begin{enumerate}
  \item $B(p_i,r)\simeq D^2\times S^2;$ 
  \item there exists an $S^2$-bundle structure on $B(p_i,2r)$  
        compatible to the $S^2$-bundle structure on $A(p_i;r,2r)$.
\end{enumerate}
\end{thm}

Since $S_{\delta}(X^2)$ is discrete for any $\delta>0$ 
when $X^2$ has no boundary,
together with Fibration Theorem \ref{thm:orig-cap}, 
Theorem \ref{thm:patchs^2} yields 
Theorem \ref{thm:dim2nobdy} in the case when the general 
fibre is a sphere.

First we begin with 

\begin{lem} \label{lem:notd^3s^1}
The fundamental group of $B(p_i,r)$ is finite and  cyclic
for any small $r>0$ and any sufficiently large $i$.
\end{lem}

\begin{proof}
Let $\tilde B(p_i,r)$ be the univerasal covering space of $B(p_i,r)$
with the deck transformation group  $\Gamma_i$.
Take a sequence $r_i\to 0$ of positive numbers satisfying
$(\frac{1}{r_i} B(p_i,r),p_i)\to (K_p,o_p)$.
We may assume that 
$(\frac{1}{r_i}\tilde B(p_i,r),\tilde p_i,\Gamma_i)$
converges to a triplet $(Z ,z_0,G)$.
We assert that $G$ is discrete and hence $\Gamma_i$ is 
isomorphic to $G$ for large $i$.
Otherwise, we would have a sequence $\gamma_i\neq 1\in\Gamma_i$
converging to the identity of $G$ under the above convergence.
Applying Fibration Theorem \ref{thm:orig-cap} to a contractible ball $B$ 
in $K_p-\{ o_p\}$, 
we have a closed domain $U_i$ in $B(p_i,r)$ fibers over 
$B$. Let $x_i\in U_i$ be a point converging to the center of 
$B$. Obviously, the geodesic loop $c_i$ at $x_i$ represented by
$\gamma_i$ is contained in $U_i$.
Since $U_i\simeq B\times S^2$, $c_i$ must be 
null homotopic, a contradiction to $\gamma_i\neq 1$.
Since $Z/G$ is the flat cone $K_p$, it follows that 
$Z$ is also a flat cone and $G$ is a finite cyclic group.
\end{proof}

Now we consider 

\begin{proclaim}{\emph{Case} A.}
    $\dim Y = 4$.
\end{proclaim}

Intuitively, this is the case when the sphere fibre
converges to a sphere under the rescaling $\frac{1}{\delta_i} M_i^4$.
In fact we have

\begin{prop} \label{prop:y4}
If $\dim Y = 4$, then  
 \begin{enumerate}
  \item  $S\simeq S^2;$
  \item  $B(p_i,r)\simeq S^2\times D^2.$
 \end{enumerate}
\end{prop}

\begin{proof}
In view of Corollary \ref{cor:soul} together with 
\eqref{eq:s^1s^2} and Lemma \ref{lem:notd^3s^1}, we have 
the only possibility $\dim S=2$.
It follows that  $S$ is homeomorphic to 
$S^2$, $P^2$, $T^2$ or $K^2$.
If $S\simeq P^2$, then $B(p_i,r)$ must be homeomorphic to
a $D^2$-bundle over $P^2$. 
It turns out that $\partial B(p_i,r)$ is homeomorphic to 
an $S^1$-bundle over $P^2$,
a contradiction to \eqref{eq:s^1s^2}.
Similarly, if $S$ was homeomorphic to $T^2$ or $K^2$, we would have 
a contradiction.
Thus $S$ must be homeomorphic to $S^2$ and 
$B(p_i,r)$ is homeomorphic to a $D^2$-bundle over $S^2$.
In view of \eqref{eq:s^1s^2}, this bundle must be 
trivial.
\end{proof}

Next we consider 

\begin{proclaim}{\emph{Case} B.}
    $\dim Y = 3$.
\end{proclaim}

Intuitively, this is the case when the sphere fibre 
collapses to a closed interval under the rescaling 
$\frac{1}{\delta_i} M_i^4$ (see Propositions 
\ref{prop:y3nobdy} and \ref{prop:y3wbdy}).

Let 
$$
 Y\supset C(0)\supset C(1)\supset \cdots\supset C(k),
$$
be as in Section \ref{sec:ideal}.
Applying Theorem \ref{thm:dim3} to the convergence
$(\frac{1}{\delta_i} M_i^4,\hat p_i) \to (Y,y_0)$,
we have a locally smooth, local $S^1$-action $\psi_i$ 
on $B(p_i,r)\simeq B(\hat p_i, R\delta_i)$ whose orbit
space is homeomorphic to $B(y_0, R)$, where $R$
is a large positive number.
Let $F_i^*:=F^*(\psi_i)$, $E_i^*:=E^*(\psi_i)$, $S_i^*:=S^*(\psi_i)$,
$\cal C_i:=S^*_i - \partial Y$, $\ell_i$, $m_i$, $n_i$  
and  $\pi_i:B(\hat p_i,R)\to B(y_0,R)$ be as in the previous 
section.

\begin{prop} \label{prop:y3nobdy}
If $\dim Y=3$ and  $Y$ has no boundary, then 
\begin{enumerate}
 \item $Y\simeq \R^3;$
 \item $C(0)$ is isometric to a $1$-dimensional closed interval$;$
 \item $\partial C(0)=F^*_i$. In particular,
        ${\rm Ext}(Y)=\partial C(0);$
 \item $\cal C_i$ coincides with one of the following:
   \begin{gather*}
    \partial C(0),\quad C(0),\quad \gamma_1\cup\gamma_2\\
    C(0) \cup \gamma_1\cup\gamma_2,
   \end{gather*}
   where $\gamma_1$ and $\gamma_2$ denote quasigeodesics in  
   $B(y_0,R)$ starting from the endpoints $\partial C(0)$, 
   reaching $\partial B(y_0,R)$, and being perpendicular to $C(0);$
 \item $B(p_i,r)\simeq S^2\times D^2$.
\end{enumerate}
\end{prop}

\begin{proof}
Note $\dim S\le 1$.  Suppose first 

\begin{proclaim}{\emph{Case} B-I-1.}
$Y$ has no boundary and $\dim S=1$.    
\end{proclaim}

Case A II-(1) of Theorem \ref{thm:dim3class}  yields 
a contradiction to \eqref{eq:s^1s^2}.

Therefore we have the following case:

\begin{proclaim}{\emph{Case} B-I-2.}
$Y$ has no boundary and $\dim S=0$.    
\end{proclaim}

Note that $B(y_0,R)\simeq D^3$ for large $R$
(\cite{SY:3mfd}).
Since $\dim Y(\infty)\ge 1$,  we may assume  $\dim C(0)\le 1$
(Lemma \ref{lem:C0codim1}).
In view of Case A-III of Theorem \ref{thm:dim3class} combined with
$\eqref{eq:s^1s^2}$, we can exclude the case of $\dim C(0)=0$.
Thus $C(0)$ must be a geodesic segment.
Case A-III of Theorem \ref{thm:dim3class} implies 
$(m_i,n_i)=(0,2)$ and $B(p_i,r)\simeq S^2\tilde\times_{\omega} D^2$,
where $\omega=0$ if and only if 
we have the cases (a), (c), (d) and (f) in Proposition 
\ref{prop:plumbing}.
Thus we obtain the conclusions (4) and (5).
\end{proof}

Next we consider the case when $Y$ has nonempty boundary.

\begin{prop} \label{prop:y3wbdy}
If $\dim Y=3$ and $Y$ has nonempty boundary, then 
\begin{enumerate}
 \item $Y$ is isometric to a product $Z\times I$, where 
       $Z\simeq \R^2$ has at most one essential singular point$;$
 \item $B(p_i,r)\simeq S^2\times D^2$.
\end{enumerate}
\end{prop}

\begin{proof}
We first assume

\begin{proclaim}{\emph{Case} B-II}
$\partial Y$ is disconnected.
\end{proclaim}

In this case, by Theorem \ref{thm:cptcomp}, 
$Y$ is isometric to the product $Z\times I$, where $I$ is a 
closed interval and  $Z$ is a two-dimensional open Alexandrov 
surface with nonnegative curvature 
homeomorphic to $\R^2$ and with 
$\dim Y(\infty)=\dim Z(\infty)\ge 1$.
By Case C-I of Theorem \ref{thm:dim3class}, 
$B(p_i,r)$ must be homeomorphic to $D^2\times S^2$.

Next we assume the case when $\partial Y$ is connected.

\begin{proclaim}{\emph{Case} B-III-1 }
$\partial Y$ is connected and $\dim S=1$.
\end{proclaim}

In this case, by Theorem \ref{thm:dim3class}, 
$B(p_i,r)\simeq S^1\times D^3$, which contradicts
Lemma \ref{lem:notd^3s^1}.

\begin{proclaim}{\emph{Case} B-III-2}
$\partial Y$ is connected and $\dim S=0$.
\end{proclaim}

In this case, in view of Case B-III of Theorem \ref{thm:dim3class}, 
we would have a contradiction to Lemma \ref{lem:notd^3s^1}.
Thus we have proved Proposition \ref{prop:y3wbdy}.
\end{proof}


\begin{proof}[Proof of Theorem \ref{thm:patchs^2}]
By Propositions \ref{prop:y4}, \ref{prop:y3nobdy} and 
\ref{prop:y3wbdy}, we already know that $B(p_i,r)\simeq D^2\times S^2$.
It is known (cf. \cite{CR:homotopy}, \cite{Ht:S1S2})
that the mapping class group 
$\cal M_+(S^1\times S^2)$ of orientation preserving homeomorphisms 
of $S^1\times S^2$ is isomorphic to $\Z_2\oplus \Z_2$. 
It is generated by the homeomorphisms
$f=(f_1, f_2)$ and $g$, where $f_1$ and $f_2$
are orientation reversing homeomorphisms of $S^1$ and $S^2$ respectively
and $g$ is defined as 
\begin{equation*}
 g(e^{i\theta}, x)=(e^{i\theta}, R(\theta) x),\qquad
        R(\theta)=\begin{pmatrix} 
              \cos \theta & -\sin \theta & 0 \\
              \sin \theta & \cos \theta & 0 \\
               0 & 0 & 1
                 \end{pmatrix}.
\end{equation*}
Note that $g^2$ is isotopic to the identity since
$\pi_1 SO(3)=\Z_2$.
From the  $S^2$-bundle structure on  $A(p_i;r,2r)$,
we have a homeomorphism $\varphi_i:\partial B(p_i,r)\to S^1\times S^2$.
From $B(p_i,r)\simeq D^2\times S^2$, we also have a 
homeomorphism $\varphi_i':\partial B(p_i,r)\to S^1\times S^2$.
Note that 
$$
 B(p_i,2r)\simeq [r,2r]\times S^1\times S^2
          \bigcup_{\varphi_i'\circ\varphi_i^{-1}}
                         \{ r\}\times D^2\times S^2,
$$
where 
$\varphi_i'\circ\varphi_i^{-1}$ is understood as a homeomorphism
of $\{ r\}\times S^1\times S^2$.
Since $\varphi_i'\circ\varphi_i^{-1}$ is isotopic to either the identity,
$f$, $g$ or $g\circ f$, it is now easy to verify that 
$B(p_i,2r)$ admits a compatible $S^2$-bundle structure
in either case.
\end{proof}


\section{Collapsing to two-spaces without boundary \\
             (torus fibre case)} 
\label{sec:dim2nobdytor}

 We consider the situation of the previous section that
a sequence of pointed complete $4$-dimensional orientable Riemannian
manifolds $(M_i^4,p_i)$ with $K \ge -1$ converges to a pointed
$2$-dimensional Alexandrov space $(X^2,p)$. 
Throughout this section, we assume that 
$p$ is an interior singular point of $X^2$.
\par

Now consider the local convergence $B(p_i,2r)\to B(p,2r)$
for a  sufficiently small positive number $r$.
By Fibration Theorem \ref{thm:orig-cap},
$A(p_i;r,2r)$ is homeomorphic to an $F_i$-bundle over 
$A(p;r,2r)\simeq S^1\times I$,
where $F_i$ is either $S^2$ or $T^2$.
In this section, we consider the case $F_i=T^2$.
First note 

\begin{equation}
 \partial B(p_i,r)\,\text{is homeomorphic to a $T^2$-bundle over $S^1$}
         \label{eq:s^1t^2}
\end{equation}

The purpose of this section is to prove the following

\begin{thm} \label{thm:locstrtor}
If $F_i=T^2$, then we have
 \begin{enumerate}
  \item  $B(p_i,r)$ is homeomorphic to $T^2\times D^2;$
  \item  for some $m\le 2\pi/L(\Sigma_p)$, 
         an $m$-fold cyclic cover $B'(p_i,r)$ of $B(p_i,r)$ 
         satisfies the following commutative diagram $:$
    \begin{equation*}
     \begin{CD}
       B'(p_i,r) @>{\simeq} >> T^2\times D^2 \\
          @V{\pi_i} VV      @VV{\pi} V    \\
       B(p_i,r)  @> {\simeq} >> (T^2\times D^2)/\Z_m
     \end{CD}
    \end{equation*}
        where the diagonal $\Z_m$-action on $T^2\times D^2$ 
        is free on $T^2$-factor and rotational on $D^2$-factor;
  \item  the Seifert $T^2$-bundle structure on $B(p_i,r)$ in $(2)$ 
        is compatible with the $T^2$-bundle structure on $A(p_i;r,2r)$.
 \end{enumerate}
\end{thm}

It should be noted that there is no restriction about the
$\Z_m$-action on the $T^2$-factor because the $T^2$-factor is 
collapsing to a point.
\par\medskip

Since $S_{\delta}(X^2)$ is discrete for any $\delta>0$ 
when $X^2$ has no boundary,
together with Fibration Theorem \ref{thm:orig-cap}, 
Theorem \ref{thm:locstrtor} yields 
Theorem \ref{thm:dim2nobdy} in the case when the general 
fibre is a torus.

Let the sequences $\delta_i\to 0$, $\hat p_i\to p$,
the pointed  complete noncompact Alexandrov space $(Y,y_0)$ with 
nonnegative curvature and 
its soul $S$ be as in the previous section.

\begin{prop} \label{prop:y4y3tor} 
Under the situation above, we have the following: 
\begin{enumerate}
 \item if $\dim Y = 4$, then  $S$ is isometric to a flat torus$;$
 \item if $\dim Y=3$,  then   
  \begin{enumerate}
   \item $Y$ has no boundary$;$
   \item $S$ is a circle.
  \end{enumerate}
\end{enumerate}
\end{prop}

Intuitively, $\dim Y=4$ (resp. $\dim Y=3$) occurs 
when a (singular) torus fibre converges to a torus (resp. a circle)
under the rescaling $\frac{1}{\delta_i} M_i^4$.
We can think of $S$ as the limit of  the ``singular torus fibre'',
which is invisible yet.

For the proof of Proposition \ref{prop:y4y3tor}, 
first we need

\begin{lem} \label{lem:d2bdle}
 \begin{enumerate}
  \item If $\dim Y=4$, then $S$ is isometric to a flat torus 
        or a flat Klein bottle$;$
  \item If $\dim Y=3$, then $Y$ has no boundary and $S$ is either 
        a circle or a point with $C(0)=I;$
  \item $B(p_i,r)$ is homeomorphic to a $D^2$-bundle over $T^2$ or
        $K^2$.
 \end{enumerate}
\end{lem}


\begin{proof}
First suppose $\dim Y=4$. If $\dim S\le 1$, Corollary 
\ref{cor:soul} implies  
\begin{equation*}
  B(p_i,r) \simeq 
         \begin{cases} \quad   D^4 \qquad &\mbox{if \, $\dim S=0$} \\
                     S^1\times D^3 \qquad &\mbox{if \, $\dim S=1$},
         \end{cases}
\end{equation*}
which contradicts \eqref{eq:s^1t^2}. 
Thus  $\dim S=2$, and hence $S$ is either homeomorphic to one of 
$S^2$ and $P^2$ or isometric to a flat torus or a flat Klein bottle.
By Generalized Soul Theorem \ref{thm:soul},  
$B(p_i,r)$ is a $D^2$-bundle over $S$.
Hence if $S$ was homeomorphic to either $S^2$ or $P^2$, 
we would have a contradiction to \eqref{eq:s^1t^2}.
Therefore $S$ is isometric to a flat torus or a flat Klein bottle
and we obtain the conclusion.

Next suppose that $\dim Y=3$. In view of \eqref{eq:dimyinfty}
and \eqref{eq:dims}, Theorem \ref{thm:dim3class} shows 
that $Y$ has no boundary and (2),(3) hold. 
\end{proof}

We take a sequence $\delta_i\ll \mu_i\to 0$ such that 
$(\frac{1}{\mu_i} B(p_i,r),p_i)$ converges to $(K_p,o_p)$.
Let $\Gamma_i$ be the deck transformation group of the universal 
covering $\tilde B(p_i,r) \to B(p_i,r)$, $\tilde p_i\in \tilde B(p_i,r)$ 
a point over $p_i$.
Passing to a subsequence, we assume that 
$(\frac{1}{\mu_i} \tilde B(p_i,r),\tilde p_i, \Gamma_i)$ converges to
a triplet $(Z,z_0,G)$ with respect to the pointed equivariant 
Gromov-Hausdorff convergence. 
Note that  $Z/G$ is isometric to $K_p$ and 
$G$ contains an abelian subgroup 
of index $\le 2$ (Lemma \ref{lem:d2bdle}).

\begin{prop} \label{prop:dim2univ} 
Under the situation above, we have 
 \begin{enumerate}
  \item $Z$ is isometric to a product $\R^2\times N^2$ preserved
        by the $G$-action, where $N^2$ is a flat cone$;$
  \item $G$ is abelian and isomorphic to $\R^2\oplus \Z_m$ 
        for some integer $m\le 2\pi/L(\Sigma_p);$
  \item $B(p_i,r)$ is a $D^2$-bundle over $T^2$.
 \end{enumerate}
\end{prop}

\begin{proof} 
Let $\Gamma_i^*$ be the abelian subgroup of $\Gamma_i$ of index 
$\le 2$, and $G^*$ the limit of $\Gamma_i^*$ under the convergence 
$(\frac{1}{\mu_i} \tilde B(p_i,r),\tilde p_i, \Gamma_i)\to
(Z,z_0,G)$, which is an abelian subgroup of $G$ of index 
$\le 2$.

\begin{clm} \label{clm:product}
 \begin{enumerate}
  \item The identity component $G_0^*$  of $G^*$ is 
        isomorphic to $\R^2;$
  \item $Z$ is isometric to a $G_0^*$-invariant product 
        $\R^2\times N^2$ such that $p_2(G_0^*)$ is compact,
        where $p_2:G_0^*\to\isom(N^2)$ is the projection.
 \end{enumerate}
\end{clm}

\begin{proof}
Consider the norm on $\Gamma_i^*$ defined by
$\|\gamma\| := d(\gamma \tilde p_i, \tilde p_i)/\mu_i$.
From the proof of Lemma \ref{lem:d2bdle} $(3)$ together with
$1/\delta_i\gg 1/\mu_i$, we can find generators $\gamma_i$,
$\tau_i$ of $\Gamma_i^*$ such that 
$\|\gamma_i\|\to 0$ and $\|\tau_i\|\to 0$ 
as $i\to\infty$.
Denoting by $H_i$ the subgroup of $\Gamma_i^*$ generated by $\gamma_i$,
take a minimal element $\sigma_i\in \tau_i + H_i$ with 
$\|\sigma_i\|=\min \{\,\|\gamma\|\,|\,
\gamma\in \tau_i+H_i\,\}$. 
Since $\Gamma_i^*$ is properly discontinuous, 
$\sup \{ \, d(H_i\tilde p_i, \sigma_i^n\tilde p_i)\,|\,n\in\Z\,\}
= \infty$.
Therefore for any large $R>0$ there exists an $n_i$ such that 
\begin{equation}
  R< d(H_i\tilde p_i,\sigma_i^{n_i}\tilde p_i) < R+1.
                                      \label{eq:far}
\end{equation}
Take $g_i\in H_i$ such that 
$d(g_i\tilde p_i,\sigma_i^{n_i}\tilde p_i)=
d(H_i\tilde p_i,\sigma_i^{n_i}\tilde p_i)$.
Since $\Gamma_i^*$ is abelian, it follows from homogeneity
that 
$(\frac{1}{\mu_i} \tilde B(p_i,r),g_i\tilde p_i, \Gamma_i^*)$ 
also converges to $(Z,z_0,G^*)$. 
Let $H$ and $K$ be the limit of $H_i$ and the group generated by 
$\sigma_i$ respectively,
and let $I_H$ and $I_K$ denote the minimal totally convex subsets
containing $Hz_0$ and $Kz_0$ respectively.
Since $H$ and $K$ are noncompact, both $I_H$ and $I_K$ contains lines
and the splitting theorem implies that 
$I_H$ and $I_K$ isometrically splits as 
$I_H=E_1\times\R$, $I_K=E_2\times\R$, where $E_j$ are compact
(see \cite{CG:str}).
Therefore we have an $H$-invariant line $\ell_1$ in $I_H$ and
a $K$-invariant line $\ell_2$ in $I_K$. 
\eqref{eq:far} implies that $\ell_1$ is not parallel to $\ell_2$.
It follows that 
$Z$ is isometric to a product $\R^2\times N$.

Obviously  $G_0^*$ preserves  $\R^2\times N$ and 
is isomorphic to the vector group $\R^2$. In particular, 
$\dim N=2$.
If $\isom(N^2)$ is compact, clearly $p_2(G_0^*)$ is compact.
If $\isom(N^2)$ is noncompact, then $Z=\R^4$.
Let $\Omega$ be the set of minimal 
displacement of $G_0^*$:
$$
\Omega=\{ x\in\R^4\,|\,\delta_{\gamma}(x)=\min\delta_{\gamma},
    \,\text{for all}\, \gamma\in G_0^*\},
$$
where $\delta_{\gamma}(x)=d(\gamma x,x)$.
Then 
$\Omega$ is isometric to $\R^2$ and we have the decomposition 
$Z=\Omega\times\R^2$ satisfying the required properties.
\end{proof}

\begin{clm} \label{clm:inf-dim}
 $\dim Z(\infty)=\dim Z-1=3$. In particular, $N^2\simeq\R^2$.
\end{clm}

\begin{proof}
Choose a sequence $R_i\to\infty$ such that 
the pointed Gromov-Hausdorff distance between 
$(B(\tilde p_i,R_i;\frac{1}{\mu_i}\tilde B(p_i,r)),\tilde p_i)$ and
$(B(z_0, R_i), z_0)$ is less than $1/R_i$.
Take $\epsilon_i\to 0$ with $\lim \epsilon_iR_i = \infty$ and 
$\lim \epsilon_i/\mu_i =\infty$.
By $d_{p.GH}((\epsilon_i Z,z_0), (K(Z(\infty)), o_{\infty}))\to 0$,
$(\frac{\epsilon_i}{\mu_i}\tilde B(p_i,r),\tilde p_i, \Gamma_i)$
converges to a triplet $(K(Z(\infty)),o_{\infty},G_{\infty})$,
where $o_{\infty}$ is the vertex of the cone $K(Z(\infty))$.
In a way similar to Claim \ref{clm:product},
we have 
$\dim G_{\infty}=2$ and $\dim K(Z(\infty))=4$.
\end{proof}

\begin{lem} \label{lem:screw}
Let $W^n=\R^k\times N^{n-k}$ be an $n$-dimensional complete Alexandrov 
space with nonnegative curvature and with 
$\dim N(\infty)=n-k-1$.
Let $\Gamma$ be a group of isometries of $W$ isomorphic to $\R^k$ 
preserving the splitting $W^n=\R^k\times N^{n-k}$ such that 
$p_1(\Gamma)\simeq\R^k$ and $p_2(\Gamma)$ is compact,
where $p_1: \Gamma\to\isom(\R^k)$ and $p_2:\Gamma\to\isom(N)$ are
the projections. Then 
$$
   \dim \left(\frac{W}{\Gamma}\right)(\infty)
         = \dim N^{n-k}(\infty) - \dim p_2(\Gamma).
$$ 
\end{lem}

\begin{proof}
From the assumption, $p_2(\Gamma)$ is isomorphic to a torus,
say $T^m$. Let $(\R\times K(N(\infty)),o,\Gamma_{\infty})$ be 
any limit of $(\epsilon W,w_0,\Gamma)$ as $\epsilon\to 0$.
It is not difficult see that 
$\Gamma_{\infty}\simeq\R^k\times T^m$.
Therefore
$$
 \left(\frac{W}{\Gamma}\right)(\infty)
     =  \left(\frac{K(N(\infty))}{T^m}\right)(\infty).
$$
From the assumption, there is a $p_2(\Gamma)$-invariant
point (soul) of $N$, say $p_0$.
Since we have an expanding map 
$$
 \frac{N^{n-k}(\infty)}{T^m} \to \frac{\Sigma_{p_0}(N^{n-k})}{T^m},
$$
it follows from $\dim N(\infty)=n-k-1$ 
that $\dim(N^{n-k}(\infty)/T^m)=n-k-m-1$.
Thus, the $T^m$-action on $N(\infty)$ is effective,
and we have the required equality.
\end{proof}

Since $\dim(Z/G_0^*)(\infty)=1$, it follows from 
Claims \ref{clm:product}, \ref{clm:inf-dim} and
Lemma \ref{lem:screw} that 
$p_2(G_0^*)=\{ 1\}$, where $p_2:G_0^*\to\isom(N^2)$ is the projection.
It follows that
$G^*$ is isomorphic to the direct sum $\R^2\oplus (G^*/G^*_0)$ 
which preserves the splitting $Z=\R^2\times N^2$.
Since $N^2/(G^*/G_0^*)$ is a flat cone,  so is $N^2$, and 
$G^*/G_0^*$ must be a finite cyclic group $\Z_m$
for some integer $m$. 

Next we show (3), or equivalently $\Gamma_i=\Gamma_i^*$.
Suppose that (3) does not hold.
Since the nontrivial deck transformation of the covering
$T^2\to K^2$ reverses the orientation of $T^2$
and since any element of $G$ is orientation-preserving,
any nontrivial element $g\in G-G^*$ 
has the property that both $p_1(g)\in\isom(\R^2)$  and 
$p_2(g)\in\isom(N^2)$ are orientation-reversing.
Therefore $Z/G$ must have nonempty boundary, a contradiction.

Let $x_0\in N^2$ be a fixed point of
$G/G_0$-action. Since $L(\Sigma_{x_0})/\#(G/G_0)$ $=L(\Sigma_p)$,
it follows that $m\le 2\pi/L(\Sigma_p)$.
This completes the proof of Proposition \ref{prop:dim2univ}.
\end{proof}

From Lemma \ref{lem:d2bdle} and Proposition \ref{prop:dim2univ}, 
the proof of Proposition \ref{prop:y4y3tor} is now complete.
\par\medskip
Let $\check p_i\in\tilde B(p_i,r)$ be a point over 
$\hat p_i\in B(p_i,r)$.
Passing to a subsequence, we may assume that 
$(\frac{1}{\delta_i} \tilde B(p_i,r), \check p_i, \Gamma_i)$ converges to
a triplet $(W,w_0,\Gamma)$.
Note that $\Gamma$ is abelian and 
$W/\Gamma=Y$.

\begin{prop} \label{prop:dim2univ2} 
Under the situation above, 
$W$ is isometric to a product $\R^2\times L^2$ preserved by
the $\Gamma$-action with $L^2\simeq \R^2$ satisfying the following:
 \begin{enumerate}
  \item Let $q_1:\Gamma\to\isom(\R^2)$ and $q_2:\Gamma\to\isom(L^2)$ be the 
        projections. Then $q_1(\Gamma)$ uniformly acts as 
        translation and $q_2(\Gamma) \simeq\Z_n$ 
        for some positive integer $n$.\par 
        Consequently, $\Gamma$ is isomorphic to one of $\Z^2$, 
        $\R\oplus\Z$ or $\R\oplus\Z\oplus\Z_n;$
  \item Letting $p_0\in L$ be a $q_2(\Gamma)$-invariant point,
        we have 
   \begin{equation*}
     L(\Sigma_p)\le\frac{L(\Sigma_{p_0}(L^2))}{n}\le\frac{2\pi}{n}.
          \label{eq:leng-ineq}
  \end{equation*}
 \end{enumerate}
\end{prop}

\begin{proof}
Since $\delta_i\ll \mu_i$, the Alexandrov convexity in nonnegative
curvature yields an expanding map 
$(Z,z_0)\to (W,w_0)$, where $Z$ is as in Proposition
\ref{prop:dim2univ} (compare Lemma \ref{lem:expanding2}).
In particular, $\dim W=4$. 
From the noncompactness of $\Gamma$, we have a splitting 
$W=\R\times Q^3$.

We show the existence of $\Gamma$-invariant splitting
$W=\R^2\times L^2$.
First assume $\dim Y=4$. From Proposition \ref{prop:y4y3tor},
$\Gamma\simeq\Gamma_i\simeq\Z^2$ and we have 
a splitting $W=\R^2\times L^2$ preserved by
the $\Gamma$-action.

Next assume that $\dim Y=3$, yielding $\dim \Gamma=1$. 
Since any limit group of a nontrivial subgroup of $\Gamma_i$
under the convergence $\Gamma_i\to\Gamma$ is noncompact, 
we have $\Gamma_0\simeq \R^1$.
By Proposition \ref{prop:y4y3tor},
the soul of $Y$ is a circle.
It follows that $\Gamma/\Gamma_0$ is infinite.

Now we show that $W$ isometrically splits as $W=\R^2\times L^2$.
Otherwise, the group of isometries of $Q^3$ is compact,
and $\Gamma$ preserves the splitting $\R\times Q^3$.
Obviously $q_1(\Gamma)=\R$, 
where $q_1:\Gamma\to\isom(\R)$ is the projection.
It turns out that $\ker q_1$ is a infinite discrete group in the 
compact group $\isom(Q^3)$. Since $\Gamma$ is closed in 
$\isom(W)$, it is a contradiction.

Next we show that there exists a $\Gamma$-invariant 
splitting $W=\R^2\times L^2$ such that 
$q_1(\Gamma)\simeq\R\oplus\Z$ and $q_2(\Gamma)$ is 
compact.
If $\isom(L^2)$ is compact, it is clear.
If $\isom(L^2)$ is noncompact, then $W=\R^4$.
Let $\Gamma_1$ be a subgroup of $\Gamma$ isomorphic to 
$\R\oplus\Z$, and let $\Omega$ be the set of minimal 
displacement of $\Gamma_1$. Then 
$\Omega=\R^2$ and we have the decomposition 
$W=\Omega\times\R^2$ satisfying the required properties.

We show $(1)$ and $(2)$.
We put $\Lambda:=q_2(\Gamma)$ for simplicity.
Let 
$(K(W(\infty)), o, \Gamma_{\infty})$ and 
$(K(L^2(\infty)),o,\Lambda_{\infty})$
be any limits of  
$(\epsilon W, w_0, \Gamma)$ and 
$(\epsilon L^2, u_0,\Lambda)$ 
as $\epsilon\to 0$ respectively,
where $K(W(\infty))=\R^2\times K(L^2(\infty))$.
Note that $q_1(\Gamma_{\infty})=\R^2$.
It follows that 
$$
K(Y(\infty)) = \frac{\R^2\times K(L^2(\infty))}{\Gamma_{\infty}}
             = \frac{K(L^2(\infty))}{\Lambda_{\infty}}.
$$
Recall that we have an expanding map
$$
 K(\Sigma_p) \to K(Y(\infty))=\frac{K(L^2(\infty))}{\Lambda_{\infty}}.
$$
It follows that $\Lambda_{\infty}$ is finite and therefore
$\Lambda\simeq\Lambda_{\infty}$.
Taking a $\Lambda$-invariant point $p_0\in L$,
we have an injective homomorphism 
$\rho:\Lambda\to SO(2)$.
Note that we also have an expanding map
$$
 \frac{K(L^2(\infty))}{\Lambda_{\infty}} 
                   \to \frac{K_{p_0}(L^2)}{\Lambda},
$$
which concludes that 
\begin{equation*}
   L(\Sigma_p)\le\frac{L(\Sigma_{p_0}(L^2))}{n}\le\frac{2\pi}{n}.
\end{equation*}

Finally remark that 
if $\Gamma$ has no torsion,  $\Gamma$ is isomorphic to either $\Z^2$
or $\R\oplus\Z$, and if $\Gamma$ has a torsion,
$\Gamma$ is isomorphic to $\R\oplus\Z\oplus\Z_n$.
\end{proof}

\begin{lem} \label{lem:finitecover}
A finite covering space of $B(p_i,r)$ is homeomorphic to 
$T^2\times D^2$.
\end{lem}

\begin{proof}
Put $\Gamma':=\ker(q_2)$ and take the subgroup $\Gamma_i'\subset \Gamma_i$
such that
\begin{enumerate}
 \item  $\Gamma_i'$ converges to $\Gamma'$ under the convergence
        $(\frac{1}{\delta_i}\tilde B(p_i,r),\check p_i, \Gamma_i)\to
        (W, w_0,\Gamma)$;
 \item  $\Lambda_i:= \Gamma_i/\Gamma_i'\simeq 
            \Gamma/\Gamma'\simeq \Lambda\simeq\Z_n$.
\end{enumerate}
We consider the quotient $B'(p_i,r):=\tilde B(p_i,r)/\Gamma_i'$.
Let $\breve p_i\in B'(p_i,r)$ be a point over $\hat p_i$.
Lifting the $d_{\hat p_i}$-gradient flow on 
$B(p_i,r)-B(\hat p_i, R\delta_i)$
to that on $B'(p_i,r)-B(\breve p_i, R\delta_i)$ 
combined with $d_{\breve p_i}$-gradient flow,
we obtain
$$
      B'(p_i,r)\simeq B(\breve p_i, R\delta_i)
$$
(see Theorem \ref{thm:rescal}).

First assume $\dim Y=4$, and note that 
$(\frac{1}{\delta_i} B'(p_i,r),\breve p_i, \Lambda_i)$ 
converges to $(T^2\times L^2, w_o', \Lambda)$.
By Stability Theorem \ref{thm:stability},
$B'(p_i,r)\simeq B(\breve p_i, R\delta_i)$ is homeomorphic to
$T^2\times D^2$ for a sufficiently large $R$.

Next assume $\dim Y=3$.
Then $(\frac{1}{\delta_i} B'(p_i,r),\breve p_i, \Lambda_i)$ converges to 
$(\R^2/\Gamma'\times L^2, w_o', \Lambda)$, where 
$\R^2/\Gamma'\simeq S^1$. Recall that 
$$
   L(\Sigma_{p_0}(L^2))/n\ge L(\Sigma_p(X^2)),
$$
for the $\Lambda$-fixed point $p_0\in L^2$ (see 
Proposition \ref{prop:dim2univ2}). 
By  Theorem \ref{thm:patch},
$B'(p_i,r)$ is homeomorphic to either an 
$S^1$-bundle or a Seifert $S^1$-bundle over $S^1\times D^2$.
Thus $B'(p_i,r)$ is homeomorphic to either $T^2\times D^2$ or 
a $S_i(L^2)$-bundle over $S^1$, denoted $S^1\tilde\times S_i(L^2)$, 
where $S_i(L^2)$ is a Seifert $S^1$-bundle over $L^2$ 
with exactly one singular orbit. 
Suppose that $B'(p_i,r)$ is not homeomorphic to $T^2\times D^2$,
and consider 
$I\times S_i(L^2)\subset S^1\tilde\times S_i(L^2)=B'(p_i,r)$,
where $I\subset S^1$ is a closed interval.
For a point $q_i\in S_i(L^2)$ converging to 
the singular point locus $\in L^2$ 
and for any fixed $t_0$ in the interior of 
$I$, put $\bar p_i :=(t_0,q_i)\in I\times S_i(L^2)\subset B'(p_i,r)$.
Take $\nu_i\to 0$ such that 
$(\frac{1}{\nu_i} (I\times S_i(L^2)),\bar p_i)$ converges to
$(\R\times K(\Sigma_{p_0}(L^2)), o)$,
and let $\widetilde S_i(L^2)$ denotes the universal cover of 
$S_i(L^2)$ with the deck transformation group
$H_i\subset \Gamma_i$.
We may assume that 
$(\frac{1}{\nu_i} (I\times \widetilde S_i(L^2)),\bar p_i, H_i)$ 
converges to a triplet $(\R\times \R \times L_1^2, o, H)$.
From our assumption on the existence of singular orbit
in $S_i(L^2)$, it follows that  $H\simeq\R\oplus \Z_{n_1}$ for some 
$n_1 > 1$. Note that $n_1$ is related with the Seifert invariants 
of the singular orbit of $S_i(L^2)$ (see Proposition 4.3 in
\cite{SY:3mfd}).
 Take $H_i'\subset H_i$ such that 
\begin{enumerate}
 \item $(\frac{1}{\nu_i} (I\times\widetilde S_i(L^2)),\bar p_i,
       H_i')$ converges to $(\R\times\R\times L_1^2, o,H_0) ;$
 \item $H_i/H_i'\simeq \Z_{n_1}$.
\end{enumerate}
Note that $K(\Sigma_{p_1}(L^2_1))/\Z_{n_1}=K(\Sigma_{p_0}(L^2))$
for a $\Z_{n_1}$-invariant point $p_1\in L^2_1$. Therefore 
$$
     L(\Sigma_{p_1}(L^2_1)) \ge n_1n L(\Sigma_p(X^2)).
$$
Repeating this procedure finitely many times, 
we obtain  finite sequences of complete nonnegatively curved
surfaces $L^2$, $L^2_1$, $\ldots$, $L^2_k$, Seifert bundles
$S_i(L^2)$, $S_i(L^2_1)$, $\ldots$, $S_i(L^2_k)$
and groups 
$H_i\supset H_i^{(1)}\supset\cdots\supset H_i^{(k)}$
with  $H_i^{(j)}=\pi_1(S_i(L^2_j))$ such that the last 
$S_i(L^2_k)$ contains no singular orbits.
Letting $\rho_i:\Gamma_i'\to\pi_1(S_i(L^2))\oplus\Z$ be an 
isomorphism, 
put $\Gamma_i'':= \rho_i^{-1}(H_i^{(k)}\oplus\Z)$.
Then $\tilde B(p_i,r)/\Gamma_i''$ is homeomorphic to 
an $S^1$-bundle over $S^1\times D^2$, which is homeomorphic 
to $T^2\times D^2$.
\end{proof}

\begin{proof}[Proof of Theorem \ref{thm:locstrtor}]
Since $\partial B(p_i,r)$ is orientable, it follows from
Lemma \ref{lem:finitecover}
that $\partial B(p_i,r)\simeq T^3$
(see \cite{Sc:geometry}), and therefore $B(p_i,r)$
has Euler number $0$ as a $D^2$-bundle over $T^2$ and 
is homeomorphic to $T^2\times D^2$.

We put a compatible Seifert $T^2$-fibre structure on 
$B(p_i,r)$ as follows:
We again go back to the convergence 
$(\frac{1}{\mu_i}\tilde B(p_i,r), \tilde p_i, \Gamma_i) \to
(\R^2\times N^2,z_0,G)$ in Proposition \ref{prop:dim2univ},
and consider $B'(p_i,r)=\tilde B(p_i,r)/\Gamma_i'$, where 
$\Gamma_i'$ is a subgroup of $\Gamma_i$  converging to 
$G_0$ under the above convergence. 
Let $A'(p_i;r,2r):=\pi_i^{-1}(A(p_i;r,2r))$, where
$\pi_i:B'(p_i,2r)\to B(p_i,2r)$ is the covering projection.
Since 
$(\frac{1}{\mu_i}B'(p_i,r), p_i', \Gamma_i/\Gamma_i')$ converges to 
$(N^2, \bar z_0,G/G_0)$, by Equivariant Fibration Theorem \ref{thm:cap},
we have a $\Z_m$-equivariant $T^2$-bundle 
$A'(p_i;r,2r)\to A(\bar z_0;r,2r)$
and hence a $\Z_m$-equivariant homeomorphism 
$A'(p_i;r,2r)\simeq T^2\times A(\bar z_0;r,2r)$ where
$\Z_m=\Gamma_i/\Gamma_i'$ acts diagonally on 
$T^2\times A(\bar z_0;r,2r)$; freely on the $T^2$-factor
and rotationally on $A(\bar z_0;r,2r)$.
Therefore
$A(p_i;r,2r)$ is homeomorphic to the diagonal quotient
$(T^2\times A(\bar z_0;r,2r))/\Z_m$.
Finally we fill $B(p_i,r)$ with an obvious gluing by the $T^2$-fibred
$T^2\times D^2\simeq (T^2\times D^2(r))/\Z_m$, of the same type
of $\Z_m$-quotient
as $(T^2\times A(\bar z_0;r,2r))/\Z_m$.
We therefore obtain $B(p_i,2r)\simeq (T^2\times D^2)/\Z_m$.
This completes the proof of Theorem \ref{thm:locstrtor}.
\end{proof}


\section{Collapsing to two-spaces with boundary } 
\label{sec:dim2wbdy}

In this section, 
we prove Theorem \ref{thm:dim2wbdy}.
First we define the fibre space $\cal F(X)$ stated there.

Let $X$ be a compact $2$-dimensional topological manifold with 
boundary, and let $F$ denote $S^2$ or $T^2$.
Let $\cal F_{\rm int}(X)$ denote either an $S^2$-bundle over $X$ 
(if $F=S^2$) or a Seifert $T^2$-bundle over $X$ (if $F=T^2$) 
which can be thought of as a main building block of 
$\cal F(X)$.
The building blocks of the complement of $\cal F_{\rm int}(X)$ are
constructed as follows.
Consider the following families of compact orientable $3$- 
or $4$-manifolds with boundary: 
\begin{enumerate}
\item The family $\cal A_s$ consists of $D^3$ and a twisted 
      $I$-bundle $P^2\tilde\times I$ over the projective plane $P^2$;
\item The family $\cal A_t$ consists of $S^1\times D^2$ and a twisted 
       $I$-bundle $K^2\tilde\times I$ over the Klein bottle $K^2$;
\item The family $\cal B_s$ consists of $D^4$ and $D^2$-bundles over 
       $S^2$ with Euler numbers $\pm 1$, $\pm 2$, 
       and of a $D^2$-bundle $P^2\tilde\times_0 D^2$ over $P^2$
       with Euler number $0;$
\item The family $\cal B_t$ consists of $D^4$, $S^1\times D^3$ 
      and  $D^2$-bundles over $S^2$, $P^2$ or $K^2$.
\end{enumerate}
Note that $P^2\tilde\times_0 D^2$ can be characterized as the 
$D^2$-bundle over $P^2$ with boundary homeomorphic to 
$P^3\# P^3$.

Let successive points (possibly empty) 
$q_0,\ldots, q_{k-1}$ of $C$  be associated with
each component $C$ of $\partial X$. 
Let  $r$ represent $s$ (resp. $t$ ) if $F=S^2$ (resp. if $F=T^2$).
Suppose that 
an element $Q_{\alpha-1, \alpha}\in \cal A_r$, called a {\it section},
is associated with each edge $\widehat{q_{\alpha-1}q_{\alpha}}$
of $C$,  and that 
an element $R_{\alpha}\in\cal B_r$, called a {\it connecting part},
the part connecting $Q_{\alpha-1, \alpha}$ and $Q_{\alpha, \alpha+1}$,
is associated with each point $q_{\alpha}$, 
so as to satisfy the following:
\begin{enumerate}
 \item $\partial R_{\alpha}$ is a gluing 
   of $Q_{\alpha-1, \alpha}$ and $Q_{\alpha, \alpha+1}$
   along their boundaries;
 \item  Let $I=[0,1]$ and let $\cal F_{\rm cap}(C)$ be an identification 
   space of $R_{\alpha}$,
   $Q_{\alpha-1,\alpha}\times I$,
   $0\le\alpha\le k-1$\,\,$({\rm mod}\, k)$, 
   where $Q_{\alpha-1,\alpha}\times 1$ and 
   $Q_{\alpha,\alpha+1}\times 0$ are glued with 
   $Q_{\alpha-1,\alpha}\subset\partial R_{\alpha}$ and
   $Q_{\alpha,\alpha+1}\subset\partial R_{\alpha}$ respectively$;$
 \item  Note that $\partial \cal F_{\rm cap}(C)$ has an $F$-bundle 
   structure over $C$. 
   Then $\partial \cal F_{\rm cap}(C)$ is required to be fibre-wise 
   homeomorphic to the 
   component of $\partial \cal F_{\rm int}(X)$ corresponding to 
   $C$. 
\end{enumerate}
Letting $\cal F_{\rm cap}(\partial X)$ denote
the disjoint union $\displaystyle{\amalg_{C}\cal F_{\rm cap}(C)}$,
we can glue $\cal F_{\rm int}(X)$ and $\cal F_{\rm cap}(\partial X)$ 
along their boundary fibres, which is denoted by 
$$
 \cal F(X):= \cal F_{\rm int}(X)\bigcup \cal F_{\rm cap}(\partial X).
$$
The set of points $\{ q_{\alpha}\}$ are called the {\it break points}
of $\partial X$ associated with the fibre space $\cal F(X)$.
Figure 4 illustrates the decomposition in $\cal F(X)$. 
\par

From the construction,
$\cal F(X)$ has a singular fibre structure over $X$.
More explicitly, there is a continuous surjective map
$f:\cal F(X)\to X$ such that $f$ restricted to $\interior X$
is either an $S^2$-bundle or a Seifert $T^2$-bundle and 
$f$ restricted to $\partial X$ is a singular fibration 
whose fibres are ones of a point, $S^1$, $S^2$, $P^2$ and $K^2$
determined by the topological types of sections and connecting parts
involved in $\cal F_{\rm cap}(\partial X)$
(see the figures (1) $\sim$ (10) after Lemma \ref{lem:deform-cap}).
 \par

\vspace{0.4cm}
\begin{center}
\begin{tikzpicture}
[scale = 0.5]
\draw [-,thick] (-4,0) -- (4.5,0);
\draw [-,thick] (-6,-2.5) -- (4.5,-2.5);
\draw [-, thick] (-6,-2) to [out=90, in=180] (-4,0);
\fill (-5,-1.5) circle (0pt) node [right]{$R_\alpha$};
\fill (-0.3, -1.3) circle (0pt) node [right]{$Q_{\alpha-1,\alpha}\times I$};
\draw [-,thick] (-1.7,0) -- (-1.7,-6.5);
\draw [-,thick] (-6,-2) -- (-6,-6.5);
\fill (-6.1,-5) circle (0pt) node [right]{$Q_{\alpha,\alpha+1}\times I$};
\fill (-0.3,-5) circle (0pt) node [right]{$\mathcal F_{\rm int}(X)$};
\fill (-2, -7.5)  circle (0pt) node [below] {Figure 3};

\end{tikzpicture}
\end{center}


Now Theorem \ref{thm:dim2wbdy} is reformulated in more detail
as follows.

\begin{thm} \label{thm:dim2wbdy-break}
Suppose that a sequence of $4$-dimensional closed orientable 
Riemannian manifolds $M_i^4$ 
collapses to a two-dimensional compact Alexandrov
space $X$ with boundary under $K\ge -1$. 
Then $M_i^4$ is homeomorphic to a fibre space 
$\cal F(X)$ defined above such that the break points of
$\partial X$ are contained in 
${\rm Ext}(X)$.
\end{thm}

To prove Theorem \ref{thm:dim2wbdy-break}, we need to understand 
the topology of a small metric ball in a collapsed $4$-manifold 
near a boundary point in the limit space. For this reason,
we again concentrate on a local problem, and 
consider the following situation that
a sequence of pointed complete $4$-dimensional orientable Riemannian
manifolds $(M_i^4,p_i)$ with $K \ge -1$ converge to a pointed
$2$-dimensional Alexandrov space $(X^2,p)$. 
Throughout the rest of this section, we assume that 
$p$ is a boundary point of $X^2$.
\par

We consider the local convergence $B(p_i,r)\to B(p,r)$
for a sufficiently small positive number $r$.
By Theorem \ref{thm:rescal}, we have sequences $\delta_i\to 0$ and 
$\hat p_i\to p$ such that
\begin{itemize}
  \item  for any limit $(Y,y_0)$ of $(\frac{1}{\delta_i}M_i, \hat p_i)$, 
     we have  $\dim Y\ge 3$;
  \item  $B(p_i,r)$ is homeomorphic to $B(\hat p_i,R\delta_i)$ for every 
      $R\ge 1$ and large $i$ compared to $R$.
\end{itemize}

In the sequel, we shall study the topology of $B(p_i,r)$
to  prove Theorem \ref{thm:dim2wbdy}.
For a fixed small number $r>0$, let $\nu$ be a sufficiently 
small positive number with $\nu/(\tan\theta_0/2) < r/2$, 
where $\theta_0=L(\Sigma_p) (\le\pi)$.
For arbitrary fixed $r_0\in (\nu/(\tan\theta_0/2), r/2)$,
applying Fibration-Capping Theorem \ref{thm:orig-cap}
to $B(p,r)\cap X_{\nu}$ and $A(p;r_0,r)\cap X_{\nu}$, we have 
decompositions 
\begin{gather}
B(p_i,r)=B_{\rm int}(p_i,r)\cup B_{\rm cap}(p_i,r),
           \label{eq:decomp-1}\\ 
A(p_i;r_0,r)=A_{\rm int}(p_i;r_0,r)\cup A_{\rm cap}(p_i;r_0,r),
            \label{eq:decomp-2}
\end{gather}
together with compatible fibre bundle maps
\begin{gather}
f_{i,{\rm int}}:B_{\rm int}(p_i,r)\to B(p,r)\cap X_{\nu},
                    \label{eq:decomp-3}\\
f_{i,{\rm cap}}: A_{\rm cap}(p_i;r_0,r)\to 
A(p;r_0,r)\cap\partial X_{\nu}.\label{eq:decomp-4}
\end{gather}
Let us denote by $F_i$ the fibre, either $S^2$ or $T^2$,
of $f_{i,{\rm int}}$.
Note that $f_{i,{\rm cap}}$ has two types of fibres,
one of which is denoted $F_{i,{\rm cap}}$.

\begin{thm} \label{thm:dim2wbdy-text}
Under the situation above the following holds$:$
 \begin{enumerate}
  \item  If $F_i=S^2$, then 
    \begin{enumerate}
     \item $F_{i,{\rm cap}}$ is homeomorphic to either
           $D^3$ or the twisted product $P^2\tilde\times I ;$
     \item $B(p_i,r)$ is homeomorphic to either 
           $D^4$, $S^2\tilde\times_{\omega} D^2$ 
           with $|\omega|\in\{1, 2\}$, 
           or $P^2\tilde\times_0 D^2 ;$
     \item if $B(p_i,r)$ is homeomorphic to $P^2\tilde\times_0 D^2$, then
      \begin{enumerate}
       \item the universal cover $\tilde B(p_i,r)$ of $B(p_i,r)$ 
             satisfies the following commutative diagram $:$
           \begin{equation*}
            \begin{CD}
              \tilde B(p_i,r) @>{\simeq} >> S^2\times D^2 \\
                 @V{\pi_i} VV        @VV{\pi} V    \\
              B(p_i,r)  @> {\simeq} >> (S^2\times D^2)/\Z_2,
            \end{CD}
           \end{equation*}
             where the diagonal $\Z_2$-action is free on the $S^2$-factor 
             and by reflection on the $D^2$-factor$;$
       \item the singular $S^2$-bundle structure on 
             $B(p_i,r)$ in $(1)$-$(c)$-$\text{\rm (i)}$ is compatible with 
             the fibre structures  on $B_{\rm int}(p_i,r)$ and 
             $A_{\rm cap}(p_i;r_0,r)$ induced from 
             $f_{i,{\rm int}}$ and $f_{i,{\rm cap}}$ respectively.
       \end{enumerate}
    \end{enumerate}
  \item  If $F_i=T^2$, then 
    \begin{enumerate}
     \item $F_{i,{\rm cap}}$ is homeomorphic to either
           $S^1\times D^2$ or the twisted product $K^2\tilde\times I ;$
     \item $B(p_i,r)$ is homeomorphic to either
           $D^4$, $S^1\times D^3$, or a $D^2$-bundle over $S^2$,
           $P^2$ or $K^2 ;$
     \item if $B(p_i,r)$ is homeomorphic to a $D^2$-bundle
           over $K^2$, then 
      \begin{enumerate}
       \item the double cover of $B(p_i,r)$ is homeomorphic to
             $T^2\times D^2;$
       \item for some $m\le \pi/L(\Sigma_p)$, 
             an $2m$-fold dihedral cover $B'(p_i,r)$ of $B(p_i,r)$ 
             satisfies the following commutative diagram $:$
         \begin{equation*}
          \begin{CD}
             B'(p_i,r) @>{\simeq} >> T^2\times D^2 \\
                @V{\pi_i} VV      @VV{\pi} V    \\
             B(p_i,r)  @> {\simeq} >> (T^2\times D^2)/D_{2m},
          \end{CD}
         \end{equation*}
             where the diagonal dihedral $D_{2m}$-action is free 
             on the $T^2$-factor and the action on the $D^2$-factor
             is generated by a rotation and a reflection $;$
       \item the $($generalized$)$ Seifert $T^2$-bundle structure on 
             $B(p_i,r)$ in $(2)$-$(c)$-$(ii)$ is compatible with 
             the fibre structures  on $B_{\rm int}(p_i,r)$ and 
             $A_{\rm cap}(p_i;r_0,r)$ induced from 
             $f_{i,{\rm int}}$ and $f_{i,{\rm cap}}$ respectively.
   \end{enumerate}
   \end{enumerate}
  \item  There is a singular fibre structure on $B(p_i,r)$
         compatible with the fibre structures on $B_{\rm int}(p_i,r)$
         and $A_{\rm cap}(p_i;r_0,r)$
         induced from $f_{i,{\rm int}}$ and $f_{i,{\rm cap}}$ 
         respectively.
 \end{enumerate}
\end{thm}

First we investigate 

\begin{proclaim}{\emph{Case} A.}
   $p\in \partial X$ is a regular boundary point, namely,
$D(\Sigma_p(X))$ is isometric to $S^1(1)$.
\end{proclaim}

The argument below provides a parametrized version of collapsing.
Namely, when $(M_i^4,p_i)$ collapses to a product $(X_0\times\R, x_0)$,
then collapsing phenomenon can be described as a one-parameter
family of collapsing of $3$-manifolds.
The local version of this is already proved in Section \ref{sec:rescal}.

For sufficiently small $R\gg r\gg \nu >0$,
from \eqref{eq:decomp-1} $\sim$ \eqref{eq:decomp-4} 
we have a decomposition 
$B(p_i,r)=B_{\rm int}(p_i,r)\cup B_{\rm cap}(p_i,r)$
and a map $f_i:B(p_i,r)\to B(p,r)\cap X_{\nu}$ such that 
both $f_{i,{\rm int}}$ and $f_{i,{\rm cap}}$ are fibre bundles.
To prove Theorem \ref{thm:dim2wbdy-text} $(1)$-$(a)$ and $(2)$-$(a)$,
it suffices to show that 

\begin{prop} \label{prop:capfibre}
\begin{equation*}
 F_{i,{\rm cap}}\simeq
  \begin{cases}
   \quad D^3\,\,{\rm or}\,\, P^2\tilde\times I\hspace{1.45cm}
                                          {\rm if}\,\, F_i=S^2 \\
         S^1\times D^2\,\,{\rm or}\,\, K^2\tilde\times I\qquad 
            {\rm if}\,\, F_i=T^2.
  \end{cases}
\end{equation*}
\end{prop}

Take $q_i$ with $d(p_i, q_i)=R$ and $\varphi_i(q_i)\in\partial X$,
where $\varphi_i:(M_i,p_i)\to (X,p)$ is an $\epsilon_i$-approximation
with $\lim\epsilon_i=0$.
Since $d_{q_i}$-flow curves are transversal to $F_{i,{\rm cap}}$ and 
since $(d_{p_i}, d_{q_i})$ are regular on 
$\partial F_{i,{\rm cap}}$, 
$F_{i,{\rm cap}}$ is homeomorphic to 
$U_r(q_i,p_i):=\partial B(q_i,R)\cap B(p_i,r)$.
Now consider the convergence 
$(\frac{1}{\delta_i} M_i,\hat p_i)\to (Y,y_0)$.
Since $Y$ contains a line, 
$Y$ splits isometrically as $Y=Y_0\times \R$. 
We may think of a soul $S$ of $Y$ as a soul of $Y_0$.
Let $R_0\gg R_1\gg 1$ be sufficiently large, and 
take $z_0\in \{ y_0\}\times \R$ with $d(y_0,z_0)=R_0$
and $\hat q_i\in M_i$ with $d(\hat p_i,\hat q_i)=\delta_i R_0$ 
and $\hat q_i\to z_0$.
From critical point theory, we obtain 
$U_r(q_i,p_i)\simeq U_{\delta_i R_1}(\hat q_i,\hat p_i)$,
and therefore
\begin{align}
     & F_{i,{\rm cap}}\simeq 
         U_{\delta_i R_1}(\hat q_i,\hat p_i),  \\
     & \partial U_{\delta_i R_1}(\hat q_i,\hat p_i) \simeq F_i.
                 \label{eq:s^2}
\end{align}

First we consider 
\begin{proclaim}{\emph{Case} A-I.}
       $\dim Y=4$.
\end{proclaim}

By Theorem \ref{thm:stability-respect},
$U_{\delta_i R_1}(\hat q_i,\hat p_i)\simeq U_{R_1}(z_0,y_0)$.
Since $Y$ has nonnegative curvature, 
$U_{R_1}(z_0,y_0)$ is homeomorphic to 
$B(y_0, R;Y_0)$, by which the topology of $F_{i,{\rm cap}}$
is determined.

\begin{lem}
Suppose that $p\in \partial X$ is a regular boundary point and 
$\dim Y=4$.
\begin{enumerate}
 \item If $F_i=S^2$, then the soul $S$ is either a point or 
       homeomorphic to $P^2$, and Proposition \ref{prop:capfibre}
       holds;
 \item If $F_i=T^2$, then $S$ is isometric to either a circle or 
       a flat Klein bottle and Proposition \ref{prop:capfibre}
       holds.
\end{enumerate}
\end{lem}

\begin{proof}
If $\dim S=0$, then $B(y_0, R;Y_0)\simeq D^3$, which is 
possible when $F_i=S^2$.
If $\dim S=1$, then $B(y_0, R;Y_0)\simeq S^1\times D^2$,
which is possible when $F_i=T^2$.
Suppose $\dim S=2$. If $S$ is orientable, then 
$B(y_0, R;Y_0)\simeq S\times I$, which is impossible 
because $\partial F_i$ is connected.
If $S$ is nonorientable, then 
$B(y_0, R;Y_0)$ is homeomorphic to a twisted 
$I$-bundle $S\tilde\times I$, which is possible when
$F_i=S^2$ and $S\simeq P^2$ and when $F_i=T^2$ and $S\simeq K^2$.
\end{proof}

\begin{proclaim}{\emph{Case} A-II.}
     $\dim Y=3$.
\end{proclaim}

By Theorem \ref{thm:dim3}, we have a locally smooth, local 
$S^1$-action $\psi_i$ on $B(\hat p_i,\delta_i R_1)$ whose orbit space is
homeomorphic to $B(y_0,R_1)$. Let 
$\pi_i:B(\hat p_i,\delta_i R_1)\to B(y_0,R_1)$ be the orbit map.
Since $S_{\delta_3}(Y)$ consists of parallel lines, 
with the critical point theory, one can construct such a $\psi_i$ 
satisfying
\begin{equation}
   U_{\delta_i R_1}(\hat q_i,\hat p_i) \simeq 
           \pi_i^{-1}(\{ y_0\}\times B(y_0,R_1;Y_0)),
              \label{eq:slicetop}
\end{equation}
(recall the construction of $\psi_i$ in Sections \ref{sec:dim3nobdyloc}, 
\ref{sec:dim3nobdyglob} and \ref{sec:dim3wbdy}).

\begin{lem}
Suppose that $p\in \partial X$ is a regular boundary point and
$\dim Y=3$. 
 \begin{enumerate}
  \item If $F_i=S^2$, then $Y_0$ is homeomorphic to $\R^2_+$ 
        and $F_{i,{\rm cap}}$ is homeomorphic to either $D^3$
        or $P^2\tilde\times I;$
  \item Suppose $F_i=T^2$. 
   \begin{enumerate}
    \item If $Y_0$ has no boundary, then either $Y_0\simeq\R^2$ and
        Proposition \ref{prop:capfibre} holds, or $Y_0$ is isometric to
        a flat M\"obius strip and 
        $F_{i,{\rm cap}}$ is homeomorphic to
        $K^2\tilde\times I;$
    \item If  $Y_0$ has nonempty boundary, then 
        $Y_0$ is isometric to a flat half cylinder and 
        $F_{i,{\rm cap}}$ is homeomorphic to
        $S^1\times D^2$.
    \end{enumerate}
 \end{enumerate}
\end{lem}

\begin{proof}
Suppose that $Y_0$ has no boundary.
If $\dim S=0$, then $Y_0\simeq R^2$ and 
the number $m_i$ of singular orbits of $\psi_i$ over 
$\{y_0\}\times B(y_0,R_1;Y_0)$ is at most two.
I follows that 
$\pi_i^{-1}(\{ y_0\}\times B(y_0,R_1;Y_0))$ is homeomorphic to
either $D^2\times S^1$  (if $m_i\le 1$) or 
$K^2\tilde\times I$ (if $m_i=2$),
which is possible when $F_i=T^2$.
If $\dim S=1$, $Y_0$ is isometric to either a flat cylinder or 
a flat M\"obius strip.
If $Y_0$ is isometric to a flat cylinder, then 
$\pi_i^{-1}(\{ y_0\}\times B(y_0,R_1;Y_0))$ is homeomorphic to
$I\times T^2$, which contradicts \eqref{eq:s^2}.
If $Y_0$ is isometric to a flat M\"obius strip, then 
$\pi_i^{-1}(\{ y_0\}\times B(y_0,R_1;Y_0))$ is homeomorphic to
$K^2\tilde\times I$, which is possible when $F_i=T^2$.

Suppose that $\partial Y_0$ is disconnected. 
Then $Y_0$ is isometric to a product $\R\times I$.
It follows that 
$\pi_i^{-1}(\{ y_0\}\times B(y_0,R_1;Y_0))$ is homeomorphic to
$I\times S^2$, which contradicts \eqref{eq:s^2}.
Therefore $\partial Y_0$ is connected.
If $\dim S=1$, $Y_0$ is isometric to a flat half cylinder.
It follows from Theorem \ref{thm:patchwbdy} that 
$\pi_i^{-1}(\{ y_0\}\times B(y_0,R_1;Y_0))$ is homeomorphic to
$S^1\times D^2$, which is possible when $F_i=T^2$.
If $\dim S=0$, $Y_0 \simeq \R^2_+$ and 
the number $m_i$ of singular orbits of $\psi_i$ over 
$\{y_0\}\times \interior B(y_0,R_1;Y_0)$ is at most one.
Hence
$\pi_i^{-1}(\{ y_0\}\times B(y_0,R_1;Y_0))$ is homeomorphic to
either $D^3$ (if $m_i=0$) or $P^2\tilde\times I$ (if $m_i=1$).
\end{proof}

We have just proved Proposition \ref{prop:capfibre},
and hence Theorem \ref{thm:dim2wbdy-text} $(1)$-$(a)$ and $(2)$-$(a)$.

Now consider any singular boundary point $p\in\partial X$.

\begin{proclaim}{\emph{Case} B.}
    $p\in \partial X$ is any singular boundary point of $X$.
\end{proclaim}

From Theorem \ref{thm:dim2wbdy-text} $(1)$-$(a)$ and $(2)$-$(a)$, 
the topology of 
$\partial B(p_i,r)$ is classified as follows:
If $F_i=S^2$, then 
\begin{equation} \label{eq:s^2bdy}
  \partial B(p_i,r) \simeq
     \begin{cases}
         & \quad S^3 = D^3 \cup D^3   \\
         & \quad P^3 = D^3 \cup P^2\tilde\times I  \\
         & P^3 \# P^3 = P^2\tilde\times I \cup P^2\tilde\times I,
     \end{cases}
\end{equation}
and if  $F_i=T^2$, then 
\begin{equation} \label{eq:t^2bdy}
  \partial B(p_i,r) \simeq
     \begin{cases}
         & S^1\times D^2 \bigcup S^1\times D^2  \\
         & \,\, S^1\times D^2 \bigcup K^2\tilde\times I  \\
         & \,\,\,\,K^2\tilde\times I \bigcup K^2\tilde\times I.
     \end{cases}
\end{equation}
We determine the topology of $B(p_i,r)$ 
under the boundary conditions \eqref{eq:s^2bdy} and 
\eqref{eq:t^2bdy} as follows. 

\begin{prop} 
\begin{enumerate}
 \item If $F_i=S^2$, then $B(p_i,r)$ is homeomorphic to either
       $D^4$, $S^2\tilde\times_{\omega} D^2$ $(|\omega|\in \{ 1,2\})$,
       or $P^2\tilde\times_0 D^2;$
 \item If $F_i=T^2$, then $B(p_i,r)$ is homeomorphic to either
       $D^4$, $S^1\times D^3$, or a $D^2$-bundle over $S^2$, 
       $P^2$ or $K^2$.
\end{enumerate}
\end{prop}

\begin{proof}
First suppose $\dim Y=4$.
If $F_i=S^2$, in view of \eqref{eq:s^2bdy}, we can eliminate the cases
when $\dim S=1$ or $S$ is a flat surface, and obtain $(1)$.
In the case of $F_i=T^2$,
it suffices to eliminate the case when $S$ is isometric to 
a flat torus.
This is done in the following

\begin{slem} \label{slem:not-t^2}
If $p\in\partial X^2$ and $F_i=T^2$, then $B(p_i,r)$ cannot 
be homeomorphic to a $D^2$-bundle over $T^2$.
\end{slem}

\begin{proof}
Suppose that $B(p_i,r)$ is homeomorphic to a 
$D^2$-bundle over $T^2$.
Then $\partial B(p_i,r)$ is an $S^1$-bundle over $T^2$,
and hence $\Gamma:=\pi_1(\partial B(p_i,r))$ is nilpotent.
On the other hand, it follows from \eqref{eq:t^2bdy}
that 
$\partial B(p_i,r)\simeq K^2\tilde\times I \cup K^2\tilde\times I$.
By Van Kampen's theorem, 
$\Gamma$ is isomorphic to a form 
$\Lambda_1 * _{\Z^2} \Lambda_2$, where 
$\Lambda_j\simeq\pi_1(K^2)$ and is considered as the $\Z_2$-extension 
of $\Z^2$. Let $\iota:\Lambda_1\to\Lambda_1*\Lambda_2$ be 
a natural inclusion and 
$\pi:\Lambda_1*\Lambda_2\to\Gamma$ the projection.
Then $\pi\circ\iota:\Lambda_1\to\Gamma$ is an injective
homomorphism.
Since $\pi_1(K^2)$ is not nilpotent, this is a contradiction.
\end{proof}

Next suppose $\dim Y=3$.
Then in view of \eqref{eq:s^2bdy},\eqref{eq:t^2bdy} together with
$\dim Y(\infty)\ge 1$, 
Theorem \ref{thm:dim3class} yields the conclusion.
\end{proof}

We have just proved Theorem \ref{thm:dim2wbdy-text}
(1)-(b) and (2)-(b).\par

Next we show Theorem \ref{thm:dim2wbdy-text} (1)-(c).
Suppose $B(p_i,r)\simeq P^2\tilde\times_0 D^2$, and 
let $\Gamma_i$ be the deck transformation group of the 
universal cover $\pi_i:\tilde B(p_i,r)\to B(p_i,r)$.
Take a sequence $\mu_i\to 0$ such that 
$(\frac{1}{\mu_i} B(p_i,r),p_i)$ converges to 
$(K_p,o_p)$.
We may assume that 
$(\frac{1}{\mu_i}\tilde B(p_i,r),\tilde p_i,\Gamma_i)$
converges to a triplet $(Z^2,z_0,\Gamma)$.
If $Z^2$ had nonempty boundary,  then 
we would have a contradiction 
to Theorem \ref{thm:dim2wbdy-text} (1)-(b) since
$\tilde B(p_i,r)\simeq S^2\times D^2$.
Therefore $Z^2$ must be isometric to a flat cone
without boundary and $\Gamma\simeq\Z_2$.
It follows from $Z^2/\Gamma=K_p$ that 
the action of $\Gamma$ on $Z^2$ is by reflection.
Let $\tilde A(p_i;r/2,r):=\pi_i^{-1}(A(p_i;r/2,r))$.
By Equivariant Fibration Theorem \ref{thm:cap},
we have a $\Z_2$-equivariant $S^2$-bundle 
$\tilde A(p_i;r/2,r)\to A(z_0;r/2,r)$
and hence a $\Z_2$-equivariant homeomorphism 
$\tilde A(p_i;r/2,r)\simeq S^2\times A(z_0;r/2,r)$, where
$\Z_2$ acts diagonally on 
$S^2\times A(z_0;r/2,r)$; freely on the $S^2$-factor
and by reflection on the $A(z_0;r/2,r)$-factor.
Therefore
$A(p_i;r/2,r)$ is homeomorphic to the diagonal quotient
$(S^2\times A(z_0;r/2,r))/\Z_2$.
Finally we fill $B(p_i,r/2)$ with an obvious gluing by 
$P^2\tilde\times_0 D^2\simeq (S^2\times D^2(r/2))/\Z_2$.
We therefore obtain $B(p_i,r)\simeq (S^2\times D^2)/\Z_2$.
This completes the proof of Theorem \ref{thm:dim2wbdy-text} (1)-(c).

In view of Theorem \ref{thm:locstrtor} (2), 
the proof of Theorem \ref{thm:dim2wbdy-text}(2)-(c)
is similar and hence omitted.

Finally we show Theorem \ref{thm:dim2wbdy-text} (3).

\begin{lem} \label{lem:deform-cap}
$B(p_i,r)$ is homeomorphic to 
$B_{{\rm cap}}(p_i,r)$.
\end{lem}

\begin{proof}
For a sufficiently small $\mu\ll\nu$, take
a $\mu$-net $N_{\mu}$ of $\partial X\cap B(p,r)$, and 
let $N_{\mu,i}$ be a subset of $M_i^4$ converging
to $N_{\mu}$. Using the flow curves of $\tilde d_{N_{\mu,i}}$,
one can easily prove the lemma.
\end{proof}

Let us denote the two fibres of $f_{i,{\rm cap}}$ in 
\eqref{eq:decomp-4} by $F_{i,{\rm cap}}$ and $F_{i,{\rm cap}}'$.
We call the topological types of the triplet 
$(B(p_i,r);F_{i,{\rm cap}},F_{i,{\rm cap}})$ 
the {\it the collapsing data} around $p_i$.

Using the topological information on
$B_{{\rm cap}}(p_i,r)\simeq B(p_i,r)$ given in (1)-(b) and 
(2)-(b), we can define a compatible singular fibre
structure on $B_{{\rm cap}}(p_i,r)$ by 
extending the fiber structure on 
$\partial B_{{\rm cap}}(p_i,r)$ by cone, which is indicated 
below in terms of the collapsing data.
\par\medskip

\begin{proclaim}{\emph{Case} A.}
      $F_i=S^2$.
\end{proclaim}

\begin{center}
\begin{tikzpicture}
[scale = 0.5]
\fill (-2, 1)  circle (0pt) node [left] {$(1)$};
\draw [-,thick] (0,0) -- (0,-3);
\draw [-,thick] (0,0) -- (3,0);
\fill (0.1,0.3) circle (0pt) node [left]{${\rm pt}$};
\fill (-0.1,-1.9) circle (0pt) node [left]{${\rm pt}$};
\fill (1.7,0.1) circle (0pt) node [above]{${\rm pt}$};
\fill (1.7,-1) circle (0pt) node [below]{$S^2$};
\fill (1,-3.5) circle (0pt) node [below]{$(D^4;D^3,D^3)$};
\fill[shift={(12,0)}] (-2, 1)  circle (0pt) node [left] {$(2)$};
\draw[shift={(12,0)}] [-,thick] (0,0) -- (0,-3);
\draw[shift={(12,0)}][-,thick] (0,0) -- (3,0);
\fill[shift={(12,0)}] (0.1,0.3) circle (0pt) node [left]{$S^2$};
\fill[shift={(12,0)}] (-0.1,-1.9) circle (0pt) node [left]{${\rm pt}$};
\fill[shift={(12,0)}] (1.7,0.1) circle (0pt) node [above]{${\rm pt}$};
\fill[shift={(12,0)}] (1.7,-1) circle (0pt) node [below]{$S^2$};
\fill[shift={(12,0)}] (1,-3.5) circle (0pt) node [below]{$(S^2\tilde\times_{\pm 1} D^2;D^3,D^3)$};
\end{tikzpicture}
\end{center}
\pbig
\begin{center}
\begin{tikzpicture}
[scale = 0.5]
\fill (-2, 1)  circle (0pt) node [left] {$(3)$};
\draw [-,thick] (0,0) -- (0,-3);
\draw [-,thick] (0,0) -- (3,0);
\fill (0.1,0.3) circle (0pt) node [left]{$S^2$};
\fill (-0.1,-1.9) circle (0pt) node [left]{$P^2$};
\fill (1.7,0.1) circle (0pt) node [above]{${\rm pt}$};
\fill (1.7,-1) circle (0pt) node [below]{$S^2$};
\fill (1,-3.5) circle (0pt) node [below]{$(S^2\tilde\times_{\pm 2} D^2;D^3,P^2\tilde\times I)$};
\fill[shift={(12,0)}] (-2, 1)  circle (0pt) node [left] {$(4)$};
\draw[shift={(12,0)}] [-,thick] (0,0) -- (0,-3);
\draw[shift={(12,0)}][-,thick] (0,0) -- (3,0);
\fill[shift={(12,0)}] (0.1,0.3) circle (0pt) node [left]{$P^2$};
\fill[shift={(12,0)}] (-0.1,-1.9) circle (0pt) node [left]{$P^2$};
\fill[shift={(12,0)}] (1.7,0.1) circle (0pt) node [above]{$P^2$};
\fill[shift={(12,0)}] (1.7,-1) circle (0pt) node [below]{$S^2$};
\fill[shift={(12,0)}] (1,-3.5) circle (0pt) node [below]{$(P^2\tilde\times_0 D^2;P^2\tilde\times I,P^2\tilde\times I)$};
\end{tikzpicture}
\end{center}
\bigskip
\begin{proclaim}{\emph{Case} B.}
      $F_i=T^2$.
\end{proclaim}

\par\bigskip
\begin{center}
\begin{tikzpicture}
[scale = 0.5]
\fill (-2, 1)  circle (0pt) node [left] {$(5)$};
\draw [-,thick] (0,0) -- (0,-3);
\draw [-,thick] (0,0) -- (3,0);
\fill (0.1,0.3) circle (0pt) node [left]{${\rm pt}$};
\fill (-0.1,-1.9) circle (0pt) node [left]{$S^1$};
\fill (1.7,0.1) circle (0pt) node [above]{$S^1$};
\fill (1.7,-1) circle (0pt) node [below]{$T^2$};
\fill (1,-3.5) circle (0pt) node [below]{$(D^4;S^1\times D^2,S^1\times D^2)$};
\fill[shift={(12,0)}] (-2, 1)  circle (0pt) node [left] {$(6)$};
\draw[shift={(12,0)}] [-,thick] (0,0) -- (0,-3);
\draw[shift={(12,0)}][-,thick] (0,0) -- (3,0);
\fill[shift={(12,0)}] (0.1,0.3) circle (0pt) node [left]{$S^1$};
\fill[shift={(12,0)}] (-0.1,-1.9) circle (0pt) node [left]{$S^1$};
\fill[shift={(12,0)}] (1.7,0.1) circle (0pt) node [above]{$S^1$};
\fill[shift={(12,0)}] (1.7,-1) circle (0pt) node [below]{$T^2$};
\fill[shift={(12,0)}] (1,-3.5) circle (0pt) node [below]{$(S^1\times D^3;S^1\times D^2,S^1\times D^2)$};
\end{tikzpicture}
\end{center}
\pbig

\begin{center}
\begin{tikzpicture}
[scale = 0.5]
\fill (-2, 1)  circle (0pt) node [left] {$(7)$};
\draw [-,thick] (0,0) -- (0,-3);
\draw [-,thick] (0,0) -- (3,0);
\fill (0.1,0.3) circle (0pt) node [left]{$S^1$};
\fill (-0.1,-1.9) circle (0pt) node [left]{$K^2$};
\fill (1.7,0.1) circle (0pt) node [above]{$S^1$};
\fill (1.7,-1) circle (0pt) node [below]{$T^2$};
\fill (1,-3.5) circle (0pt) node [below]{$(S^1\times D^3;S^1\times D^2,K^2\tilde\times I)$};
\fill[shift={(12,0)}] (-2, 1)  circle (0pt) node [left] {$(8)$};
\draw[shift={(12,0)}] [-,thick] (0,0) -- (0,-3);
\draw[shift={(12,0)}][-,thick] (0,0) -- (3,0);
\fill[shift={(12,0)}] (0.1,0.3) circle (0pt) node [left]{$S^2$};
\fill[shift={(12,0)}] (-0.1,-1.9) circle (0pt) node [left]{$S^1$};
\fill[shift={(12,0)}] (1.7,0.1) circle (0pt) node [above]{$S^1$};
\fill[shift={(12,0)}] (1.7,-1) circle (0pt) node [below]{$T^2$};
\fill[shift={(12,0)}] (1,-3.5) circle (0pt) node [below]{$(S^2\tilde\times_{\omega} D^2;S^1\times D^2,S^1\times D^2)$};
\end{tikzpicture}
\end{center}
\pbig
\begin{center}
\begin{tikzpicture}
[scale = 0.5]
\fill[shift={(-2,0)}] (-2, 1)  circle (0pt) node [left] {$(9)$};
\draw[shift={(-2,0)}] [-,thick] (0,0) -- (0,-3);
\draw[shift={(-2,0)}] [-,thick] (0,0) -- (3,0);
\fill[shift={(-2,0)}] (0.1,0.3) circle (0pt) node [left]{$P^2$};
\fill[shift={(-2,0)}] (-0.1,-1.9) circle (0pt) node [left]{$K^2$};
\fill[shift={(-2,0)}] (1.7,0.1) circle (0pt) node [above]{$S^1$};
\fill[shift={(-2,0)}] (1.7,-1) circle (0pt) node [below]{$T^2$};
\fill[shift={(-2,0)}] (1,-3.5) circle (0pt) node [below]{$(P^2\tilde\times_0 D^2;S^1\times D^2,K^2\tilde\times I)$};
\fill[shift={(12,0)}] (-2, 1)  circle (0pt) node [left] {$(10)$};
\draw[shift={(12,0)}] [-,thick] (0,0) -- (0,-3);
\draw[shift={(12,0)}][-,thick] (0,0) -- (3,0);
\fill[shift={(12,0)}] (0.1,0.3) circle (0pt) node [left]{$K^2$};
\fill[shift={(12,0)}] (-0.1,-1.9) circle (0pt) node [left]{$K^2$};
\fill[shift={(12,0)}] (1.7,0.1) circle (0pt) node [above]{$K^2$};
\fill[shift={(12,0)}] (1.7,-1) circle (0pt) node [below]{$T^2$};
\fill[shift={(12,0)}] (1,-3.5) circle (0pt) node [below]{$(\text{ a $D^2$-bundle over $K^2$};
       K^2\tilde\times I,K^2\tilde\times I)$};
\end{tikzpicture}
\end{center}
\pbig
This completes the proof of Theorem \ref{thm:dim2wbdy-text} (3). 
\par
\medskip

In the cases of (1), (4) and (6) in the above list, 
$B_{\rm cap}(p_i,r)$
has a product fibre structure. More precise fiber structures
for the cases of (4) and  (10) are
described in Theorem \ref{thm:dim2wbdy-text}(1)-(c) and 
(2)-(c) respectively.
The examples of collapsing related with (5) and (7) are given 
in Examples \ref{ex:d^2d^2} and \ref{ex:p2d2}.
\par\medskip

To finish the proof of Theorem \ref{thm:dim2wbdy-break}, it suffices 
to show the following 

\begin{lem} \label{lem:non-ext}
If $p\in\partial X$ is not an extremal point of $X$,
then  $B_{\rm cap}(p_i,r)$ 
has a product $F_{i,{\rm cap}}$-fibre structure over $I$
compatible with those of 
$B_{\rm int}(p_i,r)$ and $A_{\rm cap}(p_i;r_0,r)$.
\end{lem}

\begin{proof}
Since both $B_{\rm int}(p_i,r)$ and $A_{\rm cap}(p_i;r_0,r)$
depend on $\nu$, we rewrite them as 
$B_{{\rm int},\nu}(p_i,r)$ and $A_{{\rm cap},\nu}(p_i;r_0,r)$ 
respectively.
We have an obvious homeomorphism 
$$
A_{{\rm cap},\nu}(p_i;r_0,r)-\interior A_{{\rm cap},\nu/2}(p_i;r_0,r)
     \simeq (F_i\times I\times\partial I)\times I.
$$
In a way similar to Lemma \ref{lem:dim3balltop},
we have 
\begin{equation}
   A_{{\rm cap},\nu/2}(p_i;r_0,r)\simeq
        F_{i,{\rm cap}}\times I\times\partial I,
                                   \label{eq:prod-cap}
\end{equation}
together with a flow $\phi$ giving the above product
structure \eqref{eq:prod-cap} on $A_{{\rm cap},\nu/2}(p_i;r_0,r)$
(compare also the argument in Section \ref{sec:dim3wbdy}).
This provides a trivial $F_i$-bundle
$\xi_1$ on $\partial A_{{\rm cap},\nu/2}(p_i;r_0,r)
\simeq F_i\times I\times\partial I\times \{ 1\}$.
We also have the trivial $F_i$-bundle
$\xi_0$ on $\partial A_{{\rm cap},\nu}(p_i;r_0,r)
\simeq F_i\times I\times\partial I\times \{ 0\}$.
For any $(x,u)\in F_i\times I\times\partial I\times \{ 1\}$,
let $t(x,u)\in\R$ denote the parameter at which 
the flow curve $\phi_{(x,u)}$ starting from $(x,u)$
meets the fibre $F(\xi_1)_u$ of $\xi_1$ over $u$.
Note that the fibres $F(\xi_1)_u$ are not compatible with
the fibres $\partial F_{i,{\rm cap}}$.
Then $\varphi_s(x,u):=\phi_{(x,u)}(s t(x,u))$, $0\le s\le 1$,
presents a one-parameter family of homeomorphisms $\varphi_s$ of 
$F_i\times I\times\partial I$. Define a homeomorphism 
$\Phi$ of $F\times I\times\partial I\times I$ by
$\Phi(x,u,s):=(\varphi_s(x,u),s)$.
This gives a one-parameter family of trivial $F_i$-bundle
joining $\xi_0$ and $\xi_1$, which defines the required
product $F_{i,{\rm cap}}$-fibre structure on 
$B_{\rm cap}(p_i,r)$.
\end{proof}

\begin{ex} \label{ex:d^2d^2}
Let $T^2=SO(2)\times SO(2)$ act on $D^2\times D^2$ by
$$
(e^{i \varphi},e^{i \theta})\cdot (r_1e^{i\alpha_1}, r_2e^{i\alpha_2})
      =(r_1e^{i(\alpha_1+m_1\varphi+n_1\theta)}, 
           r_2e^{i(\alpha_2+m_2\varphi+n_2\theta)}),
$$
where we assume 
\begin{equation*}
\begin{vmatrix}
    m_1 & m_2 \\
    n_1 & n_2
\end{vmatrix}
   = \pm 1
\end{equation*}
to make the action effective.
Then we have a sequence of metrics $g_i$ on $D^2\times D^2$ 
of nonnegative sectional curvature 
such that $(D^2\times D^2, g_i)$ collapses to 
$I\times I$. In this case, we have the same collapsing data
as (5).
\end{ex}

\begin{ex}
By gluing $T^2$-equivariantly suitable two
$T^2$-action on $D^2\times D^2$ as described in 
Example \ref{ex:d^2d^2}, one can construct 
an $T^2$-action on $S^2\tilde\times_{\omega} D^2$ for any 
$\omega$  with the orbit space homeomorphic $I^2$
having two fixed points on $\partial I^2$
(see \cite{Fn:circleI}).

To be more explicit, let us restrict our attention to
the case $\omega =1$.
Let $E$ be an $\R^2$-vector bundle over $S^2$ with 
a fiber metric such that the unit disk-bundle
$E_1=S^2\tilde\times_{1} D^2(1)\subset E$ has boundary
homeomorphic to $S^3$.
For any $r>0$, we have a faithful $T^2$-action on 
$\partial (S^2\tilde\times D^2(r))$.
This induces the $T^2$-action on $E$ with 
$E/T^2\simeq I\times [0,\infty)$.
For a $T^2$-invariant metric  $g$ on $E_1$,
we have a sequence of metrics $g_i$ on $E_1$ 
such that $(E_1, g_i)$ collapses to 
$(E_1/T^2,\bar g)$.
In this case for any extremal point $p\in\partial I\times [0,\infty)$,
$B(p_i,r)\simeq D^4$ with the same collapsing data as (5).
\end{ex}

\begin{ex} \label{ex:p2d2}
In the product $D^2\times \R^2$, identify 
$(x,t)$ and $(x^*, t')$ for any $x\in\partial D^2$ and $t\in\R^2$,
where $x^*$ and $t'$ denote the antipodal point of $x^*$ and the image
of $t$ by a reflection respectively.
We consider the resulting identification space 
$D^2\times\R^2/\sim$, which is an $\R^2$-bundle over 
$P^2$ denoted $P^2\tilde\times_0\R^2$.
Let $g_i$ and $h_i$ be sequences of rotationally symmetric 
nonnegatively curved metrics on $D^2$ and $\R^2$
such that $(D^2,g_i)$ and $(\R^2,o,h_i)$ converge
to $I$ and $[0,\infty)$ respectively.
Then the product $(D^2\times \R^2, g_i\times h_i)$
gives a nonnegatively curved metric $k_i$ on $P^2\tilde\times_0 \R^2$
such that $(P^2\tilde\times_0\R^2,o,k_i)$ collapses to 
the product $I\times [0,\infty)$.
For the extremal point $p\in\partial (I\times [0,\infty))$ corresponding
to $\partial D^2\times 0$,
$B(p_i,r)$ is homeomorphic to a $D^2$-bundle over $\Mo$, 
that is $S^1\times D^3$, 
and has the same collapsing data as (7).
\end{ex}

\begin{prob}
Determine if there exists a sequence of collapsed metrics on a 
$D^2$-bundle over $S^2$ or $P^2$ with a definite lower 
sectional curvature bound having the same collapsing data as 
(2), (3), (8) or (9).
\end{prob}


\section{Classification of collapsing to noncompact
          two-spaces with nonnegative curvature}
                     \label{sec:dim2positive}

Let a sequence of pointed complete $4$-dimensional orientable 
Riemannian manifolds $(M_i^4,p_i)$ with $K \ge -1$ collapses to 
a pointed complete noncompact $2$-dimensional Alexandrov 
space $(Y^2,y_0)$ with 
nonnegative curvature. In this section, using the results of 
Sections \ref{sec:dim2nobdysph}, \ref{sec:dim2nobdytor} and 
\ref{sec:dim2wbdy},
we classify the topology of a large metric ball 
$B(p_i,R)$ in terms of geometric properties of $Y^2$.
The classification result will be used in the next section
to describe the phenomena
of orientable $4$-manifolds collapsing to a closed interval.

Applying Theorems \ref{thm:dim2nobdy} and \ref{thm:dim2wbdy-break} to 
the convergence $(M_i^4, p_i) \to (Y^2,y_0)$,
we have a singular fibration $\pi_i:B(p_i,R)\to B(y_0, R)$
with general fibre $F_i$, either $S^2$ or $T^2$, 
where $R$ is a large positive number.
Actually, we have such a singular fibre structure
on a small perturbation of $B(p_i,R)$, which is homeomorphic to 
$B(p_i,R)$.
Let $m_i$ and $n_i$ denote the numbers of singular fibres of 
$\pi_i:B(p_i,R)\to B(y_0, R)$ over $B(y_0,R)\cap\interior Y$ 
and over $B(y_0,R)\cap\partial Y$ respectively.
From the classification of Alexandrov surfaces with nonnegative
curvature (cf. \cite{SY:3mfd}) together with Theorems
\ref{thm:dim2nobdy} and \ref{thm:dim2wbdy-break}, 
if $Y$ has no boundary, then $m_i\le 2$, and 
if $Y$ has nonempty boundary, then 
$m_i\le 1$, $n_i\le 2$ and the possible cases are
$$
   (m_i, n_i)=(0,0),\quad (0,1),\quad (0,2),\quad (1,0),
$$
in the latter case. 
Note also that $m_i=0$ if $F_i=S^2$.

\begin{thm} \label{thm:dim2class}
Under the situation above, the topology of $B(p_i,R)$ can be 
classified as follows: 
\begin{proclaim}{\emph{Case} I.}
      $\dim Y(\infty)=1$.
\end{proclaim}
\begin{enumerate}
 \item Suppose $Y$ has no boundary. 
  \begin{enumerate}
   \item If $F_i=S^2$, $B(p_i,R)\simeq D^2\times S^2;$
   \item If $F_i=T^2$, $B(p_i,R)\simeq D^2\times T^2$.
  \end{enumerate}
 \item  Suppose $Y$ has nonempty boundary. 
  \begin{enumerate}
   \item If $F_i=S^2$, $B(p_i,R)$ is homeomorphic to either 
         $D^4$, $S^2\tilde\times_{\omega} D^2$ with 
         $|\omega|\in\{ 1,2\}$ or $P^2\tilde\times_0 D^2;$
   \item If $F_i=T^2$, $B(p_i,R)$ is homeomorphic to either 
         $D^4$, $S^1\times D^3$, or a $D^2$-bundle over 
         $S^2$, $P^2$ or $K^2$.
  \end{enumerate}
\end{enumerate}
\begin{proclaim}{\emph{Case} II.}
      $\dim Y(\infty)=0$.
\end{proclaim}
\begin{enumerate}
 \item Suppose $Y$ has no boundary.
  \begin{enumerate}
   \item[(i)] If $Y\simeq \R^2$, then the following holds $:$
    \begin{enumerate}
     \item[(a)] If $F_i=S^2$, $B(p_i,R)$ is homeomorphic to the space in 
            Case $I$-$(1)$-$(a)$ $;$
     \item[(b)] If $F_i=T^2$, $B(p_i,R)$ is homeomorphic to either 
           the space in  Case $I$-$(1)$-$(b)$, or an $I$-bundle 
           over $T^2\tilde\times S^1$, a twisted $S^1$-bundle over $T^2$
           doubly covered by $T^3$.
    \end{enumerate}
   \item[(ii)] If $Y$ is isometric to either a flat cylinder or a 
           flat M\"obius strip, then 
           $B(p_i,R)$ is homeomorphic to
           an $I$-bundle over an $F_i$-bundle over $S^1$.
  \end{enumerate}
 \item Suppose $Y$ has nonempty boundary.
  \begin{enumerate}
   \item[(i)] If $Y\simeq \R^2_+$, then the following holds$:$
    \begin{enumerate}
     \item[(a)] If $(m_i,n_i)=(0,2)$, then $B(p_i,R)$ is homeomorphic to 
       a gluing of two disk-bundles, denoted 
       $N_j^{k_j}\tilde\times D^{4-k_j}$, $j=1,2$,
       over $k_j$-dimensional closed manifolds $N_j^{k_j}$
       with nonnegative Euler numbers,
       where $0\le k_j\le 2$. The gluing is done along
       one of $D^3$, $P^2\tilde\times I$ $($if $F_i=S^2$$)$, 
       and one of
       $S^1\times D^2$, $K^2\tilde\times I$ $($if $F_i=T^2$$)$,
       imbedded in the boundaries
       of $N_j^{k_j}\tilde\times D^{4-k_j};$
     \item[(b)]  If $(m_i,n_i)=(1,0)$, then $B(p_i,R)$ is 
       homeomorphic to one of the gluings$:$
       \begin{equation*}
         \hspace*{2.5cm}
          S^1\times D^3 \bigcup_{T^2\times I} T^2\times D^2,\quad
          K^2\tilde\times D^2 \bigcup_{T^2\times I} T^2\times D^2.
       \end{equation*}
    \end{enumerate}
   \item[(ii)] If $Y$ is isometric to a flat half cylinder, then 
               the following holds$:$
    \begin{enumerate}
     \item[(a)] If $F_i=S^2$, $B(p_i,R)$ is homeomorphic to either 
           $S^1\times D^3$ or an $I$-bundle over a $S^1\times P^2;$
     \item[(b)] If $F_i=T^2$, $B(p_i,R)$ is homeomorphic to either
           $D^2$-bundle over $T^2$ or $K^2$, or an 
           $I$-bundle over a $K^2$-bundle over $S^1$.
    \end{enumerate}
   \item[(iii)] If $Y$ is isometric to a product $I\times \R$, then 
               the following holds$:$
    \begin{enumerate}
     \item[(a)] If $F_i=S^2$, $B(p_i,R)$ is homeomorphic to either 
           $S^3\times I$, $P^3\times I$ or $(P^3\# P^3)\times I;$
     \item[(b)] If $F_i=T^2$, $B(p_i,R)$ is homeomorphic to 
           a form $(U^3\cup V^3)\times I$, where $U^3$ and 
           $V^3$ are ones of $S^1\times D^2$ and $K^2\tilde\times I$,
           and $U^3\cup V^3$ denotes a gluing along their boundary
           tori.
    \end{enumerate}
  \end{enumerate}
  \end{enumerate}
\end{thm}

\begin{proof}
In Case I-(1), $Y$ is homeomorphic to $\R^2$ and has at most 
one essential singular point. 
In Case I-(2), there are no essential singular points in $\interior Y$ 
and at most one extremal point on $\partial Y$.
Hence the conclusions are
the direct consequences of 
Theorems \ref{thm:dim2nobdy} and \ref{thm:dim2wbdy-break}.

In Case II-(1)-(i), in view of Case I-(1) and Theorem 
\ref{thm:dim2nobdy} we may assume that $F_i=T^2$
and $m_i=2$.
Therefore $Y$ is isometric to the double
$D(I\times [0,\infty))$ for some closed interval $I$.
Split $B(y_0,R)$ with a proper segment into two closed domain 
$B_1$ and $B_2$ each of which contains one of the two essential 
singular points of $Y$. 
Applying Theorem \ref{thm:locstrtor} to each
$B_j$, we have 
$$
   B(p_i, R)\simeq T^2\times D^2 \bigcup_{I\times T^2} T^2\times D^2.
$$
But actually $\interior B(p_i,R)$ is homeomorphic to 
the complete flat manifold defined as the $\Z_2$-quotient
of $T^2\times S^1\times \R$ by an  involution,
where the action of $\Z_2$ is diagonal, free on $T^2$,
and orientation reversing on both factors $S^1$ and $\R$.
Therefore $B(p_i,R)\simeq ((T^2\times S^1)/\Z_2)\tilde\times I$.

Consider Case II-(1)-(ii), the case of $\dim S=1$.
Obviously $B(p_i,R)$ is an $F_i$-bundle over
$B(y_0, R)$, and the conclusion follows.

Next consider Case II-(2)-(i). 
If $(m_i,n_i)=(0,2)$, then $Y$ is isometric to 
$I\times [0,\infty)$ for some closed interval $I$.
Split $B(y_0,R)$ with a proper segment into two closed regions
$B_1$ and $B_2$ each of which contains one of the two 
extremal points in $\partial Y$.
Applying Theorem \ref{thm:dim2wbdy-break} to each $B_j$, we obtain 
the required gluing.
If $(m_i,n_i)=(1,0)$, then $Y$ is isometric to 
$D([0,\infty)\times \{ x\ge 0\})\bigcap \{ x\le a\}$ for some 
$a>0$. By a similar cutting and pasting argument, 
we obtain the required gluing.

Consider the cases II-(2)-(ii) and (iii), the cases of $\dim S=1$ 
and $\partial Y$ being disconnected respectively.
By Theorem \ref{thm:dim2wbdy-break} together with 
the facts that the mapping class group $\cal M(P^2)$ of all 
homeomorphisms of $P^2$ is trivial and  
$\cal M_+(K^2\tilde\times I)=\cal M(K^2)$, 
we obtain the conclusions.
\end{proof}

\bigskip
We are in a position to prove Theorem \ref{thm:dim1}.
Let us consider 
a sequence of $4$-dimensional closed orientable Riemannian
manifolds $M_i^4$ with $K \ge -1$ converging to a
$1$-dimensional closed interval $X^1$.
Let $\{ p,q\}:=\partial X$ and take $p_i$ and $q_i$ in $M_i^4$ with 
$p_i\to p$ and $q_i\to q$.
By Fibration Theorem \ref{thm:orig-cap}, 
$M_i^4$ is homeomorphic to a gluing
$$
     B(p_i,r)\bigcup B(q_i,r),
$$
for any sufficiently small positive number $r$ and any large
$i$ compared to $r$, where 
$\partial B(p_i,r)\simeq\partial B(q_i,r)$ is homeomorphic to
one of the closed 3-manifolds given in Theorem \ref{thm:circle}.
Now we consider the local convergence $B(p_i,r)\to B(p,r)$.
By Theorem \ref{thm:rescal}, we have sequences $\delta_i\to 0$ and 
$\hat p_i\to p$ such that
\begin{enumerate}
  \item  for any limit $(Y,y_0)$ of $(\frac{1}{\delta_i}M_i, \hat p_i)$, 
     we have  $\dim Y\ge 2$;
  \item  $B(p_i,r)$ is homeomorphic to $B(\hat p_i,R\delta_i)$ for every 
      $R\ge 1$ and large $i$ compared to $R$.
\end{enumerate}
Applying Theorems \ref{thm:dim3class} and \ref{thm:dim2class}
to the new convergence 
$(\frac{1}{\delta_i}M_i, \hat p_i)\to (Y,y_0)$ together
with the boundary condition stated above, 
we can determine or classify the possible topological type of
$(p_i,r)$ as:

\begin{thm} \label{thm:dim1-text}
If $p\in\partial X^1$, then $B(p_i,r)$ is homeomorphic to 
a disk bundle over a $k$-dimensional closed manifold $N^k$ with
$0\le k\le 3$, or a gluing of two disk-bundles over
$k_j$-dimensional closed manifolds
$Q^{k_j}$, $j=1,2$, with $0\le k_j\le 2$,
where $N^k$ and $Q^{k_j}$ have nonnegative Euler numbers,
and if $k=3$,  $N^3$ is one of the closed 3-manifolds given in 
Theorem \ref{thm:circle}.
\end{thm}

We have a similar topological information on 
$B(q_i,r)$. This completes the proof of Theorem \ref{thm:dim1}.
\par
\medskip
In the situation of Theorem \ref{thm:dim1}, the author does not 
know a specific example of a closed $4$-manifold built of 
gluing of three or four pieces of disk-bundles
which admits a sequence of metrics collapsing to a closed interval
under $K\ge -1$.
In view of Proposition \ref{prop:find-top}, one of typical 
problems arising from the results of the present paper is 
the following:

\begin{prob}
Letting $M^4$ be either $S^2\times S^2 \# S^2\times S^2$ or 
the connected sum of three or four pieces of $\pm \C P^2$,
determine if $M^4$ admits 
a sequence of metrics collapsing to a closed interval
under $K\ge -1$.
\end{prob}


\part{Complete Alexandrov spaces 
with nonnegative curvature} \label{part:alex}
  In Part \ref{part:alex},  we develop the geometry of 
complete Alexandrov spaces with nonnegative curvature.
Most of them  are needed in Part \ref{part:collapse}.
In particular, we establish the Generalized Soul Theorem 
\ref{thm:soul}.
\par

Throughout this part, let $X$ be an $n$-dimensional complete  
Alexandov space with nonnegative curvature.
We assume $X$ to be either noncompact or having nonempty
boundary.
Applying the Cheeger-Gromoll basic construction,
we obtain a sequence of finitely many 
nonempty compact totally convex sets of $X$:
\begin{equation}
  C(0)\supset C(1)\supset C(2)\supset\cdots\supset C(k),
            \label{eq:soul-const}
\end{equation}
with $S=C(k)$ as in Section \ref{sec:ideal} except 
that $C(0)$ coincides with the minimum set of $d(\partial X, \,\cdot\,)$
if $X$ is compact.

From the construction, we have the 
filtration  $\{ X^t\}_{0\le t < t_*}$, $t^*\le\infty$, by compact totally 
convex subsets such that
\begin{enumerate}
  \item  $X^s=\{x\in X^t\,|\,d(x,\partial X^t)\ge t-s\}$\qquad for  $s\le t;$
  \item  the closure of $\bigcup_{0\le t < t_*} X^t$ coincides with $X;$
  \item  $X^0=C$.
\end{enumerate}
Note $t_*<\infty$ if and only if $X$ is compact.
\par

\section{Local regularities} \label{sec:local}

In this section, we are concerned with three local regularity 
properties of $X$.

We begin with a more general situation described as follows:
Let $X$ be an $n$-dimensional complete nonnegatively curved 
Alexandrov space, and suppose $X$ has a filtration 
$\{ X^t_*\}_{0\le t < t_*}$  by closed totally 
convex subsets satisfying the same conditions as (1), (2), (3) above,
which is determined by a convex function
$\rho$ with minimum $0$ in such a way that $X^t_*=\{ \rho\le t\}$,
where $t_*$ is the supremum of the values of $\rho$.
However we do not assume that each $X^t_*$ is compact here.
Put  $C_*:=X^0_*$ and let us assume $\dim C_*=n-1$.

A point $x\in C_*$ is called a {\it one-normal point}
(resp. {\it two-normal point})
if there exists exactly one (resp. two) geodesic ray(s)
from $x$ perpendicular to $C_*$. 
Note that each point of $C_*$ is either
a one-normal point or a two-normal point (see \cite{SY:3mfd}).

\begin{prop}  \label{prop:codim1-reg}
Under the hypothesis above, 
suppose that  $p\in\interior C_*$ is a topologically nice 
point of $X^n$. Then 
\begin{enumerate}
 \item $\Sigma_p(X)$ is isometric to the spherical suspension over
       $\Sigma_p(C_*);$
 \item $p$ is a two-normal point. 
\end{enumerate}
\end{prop}

\begin{proof}
In view of Proposition \ref{prop:concave-rigid}, it suffices to 
show only (2).
This is done by induction on $n$.
Under the convergence $(\frac{1}{r} X, p)\to (K_p,o_p)$ as 
$r\to 0$, the filtration $\{ X^t_*\}_{0\le t < t_*}$
gives rise to a filtration $\{ K_p^t\}_{0\le t < \infty}$
with $K_p^0=K_p(C_*)$
satisfying the same conditions as  (1), (2), (3) above.
This is done as follows. Let $\rho_{\infty}$ be the limit of 
$m\rho$ under the convergence 
$(mX,p)\to (K_p,o_p)$, $m\to\infty$. Then 
$K_p^t=\{ \rho_{\infty}\le t\}$.
By Proposition \ref{prop:concave-rigid},
$K_p^0=K_p(C_*)$.
Obviously 
\[
  \Sigma_p(C_*)= \Sigma_p(C_*)\times 1\subset K_p(C_*)\subset K_p(X).
\]
We show that any point $v$ of $\Sigma_p(C_*)$ is a two-normal 
point of $K_p(C_*)$.
Note that $K_v(K_pX)=\R\times K_v(\Sigma_p X)$ and that 
$K_v(\Sigma_p X)$ has a filtration $\{ K_v^t\}_{0\le t<\infty}$ 
with $K_v^0=K_v(\Sigma_p C_*)$
satisfying the conditions as above.
From the assumption, 
$\Sigma_v(\Sigma_p X)\simeq S^{n-2}$ and $v$ is a topologically nice 
point of $\Sigma_p(X)$.
Thus we can apply the induction hypothesis to the
filtration $\{ K_v^t\}_{0\le t<\infty}$ to conclude that
$v=o_v$ is a two-normal point of $K_v(\Sigma_p C_*)$ and therefore of
$K_p(C_*)$.
   
Now suppose that $o_p$ is a one-normal point of $K_p(C_*)$ with 
a unique direction $\xi\in\Sigma_p(X)$ normal to $\Sigma_p(C_*)$. 
From the previous argument together with
Proposition \ref{prop:concave-rigid},
$\Sigma_{\xi}(\Sigma_p X)$ is 
a double cover of $\Sigma_p(C_*)$. In particular
$\pi_1(\Sigma_p(C_*))=\Z_2$. Furthermore
$\Sigma_p(C_*)$ is a deformation retract of 
$\Sigma_p(X)-\interior B(\xi,\epsilon)$ for a small $\epsilon>0$.
However by the assumption, 
$\Sigma_p(X)-\interior B(\xi,\epsilon)\simeq D^{n-1}$,
a contradiction.
\end{proof}

For a subset $D\subset C_*$, ${\cal N}(D)$ denotes  
the union of all the geodesic rays starting from the points of $D$ 
perpendicularly to $C_*$.
By Propositions \ref{prop:concave-rigid} and \ref{prop:codim1-reg}, 
if $D\subset\interior C_*$ and if any point of $D$ is topologically 
nice in $X$, then ${\cal N}(D)$ is a line-bundle over $D$.\par
\medskip
Now we go back to our situation in \eqref{eq:soul-const}.
For simplicity we put $C:=C(0)$.

Let us first assume that $\partial C$ is nonempty.
For a small $\epsilon >0$, consider the following function
$$
     f_{\epsilon}(x)=d(C_{\epsilon}, x)
$$
on $X-C_{\epsilon}$, where
$$
   C_{\epsilon}=\{ x\in C\,|\,d(\partial C,x)\ge\epsilon\}.
$$
The local regularity property we next study is related with 
the regularity of $f_{\epsilon}$. Note that 
the critical point set of $f_{\epsilon}$
is contained in $\partial C$.

Let $p\in\partial C$ be given.
In view of the convergence 
$(\frac{1}{\delta} X,p) \to (K_p,o_p)$, the 
following lemma is obvious.

\begin{lem} \label{lem:corner}
For any $p\in\partial C$ there exist positive numbers
$\epsilon_p$, $\delta_p$ and $c_p>1$ such that 
$$
   B(p,\delta)\cap\partial C \subset 
      \{\epsilon/2\le f_{\epsilon}\le c_p\epsilon\}
$$
for every $\delta\le\delta_p$ and $\epsilon$ with
$\epsilon/\delta\le\epsilon_p$.
\end{lem}

From the filtration $\{ X^t\}_{t\ge 0}$, we obtain the 
filtration $\{ K_p^t\}_{t\ge 0}$, satisfying the same 
conditions as (1), (2), (3) in the begining of Part \ref{part:alex},
of $K_p$ by totally convex sets,
in a way similar to the proof of Proposition \ref{prop:codim1-reg}.
Let $C_{\infty}$ be the limit of $C$ under the above convergence and 
$$
   C_{\infty\epsilon}=\{ x\in C_{\infty}\,|\,d(\partial C_{\infty}, x)\ge
                           \epsilon \,\},
$$ 
which is the limit of $mC_{\epsilon/m}$.
Note that $C_{\infty}$ does not necessary coincide with $K_p^0$
because of $p\in\partial C$.

We shall also consider the function
\[
  f_{\infty \epsilon}(x)= d(C_{\infty\epsilon},x),
\]
on $K_p-C_{\infty\epsilon}$,
where
$$
   C_{\infty\epsilon}=\{ x\in C_{\infty}\,|\,
                    d(\partial C_{\infty},x)\ge\epsilon\}.
$$

\begin{lem} \label{lem:loc-regular}
Suppose $\dim X=4$ and $\dim C\in\{ 2,3\}$. Then 
for any $p\in\partial C$ and $c\ge c_p$, there exist positive numbers
$\epsilon_{p,c}$, $\delta_{p,c}$ and $\mu_p$ such that 
for every $\delta'<\delta\le\delta_{p,c}$ and $\epsilon$
with $\epsilon/\delta'\le\epsilon_{p,c}$ 
\begin{enumerate}
 \item $(f_{\epsilon}, d_p)$ is regular on 
   $\{\delta'\le d_p\le\delta, \epsilon/2\le f_{\epsilon}\le c\epsilon\}
    -B(\partial C,\epsilon/100);$
 \item $(f_{\epsilon},d_C)$ is regular on 
   $\{ d_p\le\delta, \epsilon/3\le f_{\epsilon}\le 2\epsilon/3,
         0<d_C\le\mu_p\epsilon\};$
 \item $(f_{\epsilon}, d_p, d_C)$ is regular on 
   $\{\delta'\le d_p\le\delta, \epsilon/3\le f_{\epsilon}\le 2\epsilon/3,
         0<d_C\le\mu_p\epsilon\}$.
\end{enumerate}
\end{lem}

\begin{proof}
Under the convergence $(\frac{1}{\delta} X,p) \to (K_p,o_p)$,
$C_{\delta\epsilon}$, $f_{\delta\epsilon}$, 
$d_{p}$ and $d_C$ converge to 
$C_{\infty\epsilon}$, $f_{\infty\epsilon}$, 
$d_{o_p}$ and $d_{C_{\infty}}$ respectively.
Therefore it suffices to show that 
\begin{enumerate}
 \item[$(1)^{\prime}$] $(f_{\infty\epsilon}, d_{o_p})$ is regular on 
   $\{\delta'\le d_{o_p}\le 1, \epsilon/2\le f_{\infty\epsilon}\le c\epsilon\}
    -B(\partial C_{\infty},\epsilon/100);$
 \item[$(2)^{\prime}$] $(f_{\infty\epsilon},d_{C_{\infty}})$ is regular on 
   $\{ d_{o_p}\le 1, \epsilon/3\le f_{\infty\epsilon}\le 2\epsilon/3,
         0<d_C\le\mu_p\epsilon\};$
 \item[$(3)^{\prime}$] $(f_{\infty\epsilon}, d_{o_p}, d_{C_{\infty}})$ 
   is regular on 
   $\{\delta'\le d_{o_p}\le 1, \epsilon/3\le 
   f_{\infty\epsilon}\le 2\epsilon/3, 0<d_{C_{\infty}}\le\mu_p\epsilon\}$.
\end{enumerate}
First suppose $\dim C=3$. We show $(1)^{\prime}$. 
For every $x\in 
\{\delta'\le d_{o_p}\le 1, \epsilon/2\le f_{\infty\epsilon}\le c\epsilon\}
    -B(\partial C_{\infty},\epsilon/100)$,
let $y\in \partial C_{\infty\epsilon}$,
$z\in \partial C_{\infty}$ and  $u\in C_{\infty}$
be  nearest points of $\partial C_{\infty\epsilon}$,
of $\partial C_{\infty}$ and of $C_{\infty}$ respectively  from $x$.
Note that 
\[
   |\tilde\angle o_pxy - \pi/2| < \tau(\delta',\epsilon/\delta').
\]
We must show that 
\begin{equation*}
  \tilde\angle o_pxw >\pi/2+c,\quad
          \tilde\angle yxw >\pi/2+c, 
\end{equation*}
for some point $w$, where $c$ is a uniform positive constant
not depending on $\epsilon$.
Letting $a$ be a point on 
the ray from $o_p$ through $x$ with 
$d(o_p,a)>d(o_p,x)$, we obtain 
$\tilde\angle yxa > \pi/2-\tau(\delta',\epsilon/\delta')$.

We consider the following three cases.

\begin{proclaim}{\emph{Case} 1.}
  $d(z,u)\ge\epsilon/1000$.
\end{proclaim}
  
  Let $b\in\partial C_{\infty}$ and $v\in C_{\infty}$ be such that
$d(o_p,b)=d(o_p,z)+d(z,b)$ and $zbvu$ forms a square in $C_{\infty}$.
Now observe that 
the normal bundle ${\cal N}(\interior C_{\infty})$ over 
$\interior C_{\infty}$ is naturally 
imbedded in $K_p$.
Let $z_1$, $b_1$ and $v_1$ be points in ${\cal N}(\interior C_{\infty})$ 
such that
$uzbvxz_1b_1v_1$ forms a parallelepiped in ${\cal N}(\interior C_{\infty})$. 
Then we have
$$
 \tilde\angle yxb_1 >\pi/2 + c_1, \qquad \tilde\angle o_pxb_1 >\pi/2 + c_1,
$$
for some uniform constant $c_1>0$.
This implies that $(f_{\epsilon},d_{o_p})$ is 
$(c_1,\tau(\delta',\epsilon/\delta'))$-regular at $x$.
\par
 
\begin{proclaim}{\emph{Case} 2.}
  $d(z,u)\le\epsilon/1000$ and $\tilde\angle xyz\ge 1/100$.
\end{proclaim}

Let $x_1$ be the point on $xy\cap {\cal N}(C_{\infty})$
such that $\tilde\angle yux_1=\pi/2$ ($x_1=x$ if $z\neq u$).
Take $v\in {\cal N}(C_{\infty})$ such that  
$d(u,v)=d(u,x_1)+d(x_1,v)$ and
$d(x,x_1)/d(v,x)$ is sufficiently small.
Then 
$\tilde\angle yxv > \pi/2 +c_2$.
Let $w$ be a point such that $w_x^{\prime}$ is a midpoint between
$v_x^{\prime}$ and $a_x^{\prime}$.
Then we have
$$
  \tilde\angle zxw >\pi/2 + c_2^{\prime},  
    \qquad \tilde\angle o_pxw >\pi/2 + c_2^{\prime},
$$
for some uniform constant $c_2^{\prime}>0$.
This implies that $(f_{\epsilon},d_{o_p})$ is 
$(c,\tau(\delta',\epsilon/\delta'))$-regular at $x$.

\begin{proclaim}{\emph{Case} 3.}
  $d(z,u)\le\epsilon/1000$ and $\tilde\angle xyz\le 1/100$.
\end{proclaim}

 Note that $z=u$ in this case. Let $\delta_1$ be a small positive number, 
and let $R_{\delta_1}$ denote the set of $(4,\delta_1)$-strained points 
in $\partial C_{\infty}\cap B(o_p,1)$. By Lemma 1.8 of \cite{Ym:conv},
we have the following sublemma.

\begin{slem} \label{slem:extend}
There is a small $\epsilon>0$ so that 
for every $x\in B(R_{\delta_1},\epsilon)$ and 
$z\in \partial C_{\infty}$ with $d(x,z)=d(x,\partial C_{\infty})$,
there exists a point $v$ satisfying $\tilde\angle zx v>\pi-\delta_1$.
\end{slem}

Now suppose that the required regularity property does not hold
in Case 3.
Then we have a sequence
$\epsilon_i$ of positive numbers tending to $0$ and
a sequence $x_i\in
\{ \epsilon_i/2\le f_{\infty\epsilon_i}\le c\epsilon_i,
\delta'\le d_{o_p}\le 1\}-B(\partial C_{\infty}, \epsilon_i/100)$
such that 
\begin{equation*}
  \tilde\angle o_px_iw \le \pi/2+o_i,\quad\text{\rm or} \quad
         \tilde\angle y_ix_iw \le \pi/2+o_i
\end{equation*}
for any point $w$, where $y_i\in (C_{\infty})_{\epsilon_i}$ and 
$z_i\in \partial C_{\infty}$ denote  nearest points of
$(C_{\infty})_{\epsilon_i}$  and of $C_{\infty}$ from $x_i$
respectively, and $\lim o_i=0$.
Passing to a subsequence, we may assume that 
$(\frac{1}{\epsilon_i} K_p, x_i)$ converges to a pointed 
nonnegatively curved Alexandrov space $(\hat X,x_{\infty})$.
There exists a filtration $\{ \hat X^t\}_{t\ge 0}$
of $\hat X$ by totally convex subsets induced from 
$\{ K_p^t\}_{t\ge 0}$ as before.
Note that  $\hat X^0$ is not necessary the limit of $C_{\infty}$.
Let  $o_{\infty}\in\hat X(\infty)$, 
$y_{\infty}\in\hat X$ and $z_{\infty}\in\hat X$ are the limits of
$o_p$,  $y_i$  and $z_i$ respectively.
If  $d(x_{\infty}, \hat X^0)>0$, letting $t_0$ be such that 
$x_{\infty}\in\partial\hat X^{t_0}$, we can find 
a point $w\in \hat X-\hat X^{t_0}$ such that 
\begin{equation*}
  \tilde\angle o_{\infty}x_{\infty}w \ge \pi/2,\quad
       \tilde\angle y_{\infty}x_{\infty}w \ge \pi/2+c,
\end{equation*}
for some $c>0$.
Now it is possible to take  $w_1$ near $w$ such that  
\begin{equation*}
  \tilde\angle o_{\infty}x_{\infty}w_1 \ge \pi/2+c_1,\quad
     \tilde\angle z_{\infty}x_{\infty}w_1 \ge \pi/2+c_1,
\end{equation*}
for some $c_1>0$. This yields a contradiction.

Thus we may assume that 
$x_{\infty}\in \hat X^0$. Since this implies 
$z_{\infty}\in\interior \hat X^0$, there are
exactly two directions at $z_{\infty}$ normal to
$\hat X^0$ (see Proposition \ref{prop:codim1-reg}). 
Since we may assume that
$\lim \tilde\angle x_iz_iy_i=\pi$, it follows that 
$z_{\infty}$ is a $(4,0)$-straind point.
Therefore $z_i\in R_{\delta_1}$ for large $i$.
From the choice of $\epsilon$, it is possible to take $v_i$ satisfying
$\tilde\angle z_ix_iv_i >\pi/2 + c_3$ for large $i$.
Obviously $\tilde\angle y_ix_iv_i >\pi/2 + c_3/2$.
Taking  $w_i$ as in Case 2, we would obtain the regularity of
$(f_{\epsilon},d_{o_p})$ at $x_i$, a contradiction.

$(2)'$ follows from the existence of ${\cal N}(C_{\infty})$.
$(3)'$ follows from $(2)'$ and the idea of the proof of 
$(1)'$. 
Therefore the details are omitted.

Next suppose $\dim C=2$. In this case, Sublemma \ref{slem:extend}
obviously holds, and similar arguments applies.
Therefore the details are omitted.
\end{proof}

The last local regularity property is a purely metric one
on the tangent cone.
The following Proposition is also important in the proof of 
Theorem \ref{thm:soul}.

\begin{prop} \label{prop:product}
Let $X^4$ be a $4$-dimensional complete Alexandrov 
space with nonnegative curvature, and $C(i)$ as in 
\eqref{eq:soul-const}, $0\le i\le k$.
If $p\in\interior C(i)$ is a topological regular point of $X^4$, then  
$K_p(X^4)$ is isometric to the 
product of $K_p(\interior C(i))$ and a Euclidean cone.
\end{prop}

 Proposition \ref{prop:product} is true for the $3$-dimensional
complete open Alexandrov spaces with nonnegative curvature
(see \cite{SY:3mfd}).

\begin{conj} Proposition \ref{prop:product} is true 
for every complete Alexandrov space with 
nonnegative curvature if $p\in\interior C(i)$ is a topological 
regular point of $X^n$.
\end{conj}

\begin{rem}
Let $X$ be the Euclidean cone over the round sphere $S^n$
of diameter $< \pi$.
For a great sphere $S^{n-1}$ of $S^n$, the subcone $K_0\subset X$
over $S^{n-1}$ is a locally convex set of $X$. 
This shows that Proposition \ref{prop:product} does not hold
for a general locally convex set $S$.
The following example also shows that  
Proposition \ref{prop:product} does not hold for a topological 
singular point $p$.
\end{rem}

\begin{ex}
Let $\gamma$ be the isometric involution on $S^2(1)\times\R^2$
defined by $\gamma(x,t)=(R_{\pi}(x), -t)$, where
$R_{\pi}$ denote a rotation by angle $\pi$.
Then $X^4=S^2(1)\times\R^2/\gamma$ is a complete Alexandrov 
space with nonnegative curvature with soul $S$
isometric to the spherical suspension over $S^1_{\pi}$.
Let $p\in S$ be one of the two topological singular points 
of $X^4$. Then $\Sigma_p(X^4)$ is isometric to the projective 
space of constant curvature $1$, and the conclusion of 
Proposition \ref{prop:product} does not hold in this case.
\end{ex}

For the proof of Proposition \ref{prop:product},
we may assume  $\dim C(i)=2$.
Let $\Sigma_0:=\Sigma_p(\interior C(i))$ and 
$$
 \Sigma_1:=\{ \xi_1\in\Sigma_p(X)\,|\,\angle(\xi_1,\xi_0)=\pi/2\,\,
             \text{for any $\xi_0\in\Sigma_0$} \}.
$$
From the construction of the soul $S$ together with Proposition 
\ref{prop:concave-rigid}, every $\xi\in\Sigma_p(X)$ 
lies on a minimal geodesic from a point of $\Sigma_0$ 
to a point of $\Sigma_1$. Thus every 
element of $K_p(X^4)$ lies on an infinite  flat rectangle
isometric to $[0,\infty)\times [0,\infty)$ spanned by 
some pair of 
elements of $\Sigma_0$ and $\Sigma_1$.

\begin{proof}[Proof of Proposition \ref{prop:product}]
We have to show that $K_p(X^4)$ is isometric to the product 
$K(\Sigma_0)\times K(\Sigma_1)$, in other words,
$\Sigma_p(X)$ is isometric to the spherical 
join $\Sigma_0 * \Sigma_1$. 
In view of the previous observation, 
it suffices to show that there exists a unique minimal geodesic 
from each point 
$\xi_0\in\Sigma_0$ to each point $\xi_1\in\Sigma_1$.
Suppose that there are two minimal geodesics joining $\xi_0$ and 
$\xi_1$ with directions, say $v_0\in\Sigma_{\xi_0}(\Sigma_p(X))$ 
and $w_0\in\Sigma_{\xi_0}(\Sigma_p(X))$.

Let $c(t)$, $0\le t\le \ell=L(\Sigma_p(\interior C(i)))$, be the 
arc-length parameter of $\Sigma_p(\interior C(i))$ with $c(0)=\xi_0$.
The direction $v_0$ 
define a parallel field $v_t$  along $c(t)$ such that
$v_0$, $v_t$ and $c|_{[0,t]}$ span a totally geodesic triangle surface
of constant curvature $1$ with vertices $\xi_0$, $c(t)$ and $\xi_1$.
Similarly $w_0$ define a parallel field $w_t$  along $c(t)$.

First suppose that $v_{\ell}=v_0$ and $w_{\ell}=w_0$.
Then it turns out that the surface 
$\Sigma_{\xi_1}(\Sigma_p(X))\simeq S^2$ 
with curvature $\ge 1$ contains two disjoint 
closed geodesics of length $\ell$, which is impossible.

Thus we may assume that $v_{\ell}\neq v_0$.
Note that $K_{\xi_1}(\Sigma_{p}(X))$ is isometric 
to a product $\R\times C_1$, where $C_1$ is a $2$-dimensional
Euclidean cone over a circle, say $S^1_1$, and the $\R$-factor is 
given by $\Sigma_1$.
Let $\bar v_0$ and $\bar v_{\ell}$ be directions at $\xi_1$
given by the minimal geodesics to $\xi_0$ with directions 
$v_0$ and $v_{\ell}$ at $\xi_0$ respectively.
Let $\bar v_t\in S^1_1$, $0\le t\le \ell$, be the parallel field joining
$\bar v_0$ and $\bar v_{\ell}$ determined 
by the geodesics joining $\xi_0$ to $\xi_1$ with directions
$v_t$.
Let $\bar v_t$, $\ell\le t\le \ell_0$, be a unit speed geodesic joining
$\bar v_{\ell}$ to $\bar v_{0}$ in $S^1_1$,
such that $\{ \bar v_t\}_{0\le t\le\ell_0}$ forms the closed geodesic
$S^1_1$.
Let $\bar c(t)\in\Sigma_p(X)$  denote the point
on the geodesic with direction $\bar v_t$ such that 
$d(\xi_1,\bar c(t))=\pi/2$.
It is easy to see that $\bar c(t)=c(t)$ $\mod \ell\Z$ 
and that $v_{\ell_0}=v_0$. In particular, $\ell_0=n\ell$
for some integer $n\ge 2$.
Thus we have a singular surface $F$ in $\Sigma_p(X)$ 
spanned by the geodesics from $c(t)$ to $\xi_1$ 
with directions $v_t$, $0\le t\le n\ell$.
Let $D$ be a small disk in $F-\Sigma_0$ around $\xi_1$.
Note that $\Sigma_0$ is a deformation retract of $F-\interior D$.
Let $\alpha\in \pi_1(F-\interior D)\simeq\Z$
be the homotopy class defined by $\partial D$,
which is the $n$-th power of a generator of $\pi_1(F-\interior D)$.
Let $c_1$ be a path intersecting $F$ with a unique point
in $\interior D$. From the construction of $F$,
it is possible to join the endpoints of $c_1$  by
a path $c_2$ in $\Sigma_p(X)-F$.
Let $K$ denote the knot defined as the composition 
$c_1$ and $c_2$. 
We set $\Gamma:=\pi_1(\Sigma_p(X)-K)$ and consider the homomorphisms
$$
 \pi_1(F-\interior D) \overset{i_*}\longrightarrow
    \Gamma \overset{\rho}\longrightarrow
      H_1(\Sigma_p(X)-K)=\Gamma/[\Gamma,\Gamma]\simeq \Z,
$$
where $\rho$ is the natural homomorphism.
The above discussion implies that $\rho\circ i_*(\alpha)\equiv 0$
$\mod n\Z$. However since $i_*(\alpha)$ is one of the generators of 
a Wirtinger presentation of $\Gamma$, 
$\rho\circ i_*(\alpha)$ is a generator of $\Z$ (cf.\cite{CrFx:knot}), 
a contradiction.
This completes the proof of the proposition.
\end{proof}



\section{Soul Theorem in dimension four I} \label{sec:soulI}

Throughout this section, let $X$ be a complete open 
Alexandrov space with nonnegative curvature.

\begin{conj} \label{conj:soul}
Suppose in addition that $X$ is topologically nice.
Then there exists a positive number $\epsilon$ such that 
\begin{enumerate}
 \item $X$ is homeomorphic to $\interior B(S,\epsilon);$ 
 \item $B(S,\epsilon)$ is homeomorphic to a disk-bundle 
       over $S$, called the normal bundle of $S$.
\end{enumerate}
\end{conj}

Conjecture \ref{conj:soul}
is certainly true for $n=3$ (\cite{SY:3mfd}).
Theorem \ref{thm:soul} asserts that 
the conjecture is also true for $n=4$.
The purpose of Sections \ref{sec:soulI} and \ref{sec:soulII}
is to prove Theorem \ref{thm:soul}.

The assumption on the topological niceness of $X$ 
is essential in Conjecture \ref{conj:soul}.

\begin{ex} \label{ex:double-cone}
\begin{enumerate}
 \item Let $S(\Sigma)$  and $S_+(\Sigma)$ denote the 
   spherical suspension and the half of the spherical suspension 
   of $\Sigma$ respectively, and 
   let $\Sigma^3$ denote the Poincare homology $3$-sphere.
   Then the gluing $X^5$ of $S_+(S(\Sigma^3))$ and 
   $S(\Sigma^3)\times [0,\infty)$ along their boundaries is a complete 
   open Alexandrov space with nonnegative curvature.
   Note that $X^5$ is homeomorphic to $\R^5$ and 
   the soul of it is a point.
   However for any $\epsilon>0$, 
   $B(S,\epsilon)$ is not homeomorphic to $D^5$ but
   to the closed unit cone $K_1(S(\Sigma^3))$ over $S(\Sigma^3)$.
 \item Let $X^5$ be as in (1) above. Then $S^1\times X^5$
   is a complete 
   open Alexandrov space with nonnegative curvature whose  soul is 
   a circle.
   Note that $S^1\times X^5$ is topologically 
   regular.
   However for any $\epsilon>0$, 
   $B(S,\epsilon)$ is not homeomorphic to $S^1\times D^5$ but
   to $S^1\times K_1(S(\Sigma^3))$.
\end{enumerate}
\end{ex}

Example \ref{ex:double-cone} (1) (resp. (2)) shows that 
the topological niceness assumption in Conjecture \ref{conj:soul} 
cannot be weakened by the assumption that $X$ being a topological 
manifold (resp. being topologically regular) at least 
in dimension $\ge 5$ (resp. $\ge 6$).
Of course, Conjecture \ref{conj:soul} is metric by nature.
The topological version of Conjecture \ref{conj:soul} is

\begin{conj} \label{conj:soul-top}
Let $X$ be a complete open nonnegatively curved Alexandrov space, 
and suppose that $X$ is a topological manifold. Then it is
homeomorphic to {\em a} small neighborhood of a soul $S$ of it which 
has the structure of open disk-bundle over $S$.
\end{conj}

Example \ref{ex:double-cone} tells us that the small neighborhood 
of $S$ required in Conjecture \ref{conj:soul-top} never be a metric one
in general. 

For simplicity we denote by $SC_{n,k}$ 
Conjecture \ref{conj:soul} for $n=\dim X$ and $k=\dim S$.
From Proposition \ref{prop:codim1-reg},
Conjecture \ref{conj:soul} is true in the case of
$\dim S = n-1$:
If $\dim S=n-1$,
then $X$ is isometric to 
the normal bundle $N(S)$ with the canonical metric.

For $\dim S=1$, we have

\begin{lem}
If the conjecture $SC_{n-1,0}$ is true, then 
$SC_{n,1}$ is also true.
\end{lem}

\begin{proof}
Let $X^n$ be an $n$-dimensional complete open Alexandrov space of 
nonnegative curvature which is topologically nice.
Applying the splitting theorem to 
the universal covering
space of $X^n$, we see that $X^n$ is isometric to a quotient
$(\R\times N)/\Z$, where $N$ is an $(n-1)$-dimensional complete open 
Alexandrov space with nonnegative curvature whose soul is a point.
Since $X^n$ is topologically nice, so is $N$. Therefore
the assumption yields $N\simeq \R^{n-1}$.
\end{proof}

Thus for $n=4$, the remaining cases are those of $\dim S=0$ or $2$.

From now on let $X$ be as in Theorem \ref{thm:soul} unless otherwise 
stated.
Let $C:= C(0)$ for simplicity.
In this section, we consider the case of $\dim C=2$.

We begin with the case of $\dim S=2$,
for which the essential case in the proof 
is that $S\simeq S^2$ or $S\simeq P^2$:
If $S$ is homeomorphic to either a torus or a Klein bottle,
then it is flat. Thus the universal covering space 
$\tilde X$ splits as $\R\times N$, where $N\simeq \R^2$, and 
Theorem \ref{thm:soul} certainly holds. 
\par

\begin{thm}   \label{thm:S2}
Theorem  $\ref{thm:soul}$ holds in the case of 
$\dim S=2$ and $\dim C=2$.
\end{thm}

\begin{proof}
Consider
$R_{\delta,r}(S):=S-B(S_{\delta}(S),r)$ for  sufficiently small 
$\delta>0$  and $r>0$ so that one can apply Fibration Theorem 
\ref{thm:orig-cap}
to $R_{\delta, r/2}$. 
Now we consider the convergence, $B(S,\epsilon)\to S$ as
$\epsilon\to 0$. 
Letting $\varphi=\iota:S\to B(S,\epsilon)$ be the inclusion, 
and $\psi:B(S,\epsilon)\to S$ be a measurable map such that 
$d(\psi\varphi x, x)<2\epsilon$ for every $x\in B(S,\epsilon)$, 
we define $f_S:S\to L^2(S)$,
$f_{B(S,\epsilon)}:B(S,\epsilon)\to L^2(S)$
and $f=f_S^{-1}\circ\pi\circ f_{B(S,\epsilon)}:B(S,\epsilon)\to S$
by the same formulae as in \cite{Ym:conv}.
$f$ is $(1-\tau(\epsilon_2,\delta))$-open and 
is an almost Lipschitz submersion(\cite{Ym:conv}).
Let $0<\epsilon_1\ll\epsilon_2\ll r$, 
and denote by $f_{\epsilon_1,\epsilon_2}$ the restriction of 
$f$ to $\interior A(S;\epsilon_1,\epsilon_2)-B(S_{\delta}(S),r)$:
$$
f_{\epsilon_1,\epsilon_2}:
 \interior A(S;\epsilon_1,\epsilon_2)-B(S_{\delta}(S),r)
   \to R_{\delta,r/2}(S).
$$
Since each point of 
$B(S,\epsilon_2)-B(S_{\delta}(S),r)-S$ is almost
regular, it follows from Fibration Theorem 
\ref{thm:orig-cap} (see the final step of the proof in Section 
\ref{sec:cap-lip}) that 
$f_{\epsilon_1,\epsilon_2}$
is a topological submersion.
For any $x\in  A(S;\epsilon_1,\epsilon_2)-B(S_{\delta}(S),r)$,
let $(a_i,b_i)_{1\le i\le 4}$ be a $(4,\delta)$-strainer 
of $X$ at $x$ such that 
\begin{enumerate}
 \item  $(a_i,b_i)_{i=1,2}$ gives a $(2,\delta)$-strainer
   of $S$ at $f(x);$
 \item $d(S,a_3)=d(S,b_3)=d(S,x).$
\end{enumerate}
Then $(a_1)_x'$ and $(a_2)_x'$ are almost orthogonal
to the fibre $F_x:=f^{-1}(f(x))$, and 
$(a_3)_x'$ and $(a_4)_x'$ are almost tangent to $F_x$
(see \cite{Ym:conv}).
Note also that $(a_i)_x'$, $1\le i\le 3$, are almost
tangent to $\{ d_S=d_S(x)\}$.
This shows that $(f,d_S,d_{a_3})$ is a 
$\tau(\epsilon_2,\delta_2)$-open and gives a bi-Lipschitz 
coordinates around $x$. Therefore 
$(f, d_S):A(S;\epsilon_1,\epsilon_2)-B(S_{\delta}(S),r)
\to  R_{\delta,r/2}(S)\times\R$ is a 
topological submersion. Thus
for $\epsilon_1<\epsilon'<\epsilon<\epsilon_2$,
$(f, d_S)$ provides an $S^1$-bundle on 
$A_{\epsilon',\epsilon,r}:= f^{-1}(R_{\delta,r}(S))\cap
\{\epsilon'\le d_S\le\epsilon\}$.
It is clear that the fibre of $f_{\epsilon_1,\epsilon_2}$ 
is homeomorphic to $S^1\times I$.
Letting $\epsilon'\to 0$,
it is now obvious that 
$f_{\epsilon}$, the restriction of $f$ to 
$f^{-1}(R_{\delta,r}(S))\cap \{ d_S\le \epsilon\}$, is 
a locally trivial $D^2$-bundle
whose restriction to 
$R_{\delta,r}(S)$ is the identity. 

\begin{lem}
 For any $p\in S_{\delta}(S)$, there are $\epsilon_p>0$ and 
$r_p>0$ such that for every $r\le r_p$ and $\epsilon$
with $\epsilon/r\le\epsilon_p$,
 \begin{enumerate}
  \item  $B(p,r)\cap B(S,\epsilon)$ is 
         a $D^2$-bundle over $B(p,r;S);$
  \item  $\partial B(p,t)\cap B(S,\epsilon)$ is the union of
         the fibres over $\partial B(p,t;S)$ for every $r/2\le t\le r$.
 \end{enumerate}
\end{lem}

\begin{proof}
Consider the convergence $(\frac{1}{r}X,p) \to (K_p, o_p)$,
$r\to 0$.
By Proposition \ref{prop:product}, 
$K_p$ is isometric to the product of $K_p(S)$ and a flat cone
$N_p$.  Since $B(o_p,1)\cap B(K_p(S),\epsilon)$ is a
$D^2$-bundle over $B(o_p,1;K_p(S))$, the lemma 
follows from the regularity of $(d_{o_p}, d_{K_p(S)})$
on $A(o_p;1/2, 1)\cap (B(K_p(S),\epsilon)-K_p(S))$.
\end{proof}

We are now going to patch the $D^2$-bundle structures on 
$A_{\epsilon,r}:=f^{-1}(R_{\delta,r}(S))\cap\{ d_S\le \epsilon\}$ 
and on $B(p,r)\cap B(S,\epsilon)$ for each 
$p\in S_{\delta}(S)$.

Let $U_{\epsilon,r}(p)$ denote the component of 
$B(S,\epsilon) - A_{\epsilon,2r}$ containing $p$,
and $L_{\epsilon,r }(p)=f_{\epsilon}^{-1}(\partial B(p,2r))$.

\begin{lem} \label{lem:patch}
For small enough $0<\epsilon\ll r$,
   $U_{\epsilon,r}(p)-\interior B(p,r/2)$ is 
   homeomorphic to $L_{\epsilon,r}(p)\times I$.
\end{lem}

\begin{proof}
Suppose the lemma does not hold, and put  $\epsilon:=\epsilon_0r$,
with $\epsilon_0=10^{-10}$.
Under the convergence $(\frac{1}{r}X,p)\to (K_p,o_p)$ as $r\to 0$,
$f_{\epsilon}:B(S,\epsilon)-B(S_{\delta}(S),r)\to R_{\delta,r/2}(S)$
converges to a Lipschitz map
$$
f_{\infty}:B(K_p(S),\epsilon_0)-B(o_p,1)\to K_p(S)-B(o_0,1/2).
$$
It is easy to see that $f_{\infty}$ is the restriction of 
the projection $K_p=K_p(S)\times N_p\to K_p(S)$, where
$N_p$ denotes a flat cone.
Let $U_{\epsilon_0,1}(o_p)$ denote the limit of $U_{\epsilon,r}(p)$
under the above convergence. The obviously
$U_{\epsilon_0,1}(o_p)-\interior B(o_p,1/2)$ is homeomorphic to
$L_{\epsilon_0,1}(o_p)\times I$, where 
$L_{\epsilon_0,1}(o_p)=f_{\infty}^{-1}(\partial B(o_p,2))$.
The lemma follows from Stability Theorem \ref{thm:stability}.
\end{proof}

From the $D^2$-bundle structures on $B(p,r/2)\cap B(S,\epsilon)$
and on $L_{\epsilon,r}(p)$,
we have homeomorphisms 
$\partial B(p,r/2)\cap B(S,\epsilon) \to S^1\times D^2$
and $L_{\epsilon,r}(p) \to S^1\times D^2$.
In view of Lemma \ref{lem:patch}, these provide a gluing homeomorphism 
$h:S^1\times D^2\to S^1\times D^2$. Letting $r\to 0$ with
$\frac{\epsilon}{r}\ll 1$,
we conclude that $h|_{S^1\times S^1}$ is homotopic, and hence
isotopic, to the identity.
Therefore we can patch the $D^2$-bundle structures 
on $B(p,r/2)\cap B(S,\epsilon)$
and on $\partial U_{\epsilon,r}(p)$ 
to get a $D^2$-bundle structure on $B(S,\epsilon)$.
This completes the proof of Theorem \ref{thm:S2}.
\end{proof}

Next we consider the case of $S\neq C$.
Recall that the critical point set of $f_{\epsilon}$
is contained in $\partial C$, where 
\[
     f_{\epsilon}(x)=d(C_{\epsilon}, x).
\]
Therefore by \cite{Pr:alex2}, 
there is an $f_{\epsilon}$-gradient flow $\psi$ on 
$X-\partial C- C_{\epsilon}$.

\begin{defn} \label{def:grad-nbd}
Let $\dim X=n$, $\dim C=k$, and $A\subset\partial C$. 
We say that a subset $V$ with $A\subset V$ is 
is an  {\em $f_{\epsilon}$-pseudo-gradient
normal bundle} over $A$  with respect to $\psi$ if 
the following conditions are satisfied:
\begin{enumerate}
 \item $V$ is a $D^{n-k+1}$-bundle over $A$ with projection, say
     $\pi:V\to A;$
 \item Let $\partial V\subset V$ denote the total space of the 
     $S^{n-k}$-bundle induced from the $D^{n-k+1}$-bundle of $(1)$.
     Then  $\partial V$ is a gluing of two $D^{n-k}$-bundles, say
     $J_i$,$i=0,1$, and an $I\times S^{n-k}$-bundle, say $K$, 
     over $A$ with projections
     $\pi|_{J_i}$ and $\pi|_{K}$ such that 
  \begin{enumerate}
   \item  there is a neighborhood $W$ of every $p\in A$ in 
     $A$ and a locally trivializing homeomorphism 
     $h:\pi^{-1}(W)\to W\times I^{n-k+1}$ with $\pi=p_1\circ h$, 
     where $p_1:W\times I^{n-k+1}\to W$ is the projection, inducing
     following homeomorphisms for $i=0,1$
      \begin{gather*}
       (\pi|_{J_i})^{-1}(W)\simeq W\times(i\times I^{n-k}), \\
       (\pi|_{K})^{-1}(W)\simeq W\times(I\times\partial I^{n-k}).
      \end{gather*}
   \item $J_1\subset \{ f_{\epsilon}=\epsilon/2\}$,
     $J_0\subset \{ f_{\epsilon}= c\epsilon\}$ for some constant
     $c>1$ and every fibre of $\pi|_{K}$ gives a flow curve of 
     $\psi;$
   \item the flow curves of $\psi$ are transversal to $J_i$
  \end{enumerate}
\end{enumerate}
The $f_{\epsilon}$-pseudo-gradient normal bundle $V$ has 
{\em height $\nu>0$} if
$d_C$ restricted to $J_1$ takes the maximal value $\nu$
at every point of $\partial J_1$.
We call the subbundles $J_i$ and  $K$,  $J_i$-part and  $K$-part of $V$
respectively.
\end{defn}

Note that if we are given an $f_{\epsilon}$-pseudo-gradient
normal bundle over $\partial C$  with respect to $\psi$,
then parturbing and extending $\psi$ along the fibres of the
normal bundle, we can define a 
new $f_{\epsilon}$-pseudo-gradient flow 
defined on $X-C_{\epsilon}$
(see Section 10 of \cite{SY:3mfd} for the definition of 
pseudo-gradient flows).

\begin{thm} \label{thm:soul1}
Let $\dim X=4$.
If $\dim C=2$ and if $\dim S=0$, then Theorem \ref{thm:soul} holds.
\end{thm}

For the proof of Theorem \ref{thm:soul1},
we apply a method similar to one used in \cite{SY:3mfd},
which actually gives a simplification of the argument there.

\begin{prop} \label{prop:flow}
Under the same hypothesis as Theorem \ref{thm:soul1},
for some $\mu>0$ and 
any small enough $\epsilon>0$, there exists an  
$f_{\epsilon}$-pseudo-gradient normal bundle over $\partial C$  
of height $\mu\epsilon$ with respect to some
$f_{\epsilon}$-gradient flow.
\end{prop}

\begin{proof}
Applying Lemma \ref{lem:loc-regular} (3), we can take 
finitely many consecutive points $p_1,\ldots,p_N$ of $\partial C$
with $d(p_j, p_{j+1})$ small enough such that 
there is an $f_{\epsilon}$-pseudo-gradient normal $D^3$-bundle
$E$ over $\{ p_1,\ldots,p_N\}$ of height 
$\mu\epsilon$ with respect to some
$f_{\epsilon}$-gradient flow.
Let $J_1(\widehat{p_jp_{j+1}})$ denote the $D^2$-bundle 
over the arc $\widehat{p_jp_{j+1}}\subset \partial C$
extending $J_1(p_j)\cup J_1(p_{j+1})$
with
$J_1(\widehat{p_jp_{j+1}})\subset f_{\epsilon}^{-1}(\epsilon/2)\}$
and of ``height $\mu\epsilon$'.
Let $A_j$ denote the union of the $f_{\epsilon}$-flow
curves contained in $\{ \epsilon/2\le f_{\epsilon}\le c\epsilon\}$
and through the total space of the $S^1$-bundle induced from 
$J_1(\widehat{p_jp_{j+1}})$. 
Let $J_0(p_j)$ denote the $J_0$-part of $E|_{p_j}$.
Since 
$(A_j\cap\{ f_{\epsilon}= c\epsilon\})\cup J_0(p_j)\cup J_0(p_{j+1})$
is homeomorphic to $S^2$ and is locally flat , 
it bounds a domain $B_j$
in $\{ f_{\epsilon}= c\epsilon\}$ homeomorphic to $D^3$.
Now it is clear that 
\[
  J_0(p_j)\cup J_0(p_{j+1})\cup J_1(\widehat{p_jp_{j+1}})
   \cup A_j\cup B_j
\]
is homeomorphic to $S^3$ and is locally flat, 
and therefore it bounds a domain homeomorphic to $D^4$,
which defines a structure of an $f_{\epsilon}$-pseudo-gradient 
normal $D^3$-bundle over $\widehat{p_jp_{j+1}}$.
This completes the proof.
\end{proof}

\begin{proof}[ Proof of Theorem \ref{thm:soul1}]
It follows from Proposition \ref{prop:flow} that 
\begin{equation*}
   X \simeq \{ f_{\epsilon} < \epsilon/2\}
     \simeq \{ f_{\epsilon} < \nu\},
\end{equation*}
for any $\nu\ll\epsilon$. We have to 
prove $\{ f_{\epsilon} \le \nu\}\simeq D^4$.
Using the $f_{\epsilon}$-flow curves, it is easy to see
that $\{ f_{\epsilon}\le\nu\}$ is contractible.
By Freedman's celebrated work (see\cite{FQ:4mfd}),
it suffices to show $\{ f_{\epsilon} = \nu\}\simeq S^3$.
Put $L:=C\cap\{ f_{\epsilon} = \nu\}\simeq S^1$, and 
consider the distance function $d_L$.
Since $(f_{\epsilon},d_L)$ is regular near
$\{ f_{\epsilon} = \nu, 0<d_L \le\lambda \}$
with $\lambda\ll\nu$, it follows from 
a method similar to Proposition \ref{prop:flow}
that 
$\{ f_{\epsilon} = \nu, d_L\le\lambda\}\simeq S^1\times D^2$.
For simplicity, we put
$C^*_{\mu}:=\{ f_{\epsilon} \le \mu\}\cap C$ with
$\mu\le\epsilon/2$.
By a method similar to Theorem \ref{thm:S2},
$\{ f_{\epsilon} \le \epsilon/2, d_C\le 2\nu\}$ is a 
$D^2$-bundle over $C^*_{\epsilon/2}$ if 
$\nu\ll\epsilon$. Let 
$\pi:\{ f_{\epsilon} < \epsilon/2, d_C\le 2\nu\}\to
C^*_{\epsilon/2}$ be the projection.

\begin{slem}
$$
\pi^{-1}\left(C^*_{\nu-\lambda^2/\nu}\right)\cap
     \{ f_{\epsilon}=\nu\}
        \simeq D^2\times S^1.
$$
\end{slem}

\begin{proof}
Note that each point of 
$\{ f_{\epsilon} \le \epsilon/2, 0<d_C\le 2\nu\}$ is 
almost regular. For
any point $x$  of $C^*_{\nu-\lambda^2/\nu}$ and 
any point $y$ of 
$\pi^{-1}\left(C^*_{\nu-\lambda^2/\nu}\right)
\cap\{ f_{\epsilon}=\nu\}$,
let $\theta$ be the angle at $y$ between 
$(C_{\epsilon})_y'$ and the fibre  $\pi^{-1}(x)$.
Since $\pi/2-\theta$ is bounded from below by a uniform constant
depending only on $\lambda/\nu$ and since 
$\pi$ is almost Lipschitz submersion, it follows that
$\pi^{-1}(x)\cap\{ f_{\epsilon}=\nu\}$ is a circle.
\end{proof}

By using the $d_L$-flow curves, we obtain
\begin{align*}
   \{ f_{\epsilon}=\nu\}- & \pi^{-1}(C^*_{\nu-\lambda^2/\nu})\\
        & \simeq  \{ f_{\epsilon}=\nu, d_L\le\lambda\}\\
        & \simeq S^1\times D^2.
\end{align*}
Thus we conclude 
$\{ f_{\epsilon}=\nu\}\simeq S^1\times D^2\cup D^2\times S^1
\simeq S^3$.
\end{proof}

\begin{thm}\label{thm:soul2}
Conjecture \ref{conj:soul}  is true in the case when 
$\dim C=1$ and $\dim S=0$, where $\dim X$ be arbitrary.
\end{thm}

The proof of Theorem \ref{thm:soul2} is 
similar to the 3-dimensional case, and hence omitted
(see Section 13 of \cite{SY:3mfd} for details).

\section{Soul Theorem in dimension four II}
                     \label{sec:soulII}

In this section, we shall consider the case of $\dim C=3>\dim S$.\par
We also assume $\dim X=4$.

\begin{thm} \label{thm:codim1-C}
Let $X$ be a $4$-dimensional complete open 
Alexandrov space with nonnegative curvature, and suppose that it is  
topologically regular.
If $\dim C=3>\dim S$,
then $X$ is homeomorphic to 
${\cal N}(\interior C)$.
\end{thm}

\begin{thm} \label{thm:norm-nbd}
Under the same hypothesis as Theorem \ref{thm:codim1-C},
for some $\mu>0$ and 
any small enough $\epsilon>0$, there exists an  
$f_{\epsilon}$-pseudo-gradient normal bundle over $\partial C$  
of height $\mu\epsilon$ with respect to some
$f_{\epsilon}$-gradient flow.
\end{thm}

\begin{proof}[Proof of Theorem \ref{thm:codim1-C} assuming
Theorem \ref{thm:norm-nbd}]
Let $U$ be an $f_{\epsilon}$-pseudo-gradient
normal bundle of $\partial C$  with respect to some
$f_{\epsilon}$-gradient flow $\psi$.
It is possible to deform $\psi$ in $U$ along the fibres of 
$\pi:U\to\partial C$ to obtain an $f_{\epsilon}$-pseudo-gradient 
flow on $X-C_{\epsilon}$. Therefore we have
\[
  X\simeq \{ f_{\epsilon}<\epsilon/2\} \simeq 
   {\cal N}(\interior C).
\]
\end{proof}

First we construct an $f_{\epsilon}$-pseudo-gradient
normal bundle over a small neighborhood of any
point $p$ of $\partial C$.

Let $c_p\ge 1$, $\epsilon_{p,c}$, $\delta_{p,c}$ and $\mu_p$ 
be as in Lemmas \ref{lem:corner} and \ref{lem:loc-regular}.

\begin{lem} \label{lem:loc-n-bdle}
For any $p\in\partial C$ and $c\ge c_p$,
let $\epsilon$, $\delta'<\delta$  be constants as in 
Lemma \ref{lem:loc-regular}. Then 
there is an $f_{\epsilon}$-pseudo-gradient normal bundle over 
$\partial C\cap B(p,\delta)$ of height $\mu_p\epsilon$ .
\end{lem}

\begin{proof}
By Lemma \ref{lem:loc-regular}, we have an 
$f_{\epsilon}$-gradient flow $\psi_p$
on $A(\partial C;\epsilon/100,\epsilon/10)$
preserving $d_p$ on $A(p;\delta',\delta)$. 
Put $\Delta^2:=\partial C\cap B(p,\delta')\simeq D^2$, and
$F:=\partial\Delta^2=\partial C\cap\partial B(p,\delta')$.
We first construct an  $f_{\epsilon}$-pseudo-gradient
normal bundle $V$ over $F$
with respect to $\psi_p$ such that $V\subset\partial B(p,\delta')$.
Since $\dim F=1$, using the regularity of $(f_{\epsilon}, d_{p})$
we can apply the method of the proof of Theorem \ref{thm:soul1}
(compare the gluing technique used in Section 12 of \cite{SY:3mfd})
to $F\subset\Delta^2\subset\partial B(p,\delta')$, 
and obtain the required normal bundle $V$ over $F$ such that
\begin{enumerate}
 \item $V\subset\partial B(p,\delta');$
 \item $V$ has height $\mu_p\epsilon$.
\end{enumerate}
Using the cone structure of $B(p,\delta)$ from $p$, 
we extend $V$ to an $f_{\epsilon}$-pseudo-gradient
normal bundle $W$ over $\partial C\cap A(p;\delta',\delta)$
with respect to $\psi_p$.

Note that 
$\partial J_1$ bounds two disjoint domains $D_1$ and $D_2$ 
of $\{f_{\epsilon}=\epsilon/2, d_C=\mu_p\epsilon\}$ homeomorphic 
to $D^2$.
Since $\partial J_1\cup D_1\cup D_2$ is homeomorphic to $S^2$ and 
is locally flat, it bounds a 
domain $E$ of $\{f_{\epsilon}=\epsilon/2, d_C\le\mu_p\epsilon\}$
homeomorphic to $D^3$.
Let $G$ denote the union of all the $\psi_p$-flow curves in 
$\{\epsilon/2\le f_{\epsilon}\le c\epsilon\}$ meeting 
$D_1\cup D_2$.
Note $G\simeq D^2\times (I\times\partial I)$.

\begin{slem} \label{slem:part-in-cell}
$\partial(V\cup E\cup G)$ is contained in a domain of
$f_{\epsilon}^{-1}(c\epsilon)$ homeomorphic to $\R^3$.
\end{slem}

\begin{proof}
Under the convergence
$(\frac{1}{\delta} X,p) \to (K_p,o_p)$, the level set
$\{ f_{\epsilon=\epsilon'\delta}=c\epsilon\}$ converges to 
$\{ f_{\infty,\epsilon'}=c\epsilon'\}$.
Since $\Sigma_p\simeq S^3$ is a topological manifold, 
the sublemma obviously follows.
\end{proof}

Since $V\cup E\cup G\simeq D^3$ and  
$\partial(V\cup E\cup G)$ is locally flat, it follows from 
Sublemma \ref{slem:part-in-cell} that $\partial(V\cup E\cup G)$
bounds a domain $H$ 
in $f_{\epsilon}^{-1}(c\epsilon)$ homeomorphic to $D^3$.
Since $V\cup E\cup G\cup H$ is homeomorphic to $S^3$ and is locally
flat, it bounds a domain  $K$ of $X$ homeomorphic to $D^4$.
Therefore the $D^2$-bundle $W$ over $\partial C\cap A(p;\delta',\delta)$
extends to a required $D^2$-bundle structure on $K$ over 
$\partial C\cap B(p;\delta)$.
\end{proof}

\begin{proof}[Proof of Theorem \ref{thm:norm-nbd}]
We are going to glue those local pseudo-gradient normal bundles to 
construct a pseudo-gradient normal bundle over $\partial C$.

By the compactness of $\partial C$, we have
\[
    c:=\sup_{p\in\partial C} c_p >0, \qquad 
        \mu:=\sup_{p\in\partial C} \mu_p >0.
\]
By Lemmas \ref{lem:loc-regular} and \ref{lem:loc-n-bdle},
for any $p\in\partial C$ and $c$ as above, there are 
$\delta_{p,c}>0$ and an 
$f_{\epsilon}$-pseudo-gradient normal bundle $U_p$ over
$\partial C\cap B(p,\delta_{p,c})$ of height $\mu\epsilon$, 
where $\epsilon$ is any 
sufficiently small positive number. 
Choose $p_1,\ldots,p_N$ of $\partial C$ such that 
$\{ B(p_i,\delta_i)\}_{1\le i\le N}$ covers $\partial C$.
Let $K$ be a sufficiently fine triangulation of $\partial C$ by 
Lipschitz curves with $\{p_i\}_{i=1}^N\subset K^0$ all of whose 
simplices of $K$ have diameters
less than $\min_{1\le i\le N} \delta_{p_i,c}$, where 
$K^j$ denotes the $j$-skeleton of $K$, $0\le j\le 2$.
Take $\delta_{p_i,c}'$ small enough compared with 
$\min_{1\le i\le N} d(p_i, K^0-\{ p_i\})$.
With the above choice of $p_i$ and $\delta_{p_i,c}'<\delta_{p_i,c}$,
take $\epsilon$  small enough compared with 
$\min \delta_{p_i,c}'$.
Now $U_{p_i}$ is an $f_{\epsilon}$-pseudo-gradient normal bundle 
over $\partial C\cap B(p,\delta_{p_i,c})$ with respect to some 
$f_{\epsilon}$-gradient flow $\psi_{p_i}$ on 
$X-\partial C-C_{\epsilon}$.
By an argument similar to that of Lemma \ref{lem:loc-n-bdle},
it suffices to construct to an 
$f_{\epsilon}$-pseudo-gradient normal bundle 
over $K^1$ of height $\mu\epsilon$.

Let us suppose that we have already constructed an 
$f_{\epsilon}$-pseudo-gradient normal bundle $U$ over 
$L\subset K^1$ of height $\mu\epsilon$ such that 
$U$ restricted to a small neighborhood of each vertex of
$L$ is defined by the restriction of the bundle 
$U_{p_i}$ for some $p_i$.
Let $\sigma\in K^1-L$. We shall extend $U$ to an
$f_{\epsilon}$-pseudo-gradient normal bundle $V$ over 
$L\cup\sigma$ of height $\mu\epsilon$.
Suppose $\sigma\subset B(p_1,\delta_1)\cap B(p_2,\delta_2)$.
Let $x_0$ and $y_0$ be the endpoints of $\sigma$, and take
$x_1, y_1, y_2\in\interior\sigma$ with
$x_0<x_1<y_1<y_2<y_0$. Suppose 
$\sigma\cap L$ is nonempty and consists of a point, say $x_0$.
The other cases are similar and hence omitted.
We may assume that $U=U_{p_1}$ on a neighborhood of $x_0$ in $L$.
Let $V_1$ and $V_2$ be the $f_{\epsilon}$-pseudo-gradient normal bundles
over $\sigma|_{[x_0,x_1]}$  and $\sigma|_{[y_1,y_0]}$ 
of height $\mu\epsilon$ determined by 
the restriction of the bundles $U_{p_1}$ and $U_{p_2}$ respectively. 

We are going to glue $V_1$ and $V_2$.
Let $W_1\supset V_1$ and $W_2\supset V_2$ be 
$f_{\epsilon}$-pseudo-gradient normal bundles
over normal (tubular) neighborhoods of $\sigma|_{[x_0,x_1]}$  
and $\sigma|_{[y_1,y_0]}$ of height $\mu\epsilon$ determined by 
$U_{p_1}$ and $U_{p_2}$ respectively. 
Take an $f_{\epsilon}$-gradient flow $\psi_{p_1p_2}$ on 
$X-\partial C-C_{\epsilon}$ such that 
it coincides with $\psi_{p_j}$ on $W_j$, $j=1,2$.
Let $\tilde V_1$ be the $f_{\epsilon}$-pseudo-gradient normal bundle
over $\sigma|_{[x_0,y_1]}$ of height $\mu\epsilon$ determined by 
$U_{p_1}$ and extending $V_1$.
Let $\tilde W_1\supset W_1$ be the $f_{\epsilon}$-pseudo-gradient 
normal bundle over a normal (tubular) neighborhood of 
$\sigma|_{[x_0,y_1]}$ in 
$\partial C$ determined by the restriction of 
$U_{p_1}$. Let $H:=\tilde W_1\cap \{ f_{\epsilon}=c\epsilon\}$.
By definition,
\[   
      \tilde W_1 \simeq H\times I,
\]
via a homeomorphism induced from $(U_{p_1}, \psi_{p_1})$,
where $I:=[\epsilon/2,c\epsilon]$.
Consider the family 
${\cal S}:=\{ H\times t\,|\, t\in I\}$.
Let $H'\subset H$ and $I'\subset I$
be such that $H'\times I'$ provides a small neighborhood of 
$\partial C\cap \tilde W_1$.
Note that any $f_{\epsilon}$-pseudo-gradient flow $\psi$
defines an embedding $f_{\psi}:H\times I\to H_0\times I$
preserving ${\cal S}$ on $H\times I-H'\times I'$,
where $H_0$ is an open set of $\{ f_{\epsilon}=c\epsilon\}$
containing the closure of $H$.
By using the topological Morse theory in 
\cite{Si:stratified}, one can slightly deform $\psi$ (on 
$H'\times I'$) to an
$f_{\epsilon}$-pseudo-gradient flow $\psi'$
such that $f_{\psi'}$ preserves ${\cal S}$ on $H\times I$.
Consider the $J_0$-subbundles in $\{ f_{\epsilon}=c\epsilon\}$, 
denoted $J_0(\tilde V_1)$ and $J_0(V_2)$, of $\tilde V_1$ and $V_2$ 
respectively. Take a rectangle $R$ in $\{ f_{\epsilon}=c\epsilon\}$
by which  $J_0(\tilde V_1)$ and $J_0(V_2)$ are connected in such a 
way that 
\begin{enumerate}
 \item $R\cap J_0(\tilde V_1)=\pi_{\tilde V_1}^{-1}(y_1)
            \cap J_0(\tilde V_1);$
 \item  $R\cap J_0(V_2)=\pi_{V_2}^{-1}(y_2)
            \cap J_0(V_2);$
 \item $\hat H:=J_0(\tilde V_1)\cup R\cup J_0(V_2)\simeq I^2$,
\end{enumerate}
where $\pi_{\tilde V_1}$ and $\pi_{V_2}$ denote the bundle 
projections of $\tilde V_1$ and $V_2$ respectively.
The union $\hat V$ of all the flow curves of $\psi_{p_1p_2}'$, 
a deformation of $\psi_{p_1p_2}$ as above, through $\hat H$ 
in $\{ \epsilon/2\le f_{\epsilon}\le c\epsilon\}$ provides 
a gluing of $V_1$ and $V_2$.
Taking a smaller $\mu>0$ if necessary, we finally obtain the required
$f_{\epsilon}$-pseudo-gradient normal bundle $V\subset \hat V$
over $L\cup\sigma$ of height $\mu\epsilon$.
This completes the proof of Theorem \ref{thm:norm-nbd}.
\end{proof}

\begin{thm} \label{thm:soul-C3}
Theorem $\ref{thm:soul}$ holds in the case of $\dim C=3$.
\end{thm}

\begin{proof}
This immediately follows from Theorem \ref{thm:codim1-C}
and the following proposition.
\end{proof}

\begin{prop}
Let $C$ be a $3$-dimensional compact Alexandrov space with 
nonnegative curvature and with
boundary. If $C$ is a topological manifold,
then $C$ is homeomorphic to the normal bundle of the soul of $C$.
\end{prop}

\begin{proof}
Take a small $\epsilon>0$ with $C_{\epsilon}\simeq C$
(Theorem \ref{thm:collar}). Then one can apply the method of 
\cite{SY:3mfd} to obtain that $C_{\epsilon}$
is homeomorphic to the normal bundle of the soul of $C$.
\end{proof}


\section{The classification of nonnegatively curved
Alexandrov three-spaces with boundary} \label{sec:alexwbdy}

In this section, we assume $X^n$ to be 
an $n$-dimensional  complete  nonnegaitvely curved Alexandov space 
with boundary, and give a classification of such a space in 
dimension three.

\begin{prop} \label{prop:codim1}
Suppose that $X^n$ has nonempty boundary.
If $\dim S = n-1$ and if $X^n$ is a topological manifold, 
then $X^n$ is isometric to 
either $S\times [0,\infty)$  or an $I$-bundle over $S$
for some closed interval $I$.
\end{prop}

\begin{proof}
By Proposition \ref{prop:codim1-reg}, for any point $p\in S$,
\begin{enumerate}
 \item $\Sigma_p(X)$ is isometric to the spherical suspension 
       over $\Sigma_p(S);$
 \item for the directions $\xi_{\pm}\in\Sigma_p(X)$ perpendicular
       to $\Sigma_p(S)$ satisfying $\angle(\xi_+, \xi_-)=\pi$, there exist
       maximal geodesics 
       $\gamma_{\pm}:[0,\ell_{\pm})\to\interior X$ with 
       $\dot\gamma_{\pm}(0)=\xi_{\pm}$.
\end{enumerate}
Note that $\ell_{\pm}$ does not depend on the particular choice 
of $p\in S$,
and that $\gamma_{\pm}(\ell_{\pm})$ (if $\ell_{\pm}<\infty$)
belongs to $\partial X$.
Proposition \ref{prop:codim1-reg} then implies that 
if the normal bundle $N(S)$ is nontrivial, then 
$X^n$ is isometric to a twisted product of $S$ and $I$
for some closed interval $I$, 
and that if $N(S)$ is trivial, then 
$X^n$ is isometric to either $S\times I$ or $S\times [0,\infty)$.
\end{proof}

\begin{prop} \label{prop:dim1}
Suppose that $X^n$ has nonempty boundary.
If $\dim S=1$, then $X^n$ is isometric to a quotient 
$(\R\times X_0^{n-1})/\Lambda$, where $\Lambda\simeq\Z$ and 
$X_0^{n-1}$ is a complete, contractible Alexandrov space with 
nonnegative curvature and with boundary.
Topologically, $X^{n}$ is a $X_0^{n-1}$-bundle over $S^1$.
\end{prop}

\begin{proof}
Just apply the splitting theorem to the universal cover 
of $X^n$.
\end{proof}

\begin{thm} \label{thm:cptcomp}
For a complete nonnegatively curved Alexandrov space 
$X^n$, we have the following splitting$:$
\begin{enumerate}
 \item If $\partial X^n$  is disconnected,
       then $X^n$ is isometric to a product $X_0\times I$,
       where $X_0$ is a component of $\partial X;$
 \item If $\partial X^n$  is compact and connected and if 
       $X^n$ is noncompact,then  $X^n$ is isometric to the
       product $\partial X^n\times [0,\infty)$.
\end{enumerate}
\end{thm}

In the Riemannian case, Theorem \ref{thm:cptcomp} was 
proved in \cite{BZ}. 
However it seems to the author
that the method used in \cite{BZ} cannot be directly 
applied for the proof of Theorem \ref{thm:cptcomp}(1).
We make use of
the notion of $1$-systole, instead. For a non-simply connected
space $Y$, let $\sys_1(Y)$ denote the $1$-systole of $Y$,
the infimum of the lengths of non-null homotopic loops
in $Y$.

\begin{prop}
Let $X^n$ be a noncompact, non-simply connected, complete
Alexandrov space with nonnegative curvature, and let
$S$ be a soul of $X^n$. Then 
$$
   \sys_1(X^n) =\sys_1(S).
$$
In particular $\sys_1(X^n)>0$.
\end{prop}

\begin{proof}
The basic idea goes back to Sharafutdinov \cite{Shr:Pogorelov}.
For any non-null homotopic loop $\gamma$ in $X$,
take a large compact totally convex set $C$ such that 
\begin{enumerate}
 \item $C\supset \gamma;$
 \item there is a distance-decreasing retraction 
       of $C$ onto $S$ (the Sharafutdinov
       retraction  constructed in \cite{Pr:alex2}),
       $R: C\times [0,1]\to C$
       with $R(\,\cdot\,,0)=\mbox{identity}$, $R(\,C,1)=S$.
\end{enumerate}
Then obviously, $R_1(\gamma)$ is a non-null homotopic loop in $S$  
satisfying $L(R_1(\gamma))\le L(\gamma)$.
\end{proof}

\begin{lem} \label{lem:positive}
Let $X^n$ be a complete
nonnegatively curved Alexandrov space with disconnected boundary,
and $X_0$ and $X_1$ be any distinct components 
of $\partial X$. Then we have $d(X_0, X_1)>0$.
\end{lem}

\begin{proof} 
We may assume that $X^n$ is noncompact.
Let us consider the double $D(X)$.
Let $c$ be any path from a point of $X_0$ to a point of $X_1$, 
and $D(c)\subset D(X)$ the double of $c$.
Then $D(c)$ is non-null homotopic in $D(X)$ and therefore 
$L(D(c))\ge\sys_1(S)>0$ for a soul $S$ of $D(X)$.
This shows that $d(X_0, X_1)\ge\sys_1(S)/2$.
\end{proof}

\begin{proof}[Proof of Theorem \ref{thm:cptcomp}]
The essential part is (1).
Suppose that $\partial X$ is disconnected, and
let $X_0$  and $X_1$ be distinct components of $\partial X$,
and consider the functions $f_i=d(X_i,\, \cdot\,)$, which are
concave.
Put $X_i^t := f_i^{-1}([0,t])$.
Let $t_0$ denote the supremum of those $t$ with $d(X_0^t, X_1^t)>0$.
Lemma \ref{lem:positive} ensures $t_0>0$.
Consider the set 
$C_* := f_0^{-1}([t_0,\infty))\cap f_1^{-1}([t_0,\infty))$.
In what follows, we investigate the geometric properties of $C_*$.
Note that $C_*$ is a nonempty closed totally convex subset.
Note also that $C_{**}:=f_0^{-1}(t_0)\cap f_1^{-1}(t_0)$
is nonempty.

We claim that for any $x\in C_{**}$ and $x_i\in X_i$ with 
$d(x, x_i)=d(x, X_i)$, $i=1,2$, we have 
$\angle x_0 x x_1=\pi$.
For if $\angle x_0 x x_1 < \pi$, then $d(x_0, x_1)<2t_0$.
Letting $y$ be the midpoint of a minimal geodesic joining 
$x_0$ and $x_1$, we would have $y\in X_0^{t_1}\cap X_1^{t_1}$
for some $t_1< t_0$, a contradiction.

Next we show that $\dim C_*\le n-1$.
Let $x\in C_{**}$ and $x_i\in X_i$ satisfy
$d(x, x_i)=d(x, X_i)$, $i=1,2$.
If we set $\Sigma^{n-2}$ to be the set of $\xi\in \Sigma_x$ with 
$\angle((x_0)_x', \xi)=\angle((x_1)_x', \xi)=\pi/2$.
Then the above claim implies that $\Sigma_x$ is the 
spherical suspension over $\Sigma^{n-2}$.
By the first variation formula, for any $p\in C_*$ we have
\begin{equation}
     p_x'\in\Sigma^{n-2}, \label{eq:direC}
\end{equation}
 and hence $\Sigma_x(C_*)\subset\Sigma^{n-2}$
showing $\dim C_*\le n-1$.

We now show that $C_{*}=C_{**}$.
For any $x\in C_{**}$ and $p\in C_*$, let $\gamma:[0,1]\to X$ be a 
minimal geodesic
from $x$ to $p$. \eqref{eq:direC} yields that 
$f_i(\dot\gamma(0))=0$, $i=1,2$.
It follows from the concavity of $f_i$ that
$f_i(\gamma(t))\equiv t_0$ and $p\in C_{**}$.

Now for any $x\in X_0$, take a $x'\in X_1$ satisfying
$d(x, X_1)=d(x,x')$.
Let $y=\varphi(x)$ be a point on $xx'$ with $f_0(y)=f_1(y)$.
By the choice of $t_0$, $f_0(y)\ge t_0$ and $f_1(y)\ge t_0$.
It follows that $y\in C_*=C_{**}$.
Let $x_0\in X_0$ be a point of $X_0$ with 
$d(y,x_0)=d(y,X_0)$. Then 
$\angle x'yx_0=\pi$ and therefore $x=x_0$.
This also shows that  $y=\varphi(x)$ is uniquely determined by
$x$ and that $\varphi:X_0\to C_*$ is injective.
Thus one can conclude that $\dim C_*=n-1$.

For every two points $x_0$ and $x_0'$ of $X_0$ take
$x_1, x_1'\in X_1$ with  $d(x_0, x_1)=d(x_0', x_1')=2t_0$.
By Proposition \ref{prop:concave-rigid}, $x_0, x_1, x_1', x_1$ 
span a totally geodesic flat rectangle,
concluding that $X$ is isometric to $X_0\times [0, 2t_0]$.

Next suppose that $\partial X$ is connected and compact
and that $X$ is noncompact.
Let $\gamma:[0,\infty)\to X$ be a ray starting from a point of 
$\partial X$ and 
consider the Busemann function $b_{\gamma}$ associated with 
$\gamma$. For any sufficiently large $a\gg 1$,
the nonnegatively curved Alexandrov space 
$W:= b_{\gamma}^{-1}((-\infty, a])$ has disconnected boundary.
It follows from the previous argument that 
$W$ is isometric to $\partial X\times [0, a]$.
Letting $a\to\infty$ completes the proof.
\end{proof}

First we recall and reconstruct some  
$3$-dimensional complete open 
Alexandrov spaces with nonnegative curvature which are not  
topological manifolds(see \cite{SY:3mfd}).

\begin{ex}[\cite{SY:3mfd}]\label{ex:flat-orb}
Let $\Gamma$ be the discrete subgroup of isometries of $\R^3$
generated by 
$\gamma(x,y,z)=-(x,y,z)$ and $\sigma(x,y,z)=(x+1,y,z)$.
Then $\R^3/\Gamma$ is a complete open nonnegatively curved Alexandrov 
space having two topological singular 
points.
\end{ex}

\begin{ex}[\cite{SY:3mfd}]\label{ex:sph-orb}
Let $S$ be a nonnegatively curved Alexandrov surface 
homeomorphic to $S^2$ having two essential singular points,
say $p_1$ and $p_2$. 
Cut $S$ along a minimal geodesic joining $p_1$ and $p_2$.
The result of this cutting,  say $S_0$, is a
nonnegatively curved Alexandrov surface with boundary.
The double $\hat S$ of $S_0$ has an obvious isometric involution 
$\sigma$ such that $\hat S/\sigma=S$.
Consider the $\Z_2$-action on $\hat S\times\R$
defined by $(x,t)\to (\sigma(x), -t)$.
The orbit space $(\hat S\times\R)/\Z_2$, denoted $L_{\rm sph}(S)$, 
has the two topological singular points $p_1$ and $p_2$.
This space corresponds to 
$L(S_2;2)$ in Example 9.4 of \cite{SY:3mfd}.
\end{ex}

\begin{ex}[\cite{SY:3mfd}]\label{ex:tor-orb}
Let $S$ denote the double of a rectangle
$[0,a]\times [0,b]$, and 
$T$ the flat torus defined by the rectangle
$[-a,a]\times [-b,b]$.
Note that $S=T/\sigma$ for the isometric 
involution $\sigma$ on $T$ defined by $(x,y)\to (-x,-y)$.
Consider the $\Z_2$-action on $T\times\R$
defined by $(x,t)\to (\sigma(x), -t)$.
The orbit space $(T\times\R)/\Z_2$, denoted
$L_{\rm tor}(S)$, has  four topological singular 
points. This space corresponds to 
$L(S_4;4)$ in Example 9.4 of \cite{SY:3mfd}.
\end{ex}

\begin{ex}\label{ex:proj-orb}
Let the flat torus $T$ and the isometric 
involution $\sigma$ on $T$ be as in the previous example.
Consider the involution $\tau$ on $T\times\R$
defined by $(x,y)\to (-x+a,y+b)$.
Let $\Omega\simeq\Z_2\oplus\Z_2$ be the group generated by
$\sigma$ and $\tau$.
Consider the $\Z_2\oplus\Z_2$-action
on $T\times\R$ defined by $(x,t)\to (\sigma(x), -t)$
and $(x,t)\to (\tau(x), t)$.
The orbit space $(T\times\R)/\Z_2\oplus\Z_2$ is doubly covered by 
$L_{\rm tor}(T/\sigma)$ and has two topological singular 
points.  This orbit space is denoted by $L_{\rm proj}(S)$, 
where we put $S:=T/\Omega\simeq P^2$.
\end{ex}

Each space of types   $L_{\rm sph}(S)$,  $L_{\rm tor}(S)$
and  $L_{\rm proj}(S)$ is a $3$-dimensional complete open 
Alexandrov space with nonnegative curvature which is not a 
topological manifold, and admits the structure of a singular
line bundle; the singular fibre are the geodesic rays 
starting from the topological singular points.
Note that the zero-section $S$ of the singular line bundle 
is the unique soul in each case. 

Here it should be remarked that no space of 
type $L(S_1;1)$ in Example 9.4 of \cite{SY:3mfd}
actually exists since some compatibility condition is not 
satisfied ! (see the proof of Case $A$-$(3)$ of 
Theorem \ref{thm:dim3wbdy-sp} below.) 
Of course, Theorem 9.6 in \cite{SY:3mfd} is true for
the $3$-dimensional complete open 
Alexandrov spaces with nonnegative curvature which are 
topological manifolds (the generalized soul theorem), 
but for non-topological manifolds,
it should be modified as follows:

\begin{thm} \label{thm:modify}
Every $3$-dimensional complete open 
Alexandrov space with nonnegative curvature which is not a  
topological manifold is either homeomorphic to one of
the cone $K(P^2)$ and $\R^3/\Gamma$ in Example \ref{ex:flat-orb},
or isometric to one of the spaces of types  
$L_{\rm sph}(S)$,  $L_{\rm tor}(S)$ and  $L_{\rm proj}(S)$.
\end{thm}

The proof of Theorem \ref{thm:modify} is identical with
that of Theorem 9.6 of \cite{SY:3mfd}. Compare 
the proof of Case $A$-$(3)$ of 
Theorem \ref{thm:dim3wbdy-sp} below.\par
\medskip

Our concerns are $3$-dimensional complete nonnegatively curved 
spaces with boundary. Let $t$ be a positive number.

\begin{ex}\label{ex:nonmfd}
\begin{enumerate}
 \item 
   Define a $\Z_2$-action on $D^2(t)\times\R$ by 
   $(x,s)\to (-x,-s)$. Then the orbit space 
   $(D^2(t)\times\R)/\Z_2$ is a complete noncompact 
   nonnegatively curved Alexandrov 
   space with boundary having a topological singular point.
   Note that the boundary of $(D^2(t)\times\R)/\Z_2$ is 
   homeomorphic to a M\"obius strip, and that 
   the double of $(D^2(t)\times\R)/\Z_2$ is isometric to 
   $L_{\rm sph}(S)$ with $S=D(D^2(t)/\Z_2);$
 \item 
   The identification space
     $$
         [0, t]\times \R^2/(t,x)\sim (t,-x)
     $$
   is a complete noncompact nonnegatively curved Alexandrov space 
   with boundary
   having a topological singular point.
   Note that the boundary of this identification space is 
   homeomorphic to $\R^2$ and that 
   the double of this identification space is isometric to 
   $\R^3/\Gamma$ in Example \ref{ex:flat-orb} up to a rescaling
   of metric$;$
 \item 
   For the discrete group $\Gamma$ in Example \ref{ex:flat-orb},
   $(\R\times D^2(t))/\Gamma$ is a compact nonnegatively 
   curved Alexandrov space with boundary
   having two topological singular points.
 \item 
   Let $L_{\rm sph}^t(S)$ denote the totally convex subset 
   of $L_{\rm sph}(S)$ defined as 
    $$
       L_{\rm sph}(S) := d_{S}^{-1}([0,t]),
    $$
   which is a compact nonnegatively curved Alexandrov 
   space with boundary. Here $d_S$ is the distance function from the 
   soul $S$.  $L_{\rm tor}^t(S)$ and $L_{\rm proj}^t(S)$
   are defined similarly.
\end{enumerate}
\end{ex}

\begin{lem} \label{lem:dinting}
Let $X^3$ be a complete noncompact Alexandrov space with nonnegative
curvature and with boundary such that for a point $p\in X^3$
$B(p,R)\simeq K_1(P^2)$ for any large $R>0$. Consider one of the 
following conditions:
\begin{enumerate}
 \item $B(p,R)\cap\partial X^3$ is homeomorphic to a M\"obius band$;$
 \item $B(p,R)\cap\partial X^3$ is homeomorphic to $D^2$.
\end{enumerate}
If $X^3$ satisfies $(1)$ $($resp. $(2)$$)$, then it is homeomorphic 
to $(D^2\times\R)/\Z_2$ 
$($resp. $[0,1]\times\R^2/(1,x)\sim(1,-x)$$)$.
\end{lem}

\begin{proof} 
Note that the soul of $X^3$ is a point, say $p$, with 
$\Sigma_p\simeq P^2$.
Suppose first $(1)$. Using the method in the proof of 
Assertion \ref{ass:B(S,R)capY}, we see that 
$$
    X^3\simeq B(p,\epsilon)- U \simeq (D^2\times\R)/\Z_2,
$$
where $U$ is an open disk of $\partial B(p,\epsilon)$.
The proof in the case of $(2)$ is similar and hence omitted.
\end{proof}

Now we are ready to state the classification result.

\begin{thm} \label{thm:dim3wbdy-sp}
The $3$-dimensional complete nonnegatively curved Alexandrov 
spaces $X^3$ with boundaries are classified as follows$:$\par
\medskip
\noindent 
If $\partial X^3$ is disconnected, then $X^3$ is isometric to a 
product $X_0\times I$, where $X_0$ is a component of $\partial X^3$.
\par
\medskip
\noindent
Suppose $\partial X^3$ is connected.\par
\begin{proclaim}{\emph{Case} A.}
      $X^3$ is compact.
\end{proclaim}
\begin{enumerate}
 \item If $\dim S=0$, $X^3$ is homeomorphic to either $D^3$,
       the unit cone $K_1(P^2)$ or $(\R\times D^2(1))/\Gamma;$
 \item If $\dim S=1$, $X^3$ is isometric to the form 
       $(\R\times N^2)/\Lambda$, where $N^2$ is homeomorphic
       to $D^2$ and $\Lambda\simeq\Z$. In particular, 
       $X^3$ is homeomorphic to a $D^2$-bundle over $S^1;$
 \item If $\dim S=2$, $X^3$ is isometric to one of a flat  
       $I$-bundle over $S$, $L_{\rm sph}^t(S)$,  $L_{\rm tor}^t(S)$
       and $L_{\rm proj}^t(S)$ for some $t>0$.
\end{enumerate}
\begin{proclaim}{\emph{Case} B.}
      $X^3$ is noncompact.
\end{proclaim}
\begin{enumerate}
 \item If $\dim S=2$, $X^3$ is isometric to 
       $S\times [0,\infty);$
 \item If $\dim S=1$, $X^3$ is isometric to the form 
       $(\R\times N^2)/\Lambda$, where $N^2$ is either homeomorphic
       to $\R^2_+$ or isometric to $\R\times I$ for a closed interval $I$,
       and $\Lambda\simeq\Z$. In particular, $X^3$ is homeomorphic 
       to an $N^2$-bundle over $S^1;$
 \item Suppose $\dim S=0$. If $X^3$ has two ends, then 
       it is isometric to a product $\R\times X_0$, where $X_0\simeq D^2$.
       \par
       Suppose that $X^3$ has exactly one end, and
       let $C$ be the maximum set $($possibly empty$)$ of 
       $d_{\partial X^3}$ . Then $C$ is either empty  or of dimension 
       $\ge 1$.
  \begin{enumerate}
   \item If $C$ is empty, $X^3$ is homeomorphic to 
         $\R^3_+;$
   \item If $\dim C=1$,  $C$ is a geodesic ray and 
         $X^3$ is either homeomorphic to $\R^3_+$ or the identification space 
         $D^2\times \R/(x,y)\simeq (-x,-y);$
   \item If $\dim C=2$, $C$ is homeomorphic to either $\R^2$ or $\R^2_+$.
    \begin{enumerate}
     \item 
     If $C \approx \mathbb R^2$, then $X$ is  either homomorphic to  $\mathbb R^3_+$ or isometric to one of  the quotient spaces 
\begin{align*}
 \hspace{1cm} & \hspace{1cm}\hat C \times [-t,t] / (x,y) \sim (\sigma(x), -y), \\
& \hspace{1cm} S_{\ell}^1 \times \mathbb R \times [-t,t] /(z,x,y) \sim (\bar z,-x,-y),
\end{align*}
for some $t, \ell> 0;$,
where $\hat C$ is a branched double covering of $C$ with a single branching point that is a unique topological singular point in $C$, and  $\sigma$ is the deck transformation of the branched covering $\hat C \to C$. 
%
%
%
     \item 
      If $C \approx \mathbb R^2_+$, then $X^3$ is either  homeomorphic to $\mathbb R^3_+$ or isometric to 
the quotient  
$$
     D\times\R/(x,y)\sim (\sigma(x),-y),
$$
where $D$ is a compact Alexandrov surface of nonnegative curvature with boundary  suth that 
\begin{itemize}
\item the farthest points from $\pa D$ forms a segment
or a point, say $L\,;$
\item $D$ is symmetric with respect to the midpoint $O$ of $L$,
and $\sigma$ is the symmetry of $D$  about $O$.
\end{itemize}
  \end{enumerate}
  \end{enumerate}
 \end{enumerate}
\end{thm}

\begin{proof}
By Theorem \ref{thm:cptcomp}, we assume that $\partial X$ is 
connected. By the splitting theorem, we also assume 
$\dim S\neq 1$.
Letting $C$ be the maximum set of $d_{\partial X^3}$, we note 
that the topological singular points of $X^3$ are contained
in ${\rm Ext}(\interior X^3)\bigcap C$.
Let $t$ be the maximum of $d_{\partial X^3}$.

Suppose $X^3$ is compact. 
To discuss Case $A$, we use Theorem 9.6 in \cite{SY:3mfd} and 
the argument there. First of all, the conclusion in Case $A$
certainly holds if $X^3$ is a topological manifold.
Suppose $X^3$ is not topological manifold.
In Case $A$-(1), if $\dim C=0$, then $X^3$ 
is homeomorphic to $K_1(P^2)$, and if $\dim C\ge 1$, 
then $X^3$ is homeomorphic to either $K_1(P^2)$ or 
$(\R\times D^2(1))/\Gamma$ depending on the number $\le 2$
of the topological singular points of $\interior X^3$.

In Case $A$-(3), first note that 
the number of essential singular points of $S$ is at most
$4$. By Proposition \ref{prop:concave-rigid}, $X^3$ is isometric 
to a singular flat $I$-bundle over $S$ with possible
singular fibres over essential singular points, say $p$, of $S$.
Let $\pi:X^3\to S$ be the natural projection along fibres.
For a small disk neighborhood $B$ of $p$, 
$\pi^{-1}(B)$ is fibre-wise homeomorphic to 
$(D^2\times\R)/\Z_2$ in Example \ref{ex:nonmfd}.
Now if $S\simeq S^2$, $X^3$ is built from 
some pieces of $(D^2\times\R)/\Z_2$ by gluing along boundaries.
To realize this gluing, we encounter some compatibility condition.
This observation implies that $X$ is isometric to either
$L_{\rm sph}^t(S)$ or  $L_{\rm tor}^t(S)$.
If $S\simeq P^2$, the compatibility condition
implies that the number $m$ of the singular fibres in $X^3$
is two. Because if $m=1$, take a disk neighborhood  $B$  of 
the singular point $p$ as above.
Since the line bundle is regular one on 
$S-B$ which is homeomorphic to a M\"obius band, 
its restriction to $\partial(S-B)$ must be trivial.
On the other hand, $\partial (D^2\times\R)/\Z_2)$ is 
an open M\"obius band.
Therefore the compatibility condition does not hold.
Now the universal cover $\tilde X^3$ of $X^3$ is 
isometric to $L_{\rm tor}^t(\tilde S)$, where $\tilde S$
is the universal cover of $S$. Thus 
$X^3$ must be isometric to 
$L_{\rm proj}^t(S)=L_{\rm tor}^t(\tilde S)/Z_2$.

Next consider Case $B$. If $\dim S=2$, using Theorem \ref{thm:cptcomp},
we easily obtain the conclusion.
Suppose $\dim S=0$. In view of the splitting theorem, we assume
that $X^3$ has exactly one end. 
If $X^3$ is a topological manifold, then Assertion 
\ref{ass:B(S,R)capY} shows $X^3\simeq\R^3_+$.
If $C$ is empty, $X^3$ is a topological manifold. Therefore
we assume that $C$ is nonempty.
Let $t$ denote the maximum of $d_{\partial X^3}$.
From the concavity of $d_{\partial X^3}$, 
every geodesic rays starting from a point of $C$ is contained in $C$.
In particular $C$ is noncompact.
Suppose $\dim C=1$. Since $X^3$ has  exactly one end,
$C$ represents a geodesic ray, say $\gamma$. 
Since every point of $\interior C$ is  topologically regular,
we assume that $p:=\partial C$ is a topological singular 
point. 
The critical point theory to the map 
$(d_{\partial X^3}, d_p)$ shows that 
$X^3$ is homeomorphic to 
$\{ d_{\partial X^3}\ge t-\epsilon_1\}\cap\{ d_p<\epsilon_2\}$
for every $0<\epsilon_1\ll \epsilon_2\ll 1$.
It is easy to see that $B(p,\epsilon_2)$ is homeomorphic to 
$\{ d_{\partial X^3}\ge t-\epsilon_1\}\cap\{ d_p<\epsilon_2\}$.
It follows from Stability Theorem \ref{thm:stability} that 
the boundary of
$\{ d_{\partial X^3}\ge t-\epsilon_1\}\cap\{ d_p\le\epsilon_2\}$
is homeomorphic to $P^2$.
Since 
$\{ d_{\partial X^3}\le\epsilon_1\}\cap\{ d_p=\epsilon_2\}$
is homeomorphic to $D^2$ for any sufficiently small $\epsilon_1$
and $\epsilon_2$, 
$\partial X^3$ must be homeomorphic to an open M\"obius
band. From Lemma \ref{lem:dinting}, we conclude that
$X^3$ is homeomorphic to $(D^2\times\R)/\Z_2$.

Suppose finally $\dim C=2$. 
Using the Sharafutdinov retraction constructed in \cite{Pr:alex2},
one can prove that $X$ is homotopy equivalent to  $C$.
If $C$ was isometric to $I\times\R$, then $X^3$ would have 
two ends contradicting the assumption. 
Therefore $C$ is homeomorphic to either $\R^2$ or $\R^2_+$.

Suppose $C\simeq \R^2$.
By Proposition \ref{prop:concave-rigid}, 
there exist exactly two minimal geodesics 
from any $x\in C-ES(C)$ to $\partial X^3$.
Let  $m\le 2$ be number of the topological singular
points of $X^3$, which is contained in $ES(C)$.
For such a topological sigular point $x\in ES(C)$, 
there exists a unique geodesic from  $x$ to $\partial X^3$.
Thus we have the structure of a singular flat $I$-bundle 
on $X^3$ over $C$, which provides a map
$\pi:\partial X^3\to C$ whose
restriction to the complement of the topological singular point set
is locally isometric double covering.

If $m=0$, then $X^3$ is a topological manifold, and 
is homepmorphic to $\R^3_+$ as pointed above.

If  $m=1$, 
 $X^3$ is isometric to the singular flat $I$-bundle 
over $C$ with a singular $I$-fiber over a point of
$ES(C)$. From this regidity, we can conclude that  
$X$ is isometric to the 
quotient space
$\hat C \times [-t,t] / (x,y) \sim (\sigma(x), -y)$
as stated in Theorem \ref{thm:dim3wbdy-sp}.


If $m=2$, $C$ is isometric to the double 
$D(I\times[0,\infty))$ for a closed interval $I$,
and $X^3$ is isometric to the singular flat $I$-bundle 
over $C$ with the two singular $I$-fiber over 
$ES(C)$. From this regidity, we can conclude that  
$X$ is isometric to the 
quotient
\[
S_{\ell}^1 \times \mathbb R \times [-t,t] /(z,x,y) \sim (\bar z,-x,-y),
\] 
where  $\ell=2L(I)$.
Here $C$ corresponds to the set 
$S_{\ell}^1 \times \mathbb R \times\{ 0\} 
/(z,s,0) \sim (\bar z,-s,0)$.
\pmed
Suppose $C\simeq \R^2_+$. In this case $m\le 1$.
 By Proposition 11.3 of \cite{SY:3mfd},  any $p\in\partial C$ is 
topologically regular. 
If $m=0$, then  $X^3\simeq\R^3_+$ as before.
If $m=1$, then $C$ is isometric to 
$D(\{ x,y\ge 0\})\cap\{ y\le h\}$ for some $h>0$,
and the origin, say $p_0=[(0,0)]$, of $C$ is topologically singular.

Let $\gamma$ be the geodesic ray of $X$ from $p_0$ represented as 
$\gamma(s)= [(s,0)]\in C$, $(s\ge 0)$,
and let $b$ be the Busemann function associated with 
$\gamma$.
Set $D_0:=b^{-1}(0)$ and $D_s:=b^{-1}(s)$ for any $s>0$.
Note that $\Sigma_{p_0}(X)$ is isometric to 
$P^2$ with constant curvature $1$.
This implies that 
$D_0=\min b$, which is an 
Alexandrov surface.

By Proposition \ref{prop:concave-rigid}, we have a bijective locally isometric imbedding
$\tau:(0,\infty)\times (-h,h)\times [-t,t]
\to X$ such that 
\beq \label{eq:isometric=tau}
\begin{aligned}
 &\text{$\tau((0,\infty)\times (-h,h)\times\{\pm u\})\subset d_{\pa X}^{-1}(t-u)$
for any $0\le u\le t\,;$} \\
& \text{$\tau(s,0,0)=\gamma(s)$ for any $s>0$.}     
\end{aligned}
\eeq
Note that 
\beq \label{eq:tau(s)D}
\text{
$\tau(\{ s\} \times (-h,h)\times [-t,t])\subset D_s$ for any $s>0$.}
\eeq

\begin{clm}\label{claim:invE}
For any $s>0$,
$D_s$ is convex, and there is an isometric involucion $\sigma_s$ on $D_s$ such that 
$D_0$ is isometric to $D_s/\sigma_s$.
\end{clm}
\begin{proof}
Let $L_0:=\{ x=0,  y\le h\} = C\cap D_0$,
$L:_s=\{ x=s, y\le h\} = C\cap D$.
\eqref{eq:isometric=tau} and \eqref{eq:tau(s)D} implies that $\pa D_0\subset \pa X$ and 
$L_0=d_{\pa D_0}^{-1}(t)$.
Thus  the midpoint of $L_0$ is a soul of $X$,
and  $D_0$ is a disk.
%
%

Let $F_0:=D(\{ y\ge 0,z\ge 0\})\cap \{ y< h, z\le t\}$.
By  Proposition \ref{prop:concave-rigid}, $\tau$ extends to an isometric imbedding
$\tau_0:F_0\to D_0$ sending $\{ 0\le y< h, z=0\}$ to $L_0\setminus \{ q_0\}$,
where $q_0$ is the other endpoint of 
$L_0$ than $p_0$.
Let $G_0$ be the closure of $D_0-\tau_0(F_0)$. 
Note that $G_0$ is either an arc or  a disk of nonnegative curvature.

If $G_0$ is an arc, then 
$D_0=D([0,h]\times \{z\ge 0\})\cap \{ z\le t\}$. Cosider the quotient space
\[
     E:=[-h,h]\times [-t,t]/ (\pm h, u)\sim (\pm h, -u),\quad |u|\le t.
\]

Suppose that $G_0$ is a disk, and let $G_1, G_2$ be two copies of $G_0$.
For a fixed $s_0>0$, let $L=L_{s_0}$, and set 
$\{ q_1,q_2\}:=\pa L$, and consider the gluing  
$$
E:=L\times [-t, t]\cup_{\{ q_1\}\times[-t, t]}G_1
    \cup_{\{ q_2\} \times[-t, t]} G_2.
$$  
In either case, 
it is easy to define the isometric involution 
$\sigma$ on $E$ such that $E/\sigma=D_0$.
Proposition \ref{prop:concave-rigid}
implies that $E$ is isometric to $D_s$
with respect to the interior metric for any 
$s>0$, and $\tau$ extends to an injective 
locally  isometrc map
$\iota:E\times (0,\infty)\to  X$.
It follows that  $D_s$ is 
convex in $X$.
\end{proof}

It follows from Claim \ref{claim:invE} that 
$X$ is isimetric to $D\times\R/(x,y)\sim (\sigma(x), -y)$.
\end{proof}

As an application of our argument, we give a classification 
of $3$-dimensional complete Alexandrov spaces $X^3$ with nonnegative 
curvature in terms of the number of extremal points.
We assume that $X^3$ has nonempty boundary.
Let $e_{\interior X^3}$ denote the number of extremal points in 
$\interior X^3$. If $e_{\interior X^3}=0$, then $X^3$ is a topological manifold,
and the classification of such spaces is given by
Theorem \ref{thm:dim3wbdy-sp}. Hence we assume $e_{\interior X^3}>0$.

\begin{cor}
Let $X^3$ be a $3$-dimensional complete Alexandrov space with 
nonnegative curvature and  with nonempty boundary. 
Then $0\le e_{\interior X^3}\le 4$ and the structure of $X^3$ is described 
in terms of $e_{\interior X^3}$ as follows.
In the below, $S$ is a soul of $X^3$ and $t$ denote the maximum 
of $d_{\partial X^3}$.
\begin{enumerate}
 \item If $e_{\interior X^3}=4$,  $X^3$ is isometric to either
       $D(I_1\times I_2\times \{ x\ge 0\})\bigcap \{x\le t\}$
       with some closed intervals $I_1$ and $I_2$, 
       or $L_{\rm tor}^t(S);$
 \item If $e_{\interior X^3}=3$,  $X^3$ is homeomorphic to either 
       $D^3$ or $K_1(P^3);$ 
 \item If $e_{\interior X^3}=2$,  $X^3$ is either homeomorphic to one of $D^3$,
       $K_1(P^2)$, $(\R\times D^2)/\Gamma$ and $\R^3_+$, 
       or isometric to one of $L_{\rm sph}^t(S)$ and 
       $L_{\rm proj}^t(S);$
 \item If $e_{\interior X^3}=1$,  $X^3$ is homeomorphic to one of
       $D^3$, $K_1(P^2)$, $\R^3_+$, 
       $([0,1]\times \R^2)/(1,x)\simeq (1,-x)$ and 
       $(D^2\times\R)/\Z_2$.
\end{enumerate}
\end{cor}

\begin{proof}
Let $C$ denote the maximum set of $d_{\partial X^3}$.
We may assume $C$ to be nonempty because of $e_{\interior X^3}>0$.
Note also that ${\rm Ext}(\interior X^3)\subset {\rm Ext}(C)$.
It follows that $e_{\interior X^3}\le 4$ and that if $e_{\interior X^3}\ge 3$, then 
$C$ is a compact surface, and $X^3$ is also compact.

Suppose $e_{\interior X^3}=4$. Then $C$ is isometric to either 
a rectangle $I_1\times I_2$ or the double $D(I_1\times I_2)$.
If $C=I_1\times I_2$, the assumption $e_{\interior X^3}=4$ forces 
all the boundary points $p$ of $C$ to be one-normal points
in the sense that there is a unique direction at 
$p$ normal to $C$. This implies that $X^3$ is isometric to 
$D(I_1\times I_2\times \{ x\ge 0\})\bigcap \{x\le t\}$.
If $C=D(I_1\times I_2)$, it is the soul $S$, and
$X^3$ is a singular flat $I$-bundle over $S$ with 
four singular fibres. Thus $X^3$ must be isometric to 
$L_{\rm tor}^t(S)$. 

Suppose $e_{\interior X^3}=3$.  
If $C$ had no boundary, then 
$X^3$ would be a singular flat $I$-bundle over $S$ with 
three singular fibres, which is impossible.
Therefore $C$ has nonempty boundary. 
In view of Lemma \ref{lem:class-singsurf}, 
if there is no topological singular point, then 
$X^3\simeq D^3$ because the soul of $X^3$ is a point.
If there is a topological singular point, 
$C$ must be the result of cutting of the double 
$D(I^2)$ of a square $I^2$ along the diagonals,
and the interior singular point $p$ of $C$ is the unique
topological singular point. Thus 
Theorem \ref{thm:dim3wbdy-sp} yields 
$X^3\simeq K_1(P^2)$.

The cases $e_{\interior X^3}\le 2$ also follows 
from Theorem \ref{thm:dim3wbdy-sp} together 
with an argument similar to the above argument,
and hence we omit the detail.
\end{proof}

\begin{rem}
\begin{enumerate}
 \item In the case when $X^3$ is a complete open Alexandrov space
       with nonnegative curvature, considering the sublevel 
       set $\{ b\le t\}$ with large $t>0$, one 
       can reduce the classification problem for $X^3$ as in 
       Theorem \ref{thm:dim3wbdy-sp} to that of 
       a compact space with boundary$;$
 \item For $n$-dimensional compact Alexandrov spaces with 
       nonnegative curvature, the total number of extremal points 
       does not exceed $2^n$ (\cite{Pr:spaces}).
       Recently, the metric classification of nonnegatively curved 
       spaces with maximal extremal points has been proved in 
       \cite{Ym:extrem}.
\end{enumerate}
\end{rem}


\part{Equivariant fibration-capping theorem} \label{part:cap}

\section{Preliminaries} 
    \label{sec:cap-state}

Let $X$ be a $k$-dimensional complete Alexandrov space with curvature 
$\ge -1$. 
Now suppose that $\partial X$ is nonempty.

Let a Lie group $G$ act on $X$ as isometries.
Note that the group of isometries of $X$ is a Lie group (\cite{FY:isomgp}).
Later on we shall assume that $X/G$ is compact.

First we discuss the convergence in $G$.

Let $\Sigma(X)$ denote the union of $\Sigma_p(X)$ when $p$ runs over 
$X$. We topologize $\Sigma(X)$ by the following convergence:
Let $v_n\in\Sigma_{p_n}(X)$ and $v\in\Sigma_p(X)$.
Then $v_n$ converges to $v$ if and only if 
$p_n\to p$ and $\angle(v_n,\xi_n)\to \angle(v,\xi)$ 
for any $x\in X-\{ p,p_n\}$ and $x_n$ with
$\gamma_{p_n,x_n}\to\gamma_{p,x}$, where 
$\xi_n=\dot\gamma_{p_n,x_n}(0)$ and
$\xi=\dot\gamma_{p,x}(0)$.
Note that $\Sigma(X)$ is not necessary locally compact.

The following lemma shows that the action of $G$ on $\Sigma(X)$ is
continuous.

\begin{lem} \label{lem:gp-conv}
A sequence $g_n$ in $G$ converges to $g$ if and only if 
$g_np\to gp$ and $g_n^{-1}v_n\to g^{-1}v$ for any $p\in X$,
$v_n\in\Sigma_{g_np}$ and $v\in\Sigma_{gp}$ with
$v_n\to v$.
\end{lem}

\begin{proof}
We may assume $g$ to be the identity.
Let $g_n\to 1$, and suppose that  
$\gamma_{g_np,x_n}\to\gamma_{px}$ and $v_n\in\Sigma_{g_np}$
converges to $v\in\Sigma_p$. Then from  definition, 
$\angle(g_n^{-1}v_n, (x_n)_p')=\angle(v_n, (g_nx_n)_{g_np_n}')$
converges to $\angle(v,x_p')$. 
The converse is obvious.
\end{proof}

For any $p$, $G_p$ denotes the isotropy subgroup of $G$ at $p$.

\begin{lem} \label{lem:ell_p}
For each $p\in X$, there is a nonnegative integer $\ell(p)$ together 
with the $G_p$-invariant isometric splitting 
$K_p(X)=\R^{\ell(p)}\times K_p^{\prime}$ satisfying 
\begin{enumerate}
 \item $K_p(Gp)=\R^{\ell(p)};$
 \item $\{K_p(X)=\R^{\ell(p)}\times K_p^{\prime}\}_{p\in X}$ 
    are the $G$-invariant splittings.
\end{enumerate}
\end{lem}

\begin{proof}
Let $\Sigma_p(Gp)$ denote the set of all $\xi\in\Sigma_p$
such that there exists a sequence $g_n\in G$ satisfying
$g_np\to p$ and $(g_np)_p'\to\xi$.
First we show that for any $\xi\in \Sigma_p(Gp)$ there exists an 
$\eta\in\Sigma_p(Gp)$ such that $\angle(\xi,\eta)=\pi$.
Let $g_n\in G$ be as above.
Taking a subsequence if necessary, we may assume  $g_n\to g$ for some
$g\in G$. Considering $g^{-1}g_n$, we may also assume that
$g_n$converges to the identity.
Put $p_n:=g_np$ and 
take $q_n\in X$ with $\tilde\angle pp_nq_n>\pi-\epsilon_n$,
$\lim \epsilon_n=0$.
Passing to a subsequence, we may assume that 
$(g_n^{-1}q_n)_p'$ converges to a direction $\xi_0\in\Sigma_p$.
In view of $g_n\to 1$, $(q_n)_{p_n}'\to\xi_0$.
We conclude $\xi_0=\xi$ as follows:
Putting $\xi_n:=(p_n)_p'$,  we have for any $x\in X-\{ p\}$
\begin{align*}
  |\angle(\xi,x_p')- & \angle(\xi_0,x_p')|  \le
    |\angle(\xi,x_p')-\angle(\xi_n,x_p')| \\
   & +|\angle(\xi_n,x_p')-\tilde\angle p_npx| 
     +  |\tilde\angle p_npx-(\pi-\tilde\angle pp_nx)| \\ 
   &  + |\pi-\tilde\angle pp_nx - \tilde\angle q_np_nx| 
      +  |\tilde\angle q_np_nx - \angle q_np_nx| \\
   &  +   |\angle q_np_nx - \angle(\xi_0,x_p')| < \epsilon_n,
\end{align*}
where $\lim\epsilon_n=0$. This implies $\xi=\xi_0$.
It follows that
$(g_n^{-1}p)_p'$ converges to $-\xi$.

By the above argument together with the splitting theorem,
$\Sigma_p(Gp)$ spans a Euclidean space of 
dimension, say $\ell(p)$, in $K_p$, and 
we have a $G_p$-invariant isometric splitting 
$K_p(X)=\R^{\ell(p)}\times K_p^{\prime}$.

We show $\R^{\ell(p)}=K_p(Gp)$.
Suppose the contrary. Since $K_p(Gp)$ is closed in $K_p$,
we can find sequences $p_n\in Gp$ and $q_n\in X-Gp$ such that 
\begin{enumerate}
 \item[$(a)$] $\angle((q_n)_p',\R^{\ell(p)})\to 0;$
 \item[$(b)$] $d(q_n,p_n)=d(q_n,Gp);$
 \item[$(c)$] $(p_n)_p'$ and $(q_n)_p'$ converge to elements of 
       $\Sigma_p(Gp)$ and $\Sigma_p-\Sigma_p(Gp)$ respectively.
\end{enumerate}
We may assume that $p_n=g_np$ with $g_n\to 1$. Putting
$v_n:=(q_n)_{p_n}'$, we may assume 
$g_n^{-1}v_n$ converges to a direction $w$ in $\Sigma_p$.
By $g_n\to 1$, we obtain $v_n\to w$. 

Now we claim that $w\in\R^{\ell(p)}$.
Otherwise, since $w$ is perpendicular to any element of 
$\Sigma_p(Gp)$, $w$ must be perpendicular to $\R^{\ell(p)}$.
Choose a point $x$ in the direction $w$.
Set $s:=d(p,x)$, $\epsilon_n:=d(p,p_n)$ and 
$\delta_n:=d(p_n,q_n)$. Note that
\begin{align*}
  d(x,p_n)=&\sqrt{s^2+\epsilon_n^2} + o(\epsilon_n),\\
  d(p_n,q_n)=&\sqrt{\epsilon_n^2+\delta_n^2} + o(\epsilon_n),\\
  d(x,q_n)=&\sqrt{s^2+\epsilon_n^2+\delta_n^2} + o(\epsilon_n),
\end{align*}
where $\lim o(\epsilon_n)/\epsilon_n=0$, which implies 
\[
    \angle(v_n, (x)_{p_n}')\ge \tilde\angle xp_nq_n \ge 
      \pi/2-O(\epsilon_n),
\]
where $\lim O(\epsilon_n)=0$. Since $\angle(w,x_p')=0$,
this contradicts the convergence $v_n\to w$.

It turns out that $K_p(Gp)$ is contained in the hyperplane
of $\R^{\ell(p)}$ perpendicular to $w$, which is also a 
contradiction to the definition of $\ell(p)$.

(2) is obvious.
\end{proof}

In view of $K_p=\R^{\ell(p)}\times K_p'$, we consider 
$$
   L_p := \exp_p(K_p'\cap B(o_p,\epsilon))
$$
for a sufficiently small $\epsilon>0$,
where $\exp_p:K_p'\cap B(o_p,\epsilon)\to X$ is 
defined by using quasigeodesics (see \cite{Pt:apply}). 
Note that $L_p$ is $G_p$-invariant.
Let $G\times_{G_p} L_p$ denote the orbit space of
$G\times L_p$ by the diagonal $G_p$-action defined as
$h(g,x)=(gh^{-1},hx)$. Remark that 
$G\times_{G_p} L_p$ has a natural $G$-action and is an $L_p$-bundle 
over $G/G_p\simeq Gp$.

We assume the following:

\begin{asmp} \label{asmp:slice}
For each $p\in X$, $L_p$ gives a {\it slice} at $p$. Namely
$U_p:=GL_p$ provides a $G$-invariant tubular neighborhood of $Gp$
which is $G$-isomorphic to $G\times_{G_p} L_p$.
\end{asmp}

In the Riemannian case, Assumption \ref{asmp:slice}
is satisfied (cf. \cite{Brd:intro}). 
The author believes that Assumption \ref{asmp:slice}
is satisfied in the present general case, but he does not know
the proof yet. Note that Assumption \ref{asmp:slice} is automatically
satisfied if $G$ is discrete.

Let $M^n$ be another complete Alexandrov space with curvature 
$\ge -1$, and let $G$ also act on  $M^n$ 
as isometries. Let $d_{e.GH}((M,G), (X,G))$ 
denote the equivariant Gromov-Hausdorff distance
between the $G$-spaces (cf. \cite{FY:fundgp} for the definition).

The following theorem is an equivariant-version of Theorem
\ref{thm:orig-cap}. See Section \ref{sec:cap} for the notations
below.

\begin{thm}[Equivariant Fibration-Capping Theorem] \label{thm:cap}
Let $X$ and $G$ be as above satisfying Assumption \ref{asmp:slice}.
We also assume that $X/G$ is compact.
Given $k$ and $\mu>0$ there exist positive numbers $\delta=\delta_k$, 
$\epsilon_{X,G}(\mu)$ and $\nu=\nu_{X,G}(\mu)$ satisfying 
the following $:$\,\, Let  $Y\subset R^D_{\delta}(X)$ be a $G$-invariant
closed domain
such that $\text{$\delta_D$-{\rm str.rad}$(Y)$}>\mu$. 
Let $M$ be an $n$-dimensional complete Alexandrov space with 
curvature $\ge -1$ and with $M=R_{\delta_n}(M)$.
Suppose 
$d_{e.GH}((M,G),(X,G))<\epsilon$ for some 
$\epsilon\le\epsilon_{X,G}(\mu)$.
Then there exists a $G$-invariant closed domain $N\subset M$ and 
a $G$-invariant decomposition 
$$
   N = N_{\rm int} \cup N_{\rm cap}
$$
of $N$ into two closed domains glued along their boudaries and a
$G$-equivariant Lipschitz map  $f:N \to Y_{\nu}$ such that 
\begin{enumerate}
 \item $N_{{\rm int}}$ is the closure of $f^{-1}(\interior_0 Y_{\nu})$, and
       $N_{\rm cap} = f^{-1}(\partial_0 Y_{\nu})$;
 \item the restrictions $f|_{N_{\interior}}: N_{\interior} \to Y$ and
     $f|_{N_{\rm cap}} : N_{\rm cap} \to \partial Y$ are 
   \begin{enumerate}
     \item  locally trivial fibre bundles;
     \item  $\tau(\delta,\nu,\epsilon/\nu)$-Lipschitz submersions.
   \end{enumerate}
\end{enumerate}
\end{thm}

In the proof of Theorem \ref{thm:cap} below, 
we generalize the argument in \cite{Ym:conv}.
Although the basic procedures are similar in several steps,
we give the proof  for readers' convenience.\par

The following result is of fundamental significance to describe the 
basic local properties of a neighborhood of a strained point.

\begin{lem}[\cite{BGP}] \label{lem:almisom}
There exists a positive number $\delta_k$ satisfying the following:
Let $(a_i,b_i)$ be an $(k,\delta)$-strainer in $D(X)$  at $p\in X$
with length $\ge\mu_0$ and with  $\delta\le \delta_k$.
Then the map $f:X\to \R^k$ defined by
$$
   f(x) = (d(a_1,x), \ldots, d(a_k,x))
$$
provides a $\tau(\delta,\sigma)$-almost isometry of the metric 
ball $B(p,\sigma;X)$ onto an open subset of $\R^k$ or 
$\R^k_+$, where $\sigma\ll \mu_0$.
\end{lem}

From now on we assume
\begin{equation}
  \text{$\delta_D$-str.rad}(X) > \mu_0 
\end{equation}
for a fixed $\mu_0>0$ and a small $\delta>0$. \par

The purpose of the rest of this paper is devoted to prove 
Theorem \ref{thm:cap} in the case of 
$$ 
       Y=R^D_{\delta}(X)=X.
$$
The general case is similar,
and hence the proof will be omitted.\par

By definition, 
we may assume that for every $p\in X$ \
there exists an admissible  $(k,\delta)$-strainer 
$(a_i,b_i)$ of length $>\nu_0/2$ at all points in $B_p(\sigma)$.

\begin{lem} \label{lem:basic}
Under the situation above,

\begin{enumerate}
 \item for every $q,r,s\in B(p,\sigma)$ with 
   $1/100 \le d(q,r)/d(q,s) \le 1$, we have
   $|\angle rqs - \tilde\angle rqs| < \tau(\delta,\sigma)$;
 \item  for every $q\in X$ with $\sigma/10\le d(p,q) \le\sigma$ and 
   for every $x\in X$ with $d(p,x)\ll\sigma$, we have 
    $$
       |\angle xpq - \tilde\angle xpq| < \tau(\delta,\sigma, d(p,x)|/\sigma);
    $$
 \item  if $d(p, \partial X) \ge 2\sigma$, then 
       for every $q\in B_p(\sigma)$ and for every $\xi\in\Sigma_q$, 
       there exist points $r,s$ such that $d(q,r)\ge\sigma$, $d(q,s)\ge\sigma$
       and 
       $$
        \angle(\xi, r^{\prime}_q) < \tau(\delta,\sigma), \qquad
              \tilde\angle rqs > \pi-\tau(\delta,\sigma);
       $$
 \item  if $d(p,\partial X)\le 2\sigma$, then for every $q\in B_p(\sigma)$
       and for every $\xi\in\Sigma_q$ with 
       $\angle(\xi, (a_k)_q^{\prime})\le\pi/2$, 
       there exist points $r\in X$ and $s\in D(X)$ such that 
       $d(q,r)\ge\sigma$, $d(q,s)\ge\sigma$ and 
       $$
         \angle(\xi, r^{\prime}_q) < \tau(\delta,\sigma), \qquad 
         \tilde\angle rqs > \pi-\tau(\delta,\sigma).
       $$
\end{enumerate}
\end{lem}

\begin{proof}
(1) follows from Lemma  \ref{lem:almisom}.
(2) follows from Lemma 5.6 in \cite{BGP}.
For the proof of (3) and (4), see \cite{Ym:conv}.
\end{proof}

We consider 
$X_{\nu} := \{ x\in X\, |\, d(x,\partial X)\ge \nu \}$.

By contradiction argument, one can prove the following two lemmas.

\begin{lem} \label{lem:basic2}
There exist positive numbers $\delta\ll 1/k$
and $\nu\ll \sigma\ll \mu_0$ such that 
for every $p\in\partial X_{\nu}$ and $x\in\partial X_{\nu}$
with $\sigma/10\le d(p,x)\le\sigma$, there exists 
$y\in\partial X_{\nu}$ with $\sigma/10\le d(p,y)\le\sigma$ such that 
$\tilde\angle xpy > \pi - \tau(\delta, \sigma, \nu)$.
\end{lem}

\begin{lem} \label{lem:basic3}
There exist positive numbers $\delta\ll 1/k$
and $\nu\ll \sigma\ll \mu_0$ such that 
for every $p\in\partial X_{\nu}$ and $x\in\partial X_{\nu}$
with $d(p,x)\le\sigma$, there exists 
$y\in\partial X_{\nu}$ with $\sigma/10\le d(p,y)\le\sigma$ such that 
$\angle(x_p', y_p') < \tau(\delta, \sigma, \nu)$.
\end{lem}


\section{Embedding $X$ into $L^2(X)$ } \label{sec:embed}
\medskip

  Let $L^2(X)$ denote the Hilbert space consisting of all $L^2$ functions on 
$X$ with respect to the Hausdorff $k$-measure, where $G$ acts on 
$L^2(X)$ by $g\cdot\phi(x)=\phi(g^{-1}x)$ for any $\phi\in L^2(X)$.
In this section  
we study the map $f_X:X\to L^2(X)$  defined by  
$$
    f_X(p)(x) = h(d(p,x)),
$$
where  $h:\R \to [0,1]$ is a smooth non-increasing function 
such that\par

\begin{enumerate}
 \item  $h=1$ \quad on $(-\infty,0]$, \qquad 
                     $h=0$ \quad on $[\sigma,\infty)$;
 \item  $h^{\prime} = 1/\sigma$ \quad on 
                  $[2\sigma/10, 8\sigma/10]$;
 \item  $-\sigma^2 < h^{\prime} < 0 $ \quad on 
                  $(0, \sigma/10]$;
 \item  $|h^{\prime\prime}| < 100/\sigma^2$.
\end{enumerate}

Remark that $f_X$ is a $G$-equivariant Lipschitz map. 
\par
 From now on, we use $c_1,c_2,\dots$ to express positive 
constants depending only on the dimension $k$.
 First we remark that by Lemma \ref{lem:almisom}
 there exist constants $c_1$ and $c_2$ 
such that for every $p\in X$,

\begin{equation}
    c_1 < \frac{\cal H^k(B_p(\sigma))}{b_0^k(\sigma)} < c_2,
\end{equation}
where $\cal H^k$ and $b_0^k(\sigma)$ denote 
the Hausdorff $k$-measure and 
the volume of a $\sigma$-ball
in $\R^k$ respectively.

 We next consider the directional derivatives of $f_X$.  
For $\xi\in \Sigma_p$, putting

\begin{equation}
  df_X(\xi)(x) = -h^{\prime}(d(p,x))\cos \angle(\xi, x^{\prime}_p), 
              \quad (x\in X),  \label{eq:deriv}
\end{equation}
we have 
$$
    df_X(\xi) = \lim_{t\downarrow 0} \frac{f_X(\operatorname{exp}t\xi) - f_X(p)}
                                  {t} \qquad \text{in $L^2(X)$}.
$$
From now on we use the following norm of $L^2(X)$ with normalization:
$$
 |f|^2 = \frac{\sigma^2}{b^k_0(\sigma)}\int_X |f(x)|^2 d\cal H^k(x).
$$ 

\begin{lem} \label{lem:df_X}
There exist positive numbers $c_3$ and $c_4$ such that 
$$
       c_3 < |df_X(\xi)| < c_4
$$
for every $p\in X$ and $\xi\in \Sigma_p$. 
\end{lem}

\begin{proof}
Use Lemma \ref{lem:basic} (3), (4).
\end{proof}

\begin{lem} \label{lem:inj}
There exist positive numbers  $c_5$ and $c_6$ such that
for every $p,q\in X$ with $d(p,q)\le \sigma$, 
$$
c_5 < \frac{|f_X(p)-f_X(q)|}{d(p,q)} < c_6.
$$
In particular $f_X$ is injective.
\end{lem}

The proof is straightforward, and hence omitted.
\medskip

 Let $K_p=K(\Sigma_p)$ be the tangent cone at $p$. 
We make an identification $\Sigma_p=\Sigma_p\times \{ 1\} \subset K_p$. 
The map $df_X:\Sigma_p\to L^2(X)$ naturally extends to 
$df_X:K_p\to L^2(X)$. Next we show that $df_X(K_p)$ can be approximated by 
a $k$-dimensional subspace of $L^2(X)$. \par

\begin{lem} \label{lem:almlinear}
For any $p\in X$, let $(a_i,b_i)$ be an admissible
 $(k,\delta)$-strainer at $p$. 
Taking $\xi_i$ in $(a_i)_p^{\prime}$, we have for 
any $\xi\in \Sigma_p$,
$$
  |df_X(\xi) - \sum_{i=1}^n c_i\,df_X(\xi_i)| < \tau(\delta),
$$
where $c_i = \cos\,\angle(\xi_i,\xi)$. 
In particular, $df_X(\xi_1),\dots, df_X(\xi_k)$ are linearly
independent in $L^2(X)$.

\end{lem}

\begin{proof}
Let $\phi : \Sigma_p(D(X))\to S^{k-1}(1)$ be the $\tau(\delta)$-almost 
isometry defined by
$$
  \phi(\xi)=(\cos\,\angle(\xi_i,\xi))/|(\cos\,\angle(\xi_i,\xi))|.
$$ 
(See \cite{BGP}).  It is easy to verify 
$$
\left|\cos\,\angle(\xi,\eta) - \sum_{i=1}^k c_i\cos\,\angle(\xi_i,\eta)\right| 
                    < \tau(\delta),
$$
for every $\eta\in \Sigma_p$, from which the lemma follows. 
\end{proof}

Thus if $p\in\interior X$ (resp. if $p\in\partial X$), then 
$df_X(K_p)$ can be approximated by the $k$-dimensional 
subspace $\Pi_p$  generated by $\{ df_X(\xi_i)\}_{1\le i\le k}$
(resp. by a $k$-dimensional half-space of $\Pi_p$). 
In view of Lemma \ref{lem:almlinear}, one may say that 
$df_X$ is almost linear.
\par
\medskip

\section{$G$-invariant tubular neighborhood} 
            \label{sec:cap-tube}

In this section, we construct a 
$G$-invariant tubular neighborhood of $f_X(X_{\nu})$ 
in $L^2(X)$ in a generalized sense, 
where $\nu\ll\sigma$.

Let $G_k(L^2(X))$ be the infinite-dimensional Grassmann manifold 
consisting of all $k$-dimensional subspaces of $L^2(X)$. 

For any $p\in X$, let $K_p=\R^{\ell(p)}\times K'$ be as in Lemma
\ref{lem:ell_p}, and let $U_p=GL_p$ be as in Assumption 
\ref{asmp:slice}.

\begin{lem} \label{lem:inv}
There exists a $G_p$-invariant $k$-dimensional subspace
$\Pi_p$ of $L^2(X)$ such that 
$\angle(\Pi_p, df_X(K_p))<\tau(\delta)$.
\end{lem}

\begin{proof}
Let $\Pi'$ be a $k-\ell(p)$-dimensional
subspace of $L^2(X)$ which is $\tau(\delta)$-close to $df_X(K')$.
Then $G_p\Pi'$ is a $G_p$-invariant compact subset 
of $L^2(X)$.
Since $G_p\Pi'$ can be considered as a $G_p$-invariant
subset of $G_{k-\ell(p)}(L^2(X))$ whose diameter is small,
we can find a $G_p$-fixed point, say $\Pi$, near $\Pi'$ 
by using the center of mass technique on 
$G_{k-\ell(p)}(L^2(X))$. 
It suffices to put $\Pi_p:=\Pi\oplus df_X(K_p(Gp))$.
\end{proof}

Now fix $p$ and 
take $L_p$ so small that $G_q\subset G_p$ for all $q\in L_p$.
For any $q\in L_p$ and $x\in Gq$, we put $\Pi_q:=\Pi_p$,
$\Pi_x:=g(\Pi_{p})$, where $x=gq$.
Note that $\{ \Pi_{x}\}_{x\in U_p}$ provides a $G$-invariant field
of $k$-dimensional subspaces of $L^2(X)$ which are 
$\tau(\delta)$-almost tangent to $f_X(X)$.

\begin{lem} \label{lem:GL}
If $L_p$ is sufficiently small,
 then  $\angle(\Pi_x, \Pi_y)\le C d(x,y)$ 
 for all $x$, $y\in U_p$, where $C=C_{U_p}$.
\end{lem}

\begin{proof}
Since $G$ acts on $G_k(L^2(X))$ isometrically, 
the map $\pi:G/G_p\to G_k(L^2(X))$ defined by 
$\pi([g]) := g(\Pi_p)$ is Lipschitz with respect to a 
$G$-invariant metric on $G/G_p$.
This implies that $\angle(\Pi_p, \Pi_{gp})\le C_1 d([e], [g])$
for some constant $C_1$. 
Since it is straightforward to see that 
$d([e], [g]) \le C_2 d(p, gp)$ for some constant $C_2$,
it follows that $\angle(\Pi_p, \Pi_{gp})\le C d(p,gp)$.
If $x\in Gq$, $q\in L_p$ and $y\in Gx$, this argument shows that
$\angle(\Pi_x, \Pi_y)\le C' d(x,y)$ for sufficiently small 
$L_p$.
Next consider the general case when $x=gx_0$, $x_0\in g_0L_p$ and 
$y\in g_0L_p$ for some $g$ and $g_0$ in $G$. 
We may assume that both $x$ and $y$ are close to $x_0$.
Then the conclusion follows from the above argument and 
the fact that $\tilde\angle x x_0 y$ is close to $\pi/2$.
\end{proof}

\begin{lem} \label{lem:parallel}
For any $p,q\in\interior X$ or for any $p,q\in\partial X$,
$$
 d_H^{L^2}(df_X(\Sigma_p),df_X(\Sigma_q))<
        \tau(\delta,\sigma,d(p,q)/\sigma),
$$
where $d_H^{L^2}$ denotes the Hausdorff distance in $L^2(X)$.
\end{lem}

\begin{proof}
First suppose that  $p,q\in\interior X$.
Let $(a_i,b_i)$ be an admissible $(n,\delta)$-stainer 
in $D(X)$ at $p$.
For every $\xi\in \Sigma_q$  take a point 
$r\in D(X)$ satisfying $d(q,r)\ge\sigma$ and 
$\angle(\xi,r_q^{\prime})< \tau(\delta,\sigma)$.
Taking $\xi_1$ in $r_p^{\prime}$, we have 
$$
   |\angle(\xi,x_q^{\prime}) - \angle(\xi_1,x_p^{\prime})|
            < \tau(\delta,\sigma, d(p,q)/\sigma),
$$
for all $x$ with $\sigma/10\le d(p,x)\le \sigma$.
It follows that 
$|df_X(\xi) - df_X(\xi_1)| < \tau(\delta,\sigma, d(p,q)/\sigma)$.

If $p,q\in\partial X$, take the above $r$ from $X$ in place of 
$D(X)$. Then the conclusion follows from a similar argument.
\end{proof}

\begin{lem}  \label{lem:diff}
For any $p,q\in X$ and $\xi$ in $q_p^{\prime}$,  
\begin{equation}
 \left| \frac{f_X(q)-f_X(p)}{d(q,p)} - df_X(\xi) \right| < 
                 \tau(\delta,\sigma,d(p,q)/\sigma). \label{eq:diff}
\end{equation}
\end{lem}

\begin{proof}
Note that
$|\angle xpq-\tilde\angle xpq|<\tau(\delta,\sigma,d(p,q)/\sigma)$
for all $x$ with $\sigma/10\le d(p,x)\le\sigma$ and that
$|d(x,q)- d(x,p)+t\cos\tilde\angle xpq|<t\tau(t/\sigma)$,  $t=d(p,q)$.
It follows that 
\begin{equation}
  |d(x,q)-d(x,p) + t\cos \angle(\xi, x^{\prime})| 
                < t \,\tau(\delta,\sigma,t/\sigma), 
\end{equation}
which yields \eqref{eq:diff}.
\end{proof}

Let $\sigma_1\ll\sigma$
and let us use the simpler notation $\tau_{\delta}$ 
to denote a positive function 
of type $\tau(\delta,\sigma,\sigma_1/\sigma)$. \par

Let $\pi:X\to X/G$ be the orbit projection, and for 
$\bar p\in\pi(\partial X)$
let $p\in \partial X$ be a point over $\bar p$.
For a $G$-slice $L_p$ at $p$ with $\diam(L_p)\le\sigma_1$, let 
$U_p = GL_p$ be the $G$-invariant neighborhood of $Gp$ as in 
Assumption \ref{asmp:slice}, 
and $U_{\bar p} := \pi(U_p)$.
We take a finite open covering $\{ U_{\bar p_i} \}_{i=1,2,\ldots}$
of $X/G$ as above.

\begin{lem} \label{lem:T}
Suppose that $X/G$ is compact.
Then there exists a $G$-equivariant Lipschitz map
$T: X\to G_k(L^2(X))$ such that \par
\begin{enumerate}
 \item  $\angle(T(x),\Pi_{p_i}) < \tau_{\delta} \quad$ if $x\in U_{p_i}$;
 \item  $\angle(T(x),T(y)) < Cd(x,y), \quad$ 
          where $C=C_{G,X}$ is a constant depending 
          only on the $G$-action on $X$.
\end{enumerate}
\end{lem}

\begin{proof}
Let $\{ \bar\rho_i \}$ be a partition of unity consisting of 
Lipschitz functions associated with 
$\{ U_{\bar p_i} \}_{i=1,2,\ldots}$, and set
$\rho_i:= \bar\rho_i\circ \pi$.

We use the center of mass technique on $G_k(L^2(X))$.
For each $x\in X$, consider the weighted distance functions
$$
  \phi_x :=    \sum_{i\in I_x}  \rho_i(x) 
        d(\Pi_{p_i,x},\, \cdot\,)
$$
on $G_k(L^2(X))$ with weights $\rho_i(x)$, where
$I_x := \{ i\,|\, x\in U_{p_i} \}$.
Since $\rho_i(x)=0$ except finitely many $i$,
all $\Pi_{p_i,x}$ in the righthand  side 
actually lies in some finite dimensional Euclidean space $E$.
Since $G_k(E)$ is totally geodesic in $G_k(L^2(X))$,
$\phi_x$ has a unique minimum point, say $T(x)$, on $G_k(E)$.
It is straightforward to see that $T:X\to G_k(L^2(X))$
is $G$-equivariant. A convexity argument also shows that 
$T$ is Lipschitz.
\end{proof}

 Let $G_k^*(L^2(X))$ be the Grassmann manifold consisting of 
all subspaces of 
codimension $k$ in $L^2(X)$, and $N:X \to G_k^*(L^2(X))$ the dual 
of $T$, $N(x)=T(x)^{\perp}$, where $T(x)^{\perp}$ denotes
the orthogonal complement of $T(x)$.
The angle $\angle(N(x), N(y))$ is defined
as $\angle(N(x), N(y))=\angle(T(x),T(y))$.

Lemma \ref{lem:T} immediately implies 

\begin{lem} \label{lem:N}
Suppose that $X/G$ is compact.
Then the map $N:X\to G_k^*(L^2(X))$ is $G$-equivariant and
Lipschitz with Lipschitz constant $C=C_{G,X}$.
\end{lem}

Now we consider the set $\partial X_{\nu}$.
For each $x\in\partial X_{\nu}$, Let $V_x$ denote 
the set of directions at $x$ consisting of all minimal
segments from $x$ to $\partial X$.
Since $\diam(V_x)<\tau(\delta,\sigma)$, a center of mass technique 
on $S(L^2(X)) := \{ v\in L^2(X)\,| \, |v|=1\}$  
similar to Lemma \ref{lem:T} yields

\begin{lem} \label{lem:n}
Suppose that $X/G$ is compact.
Then there exists a $G$-equivariant Lipschitz map 
${\bf n}:\partial X_{\nu} \to S(L^2(X))$  such that 
\begin{enumerate} 
 \item  $\angle({\bf n}(x), df_X(V_x)) < \tau_{\delta}$;
 \item  $\angle({\bf n}(x), {\bf n}(y)) < Cd(x,y),$ 
        where $C=C_{G,X}$ is a constant depending 
          only on the $G$-action on $X$.
\end{enumerate}
\end{lem}

For $x\in \partial X_{\nu}$, let us denote by
$\hat H(x)$ the subspace of codimension $k-1$ generated by 
$N(x)$ and  ${\bf n}(x)$, and by 
$H(x)$ the half space of $\hat H(x)$ containing
${\bf n}(x)$ and bounded by $N(x)$.

We consider the ``normal bundle'' $\cal W$ of $f_X(X_{\nu})$ as
$\cal W := \{ (x, v)\in X_{\nu}\times L^2(X)\, |\,
v\in W(x)\}$, where
\begin{equation*}
   W(x) = 
      \begin{cases}
          H(x)    & \qquad x\in \partial X_{\nu}\\
          N(x)    & \qquad x\in \interior X_{\nu}.
      \end{cases}
\end{equation*}
Note that $\cal W$ is $G$-invariant.
For $c>0$, we put 

$$
 \cal W(c)=\{ (x,v)\in \cal W\,|\, |v|<c \}.
$$

\begin{lem} \label{lem:tubular}
There exists a positive number $\kappa= C$ such that 
$\cal W(\kappa)$ provides
a $G$-invariant tubular neighborhood of $f_X(X_{\nu})$.
Namely 
\begin{enumerate}
 \item \text{ $f_X(p_1)+v_1\neq f_X(p_2)+v_2$ for every 
      $(p_1,v_1)\neq (p_2,v_2)\in \cal W(\kappa)$};
 \item \text{ the set 
       $U(\kappa)=\{ x+v| (x,v)\in\cal W(\kappa)\}$ 
          is open in $L^2(X)$}.
\end{enumerate}
\end{lem}

\begin{proof}
Suppose that $x_1+v_1=x_2+v_2$ for $x_i=f_X(p_i)$ and  $v_i\in
W(p_i)$. We first assume  $d(p_1,p_2)\le\sigma_1$.

\medskip
Case 1.\quad $p_1, p_2\in \interior X_{\nu}$.
\medskip

Put $K=\{ x_1+N(p_1)\}\cap \{x_2+N(p_2)\}$, and let 
$y\in K$ and $z\in x_2+ N(p_2)$ be such that 
$|x_1-y|=d(x_1,K)$,$|x_1-y|=|y-z|$ and that 
$\angle x_1yz = \angle(x_1-y, N(p_2))\le\angle(N(p_1), N(p_2))$.
Then Lemma \ref{lem:parallel} implies that $\angle x_1yz < \tau_{\delta}$. 
It follows from the choice of $z$ that 
$|\angle(x_1-z, N(p_2)) - \pi/2| < \tau_{\delta}$. 
On the other hand 
the fact $\angle(x_2-x_1,T(p_1))<\tau_{\delta}$ (Lemma \ref{lem:diff}) 
also implies that
$|\angle(x_2-x_1, N(p_2)) - \pi/2|<\tau_{\delta}$. 
It follows that $|x_2-z|<\tau_{\delta}|x_1-x_2|$. 
Putting $\ell=|y-x_1|=|y-z|$ and using Lemma \ref{lem:T}, we then have 
\begin{align*}
  |x_1-z| &\le \ell\angle x_1yz \\
       &\le \ell\angle(T(p_1),T(p_2)) \\
       &\le \ell C |x_1-x_2|.
\end{align*}
Thus we obtain $\ell \ge (1-\tau_{\delta})/C$ as required. \par

\medskip
Case 2.\quad $p_1, p_2\in \partial X_{\nu}$.
\medskip

Apply the above argument to $\hat H(p_i)$ instead of $N(p_i)$.
\par

\medskip
Case 3.\quad $p_1\in \partial X_{\nu}$ and $p_2\in \interior X_{\nu}$.
\medskip

Apply the above argument to $H(p_1)$  and $N(p_2)$ 
instead of $N(p_i)$.  Let $K=\{ x_1+H(p_1)\}\cap \{x_2+N(p_2)\}$.
If $K$ meet $x_1+N(p_1)$, we can apply the argument of Case 1 to 
$N(p_i)$. If $K$ does not meet $N(p_1)$, it is an affine subspace 
parallel to $N(p_1)$, and let $\hat K$ be the affine space generated 
by $K$ and line segment from $x_1$ to $K$.
Then we can apply the argument of Case 1
to $\hat K$ and $N(p_2)$. Thus we obtain (1).

Next we shall prove (2), which  follows from the argument above: 
We assume Case 1. The other cases are similar, and hence omitted.

For any $y\in U(\kappa)$ with 
$y\in f_X(p_0)+ N(p_0)$, $x_0:=f_X(p_0)\in f_X(X_{\nu})$
and for any $z\in L^2(X)$ 
close to $y$, let $T_0$ be the $n$-plane through $z$ and parallel to 
$T(x_0)$, and $y_0$ the intersection point of $T_0$ and $N(x_0)$. 
If $x\in f_X(X_{\nu})$ is near $x_0$, then $N(x)$ meets $T_0$ 
at a unique point, say $\alpha(x)$.
With the above argument, we can observe that 
$\alpha$ is a homeomorphism of a 
neighborhood of $x_0$ in $f_X(\partial X_{\nu})$ onto a 
neighborhood of $y_0$ in $T_0$.
Hence $z\in U(\kappa)$ as required.\par

Finally we shall finish the proof of (1).
Suppose that 
$q_0:=f_X(p_0)+v_0 = f_X(p_1)+v_1$ for some $p_i$ and $v_i$ with
$d(p_0, p_1)>\sigma_1$, $|v_i|<C$.
For a curve $p_s$ joining $p_0$ to $p_1$, put 
$v(s,t):= (1-t)f_X(p_s)+tq_0$.
We assert that there exists a continuous map
$V:[0,1]\times [0,1]\to \cal W$ such that
$Exp_{\cal W}(V(s,t))=v(s,t)$, yielding
$Exp_{\cal W}(V(s,1))=q_0$ for any $s\in I$,
a contradiction to the previous argument.
To prove this assertion,
consider the set $I\subset [0,1]$ such that for $t\in I$
there exists a lift $V:[0,1]\times [0,t]\to \cal W$
of $v$ as above. Actually $0\in I$ and the previous argument shows 
that $I$ is open. 
We define a metric of $\cal W$ by 
$d((x_1,v_1), (x_2, v_2)):= (d(x_1, x_2)^2 + |v_1-v_2|^2)^{1/2}$.
Then the proof of Lemma \ref{lem:proj} implies that 
$$
  d((x_1,v_1), (x_2, v_2))\le Cd(x_1+v_1, x_2+v_2),
$$
from which the closedness of $I$ follows.
\end{proof}

 Next let us study the properties of the projection 
$\pi:\cal W(\kappa) \to
f_X(X_{\nu})$ along $\cal W$. 
By definition, $\pi(x)=y$ if $x\in W(y)$.

\begin{lem} \label{lem:proj}
The map $\pi:\cal W(\kappa)\to f_X(X_{\nu})$ is  
Lipschitz continuous.
More precisely, if $x,y\in\cal W(\kappa)$ 
are close to each other and 
$t=|x-\pi(x)|$, then
\begin{enumerate}
 \item   $|\pi(x)-\pi(y)|/|x-y| < 1+\tau_{\delta} + Ct$;
 \item   if the angle between $y-x$ and the fibre $W(\pi(x))$ is 
         equal to $\pi/2$, then 
       $$
            |(y-x)-(\pi(y)-\pi(x))| < (\tau_{\delta} + Ct)|x-y|.
       $$
\end{enumerate}
\end{lem}

\begin{proof}
First we prove (2).

\medskip
\noindent Case 1. \quad $\pi(x)\in f_X(\interior X_{\nu})$. \par
\medskip

Let $N$ be the affine space of codimension $k$ 
parallel to $N_{\pi(x)}$
and through $y$. Let $y_1$ and $y_2$ be the intersections of 
$N_{\pi(y)}$ and 
$N$ with $T_{\pi(x)}$ respectively.
Let $z$ be the point in $K=N\cap N(\pi(y))$ such that 
$|y_2-z|=d(y_2,K)$, and $y_3\in N(\pi(y))$ the point such that 
$|y_2-z|=|y_3-z|$ and
$\angle y_2zy_3 = \angle(y_2-z, N(\pi(y)))\le\angle(N,N(\pi(y)))$.
An argument similar to that in Lemma \ref{lem:tubular} yields that  
\begin{align*}
 &  |y_1-y_3| < \tau_{\delta} |y_1-y_2|,\\
 &  |y_2-y_3|/|z-y_2| \le \angle(N(\pi(x)),N(\pi(y)))
    \le C |\pi(x)-\pi(y)|. 
\end{align*}
It follows that $|y_1-y_2|< Ct |\pi(x)-\pi(y)|$.  Furthermore
the assumption implies $|(\pi(x)-y_2)-(x-y)|<\tau_{\delta} |x-y|$. 
Therefore we get

\begin{align*}
 |(\pi(x)- & y_1) - (x-y)| \\
               & \le |(\pi(x)-y_1)-(\pi(x)-y_2)| +|(\pi(x)-y_2)-(x-y)| \\
               & \le |y_1-y_2| + \tau_{\delta}|x-y| \\
               & <  Ct |\pi(x)-\pi(y)|+\tau_{\delta}|x-y|.
\end{align*}
On the other hand,
since $\angle y_1\pi(x)\pi(y) < \tau_{\delta}$, 
$$
|(\pi(x)-\pi(y))-(\pi(x)-y_1)| < \tau_{\delta}|\pi(x)-\pi(y)|.
$$
Combining the two inequalities, we obtain that 
$$
|(\pi(x)-\pi(y))-(x-y)|<(\tau+Ct)|\pi(x)-\pi(y)|+\tau_{\delta}|x-y|,
$$
from which $(2)$ follows. 

\medskip
\noindent Case 2. \quad  The $\pi$-image of a small 
     neighborhood of $x$ is contained in $f_X(\partial X_{\nu})$. \par
\medskip

Apply the above argument to the affine subspaces 
$\hat H(\pi(x))$ 
and $\hat H(\pi(y))$ in place of $N(\pi(x))$ and $N(\pi(y))$.

\medskip
\noindent Case 3. \quad The other case. \par
\medskip

This case can be reduced to the Case 1.\par

For $(1)$, we assume Case 1. The other case is similar.
Take $y_0\in N(\pi(y))$ such that $|x-y_0|=d(x,N(\pi(y)))$.
Then $(2)$ implies 

\begin{align*}
 \frac{|\pi(x)-\pi(y)|}{|x-y|} & \le \frac{|\pi(x)-\pi(y)|}{|x-y_0|} \\
                             & \le  1+\tau_{\delta} + Ct.
\end{align*}
\end{proof}

\section{Almost Lipschitz submersion } \label{sec:cap-lip}

\medskip
In this section, we shall prove Theorem \ref{thm:cap}. \par

Let $M$ be as in Theorem \ref{thm:cap}. We assume 
$d_{\text{$G$-$GH$}}((M,G),(X,G))<\epsilon$ and 
$\epsilon\ll\sigma_1, \nu$.
Let $\varphi:X\to M$ and $\psi:M\to X$ be 
$2\epsilon$-G-approximations such that
$d(\psi\varphi(x),x)<\epsilon$, $d(\varphi\psi(x),x)<\epsilon$, 
where we can take such a $\varphi$ to be measurable
and $G$-equivariant. Then the map $f_M:M \to L^2(X)$ 
defined by 
$$
   f_M(p)(x) = h(d(p,\varphi(x))), \quad (x\in X)
$$
is $G$-equivariant.
Since $f_M(M)\subset \cal W(\tau(\nu+\epsilon))$, 
the map 
$$
  f=f_X^{-1}\circ\pi\circ f_M :M\to X_{\nu}
$$
is well defined, 
and Lemma \ref{lem:tubular} shows that 
it is $G$-equivariant.  It also follows from Lemma\ref{lem:proj} 
that $f$ is a Lipschitz map.\par

As before, $df_M(\xi)\in L^2(X)$, $\xi\in\Sigma_p$, 
is given by 
\begin{equation}
  df_M(\xi)(x)=-h^{\prime}(d(p,\varphi(x))) 
            \cos \angle(\xi,\varphi(x)_p^{\prime}).
                        \label{eq:deriv2}
\end{equation}

\begin{lem} \label{lem:deriv2}
For every $p,q\in M$ and $\xi\in q_p^{\prime}$,
$$
  \left|\frac{f_M(q)-f_M(p)}{d(q,p)} - df_M(\xi) \right| <   
    \tau(\delta,\sigma,\epsilon/\nu,d(p,q)/\sigma).
$$
\end{lem}

\begin{proof}
For every $x$ with $\sigma/10\le d(\psi(p),x)\le\sigma$, 
take $y\in X$ such that
$\tilde\angle x\psi(p)y >\pi-\tau(\delta,\sigma)$,
$\nu/10\le d(\psi(p),y)\le \nu$.
Since $\tilde\angle \varphi(x)p\varphi(y) > 
\pi-\tau(\delta,\sigma,\epsilon/\nu)$, it follows from
an argument similar to Lemma \ref{lem:diff} that 
\begin{align*}
|d(q,\varphi(x))-d(p,\varphi(x)) + d(q,p) & \cos 
        \angle(\xi, \varphi(x)_p^{\prime})| \\
    & < d(q,p)\tau(\delta,\sigma,\epsilon/\nu,d(p,q)/\sigma),
\end{align*}
which implies the required inequality. 
\end{proof}

We put 

\begin{center}
  $M_{\text{int}} := \text{ the closure of }\,\,
            \{ p\in M\,|\, f(p)\in \interior X_{\nu} \}$, \\
  $M_{\text{cap}} := \{ p\in M\,|\, f(p)\in \partial X_{\nu} \}$.
\end{center}

\begin{lem}
Every $p\in M_{\interior}$ $($resp. $p\in M_{\capp}$$)$ satisfies 
$d(f(p),\psi(p))<\tau(\epsilon)$ $($resp.
$d(f(p),\psi(p))<\tau(\nu)$$)$.
\end{lem}

\begin{proof}
Every $p\in M_{\text{int}}$ (resp. $p\in M_{\text{cap}}$)satisfies 
$d(f_M(p), f_X(X_{\nu}))<\tau(\epsilon)$ (resp.
$d(f_M(p), f_X(X_{\nu}))<\tau(\nu)$),
from which the lemma follows immediately.
\end{proof}

\begin{lem}[Comparison Lemma] \label{lem:compar}
Let $x_i\in M$ and $\bar x_i\in X$, $1\le i\le 3$, be given.
\begin{enumerate}
 \item If $x_1, x_2\in M_{\interior}$, 
      $\nu/10\le d(x_1, x_2)\le\nu$,
      $\sigma/10\le d(x_2, x_3)\le\sigma$ and 
      $d(f(x_i),\bar x_i) < \tau(\epsilon)$, $1\le i\le 3$,
    then 
     $$
       |\angle xyz - \angle \bar x\bar y\bar z|<
        \tau(\delta,\sigma,\epsilon/\nu);
     $$
 \item If $x_1, x_2\in M_{\text{cap}}$,
      $\sigma/10\le d(x_1, x_2), d(x_2, x_3)\le\sigma$,
      and 
      $d(f(x_i),\bar x_i) < \tau(\epsilon)$, $1\le i\le 3$,
   then 
     $$
      |\angle xyz - \angle \bar x\bar y\bar z|<
        \tau(\delta,\sigma,\nu/\sigma).
      $$
\end{enumerate}
\end{lem}

\begin{proof}
We prove (1). The proof of (2) is similar and hence omitted. 
From the assumption, we can take a point $\bar w\in X$ such that 
\begin{equation}
\tilde\angle \bar x_1\bar x_2\bar w > \pi-\tau(\delta,\sigma) 
            \label{eq:pi}
\end{equation}
and $d(\bar x_2,\bar w)\ge\nu/10$. Put $w=\varphi(\bar w)$. 
Then 
\begin{align*}
 \angle x_1x_2x_3 & > \angle \bar x_1\bar x_2\bar x_3  
                    - \tau(\delta,\sigma,\epsilon/\nu),\\
 \angle w x_2x_3 & > \angle \bar w\bar x_2\bar x_3 -
                    - \tau(\delta,\sigma,\epsilon/\nu).
\end{align*}
Since \eqref{eq:pi} implies 
$$
  |\angle x_1x_2w-\pi|<\tau(\delta,\sigma,\epsilon/\nu),
$$
we have the required inequality.
\end{proof}

We now fix $p\in M$, and put
\begin{center}
  $H^{\interior}_p  := \{ \xi\in\Sigma_p \,|\,
       \xi\in x^{\prime}_p, d(p,x)\ge\nu/10 \}, 
         \quad (\text{if\,\, $p\in M_{\interior}$}),$ \\
  $H^{\capp}_p  := \{ \xi\in\Sigma_p \,|\,
       \xi\in\varphi(x)^{\prime}_p, x\in\partial X_{\nu},
        \sigma/10\le d(f(p),x)\le\sigma \}$, \\
      $\hspace{5cm}   \quad (\text{if\,\,$p\in M_{\capp}$})$.
\end{center}
which can be regarded as the sets of $^{\prime\prime} \text{horizontal 
directions} ^{\prime\prime}$ at $p$. 

For $\bar p :=f(p)\in X_{\nu}$, we also put 
\begin{gather*}
  H^{\interior}_{\bar p} := \{ \xi\in\Sigma_{\bar p}\,|\,
           \xi\in x^{\prime}_{\bar p}, 
              d(\bar p,\bar x)\ge\nu/10 \},
        \quad (\text{if\,\, $\bar p\in \partial X_{\nu}$}),\\
  H^{\capp}_{\bar p} := \{ \xi\in\Sigma_{\bar p}\,|\,
           \xi\in x^{\prime}_{\bar p}, 
              \sigma/10\le d(\bar p,\bar x)\le\sigma \},
        \quad (\text{if\,\, $\bar p\in \partial X_{\nu}$}).
\end{gather*}

For any $\bar \xi\in H^{\interior}_{\bar p}$,
take $\bar q\in X$ with $\bar\xi = q_p'$ and 
$d(\bar p, \bar q)\ge\nu/10$.
Put $\xi:=(\varphi(\bar q))_p'\in H^{\interior}_p$, and consider the map
$\chi_{\interior}: H^{\interior}_{\bar p}\to H^{\interior}_p$ 
defined by $\chi_{\interior}(\bar\xi)=\xi$. 
For any $\bar \xi\in H^{\capp}_{\bar p}$,
take $\bar q\in\partial X$ with $\bar\xi = q_p'$ and 
$\sigma/10\le d(\bar p, \bar q)\le\sigma/10$.
Put $\xi:=(\varphi(\bar q))_p'\in H^{\capp}_p$, and consider the map
$\chi_{\capp}: H^{\capp}_{\bar p}\to H^{\capp}_p$ 
defined by $\chi_{\capp}(\bar\xi)=\xi$.

\begin{lem} \label{lem:almtang}
For sufficiently small $t>0$, the following holds:
\begin{enumerate}
 \item For every $\bar\xi\in H^{\interior}_{\bar p}$, 
       put $\xi :=\chi_{\interior}(\bar\xi)$. Then
    $$
      d(f(\exp\,t\xi),\exp\, t\bar\xi)
         <t\tau(\delta,\sigma,\sigma_1/\sigma,\epsilon/\nu);
    $$
 \item For every $\bar\xi\in H^{\capp}_{\bar p}$, 
       put $\xi :=\chi_{\capp}(\bar\xi)$. Then
    $$
      d(f(\exp\,t\xi),\exp\, t\bar\xi)
       <t\tau(\delta,\sigma,\sigma_1/\sigma,\nu/\sigma,\epsilon/\nu).
    $$
\end{enumerate}
Conversely for every $\xi\in H^{\interior}_p$ 
$($resp. $\xi\in H^{\capp}_p$$)$, 
there exists $\bar\xi\in H^{\interior}_{\bar p}$ 
$($resp. $\bar\xi\in H^{\capp}_{\bar p}$$)$
satisfying the above inequality $(1)$ $($resp. $(2)$$)$.
\end{lem}

 In other words, the curve $f(\exp t\xi)$ is almost tangent to 
$\exp t\bar\xi$.

\begin{proof}
We prove (1). The proof of (2) is similar and hence omitted.
By using \eqref{eq:deriv},\eqref{eq:deriv2} and 
Lemma \ref{lem:compar} we get 
   $|df_M(\xi)-df_X(\bar\xi)| < 
        \tau(\delta,\sigma,\epsilon/\nu)$.
Lemmas \ref{lem:diff} and \ref{lem:deriv2} then imply 
$$
\left|\frac{f_M(c(t))-f_M(p)}{t}-\frac{f_X(\bar c(t))-f_X(\bar p)}{t}\right|
   <\tau(\delta,\sigma,\epsilon/\nu),
$$                
for sufficiently small $t>0$. In particular $f_M(c(t))-f_M(p)$ is almost 
perpendicular to $N_{\pi(f_M(p))}$. 
It follows from \ref{lem:proj} that
$$
 \left|\frac{f_M(c(t))-f_M(p)}{t}-\frac{\pi\circ f_M(c(t))-\pi\circ f_M(p)}{t}
    \right| < \tau(\delta,\sigma,\sigma_1/\sigma,\epsilon/\nu)
$$
and hence $|\pi\circ f_M(c(t))-f_X(\bar c(t)) 
< t\tau(\delta,\sigma,\sigma_1/\sigma,\epsilon/\nu)$.
Lemma \ref{lem:inj} then implies the required inequality. 
\end{proof}

 From now on we use the simpler notation 
$\tau_{\delta,\nu,\epsilon}$ to denote a positive function of type 
$\tau(\delta,\sigma,\sigma_1/\sigma,\nu/\sigma,\epsilon/\nu)$.\par
We show that both $f|_{M_{\text{int}}}$ and $f|_{M_{\text{cap}}}$ 
are $\tau_{\delta,\nu,\epsilon}$-Lipschitz submersions.

The following fact follows from Lemma \ref{lem:almtang}.\par
\begin{equation}
 \left| \frac{d(f(\exp t\xi),\bar p)}{t} - 1 \right| <\tau_{\delta,\nu,\epsilon}, 
\end{equation}
for all $\xi\in H^{\interior}_p\cup H^{\capp}_p$ and 
sufficiently small $t>0$. 

\begin{lem} \label{lem:lip}
For every $p,q\in M_{\interior}$ $($resp.
$p,q\in M_{\capp}$$)$, we have 
$$
   \left|\frac{d(f(p),f(q))}{d(p,q)} - \cos \theta \right|<\tau_{\delta,\nu,\epsilon},
$$
where $\theta=\angle(q_p^{\prime}, H^{\interior}_p)$
$($resp. $\theta=\angle(q_p^{\prime}, H^{\capp}_p)$$)$.
\end{lem}

 For the proof of Lemma \ref{lem:lip} we need two sublemmas.\par

Lemma \ref{lem:compar} and the second half of Lemma \ref{lem:almtang}
imply the following

\begin{slem} \label{slem:HbarH}
$\chi_{\interior}$ $($resp. $\chi_{\capp}$$)$ gives 
a $\tau(\delta,\sigma,\epsilon/\nu)$-approximation
$($resp. $\tau(\delta,\sigma,\nu/\sigma)$-approximation$)$.
\par
In particular, 
$d_{GH}(H^{\interior}_p, S^{k-1}(1))<\tau_{\delta,\nu,\epsilon}$
and 
$d_{GH}(H^{\capp}_p, S^{k-2}(1))<\tau_{\delta,\nu,\epsilon}$.
\end{slem}

Let $H_p$ denote either $H^{\interior}_p$  or $H^{\capp}_p$.

\begin{slem}  \label{slem:fproj}
For any $\xi\in\Sigma_p^{\prime}$, 
put $\theta :=\angle(\xi, H_p)$ and let $\xi_1\in H_p$ satisfy 
$\theta = \angle(\xi, \xi_1)$. Then 
$$
 d(f(\exp t\xi), f(\exp t \cos\theta\xi_1))< t\tau_{\delta,\nu,\epsilon},
$$ 
for every sufficiently small $t>0$.
\end{slem}

\begin{proof}
Since $\Sigma_p$ has curvature $\ge 1$, we have an expanding map
$\rho:\Sigma_p\to S^{n-1}(1)$,$(n=\operatorname{dim}M)$. 
First we show that 
$|d(\rho(v_1),\rho(v_2))-\angle(v_1,v_2)|<\tau_{\delta,\nu,\epsilon}$
for every $v_1,v_2\in H_p$. Let $v_1^*\in H_p$ be such that
$\angle(v_1,v_1^*)>\pi-\tau_{\delta,\nu,\epsilon}$. 
Since $\rho$ is expanding, we obtain that
\begin{equation}
    |\angle(v_1,v_2)-d(\rho(v_1),\rho(v_2))|<\tau_{\delta,\nu,\epsilon},\quad 
         |\angle(v_1^*, v_2)- d(\rho(v_1^*), \rho(v_2))|<\tau_{\delta,\nu,\epsilon}.
           \label{eq:v_1v_2}
\end{equation}
Now we assume $H_p=H^{\interior}_p$. 
The case $H_p=H^{\capp}_p$ is similar and hence omitted.
The above argument also implies that $\rho(H_p)$ is 
$\tau_{\delta,\nu,\epsilon}$-Gromov-Hausdorff close to 
a totally geodesic $(k-1)$-sphere $S^{k-1}(1)$ in $S^{n-1}(1)$. 
Let $\zeta:H_p\to S^{k-1}(1)\subset  S^{n-1}(1)$ be a 
$\tau_{\delta,\nu,\epsilon}$-approximation 
such that $d(\zeta(v),\rho(v))<\tau_{\delta,\nu,\epsilon}$ for all 
$v\in H_p$.
For a given $\xi\in\Sigma_p$, an argument similar to 
\eqref{eq:v_1v_2} implies that
$|\angle(\xi, v)-d(\rho(\xi),\zeta(v))|<\tau$ for all $v\in H_p$.
Remark that for any $y$ with $\sigma/10\le d(p,y)\le \sigma$, 
an elementary geometry yields 
$$
   \cos d(\rho(\xi),\zeta(y^{\prime}_p)) 
  = \cos d(\rho(\xi),\eta)\cos d(\eta,\zeta(y^{\prime}_p)),
$$
where $\eta$ is an element of 
$S^{k-1}(1)$ such that 
$\angle(\rho(\xi),\eta)=\angle(\rho(\xi),S^{k-1}(1))$.
It follows that for sufficiently small $t>0$
\begin{align*}
 |f_M(\exp & t\xi)-f_M(\exp t\cos\theta \xi_1)|^2/t^2 \\
 &  = \frac{\sigma^2}{b(\sigma)} \int_X
 \left(\frac{h(d(\exp t\xi,\varphi(x)))-h(d(p,\varphi(x)))}{t} \right. \\
 &\qquad\qquad\qquad
    \left. - \frac{h(d(\exp t\cos\theta \xi_1,\varphi(x))) - 
       h(d(p,\varphi(x)))}{t} \right)^2 \, d\mu(x) \\
 & \le \frac{\sigma^2}{b(\sigma)} \int_X
   (h^{\prime}(d(p,\varphi(x))))^2(\cos\angle(\xi,\varphi(x)_p^{\prime}) - \\
 & \qquad\qquad\qquad
   \cos \theta \cos \angle(\xi_1,\varphi(x)_p^{\prime}))^2 \, 
             d\mu(x) + \tau_{\delta,\nu,\epsilon} \\
 & \le \frac{\sigma^2}{b(\sigma)} \int_X
  (h^{\prime})^2(\cos \angle(\xi,\varphi(x)_p^{\prime}) -
       \cos \angle(\rho(\xi),\rho(\varphi(x)_p^{\prime})) \\ 
 &\qquad\qquad\qquad
   + \cos \angle(\rho(\xi),\eta) 
          \cos \angle(\eta,\rho(\varphi(x)_p^{\prime})) - \\
 & \qquad\qquad\qquad
    \cos \angle(\xi,\xi_1) \cos \angle(\xi_1,\varphi(x)_p^{\prime}))^2
        \, d\mu(x) + \tau_{\delta,\nu,\epsilon} \\
 & \le \tau_{\delta,\nu,\epsilon}.
\end{align*}
Therefore by Lemmas \ref{lem:proj} and \ref{lem:inj}
we conclude the proof of the sublemma.
\end{proof}

\begin{proof}[Proof of Lemma \ref{lem:lip}]
Suppose $p, q\in M_{\interior}$.
Since $f_{\interior}$ is a $\tau(\epsilon)$-approximation,
we may assume that
$d(p,q)<\nu^2\ll\nu$. 
Let $c:[0,\ell]\to M$ be a minimal geodesic 
joining $p$ to $q$ where $\ell =d(p,q)$. 
By using Lemma \ref{lem:basic}, one can show 
that  
\begin{equation*}
    |\angle qc(t)x-\angle qpx|<\tau_{\delta,\nu,\epsilon}, \label{eq:qcx}
\end{equation*}
for every $t<\ell$ and for every $x\in M$ 
with $\nu/10\le d(p,x)\le\nu$.
Let $\xi_t := q^{\prime}_{c(t)}$,  
and $\eta_0\in H_{p}^{\interior}$ be such that 
$\angle(\xi_0, H_{c(t)}^{\interior})=\angle(\xi_0, \eta_0)$. 
Take a point $y$ such 
that $\eta_0=y_p^{\prime}$, $\nu/10\le |py|\le\nu$.
\eqref{eq:qcx} implies that 
  $|\angle(y^{\prime}_{c(t)}, \xi_t) - 
      \angle(\xi_t, H^{\interior}_{c(t)})| < \tau_{\delta,\nu,\epsilon}$.
Put $\theta_t=\angle qc(t)y$. It follows 
from Sublemma \ref{slem:fproj} and \eqref{eq:qcx} that
\begin{equation}
    d(f\circ c(t+s), f(\exp s\cos\theta_0\eta_t))<\tau_{\delta,\nu,\epsilon} s,
        \label{eq:almtang}
\end{equation}
where $\eta_t:= y_{c(t)}'$.
Put $\bar c(t)=f\circ c(t)$, and take any $\bar\eta_t$ in 
$\psi(y)^{\prime}_{\bar c(t)}$. 
Then by Lemma \ref{lem:almtang}
\begin{equation}
 d(f(\exp s\cos\theta_0\eta_t),\exp s\cos\theta_0\bar\eta_t)
        < \tau_{\delta,\nu,\epsilon} s.  
                              \label{eq:almtang2}
\end{equation}
By Lemma \ref{lem:basic}, 
we see that for every $z\in  X$ with $\nu/10\le d(\bar p, z)\le
\nu$, 
\begin{equation}
   |\angle \psi(y)\bar c(t)z -\angle \psi(y)\bar p z|<
    \tau_{\delta,\nu,\epsilon}.  \label{eq:angle}
\end{equation}  
Now let $(a_i,b_i)$ be a $(k,\delta)$-strainer at $\bar p$ with
$d(\bar p, a_i)=\nu$ and $ \lambda:B_{\bar p}(\nu^2)\to \R^k$
be the bi-Lipschitz map defined by  
$$
    \lambda(x)=(d(a_1,x),\dots, d(a_k,x)).
$$
Put $u(t)=\lambda\circ\bar c(t)$.  
Combining \eqref{eq:almtang}, \eqref{eq:almtang2} and \eqref{eq:angle}, 
we get
$$
  |\dot u(s) - \dot u(t)|<\tau_{\delta,\nu,\epsilon},\qquad   
   ||\dot u(s)| - \cos \theta_0|<\tau_{\delta,\nu,\epsilon}.
$$
for almost all $s,t \in [0,\ell]$. Thus we arrive at 
\begin{align*}
  |\ell\dot u(s) - (\lambda&(f(q))-\lambda(f(p)))|  \\
         &\le \int_0^{\ell} |\dot u(s) - \dot u(t)|\, dt 
                  \le \tau_{\delta,\nu,\epsilon}\ell,
\end{align*}
from which the conclusion follows.

Next suppose $p, q\in M_{\capp}$. We may assume 
$d(p,q)<\sigma^2\ll\sigma$, 
since $f_{\capp}$ is a $\tau(\nu)$-approximation,

By using Lemma \ref{lem:basic2}, we have 
\begin{equation*}
    |\angle qc(t)\varphi(x)-\angle qp\varphi(x)|
       < \tau_{\delta,\nu,\epsilon}, \label{eq:barqcx}
\end{equation*}
for every $t<\ell$ and for every $x\in\partial X_{\nu}$ 
with $\sigma/10\le d(p,x)\le\sigma$.
Let $\xi_t := q^{\prime}_{c(t)}$,  
and $\eta_0\in H_{p}^{\capp}$ be such that 
$\angle(\xi_0, H_{p}^{\capp})=\angle(\xi_0, \eta_0)$. 
Take a point $y\in\partial X_{\nu}$ such 
that $\eta_0=\varphi(y)_p^{\prime}$, 
$\sigma/10\le d(\bar p,y)\le\sigma$.
By an argument similar to the previous one using 
a $(k,\delta)$-strainer $(a_i,b_i)$ of $D(X)$ at $\bar p$ with
$d(\bar p, a_i)=\sigma$,
we obtain the required estimate.
\end{proof}

\begin{lem}
Let $p\in M$ and $x\in X_{\nu}$ be given.
\begin{enumerate}
 \item  If $p\in M_{\interior}$, then there exists a point 
        $q\in M_{\interior}$ such that
        $f(q)=x$ and $d(f(p),f(q))\ge (1-\tau_{\delta,\nu,\epsilon})d(p,q).$ 
 \item  If $p\in M_{\capp}$ and $x\in\partial X_{\nu}$, 
        then there exists a point 
        $q\in M_{\capp}$ such that
        $f(q)=x$ and $d(f(p),f(q))\ge (1-\tau_{\delta,\nu,\epsilon})d(p,q)$ 
\end{enumerate}
\end{lem}

 Namely $f_{\interior}$ and $f_{\capp}$ are 
$(1-\tau_{\delta,\nu,\epsilon})$-open in the sense of [BGP].

In view of Lemma \ref{lem:basic3}, the proof of the lemma above 
is similar to \cite{Ym:conv} and hence omitted.

We are in a position to complete the proof of 
Theorem \ref{thm:cap}.
So far we do not need the assumption that 
$S_{\delta_n}(M)$ is empty.
Now we use this assumption to prove that 
both $f_{\interior}$ and $f_{\capp}$ are locally trivial 
fibre bundle maps.

For any $p\in M_{\interior}$, set $F:=f_{\interior}^{-1}(\bar p)$,
and take an $(n,\delta)$-strainer $(a_i,b_i)$ at $p$ 
such that $(a_i)_p', (b_i)_p'\in H_p^{\interior}$ for 
all $1\le i\le k$. Note that $(a_i)_p'$  and $(b_i)_p'$ are
almost tangent to $F$ for $k+1\le i\le n$.
This implies that  the map 
$\Phi = (f, d_{a_{k+1}},\ldots, d_{a_n})$ is a bi-Lipschitz 
homeomorphism of a small neighborhood of $p$ onto an 
open subset of $X\times\R^{n-k}$.
It follows that $f_{\interior}$ is a topological submersion
and hence is a locally trivial bundle map by 
\cite{Si:stratified}.
The proof for $f_{\capp}$ is similar and hence omitted.

\begin{proof}[Proof of Corollary \ref{cor:cap}]
Since $f_{\rm cap}:N_{\rm cap}\to\partial_0Y_{\nu}$ is a 
locally trivial fibre bundle,
the conclusion (1) follows from the generalized Margulis lemma
of \cite{FY:fundgp}.
Note that  $f_{\rm cap}$ is also an almost Lipschitz submersion.
Therefore by the parametrized versions of 
Theorem 0.5 in \cite{SY:3mfd} and Theorem \ref{thm:dim1} in 
the present paper, we can get the conclusions (2) and (3)
respectively. Note that in the present case, we can prove
the rescaling theorem corresponding to Theorem \ref{thm:rescal}
by using the (generalized) argument of Lemma \ref{lem:split-rescal}.
The actual proofs of (2) and (3) are done by contradiction, 
and the details are omitted.
\end{proof}


\end{document}